\documentclass[1 leqno,11pt]{amsart}
\usepackage{amssymb, amsmath,latexsym,amsfonts,amsbsy, amsthm,mathtools,graphicx,,color}
\usepackage{float}
\usepackage{tikz}
\usetikzlibrary{arrows, decorations.markings,matrix,trees}
\usepackage{hyperref}

\numberwithin{equation}{section}

\allowdisplaybreaks

\let\al=\alpha

\let\d=\delta
\let\e=\varepsilon

\let\f=\frac

\let\na=\nabla

\let\pa=\partial

\def\cA{{\cal A}}
\def\cB{{\cal B}}

\def\bM{\bar M}
\def\tF{\tilde F}

\def\Im{\mathrm{Im}}
\def\Re{\mathrm{Re}}

\def\cA{\mathcal{A}}
\def\cB{\mathcal{B}}

\def\dv{\mbox{div}}

\def\eqdef{\buildrel\hbox{\footnotesize def}\over =}

\newcommand{\beq}{\begin{equation}}
\newcommand{\eeq}{\end{equation}}
\newcommand{\ben}{\begin{eqnarray}}
\newcommand{\een}{\end{eqnarray}}
\newcommand{\beno}{\begin{eqnarray*}}
\newcommand{\eeno}{\end{eqnarray*}}

\newtheorem{theorem}{Theorem}[section]

\newtheorem{lemma}[theorem]{Lemma}
\newtheorem{proposition}[theorem]{Proposition}

\newtheorem{remark}[theorem]{Remark}

\begin{document}

\title{Mack modes in supersonic boundary layer} 

\author[N. Masmoudi]{Nader Masmoudi}
\address{NYUAD Research Institute, New York University Abu Dhabi, Saadiyat Island, Abu Dhabi, P.O. Box 129188, United Arab Emirates, Courant Institute of Mathematical Sciences, New York University, 251 Mercer Street New York, NY 10012 USA}
\email{masmoudi@cims.nyu.edu}

\author[Y. Wang]{Yuxi Wang}
\address{School of Mathematics, Sichuan University, Chengdu 610064, P. R. China, NYUAD Research Institute, New York University Abu Dhabi, Saadiyat Island, Abu Dhabi, P.O. Box 129188, United Arab Emirates}
\email{wangyuxi@scu.edu.cn}

\author[D. Wu]{Di Wu}
\address{School of Mathematics, South China University of Technology, Guangzhou, 510640,  P. R. China}
\email{wudi@scut.edu.cn}

\author[Z. Zhang]{Zhifei Zhang}
\address{School of Mathematical Sciences, Peking University, 100871, Beijing, P. R. China}
\email{zfzhang@math.pku.edu.cn}

\maketitle

\begin{abstract}
Understanding the transition mechanism of boundary layer flows is of great significance in physics and engineering, especially due to the current development of supersonic and hypersonic aircraft.  In this paper, we  construct multiple unstable acoustic modes so-called {\bf Mack modes}, which play a crucial role during the early stage of transition in the supersonic boundary layer. To this end, we develop an inner-outer gluing iteration to solve a hyperbolic-elliptic mixed type and singular system.  
\end{abstract}

\setcounter{tocdepth}{1}


\section{Introduction}

\subsection{Physical background}
The investigation of  the transition mechanism of boundary layer flows is of great significance in physics and engineering, especially due to the current development of supersonic and hypersonic aircraft.  The early stage of boundary layer transition often involves the instability caused by viscous disturbance waves, commonly referred to as Tollmien–Schlichting (T–S) waves. The origins of T-S waves can be traced back to influential contributions by Heisenberg, Tollmien, Schlichting, C. C. Lin, etc.  \cite{Drazin-Reid, Lin, Sch}, who found lower and upper marginal stability branches. See a curve of neutral stability of type (b) in FIGURE 1. In a breakthrough work  \cite{GGN-DMJ}, Grenier, Guo and Nyugen first gave a rigorous mathematical construction of  T-S waves of temporal mode by developing the Rayleigh-Airy iteration scheme( see Remark \ref{rem: comparison iterations} for more explanation). 
Chen and the last two authors \cite{CWZ} confirm the existence of neutral curve for T-S wave mathematically, and readers also refer to  \cite{BG-2023}. Let us refer to \cite{BG-2024,CWZ1, GM, GMM-duke, GMM-arxiv,  GN} and references therein for more works on nonlinear stability/instability of boundary layer flows.  
\begin{figure}[H]
	\centering
	\includegraphics[height=5.0cm,width=8cm]{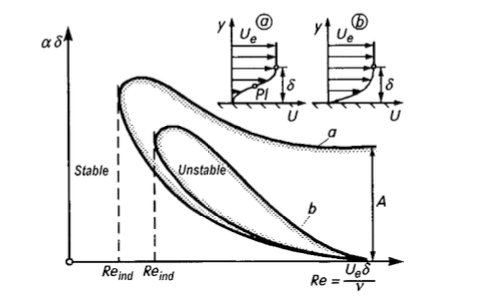}
	\caption{Curves of neutral stability of a plane boundary layer for 2D incompressible perturbations. (a)“inviscid” instability; (b)“viscous” instability. This picture is from Fig. 15.9. in Chapter 15 in \cite{Sch}. }\label{figure1}
\end{figure}
For the subsonic flow, Yang and Zhang \cite{YZ} gave the first rigorous construction of the compressible T-S waves of temporal mode, by developing a quasi-compressible and Stokes iteration. Recently, for the whole subsonic regime, we \cite{MWWZ} constructed unstable T-S waves by introducing a new quasi-incompressible-compressible iteration scheme. It should be noted that  there is only a single unstable solution for a given phase velocity. This type of unstable mode is referred to as a single mode.

In the supersonic case, the situation becomes much more complicated, and instability phenomenon is different from the subsonic case. It is documented in the physical literature by physicist Mark (see \cite{Sch}) that

{\it  Using a numerical computation shows that there is an infinite number of neutral wave numbers or modes with the same phase velocity if supersonic regime arises in the boundary layer. These perturbations can be viscous or inviscid and do not depend on whether the boundary layer has a generalized point of inflection.}

\noindent Readers refer to  \cite{Mack 1984} for an extensive summary by Mack. The multiple unstable modes are also called  {\bf Mack}(acoustic) {\bf modes}.
According to physical observation, the multiple unstable acoustic modes of supersonic boundary layer correspond to the inviscid instability and driven by the effect of {\it compressibility} instead of the profile of basic flow. This is completely different from  T-S waves in the subsonic case, where the instability mechanisms are caused by the {\it viscous sublayer}.  Smith \cite{Smith 1989}  found that T-S waves do not exist in the 2D supersonic flow. In the 3D supersonic flow, both T-S waves and acoustic modes are present, with acoustic modes exhibiting the strongest instability when $M_a\geq 2.2$ \cite{Mack 1984}. See the appearance of acoustic modes in FIGURE 2.
 Mack \cite{Mack 1984} pointed out that this instability is primarily generated by the inviscid part,  with viscosity playing a stabilizing role. We refer to \cite{LL, Lin 2003, Ray} for the inviscid instability results in the incompressible and subsonic cases.  
\begin{figure}[H]
	\centering
	\includegraphics[height=5.0cm,width=12cm]{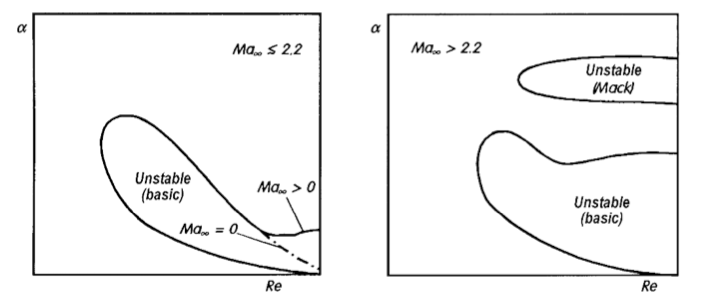}
	\caption{Appearance of acoustic modes. This picture is from Fig. 15.33. in Chapter 15 in \cite{Sch}. }\label{figure1}
\end{figure}

\subsection{Problem and Main results}

As mentioned in the previous subsection, the inviscid instability predominates in the supersonic case. Thus,  we consider the 2D non-isentropic Euler equations in a half-plane $(X, Y)\in \mathbb{R}\times \mathbb{R}_{+}$:
\begin{align}\label{eq: CEuler}
\left\{
\begin{aligned}
&\pa_t \rho+\dv(\rho\mathbf{ u})=0,\\
&\rho(\pa_t\mathbf{ u}+\mathbf{ u}\cdot \na \mathbf{ u})+\gamma^{-1} M_a^{-2}\na p=0,\\
&\rho(\pa_t \mathcal T+\mathbf{ u}\cdot \na \mathcal T)=-(\gamma-1) p~\dv \mathbf{ u},\\
&p=\rho \mathcal T,\\
&v(X, 0)=0.
\end{aligned}
\right.
\end{align}
Here $\mathbf{ u}=(u,v)$ is the velocity, $\rho$ is the density, $\mathcal T$ is the temperature,  and $p$ is the pressure respectively.  $\gamma>1$ is the ratio of the specific heats and  $M_a>0$ is the Mach number. In this paper, we take $\gamma=1.4$ in accordance with the specifications from both physical papers and numerical papers, as indicated in \cite{Mack 1984, Smith 1989}. 

Inspired by Orr-Sommerfeld approach, we explore the instability of supersonic flow by considering the linearized non-isentropic Euler equation around the basic flow $(\rho_0(Y), u_0(Y), 0, T_0(Y))$, which is given by
\begin{align}\label{eq: LCEuler}
\left\{
\begin{aligned}
&\pa_t \rho+\pa_X(u_0\rho+u\rho_0)+\pa_Y(v\rho_0)=0,\\
&\rho_0(\pa_tu+u_0\pa_X u+v\pa_Y u_0)+\gamma^{-1} M_a^{-2}\pa_X p=0,\\
&\rho_0(\pa_tv+u_0\pa_X v)+\gamma^{-1} M_a^{-2}\pa_Y p=0,\\
&\rho_0(\pa_t \mathcal T+u_0\pa_X \mathcal T+v\pa_Y T_0)=\gamma^{-1}(\gamma-1)(\pa_t p+u_0\pa_X p),\\
&p=\rho T_0+\mathcal T\rho_0,\\
&v(0)=\rho(\infty)=u(\infty)=v(\infty)=\mathcal T(\infty)=0.
\end{aligned}
\right.
\end{align}
Here the basic flow $(\rho_0(Y), u_0(Y), 0, T_0(Y))$, a steady  solution to \eqref{eq: CEuler},  satisfies 
\begin{align}
\rho_0(Y)  T_0(Y)=1, \quad  T_0(Y)=1+\f12(\gamma-1)M_a^2(1-u_0^2(Y)),
\end{align}
where the temperature $T_0(Y)$ is considered with a thermally insulating wall. We introduce the following structure assumptions on $U_B(Y)=u_0(Y)$:
\begin{align}\label{assume: U_B}
\begin{split}
&U_B(0)=0,\quad \lim_{Y\to +\infty}U_B(Y)=1,\\
&0<U_B(Y)<1,\quad\mbox{for}\quad Y\in(0,+\infty),\quad \pa_YU_B(Y)>0,\quad \mbox{for}\quad Y\in[0,+\infty),\\
&\sup_{Y\geq 0}|e^{\eta_0 Y}\pa_Y^k(U_B(Y)-1)|\leq C,\quad k=0,1,2,3.
\end{split}
\end{align}
Especially, the famous Blasius flow satisfies \eqref{assume: U_B}.

We construct  the solution of \eqref{eq: LCEuler} with the form
\begin{align*}
(\rho, u,v, \mathcal T, p)=(\varrho, U,V, T, P)(Y)e^{i\al (X-c t)}+c.c.,
\end{align*}
where $c\in\mathbb{C}$ and $\al>0$ and $(\varrho, U,V, T, P)(Y)$ satisfies the system
\begin{align}\label{eq: LCEuler-1}
\left\{
\begin{aligned}
&i\al(U_B-c)\varrho+i\al \rho_0 U+\pa_Y(V \rho_0)=0,\\
&i\al(U_B-c)(U\rho_0)+\rho_0V\pa_Y U_B+i\al \gamma^{-1} M_a^{-2}P=0,\\
&i\al(U_B-c)(V\rho_0)+\gamma^{-1} M_a^{-2}\pa_Y P=0,\\
&i\al(U_B-c)(\rho_0 T)+\rho_0 V\pa_Y T_0=\gamma^{-1}(\gamma-1)  i\al(U_B-c)P,\\
&P=\varrho T_0+T\rho_0,
\end{aligned}
\right.
\end{align}
with boundary condition
\begin{align}\label{BC: (U,V, rho,T)}
V(0)=\varrho(\infty)=U(\infty)=V(\infty)=T(\infty)= 0.
\end{align}
If the boundary value problem \eqref{eq: LCEuler-1}-\eqref{BC: (U,V, rho,T)} has a non-trivial solution for some $ c=c_r+i c_i$ with $c_i>0$ and wave number $\al>0$, then the basic flow is spectrally unstable, and the mode $(\al,c)$ is called an unstable mode. The goal of this work is to construct unstable modes to the system \eqref{eq: LCEuler-1} -\eqref{BC: (U,V, rho,T)} for $c_r\in(1-M_a^{-1},1)$ and $0<c_i\ll c_r$.

We introduce the relative Mach number $\bar{M}(Y)$ defined by
\begin{align}\label{def:r-M}
\bar{M}(Y)=\f{M_a(U_B(Y)-c)}{T_0^\f12(Y)}=\bar{M}_r(Y)+\mathrm{i}\bar{M}_i(Y).
\end{align}
Since the Mach number $M_a>1$ and $0<c_i\ll c_r$, we notice that $|\bar{M}_r|>1$ near the boundary and $|\bar{M}_r|<1$ near infinity, which indicate that the system \eqref{eq: LCEuler-1} is hyperbolic near the boundary and elliptic near infinity, see Section \ref{sec:equation of pressure }.

To state our main results, we introduce
a set $\mathcal{H}(U_B, M_a)$, which is a restriction on the background flow $U_B$ and Mach number $M_a>1$:
\begin{align}\label{set: H}
\mathcal{H}(U_B, M_a)=\left\{c_r\in[T_0^\f12(0) M_a^{-1}, 1+M_a^{-1}]: J(c_r; M_a)\neq 0\right\},
\end{align}
where the integral $J(c_r; M_a)$ is given by
\begin{align*}
J(c_r; M_a)=\int_0^{Y_0(c_r)}(\bar{M}_r^2-1)^{-\f12}\tilde Q_1(Z; c_r, M_a)dZ.
\end{align*}
Here $Y_0$ is the unique turning point satisfying $\bar{M}_r(Y)=1$ and 
\begin{align*}
\tilde Q_1(Y; c_r, M_a)=\tilde Q_1(Y; Y_0(c_r), c_r, M_a)=\f{\pa_Y^2 \bM_r}{\bM_r}-\f{2(\pa_Y\bM_r)^2}{\bM_r^2}-\f34\f{(\pa_Y^2\eta)^2}{(\pa_Y\eta)^2}-\f{\pa_Y^3 \eta}{2\pa_Y\eta},
\end{align*}
where $\eta(Y)$ satisfies $	\partial_Y (\bar{M}_r^2-1)(Y_0)\eta(Y)(\partial_Y\eta(Y))^2=(\bar{M}_r^2(Y)-1)$ with $\eta(Y_0)=0$.
\medskip

Our main results  are stated as follows.

\begin{theorem}[General monotone flow]\label{thm: main-1}
	Assume that $U_B(Y)$ satisfies the structure assumption \eqref{assume: U_B} and  $\mathcal{H}(U_B, M_a)\cap(1-M_a^{-1},1)\neq \emptyset$ for some Mach number $M_a>1$. Then for any $c_r\in\mathcal{H}(U_B, M_a)\cap(1-M_a^{-1},1)$, there exists  a sequence  $\{(\al_k, c_{i,k})\}_{k=1}^{\infty}$  with $c_{i,k}=c_i(\al_k)$  such that the linearized Euler system \eqref{eq: LCEuler} has the solutions $(\rho_k, u_k,v_k, \mathcal T_k, p_k)(t,X,Y)$ with the form 
	\begin{align*}
	(\rho_k, u_k,v_k, \mathcal T_k, p_k)(t,X,Y)=(\varrho_k, U_k,V_k, T_k, P_k)(Y)e^{\mathrm i\alpha_k X}e^{-\mathrm i\alpha_k c_k t},
	\end{align*}
	where $(\varrho_k, U_k,V_k, T_k, P_k)(Y)\in L^\infty(\mathbb R_+)\times W^{1,\infty}(\mathbb R_+)^2\times L^\infty(\mathbb R_+)\times W^{2,\infty}(\mathbb R_+)$ are solutions to \eqref{eq: LCEuler-1} for the pair $(\alpha_k, c_k)$ with $c_k=c_r+\mathrm ic_{i,k}$. Moreover, there exists $\al_0\gg1$ such that for any $k\in \mathbb{N}_+$
	\begin{align*}
	&\al_k\geq \al_0,\quad \f32 \pi<C_1(c_r)(\al_{k+1}-\al_k)<\f52 \pi,\\
	&0<c_{i,k}\sim\al_k^{-3}e^{-C(c_r)\al_k },
	\end{align*}
	where $C_1(c_r)$ and $C(c_r)$ are  positive constants dependent on $c_r$, but independent of $\alpha$. 
	\end{theorem}

\begin{remark}

Based on our construction, we know that the condition $\mathcal{H}(U_B, M_a)\cap(1-M_a^{-1},1)= \emptyset$ is related to the neutral modes.
We point out that the set of $\mathcal{H}(U_B, M_a)\cap(1-M_a^{-1},1)$ is not an empty set. Indeed, the condition $\mathcal{H}(U_B, M_a)\cap(1-M_a^{-1},1)\neq \emptyset$ holds when $U_B(Y)$ is an analytic function. The precise statement is given in Proposition \ref{rem: Non-empty condition}. 
\end{remark}

\begin{remark}
When the wave number $\al$ lies in the middle regime, the instability is stronger. For this case, the framework for constructing the unstable modes will be completely different, and we will study this subject in our forthcoming paper.
\end{remark}

 We can derive the following results from Theorem \ref{thm: main-1} for some analytic basic flows.
 
\begin{theorem}[Analytic flow]\label{thm: main-2}
	Assume that $U_B(Y)$ is an analytic function and satisfies the structure assumption \eqref{assume: U_B}. Let $M_a\geq M_0$, where $M_0>1$  is given in Proposition \ref{rem: Non-empty condition}. Then there exists $C_1(M_a),\cdots, C_N(M_a)$ with $1-M_a^{-1}<C_1<\cdots<C_N<1$, 
	such that for any $c_r\in (C_1, C_2)\cup (C_3, C_4)\cup \cdots (C_{N-1}, C_N)$, the results in Theorem \ref{thm: main-1} hold.

\end{theorem}

For the Blasius flow, we can weaken the lower bound on the Mach number $M_a$.

\begin{theorem}[Blasius flow]\label{thm: main-3}
	Let $U_B(Y)$ be the Blasius profile and $M_a\geq 3$,  then there exists $C_1(M_a),\cdots, C_N(M_a)$ with $1-M_a^{-1}<C_1<\cdots<C_N<1$, 
	such that for any $c_r\in (C_1, C_2)\cup (C_3, C_4)\cup \cdots (C_{N-1}, C_N)$, the results in Theorem \ref{thm: main-1} hold. 
\end{theorem}
The lower bound for $M_a\geq 3$ in Theorem \ref{thm: main-3} is not optimal. Our result to the Blasius flow is also valid for $M_a\geq 2.2$ as the physical literature \cite{Mack 1984,Sch} said. For $M_a\in(1, 2.2)$, the most unstable mode is the T-S waves (for 3D) and vorticity modes (for 2D), see  \cite{Mack 1984,Sch}.

\subsection{Notations} Throughout this paper, we denote by $C$ a constant independent of $\al$, which may be different from line to line. $A\lesssim B$  means that there exists a  constant C independent of $\al$ such that $A\leq CB$. Similar definitions hold for $A\gtrsim B$.  Moreover, we use the notation $A\sim B$ if $A\lesssim B$ and $A\gtrsim B$. The weighted Sobolev space $W^{k,\infty}_{w_0}$($k\in \mathbb{N}$ ) has the norm $\|f\|_{W^{k,\infty}_{w_0}}=\sum_{\ell=0}^{k}\|\pa_Y^\ell f\|_{L^\infty_{w_0}(\mathbb{R}_{+})}$, where the definition of norm $\| f\|_{L^\infty_{w_0}(\mathbb{R}_{+})}$ is given by \eqref{def: w_0}.

\section{Main ideas and sketch of proof}\label{sec: sketch of proof}

\subsection{The equation of the pressure}\label{sec:equation of pressure }

To solve the complicated eigenvalue system \eqref{eq: LCEuler-1}, we first drive an equation for the pressure, which reflects the main characteristic of the compressible effect.  

\subsubsection{Refomulation via the pressure}
From the equations of velocity in  \eqref{eq: LCEuler-1}, we can write 
\begin{align*}
&V\rho_0=-\f{\partial_Y P}{\mathrm{i}\al(U_B-c)\gamma M_a^{2}},\quad\partial_Y(V\rho_0)=-\f{\partial^2_Y P}{\mathrm{i}\al(U_B-c)\gamma M_a^2}+\f{\partial_YU_B \partial_Y P}{\mathrm{i}\al(U_B-c)^2\gamma M_a^2},\\
&\mathrm{i}\al U\rho_0=\f{\partial_YU_B\partial_YP}{\mathrm{i}\al(U_B-c)^2\gamma M_a^2}-\f{\mathrm{i}\al P}{(U_B-c)\gamma M_a^2},
\end{align*}
which along with the continuity equation implies 
\begin{align}\label{eq:for-P1}
\mathrm{i}\al(U_B-c)\varrho+\f{2\partial_YU_B\partial_YP}{\mathrm{i}\al(U_B-c)^2\gamma M_a^2}-\f{\mathrm{i}\al P}{(U_B-c)\gamma M_a^2}-\f{\partial^2_Y P}{\mathrm{i}\al(U_B-c)\gamma M_a^2}=0.
\end{align}
On the other hand, from the last two equations in  \eqref{eq: LCEuler-1}, we obtain 
\begin{align*}
\mathrm{i}\al(U_B-c)\varrho=-\f{\mathrm{i}\al(U_B-c)(T\rho_0)}{T_0}+\f{\mathrm{i}\al(U_B-c)P}{T_0},
\end{align*}
and
\begin{align*}
\mathrm{i}\al(U_B-c)(T\rho_0)=\gamma^{-1}(\gamma-1)\mathrm{i}\al(U_B-c)P+\f{\partial_YT_0\partial_Y P}{\mathrm{i}\al(U_B-c)\gamma M_a^{2}},
\end{align*}
which implies
\begin{align}\label{eq:for-P2}
\mathrm{i}\al(U_B-c)\varrho=-\f{\partial_YT_0\partial_Y P}{\mathrm{i}\al(U_B-c)\gamma M_a^{2}T_0}+\f{\mathrm{i}\al(U_B-c)P}{\gamma T_0}.
\end{align}
Therefore, by \eqref{eq:for-P1} and \eqref{eq:for-P2}, we obtain 
\begin{align*}
\partial_Y^2 P-\big(\f{2\partial_YU_B}{U_B-c}-\f{\partial_Y T_0}{T_0}\big)\partial_YP-\al^2\big(1-\f{M_a^2(U_B-c)^2}{T_0} \big)P =0,
\end{align*}
which along with $\bar{M}=M_a(U_B-c)/T_0^\f12$, $F(Y)=1-\bar{M}^2$ and the fact that
\begin{align*}
\f{\partial_Y\bar{M}}{\bar{M}}=\f{\partial_YU_B}{U_B-c}-\f{\partial_Y T_0}{2T_0},
\end{align*}
implies 
\begin{align*}
\partial_Y^2P-\frac{2\partial_Y\bar M}{\bar M}\partial_Y P-\alpha^2F(Y) P=0.
\end{align*}
The boundary condition \eqref{BC: (U,V, rho,T)} is equivalent to 
\begin{align}
\partial_YP(0)=\lim_{Y\to\infty} \partial_Y^j P(Y)=0,\quad j=0,1,2.
\end{align}

Once we obtain a solution $P(Y)$, we can rebuild $(\varrho,U,V,T)(Y)$ by the following formulas  
\begin{align}\label{eq: rebuilt1}
\begin{split}
&V(Y)=\frac{\mathrm i T_0\partial_Y P}{\gamma M_a^2\alpha(U_B-c)},\quad U(Y)=-\frac{T_0 P}{\gamma M_a^2(U_B-c)}-\frac{T_0\partial_YP \partial_YU_B}{\gamma M_a^2\alpha^2(U_B-c)^2},\\
&\varrho(Y)=\frac{ P}{\gamma M_a^2(U_B-c)^2}+\frac{\partial_YP \partial_YU_B}{\gamma M_a^2\alpha^2(U_B-c)^3}-\f{1}{\gamma M_a^2\al (U_B-c)^2}\partial_Y(\frac{\partial_Y P}{U_B-c}),\\
&T(Y)= PT_0-\varrho T_0^2.
\end{split}
\end{align}
Therefore, we only need to construct a non-zero solution $P$  to the following equation 
\begin{align}\label{eq:ray-P1}
\left\{
\begin{aligned}
&\partial_Y^2 P(Y)-\frac{2\partial_Y \bar{ M}}{\bar{M}}\partial_Y P(Y)-\alpha^2F(Y)P(Y)=0,\quad Y>0,\\
&\partial_Y P(0)=P(\infty)=0,
\end{aligned}
\right.
\end{align}
where 
\begin{align}\label{def(bM, T_0)1}
\bar{M}(Y)=\f{M_a(U_B-c)}{\sqrt{T_0}},\quad T_0(Y)=1+\f12(\gamma-1)M_a^2(1-U_B^2(Y)),\quad F(Y)=1-\bar{M}^2(Y),
\end{align}
and $c=c_r+\mathrm{i}c_i$ with $0<c_i\ll 1$.

\subsubsection{Mixed type and singularity}
 We denote
\begin{align*}
\mathcal{L}[P]:=\partial_Y^2 P(Y)-\frac{2\partial_Y \bar{ M}}{\bar{M}}\partial_Y P(Y)-\alpha^2F(Y)P(Y),\quad F(Y)=F_r(Y)+F_i(Y),
\end{align*}
and 
\begin{align}\label{def:F_r}
F_r(Y)=1-\bar{M}_r^2=1-\f{M_a^2(U_B-c_r)^2}{T_0(Y)},\quad F_i(Y)=\f{M_a^2(\mathrm{i}c_i(U_B-c_r)+c_i^2)}{T_0(Y)}.
\end{align}
For $(\al,c) $ belonging  $\mathbb H_0$ (see the definition  \eqref{def: H_0}), there are two crucial points related to $Y$. \smallskip

\textbf{$\bullet$ Turning point $Y_0$: changing the type of operator $\mathcal{L}$.}\smallskip

There exists a unique $Y_0\sim1$ such that $F_r(Y_0)=0$. Moreover,  $F_r(Y)<0$ and $F_r(Y)>0$ for $0\leq Y<Y_0$ and $Y_0<Y<+\infty$ respectively.
Indeed, by the definition of $F_r(Y)$, we have
\begin{align*}
F_r(Y_0)=0\Longleftrightarrow M_a^2(U_B(Y_0)-c_r)^2=T_0(Y_0)=1+\f12(\gamma-1)M_a^2(1-U_B^2(Y_0)),
\end{align*}
which implies that $U_B(Y_0)$ is a solution to the following quadratic equation
\begin{align*}
\f12(\gamma+1)M_a^2x^2-2M_a^2c_r x+M_a^2(c_r^2-\f12(\gamma-1))-1=0.
\end{align*}
It is easy to see that $\Delta=(\gamma^2-1)+2(\gamma+1)M_a^{-2}-2(\gamma-1)c_r^2> (\gamma+1-2c_r)^2$. Then we use the fact $0\leq U_B(Y)\leq 1$ for any $Y\geq 0$  to deduce	    
\begin{align*}
U_B(Y_0)=\frac{2c_r-\sqrt{\Delta}}{\gamma+1},\quad  Y_0\sim1.
\end{align*}	
By the definition of $F_r(Y)$, we know that 
\begin{align*}
&Y< Y_0 \Longleftrightarrow F_r(Y)<0\Longleftrightarrow |\bar{M}_r|>1: \text{supersonic regime},\\
&Y>Y_0 \Longleftrightarrow F_r(Y)>0\Longleftrightarrow |\bar{M}_r|<1: \text{subsonic regime}.
\end{align*}
Then we find that
\begin{align*}
\mathcal{L}[P]\sim\left\{
\begin{aligned}
&(\partial_Y^2 -\alpha^2 |F_r(Y)|)P-\frac{2\partial_Y\bar M}{\bar M}\partial_Y P,\quad \underbrace{Y>Y_0}_{\text{subsonic regime}}:\text{ elliptic type operator},\\
&(\partial_Y^2 +\alpha^2 |F_r(Y)|)P-\frac{2\partial_Y\bar M}{\bar M}\partial_Y P,\quad \underbrace{0\leq Y<Y_0}_{\text{supersonic regime}}:\text{ hyperbolic type operator}.
\end{aligned}
\right.
\end{align*}
Therefore, $\mathcal{L}[\cdot]$ is a hyperbolic-elliptic mixed-type operator.\smallskip

\textbf{$\bullet$ The singularity: critical layer.}\smallskip

There exists a critical layer located at $Y_c$, where $U_B(Y_c)=c_r$. By the definition of $\bar{M}$, we have
\begin{align*}
\bar{M}(Y_c)=\f{-\mathrm{i} M_a c_i}{T_0^\f12(Y_c)}\sim-\mathrm{i}M_a c_i\Longrightarrow \f{\partial_Y\bar{M}(Y_c)}{\bar{M}(Y_c)}\sim\mathrm{i}c_i^{-1}M_a^{-1}.
\end{align*}
Therefore, in a neighborhood of $Y_c$, the second term of $\mathcal{L}[\cdot]$ is singular, that is,
\begin{align}
\f{2\partial_Y \bar{M}}{\bar{M}}\partial_Y P\sim\mathrm{i}c_i^{-1}M_a^{-1}\partial_Y P.
\end{align}
In a small neighborhood of $Y_c$, we know that a solution $P$ to $\mathcal{L}[P]=0$ behaves like
\begin{align*}
\partial_Y P\sim \bar{M}\log \bar{M}\sim (U_B-c)\log(U_B-c)\Longrightarrow \partial_YV\sim\partial_Y^2 P\sim\log(U_B-c),
\end{align*}
which is singular when $Y$ passes the critical point $Y_c$. 

Let's note that the critical layer lies in the subsonic regime. That is $Y_c>Y_0$.

\subsection{Inner-outer approximation}
For any $(\al, c) $ belonging to a suitable subset of $\mathbb H_0$, we first construct a homogeneous solution $P_{\alpha, c}(Y)$ to the following equation
\begin{align}\label{eq:ray-F-nbv1}
\left\{
\begin{aligned}
&\partial_Y^2P-\frac{2\partial_Y\bar M}{\bar M}\partial_Y P-\alpha^2F(Y) P=0,\\
&P(\infty)=0.
\end{aligned}
\right.
\end{align}

\subsubsection{Outer approximation}
In this part, we show how to find an approximation of \eqref{eq:ray-F-nbv1} for all $Y$ away from the critical layer. 
We start with the following asymptotic analysis of the toy model  \begin{align}\label{eq:toy1}
\mathcal{L}_{toy}[P]= (\partial_Y^2-\alpha^2F_r(Y))P=0.
\end{align}
In a small neighborhood of $Y_0$, we can write 
\begin{align*}
\mathcal{L}_{toy}[P](Y)=\underbrace{(\partial_Y^2-\al^2\partial_YF(Y_0)(Y-Y_0))P}_{\text{change type of equation}}+\underbrace{\mathcal{O}(\al^2|Y-Y_0|^2)P}_{\text{small terms}}.
\end{align*} 
Thus, in a small neighborhood of $Y_0$, \eqref{eq:toy1} can be treated as
\begin{align*}
(\partial_Y^2-\al^2\partial_YF(Y_0)(Y-Y_0))P_{toy}(Y)=0: \text{Airy type equation}.
\end{align*}
We know that $Ai(\kappa(Y-Y_0))$ and $Bi(\kappa(Y-Y_0))$ are two linear independent solutions to the above equation, where $\kappa=\partial_YF(Y_0)^\f13\alpha^{\f23}$. Therefore, \eqref{eq:toy1} can be approximated by the Airy equation in the layer of changing the type of equation, and the thickness of this layer is of order $\al^{-\f23}$.

To use the Airy function as the building block to construct the approximate solution, we need to make the following key modifications:
\begin{itemize}
	\item Applying the Langer transformation to zoom small layer of $Y_0$ to a larger domain. Here the Langer transformation $\eta(Y)$  is the solution to $
	\partial_Y F_r(Y_0)\eta(Y)(\partial_Y\eta(Y))^2=F_r(Y)$ with $\eta(Y_0)=0$. Then we have
	\begin{align*}
	&\mathcal{L}[Ai(\kappa\eta)](Y)=\Big(-2\f{\partial_Y\bar{M}}{\bar{M}}+\f{\partial_Y^2\eta}{\partial_Y\eta}\Big)\partial_Y(Ai(\kappa\eta)),\\
	&\mathcal{L}[Bi(\kappa\eta)](Y)=\Big(-2\f{\partial_Y\bar{M}}{\bar{M}}+\f{\partial_Y^2\eta}{\partial_Y\eta}\Big)\partial_Y(Bi(\kappa\eta)).
	\end{align*}
	We notice that $Ai(\kappa\eta(Y))$ is already a good approximation in the subsonic regime except around the critical layer. However, $Ai(\kappa\eta(Y))$ is not a good approximation in the supersonic regime. Hence, we need to make a further modification.
	
	\item Introducing  a function of amplitude $E(Y)=\bar{M}/(\partial_Y\eta)^\f12$ to cancel singular terms, which is a solution to 
	\begin{align*}
	\f{\partial_Y E}{E}=\f{\partial_Y\bar{M}}{\bar{M}}-\f{\partial_Y^2\eta}{2\partial_Y\eta}.
	\end{align*}
	Then we introduce 
	\begin{align}\label{def: (A, B)}
	A(Y)=E(Y)Ai(\kappa\eta(Y)),\quad B(Y)=E(Y)Bi(\kappa\eta(Y)),
	\end{align}
	which satisfy 
	\begin{align*}
	\mathcal{L}[A]=(Q_1(Y)+Q_2(Y))A,\quad \mathcal{L}[B]=(Q_1(Y)+Q_2(Y))B.
	\end{align*}
	Here $|Q_1(Y)|\sim 1$ and $Q_2(Y)\sim c_i$ when $Y$ is away from $Y_c$. Hence, $(Q_1(Y)+Q_2(Y))A$ and $(Q_1(Y)+Q_2(Y))B$ can be treated as small errors. Therefore,
\beno
\mathcal{L}_{app}[P]=\mathcal{L}[P]-(Q_1+Q_2)P
\eeno 
is a good approximation of $\mathcal{L}[P]$ in the $\mathbb{R}_+$ except around the critical layer.     
\end{itemize}

To construct the exact solution to \eqref{eq:ray-F-nbv1}, we need to solve the non-homogeneous equation for a given $f$
\begin{align}\label{eq:non-L}
\mathcal{L}[P_{non,f}]=f,\quad P_{non,f}(\infty)=0.
\end{align}
For this, we solve the following approximate equation
\begin{align}\label{sys:mix-type}
\mathcal L_{app}[P_{non,f}](Y)=f,\quad P_{non,f}(\infty)=0,
\end{align}
where $P_{non,f}$ can be written as
\begin{align}\label{eq:airy-non-f}
\begin{split}
P_{non,f}(Y)=&-\kappa^{-1}\pi A(Y)\int_0^YB(Z)\bar M^{-2}f(Z)dZ-\kappa^{-1}\pi B(Y)\int_Y^{+\infty}A(Z)\bar M^{-2}f(Z)dZ,
\end{split}
\end{align}
and write $P_{non,f}=\mathcal L_{app}^{-1}(f)$ for convenience.

\subsubsection{Inner approximation}

In the critical layer, $\mathcal{L}_{app}[\cdot]=\mathcal{L}[\cdot]-(Q_1+Q_2)$ is not a good approximation of $\mathcal{L}[\cdot]$ due to $Q_1(Y)\sim (|Y-Y_c|+c_i)^{-2}\gg\al^2$.  Instead, we introduce $\varphi=\partial_Y P/\bar{M}$ to transform the equation of pressure to one related to the vertical velocity, that is, 
\begin{align*}
\mathcal{L}[P]=0\Longleftrightarrow \mathcal{L}_{cr}[\varphi]=\bar M\left(\partial_Y\left(\frac{\partial_Y}{F(Y)}\right)-\alpha^2\right)\varphi-\partial_Y\left(\frac{\partial_Y\bar M}{F(Y)}\right)\varphi=0.
\end{align*}
For $Y$ near $Y_c$, two linearly independent solutions behave as
\begin{align*}
\bar{M}(Y)\text{ and }\bar{M}\log\bar{M}(Y) \text{ in the critical layer}.
\end{align*}
Besides, we know that $\varphi\sim e^{-\al Y}$ away from the critical layer. Hence, we can use the following localized equation as a good approximation in the critical layer
\begin{align}\label{eq:ray-local}
\mathcal{L}_{cr}[\varphi]=0,\quad\varphi(\bar Y_1^*)=\varphi(\bar Y_2^*)=0,
\end{align}
where $Y_c\in [\bar{Y}_1^*,\bar{Y}_2^*]$ lies in the subsonic regime. This is a Rayleigh-type equation.

We also need to solve the following non-homogeneous Rayleigh type equation
\begin{align}\label{sys:citical}
\mathcal{L}_{cr}[\varphi]=\partial_Y\left(\frac{f}{1-\bar M^2}\right),\quad\varphi(\bar Y_1^*)=\varphi(\bar Y_2^*)=0,\quad\varphi_{non,f}=\partial_YP_{non,f}/\bar{M},
\end{align}
and write $\varphi=\mathcal{L}_{cr}^{-1}(f)$ for convenience.

\subsection{Inner-outer gluing iteration scheme} \label{sec:IN}
Based on the above inner-outer approximation, we shall construct a solution $P(Y)$ to 
\eqref{eq:ray-F-nbv1} via an inner-outer gluing procedure. We denote
\begin{align*}
\mathcal{G}_1(P):\eqdef-\int_Y^{+\infty}\chi_1(Z)\partial_ZP(Z)dZ.
\end{align*}
 Since the way passing the information in the critical layer to the supersonic regime is crucial, we recover the corresponding pressure in two different ways
\begin{align*}
&P^{nloc}(Y):=-\int_Y^{+\infty}\chi_2\varphi(Z)dZ:\eqdef\mathcal{G}_{2,nloc}\circ\varphi, \\
&P^{loc}(Y):=-\chi_3\int_Y^{+\infty}\chi_2\varphi(Z)dZ:\eqdef\mathcal{G}_{2,loc}\circ\varphi,
\end{align*}
where $\chi_2$ and $\chi_3$ are smooth cut-off functions such that $\mathrm{supp}(1-\chi_1)\subset\mathrm{supp}(\chi_2)\subset[\bar Y_1^*,\bar Y_2^*]\subset\mathrm{supp}(\chi_3)\subset[Y_0+\delta,+\infty)$. Then $\mathrm{supp}(P^{nloc})\cap[0,Y_0]\neq\emptyset$ and $\mathrm{supp}(P^{loc})\cap[0,Y_0]=\emptyset$. 

The non-local pressure $P^{nloc}$ is crucial to construct the leading order term of the pressure, and in the subsequent iteration, the local pressure $P^{loc}$ is crucial to ensure that the pressure related to the critical layer does not affect the leading order term.

\subsubsection{Construction of leading order term}
According to the argument in the previous part, we know that 
\begin{align*}
\mathcal L[A](Y)=(Q_1(Y)+Q_2(Y))A(Y),\quad A(\infty)=0,
\end{align*} 
where
\begin{align*}
Q_1(Y)\sim 1\text{ for $Y$ away from $Y_c$ and }Q_1(Y)\sim c_i^{-2}\text{ for $Y$ near $Y_c$}. 
\end{align*}
Hence, $A(Y)$ is the approximation of $P(Y)$ in the regime away from the critical layer. However, $A(Y)$ can not be treated as the leading order term. 

Now we show main steps constructing the leading order term $P^{(0)}(Y)=P_{0,0}(Y)+P_{0,1}(Y)+P_{0,2}(Y)$.
\begin{enumerate}
	\item We first define
	\begin{align*}
	P_{0,0}(Y):=-\int_Y^{+\infty}\chi_1(Z)\partial_ZA(Z)dZ:\eqdef\mathcal{G}_{1}(A).
	\end{align*}
	Then we have
	\begin{align*}
	\mathcal L [P_{0,0}]=Err_{0,0}^{cr}+Err^{ncr}_{0,0}.
	\end{align*}
	
	\item We next define 
	\begin{align}\label{def: P_0,1}
	P_{0,1}(Y):=\mathcal{G}_{2,nloc}\circ\mathcal{L}^{-1}_{cr}(-Err_{0,0}^{cr}).
	\end{align}
	Notice that $\mathrm{supp}(P_{0,1})\cap[0,Y_0]\neq\emptyset$, which means that the information in the critical layer is conveyed by non-local pressure $P_{0,1}$ to the supersonic regime. Then we have
	\begin{align*}
	\mathcal{L}[P_{0,1}]=-Err_{0,0}^{cr} + Err_{0,1}^{cr} +Err_{0,1}^{ncr}+Err_{0,1}^{sup},
	\end{align*}
	where $Err_{0,1}^{sup}$ supported in the supersonic regime, carries the key information (containing the term like $(U_B-c)\log(U_B-c)$) from the critical layer. The wave generated by $Err_{0,1}^{sup}$ can lead to  the instability.
	
	\item We finally construct $P_{0,2}(Y)$, which is used to correct the errors from the previous steps
	\begin{align*}
	P_{0,2}(Y):=\mathcal{G}_{1}\circ\underbrace{\mathcal L_{app}^{-1}(-Err_{0,0}^{ncr}-Err_{0,1}^{ncr})}_{\tilde{P}_{0,2}^{(1)}}+\bar\chi_1  \tilde P_{0,2}^{(2)}(Y),
	\end{align*}
	where $\bar\chi_1$ is a cut-off function supported in the supersonic regime 
	and $\tilde P_{0,2}^{(2)}$ satisfies 
	\begin{align}\label{eq: tilde P_(0,2)^(2)1}
	\mathcal L[\tilde P_{0,2}^{(2)}]=-Err_{0,1}^{sup}+(Q_1+Q_2)\tilde P_{0,2}^{(2)},
	\end{align}
	which is given by 
	\begin{align}\label{formula: tilde P_(0,2)^(2)1}
	\begin{split}
	\tilde P_{0,2}^{(2)}(Y)=&-\kappa^{-1}\pi A(Y)\int_Y^{+\infty} B(Z)\bar M^{-2}Err_{0,1}^{sup}dZ-\kappa^{-1}\pi B(Y)\int_0^{Y}A(Z)\bar M^{-2}Err_{0,1}^{sup}dZ.
	\end{split}
	\end{align}
	In fact, $\tilde P_{0,2}^{(2)}(Y)$ is a solution to $\mathcal{L}_{app}^{-1}[Err_{0,1}^{sup}]$ with $\tilde P_{0,2}^{(2)}(Y)\sim e^{\al Y}\to\infty$ as $Y\to\infty$. We find that
	\begin{align*}
	\mathcal{L}[P_{0,2}]=-Err_{0,0}^{ncr}-Err_{0,1}^{ncr}-Err_{0,1}^{sup}+Err_{0,2}^{cr}+E_{0,2}^{ncr}.
	\end{align*}
	Then we have $P^{(0)}=P_{0,0}(Y)+P_{0,1}(Y)+P_{0,2}(Y)$ with
	\begin{align*}
	\mathcal L[P^{(0)}]=Err_{cr}^{(0)}+Err_{ncr}^{(0)},
	\end{align*} 
	where 
	\begin{align*}
	Err_{cr}^{(0)}=Err_{0,1}^{cr}+Err_{0,2}^{cr}\text{ and }Err_{ncr}^{(0)}=Err_{0,2}^{ncr}.
	\end{align*}
\end{enumerate} 

\begin{figure}[h]\label{fi:1}
	\begin{tikzpicture} 
	\matrix(tree)[%
	matrix  of nodes, 
	minimum size=0.8cm, 
	column sep=2.5cm, 
	row sep=0.8cm, 
	]
	{
		$P_{0,0}$ & $P_{0,1}$ &&$P_{0,2}$\\ 
		& & $\tilde P_{0,2}^{(2)}$& \\
		& &$\tilde P_{0,2}^{(1)}$&\\
	}; 
	\draw [->](tree-1-1)-- (tree-1-2)node [midway,above]{$Err_{0,0}^{cr}$}; 
	\draw [->] (tree-1-1)--(tree-3-3)node [midway,below]{$Err_{0,0}^{ncr}$}; 
	\draw [->](tree-1-2) -- (tree-2-3)node [midway,above]{$Err_{0,1}^{sup}$}; 
	\draw [->] (tree-1-2)-- (tree-3-3)node [midway,right]{$Err_{0,1}^{ncr}$};   
	\draw[->](tree-2-3)-- (tree-1-4)node [midway,right]{};
	\draw[->](tree-3-3)-- (tree-1-4)node [midway,right]{};
	\end{tikzpicture}
	\caption{Construction of leading order term}
\end{figure}
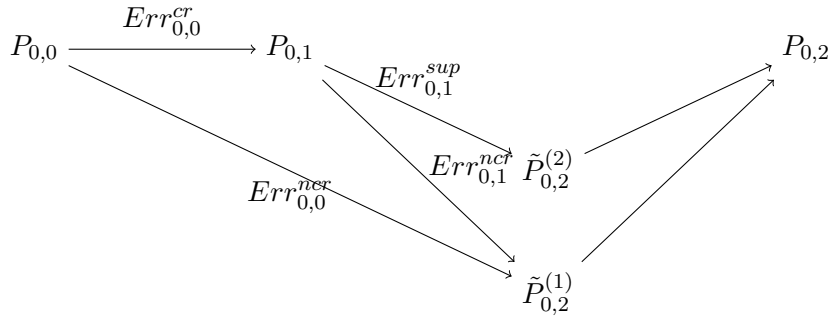

\subsubsection{Iteration scheme}
We show the iteration scheme of the construction. Assume that for $j\geq 0$, the error of $P^{(j)}(Y)$ can be written as 
\begin{align*}
\mathcal{L}[P^{(j)}](Y)=Err_{cr}^{(j)}+Err_{ncr}^{(j)}:=Err^{(j)}.
\end{align*}
We define 
\begin{align}\label{def: P_j+1,1}
P_{j+1,1}(Y)=\mathcal{G}_{2,loc}\circ\mathcal L_{cr}^{-1}(-Err_{cr}^{(j)}).
\end{align}
Then we have 
\begin{align*}
\mathcal L[P_{j+1,1}]
=&-Err_{cr}^{(j)}+Err_{j+1,1}^{cr}+Err_{j+1,1}^{ncr},
\end{align*}
Next we define 
\begin{align*}
P_{j+1,2}(Y)=\mathcal{G}_1\circ\mathcal L_{app}^{-1}(-Err_{ncr}^{(j)}-Err_{j+1,1}^{ncr}).
\end{align*}
Therefore, $P^{(j+1)}=P_{j+1,1}+P_{j+1,2}$ satisfies 
\begin{align*}
\mathcal L[P^{(j+1)}]=&-Err_{cr}^{(j)}-Err_{ncr}^{(j)}+\underbrace{Err_{cr}^{(j+1)}+Err_{ncr}^{(j+1)}}_{:=Err^{(j+1)}}\\
=&-Err^{(j)}+Err^{(j+1)},
\end{align*}
where 
\begin{align*}
Err_{cr}^{(j+1)}=Err_{j+1,1}^{cr}+Err_{j+1,2}^{cr}\text{ and }Err_{ncr}^{(j+1)}=Err_{j+1,2}^{ncr}.
\end{align*}
Thus, $\mathcal L[\sum_{j=0}^N P^{(j)}]=Err^{(N)}(Y)$ and $P=\sum_{j=0}^{\infty} P^{(j)}$ solves \eqref{eq:ray-F-nbv1} formally. 

\begin{remark}\label{rem: comparison iterations}

The inner-outer gluing iteration scheme developed in this paper has some similarities with the Rayleigh-Airy iteration developed by Grenier-Guo-Nyugen \cite{GGN-DMJ}. Both iteration methods  are based on the fact that the original equations can be approximated by the Rayleigh equation and Airy equation under different scales. However, the principles of these two methods are fundamentally different. 

The Rayleigh-Airy iteration solves the Orr-Sommerfeld (OS) equation, utilizing the Rayleigh equation and the Airy equation at low and high frequencies, respectively. Specifically, at low frequencies, the viscous term acts as a small perturbation, allowing the OS equation to be approximated by the Rayleigh equation. Conversely, at high frequencies, the nonlocal term becomes a small perturbation, leading to the approximation of the OS equation by the Airy equation. This method constitutes a gluing iterative approach from the standpoint of frequency space.

In contrast, the inner-outer gluing iteration scheme solves the system  \eqref{eq:ray-F-nbv1} from the standpoint of physical space. 
More precisely, in the region around the critical layer, \eqref{eq:ray-F-nbv1} consistently exhibits elliptic-type behavior. Thus, within this interval, we employ the Rayleigh-type equation to approximate the solution of \eqref{eq:ray-F-nbv1}. Conversely, in the region away from the critical layer, \eqref{eq:ray-F-nbv1} undergoes a change in type, prompting us to approximate them with the Airy-type equation. The Airy-type equation also effectively captures the decaying behavior at infinity. Subsequently, cut-off functions and a gluing method in physical space are introduced to construct the solution on the whole physical space (i.e., $Y\in[0,+\infty)$).

\end{remark}

\subsection{The dispersion relation}
After we construct the solution $P$ to \eqref{eq:ray-F-nbv1}, we need to match the boundary value and derive the dispersion relation.  The asymptotic formula of $\partial_Y P(0)$  takes as follows
\begin{align*}
\partial_Y P(0)={\partial_YA(0)+C_0\alpha^{-1}\partial_Y B(0)}+\mathrm iC_1\alpha c_i\partial_YB(0)+C_2\mathcal K_{cr}\partial_YA(0)\partial_YB(0)=H(\al,c).
\end{align*}
Then the dispersion relation is equivalent to $H(\al,c)=0$. Taking the real part will give
\begin{align}\label{eq:dis-cr}
c_r\sim\frac{1}{D_0\alpha\tan(\theta_0\alpha-\pi/4)+D_1},
\end{align}
where $D_0,\theta_0\neq 0$ are independent of $\alpha$. Taking the imaginary part will give
\begin{align}\label{eq:dis-ci}
c_i\sim\Im(\mathcal K_{cr})\alpha^{-\f76}c_r\cos(\theta_0\alpha-\pi/4).
\end{align}
For a fixed $c_r$, by \eqref{eq:dis-cr}, there is a sequence $\{\al_k\}_{k=1}^{+\infty}$ such that 
\begin{align*}
\al_k\tan(\theta_0\alpha_k-\pi/4)=\mbox{nonzero constant}.
\end{align*}
By using  \eqref{eq:dis-ci} and the different period of functions $\cos (\theta_0\alpha-\pi/4)$ and $\tan (\theta_0\alpha-\pi/4)$, we deduce
\begin{align*}
\mathrm{Sign}\left( \frac{c_i(\alpha_k)}{c_i(\alpha_{k+1})}\right)=\mathrm{Sign}\left( \frac{\cos(\theta_0\alpha_{k}-\pi/4)}{\cos(\theta_0\alpha_{k+1}-\pi/4)}\right)=-1.
\end{align*}
Thus, we can find a subsequence $\{\al_{n_k}\}_{k=1}^{+\infty}\subset\{\al_k\}_{k=1}^{+\infty}$  such that
\begin{align}\label{key: sign}
\mathrm{Sign}(\cos(\theta_0\alpha_{n_k}-\pi/4))=\mathrm{Sign}(\Im(\mathcal K_{cr})),
\end{align}
which implies
\begin{align*}
0<c_{i,k}=c_i(\al_{n_k})\sim \al^{-\f76}\Im(\mathcal K_{cr}).
\end{align*}

\subsection{Organization}
The paper is organized as follows. In Section 3, we introduce Langer transform and solve the mixed-type outer system \eqref{sys:mix-type}. In Section 4, we solve the inner singular system \eqref{sys:citical}. In Section 5, we construct a solution of system \eqref{eq:ray-F-nbv1} via inner-outer gluing iteration scheme. In Section 6, we give the expansion of boundary values of the solution constructed in Section 5. In Section 7, we deduce the dispersion relation and construct multiple unstable modes. Finally, we give the proof of  Theorem \ref{thm: main-1}-\ref{thm: main-3}.

\section{The hyperbolic-elliptic mixed type system}

This section  is devoted to solving the approximate hyperbolic-elliptic mixed-type
system in the regime away from the critical layer, which takes as follows
\begin{align}\label{sys:app-f-1}
	\left\{
	\begin{aligned}
		&\mathcal L_{app}[P]= \partial_Y^2 P-\frac{2\partial_Y \bar{M}}{\bar{M}}\partial_Y P-\alpha^2 \tilde F(Y) P-Q_1(Y) P=f,\\
		&P(\infty)=0.
	\end{aligned}
	\right.
\end{align}
 Here $\tilde F(Y)$ is a small modification of $F_r(Y)$ given by
\begin{align}
 \tilde F(Y)=&F_r(Y)+(1-\chi_0(Y))F_i(Y),\label{def: tF}
\end{align}
where  $\chi_0$ is a smooth cut-off function satisfying $\chi_0=1$ and $\chi_0=0$ for $Y\in[0,\bar{Y}_2^*+\delta_0]$ and $Y\in[\bar{Y}_2^*+2\delta_0,+\infty)$ respectively. We point out that on both sides of $Y_0$, the function $\tilde F(Y)$ changes the sign, so \eqref{sys:app-f-1} is a mixed-type equation. In addition, $Q_1(Y)$ is given by
\begin{align}\label{def: Q_1}
	Q_1(Y)=\f{\pa_Y^2 \bM}{\bM}-\f34\f{(\pa_Y^2\eta)^2}{(\pa_Y\eta)^2}-\f{\pa_Y^3 \eta}{2\pa_Y\eta}-\f{2(\pa_Y\bM)^2}{\bM^2}.
\end{align}
Moreover, $\eta(Y)$ is  the solution to 
\begin{align}\label{eq: eta}
	\partial_Y \tilde F_r(Y_0)\eta(Y)(\partial_Y\eta(Y))^2=\tilde F_r(Y)\quad\mbox{ with}\quad \eta(Y_0)=0,
\end{align}
where $Y_0$ is the turning point satisfying $F_r(Y_0)=0$ with $Y_0\sim1$. It is easy to see $\partial_Y \tF(Y_0)=\partial_Y F_r(Y_0)$. By solving  the equation \eqref{eq: eta}, we obtain
\begin{align}\label{def: Langer}
\eta(Y):=\left\{
\begin{aligned}
&\Big(\f32\int_{Y_0}^Y\Big(\f{\tilde F(Z)}{\partial_Y \tF(Y_0)}\Big)^\f12dZ\Big)^\f23,\quad Y\geq Y_0,\\
&-\Big(\f32\int_{Y_0}^Y\Big(\f{-\tilde F(Z)}{\partial_Y \tF(Y_0)}\Big)^\f12dZ\Big)^\f23,\quad Y\leq Y_0.
\end{aligned}
\right.
\end{align}
Let's mention that by the definition of $\tF(Y)$, we know that 
\begin{align}\label{est: Im tF}
\begin{split}
&\mathrm{Im}(\tF(Y))=0,\quad \tF(Y)=F_r(Y),\quad  Y\leq \bar{Y}_2^*+\delta_0,\\
&|\tF(Y)-F_r(Y)|\leq C c_i,\quad Y> \bar{Y}_2^*+\delta_0.
\end{split}
\end{align}
We infer from \eqref{def: Langer}-\eqref{est: Im tF} that
\begin{align}\label{eq:eta-i}
\Im \eta(Y)=0,\quad\forall Y\leq \bar{Y}_2^*+\delta_0\quad \text{ and }\quad |\pa_Y^k\eta(Y)|\leq C|Y-Y_0|^{\f23-k},\quad k\geq 0, \quad\forall Y> \bar{Y}_2^*+\delta_0.
\end{align}

In this section, we always assume that $(\alpha,c)\in\mathbb H_0$, where
 {\small
 \begin{align}\label{def: H_0}
	\mathbb H_0:=\{(\al,c)\in\mathbb R_+ \times\mathbb C: c_r\in(1-M_a^{-1},1),~ Y_c-Y_0\gtrsim 1, 0<c_i\ll \al^{-n}, \forall n>0\text{ and }\alpha\geq \alpha_0\gg 1\}.
\end{align}
}
As in Section 2.2.1, we introduce
\begin{align}\label{def: (A,B)}
A(Y)= E(Y)Ai(\kappa\eta(Y)),\quad B(Y)=E(Y)Bi(\kappa\eta(Y)),\quad E(Y)=\bM(Y)/(\pa_Y\eta(Y))^\f12,
\end{align}
where 
\begin{align}\label{relation kappa and alpha}
\kappa^3=\alpha^2\partial_Y\tF(Y_0)=\alpha^2\partial_YF_r(Y_0)\quad \mbox{ is a real number with relation}\quad \al\sim \kappa^{\f32}.
\end{align} 
Here $Ai(z)$ and $Bi(z)$ are two Airy functions satisfying the equation $f''-zf=0$, with asymptotic expansions presented in Lemma \ref{lem: Airy-original}. It is easy to check that $\mathcal L_{app}[A](Y)=\mathcal L_{app}[B](Y)=0$ and the Wronskians
\begin{align*}
\partial_Y A(Y) B(Y)-A(Y)\partial_Y B(Y)
	=&-\kappa\pi^{-1}\partial_Y\eta E(Y)^2
	=-\kappa\pi^{-1}\bar{M}^2.
\end{align*} 

For any given suitable source term $f(Y)$ with $f(Y)\to0$ as $Y\to\infty$,
\begin{align}\label{eq:app-f-green-1}
		P(Y)=&\mathcal{A}_f(Y)A(Y)+\mathcal{B}_f(Y)B(Y)
\end{align}
 is a solution to the system \eqref{sys:app-f-1}, where 
 \begin{align}\label{def: cA_f, cB_f}
&\mathcal{A}_f(Y)=-\kappa^{-1}\pi \int_0^YB(Z)f(Z)\bar{M}^{-2}(Z)dZ,\quad \mathcal{B}_f(Y)=-\kappa^{-1}\pi \int_Y^{+\infty}A(Z)f(Z)\bar{M}^{-2}(Z)dZ.
\end{align}

We introduce a real function
\begin{align}\label{def: w_0}
w_0(Y)=\mathbf{ 1}_{\{Y\geq Y_{0}\}}(Y)\int_{Y_0}^Y\Re(\tF^\f12)(Z)dZ,
\end{align}
and the weighted norms
\begin{align}\label{def: norm w_0}
&\|f\|_{L^\infty_{w_0}}=\|e^{\al w_0} f\|_{L^\infty([0,\infty))},\quad \|f\|_{L^\infty_{w_0}([a,b])}=\|e^{\al w_0} f\|_{L^\infty([a,b])}.
\end{align}
We see from \eqref{def: w_0} that
\begin{align*}
w_0(Y)=\int_{Y_0}^YF_r^\f12(Z)dZ \quad \mbox{for}\quad  Y\in[Y_0, \bar Y_2^*+\d_0].
\end{align*}

 We will construct a solution to \eqref{sys:app-f-1} via \eqref{def: cA_f, cB_f} for a source term
 \begin{align}\label{def: general source term}
f(Y)=a(Y)A(Y)+b(Y)B(Y) \quad\mbox{with}\quad \|(1+Y)^2a\|_{L^\infty}+\|(1+Y)^2b\|_{L^\infty_{2w_0}}\leq C,
 \end{align}
  and  the following four kinds of special source terms:
\begin{align}\label{def: special source terms}
\begin{split}
		&(1-\chi)\int_Y^{+\infty}g(Z)dZ,\quad f,\quad (1-\chi)\int_Y^{\infty}f(Z)dZ,\quad (1-\chi)h\int_Y^{\infty}f(Z)dZ.
		\end{split}
	\end{align}
Five kinds of source terms involve different information, which is essential to guarantee convergence in the iteration scheme and capture the information of the imaginary part in the dispersion relation. Here $\chi$ is a smooth cut-off function such that $\chi=1$ for $Y\in[Y_1-\d_0, Y_2+\d_0]$ and $\chi=0$ for $Y\in[0,Y_1-2\d_0]\cup[Y_2+2\d_0,+\infty)$, and the supports  of functions $f, g, h$ are given by
	\begin{align*}
		&\mathrm{supp}(g)\subseteq[Y_1, Y_1+\d_0]\cup[Y_2-\d_0, Y_2],\quad \mathrm{supp}(f)\subseteq[Y_1^*-\d_0, Y_1^*]\cup [Y_2^*, Y_2^*+\d_0],\\
		& \mathrm{supp}(h)\subseteq [\bar Y_1^{**},+\infty),\quad \|g\|_{L^\infty_{w_0}}+\| f\|_{L^\infty_{w_0}}+\|h\|_{L^\infty}\leq C.
	\end{align*}
Moreover, the relations for points are given by	
\begin{align*}
Y_0<\bar Y_1^{**}<Y_1^*-\d_0<Y_1^*<Y_1-2\d_0<Y_1+2\d_0<Y_c<Y_2-2\d_0<Y_2+2\d_0<Y_2^*<Y_2^*+\d_0<\bar Y_2^*,
\end{align*}
and $Y_c$ is the critical point satisfying $U_B(Y_c)=c_r$.

The following decomposition is used frequently in this section. For a complex-valued function $f$ in the formula \eqref{def: cA_f, cB_f}, we divide $f=\Re f+i\Im f$, and define $\mathcal{A}_{\Re f}$, $\mathcal{A}_{\Im f}$, $\mathcal{B}_{\Re f}$, $\mathcal{B}_{\Im f}$, which are four complex-valued functions, as
\begin{align}\label{relation: (cA, cB)}
\mathcal{A}_f=\mathcal{A}_{\Re f}+i\mathcal{A}_{\Im f},\quad \mathcal{B}_f=\mathcal{B}_{\Re f}+i \mathcal{B}_{\Im f}.
\end{align}
Thus we have
\begin{align}\label{relation: (cA^1, cB^1)-(cA^2, cB^2)}
\begin{split}
&|\mathcal{A}_f|\leq |\mathcal{A}_{\Re f}|+|\mathcal{A}_{\Im f}|,\quad |e^{2\al w_0}\mathcal{B}_f|\leq |e^{2\al w_0}\mathcal{B}_{\Re f}|+|e^{2\al w_0}\mathcal{B}_{\Im f}|,\\
&\Im \mathcal{A}_f=\Im \mathcal{A}_{\Re f}-\Re \mathcal{A}_{\Im f},\quad \Im \mathcal{B}_f=\Im \mathcal{B}_{\Re f}-\Re \mathcal{B}_{\Im f}.
\end{split}
\end{align}
As a result, to get the estimates of $\cA_f$ and $\cB_f$ with a complex-valued function $f$, we first assume $f$ is a real function and use the decomposition \eqref{relation: (cA, cB)} and \eqref{relation: (cA^1, cB^1)-(cA^2, cB^2)} to get the results for general $f$.

\subsection{Langer transformation and Airy function}

We first show the estimates of the Langer transformation $\eta(Y)$.
\begin{lemma}\label{lem:est-eta}
	 There exists a small positive constant $L\in(0,1)$ such that  $\eta(Y)\in C^{\infty}(\mathbb{R}_{+})$ satisfies the following properties: for any $Y\geq 0$, we have
\begin{align*}
|\pa_Y^k\eta(Y)|\leq C(1+|Y-Y_0|)^{\f23-k},\quad k=0,1,2,3.
\end{align*}
Moreover, for any $|Y-Y_0|\leq L$, we have
\begin{align*}
		|\eta(Y)-(Y-Y_0)|\leq C|Y-Y_0|^2,\quad |\partial_Y \eta(Y)-1|\leq C|Y-Y_0|,
		\end{align*}
and for any $|Y-Y_0|\geq L$, we have
\begin{align*}
|\pa_Y^k\eta(Y)|\geq C^{-1}(1+|Y-Y_0|)^{\f23-k},\quad k=0,1.
\end{align*}
\end{lemma}

\begin{proof}
The proof is very similar to the process in Lemma A.6 in our previous paper \cite{MWWZ}, by using \eqref{def: Langer}-\eqref{eq:eta-i}, and we omit details here.  
\end{proof}

With the estimates of $\eta(Y)$ at hand, we present the estimate for $Q_1(Y)$ in the regime away from the critical layer.

\begin{lemma}\label{lem: Q_1}
In the regime away from the critical layer, we have 
\begin{align}
&|Q_1(Y)|+c_i^{-1}|\Im Q_1(Y)|\leq C(|\pa_Y^2 U_B|+|\pa_Y U_B|^2+(1+|Y-Y_0|)^{-2}).\label{est: Q_1}
\end{align}

\end{lemma}
\begin{proof}
In the regime away from the critical layer, we have $C^{-1}\leq |\bM(Y)|\leq C$, which uses a direct calculation to  give 
\begin{align}\label{est: Q_1-1}
\left|\f{\pa_Y^2 \bM}{\bM}\right|+\left|\f{(\pa_Y\bM)^2}{\bM^2}\right|+c_i^{-1}\left(\left|\Im\left(\f{\pa_Y^2 \bM}{\bM}\right)\right|+\left|\Im\left(\f{(\pa_Y\bM)^2}{\bM^2}\right)\right|\right)\leq&C\left(\left|\pa_Y^2 U_B\right|+\left|\pa_Y U_B\right|^2\right).
\end{align}
By Lemma \ref{lem:est-eta}, we have
\begin{align}\label{est: Q_1-3}
\left|\f{(\pa_Y^2\eta)^2}{(\pa_Y\eta)^2}\right|+\left|\f{\pa_Y^3 \eta}{2\pa_Y\eta}\right|+c_i^{-1}\left(\left|\Im\left(\f{(\pa_Y^2\eta)^2}{(\pa_Y\eta)^2}\right)\right|+\left|\Im\left(\f{\pa_Y^3 \eta}{2\pa_Y\eta}\right)\right|\right)\leq C(1+|Y-Y_0|)^{-2}.
\end{align}
Combining \eqref{est: Q_1-1}-\eqref{est: Q_1-3} and using the definition of $Q_1$ in \eqref{def: Q_1}, we get the estimate \eqref{est: Q_1}. 
\end{proof}

Next we present some asymptotic formulas of the Airy functions, which can be found in the book \cite{OLBC}. For convenience, we introduce a function
\begin{align}\label{def: Theta}
\Theta(z)=\f23 z^\f32-\f{\pi}{4}.
\end{align}

\begin{lemma}(\cite{OLBC})\label{lem: Airy-original}
As $z\to \infty$ the following asymptotic expansions are valid uniformly in the stated sectors. There exists a constant $M>0$ such that
\begin{enumerate}
\item for $|arg(z)|\leq\pi-\d$ with $\d>0$ and $|z|\geq M$,
\begin{align*}
&Ai(z)= \f{1}{2\sqrt\pi }z^{-\f14}e^{-\f23 z^\f32}(1+O(z^{-\f32})),\quad \pa_zAi(z)=- \f{1}{2\sqrt\pi }z^{\f14}e^{-\f23 z^\f32}(1+O(z^{-\f32})).
\end{align*}
\item for $|arg(z)|\leq\f13\pi-\d$ with $\d>0$ and $|z|\geq M$,
\begin{align*}
&Bi(z)=\f{1}{\sqrt\pi }z^{-\f14}e^{\f23 z^\f32}(1+O(z^{-\f32})),\quad \pa_zBi(z)= \f{1}{\sqrt\pi }z^{\f14}e^{\f23 z^\f32}(1+O(z^{-\f32})).
\end{align*}
\item for  $z\in \mathbb{R}$ and positive constants $0<a_0\neq b_0>0$ with $z<-M$,
\begin{align*}
&Ai(z)= \f{1}{\sqrt\pi (-z)^\f14}\Big(\cos\Theta(-z))+a_0(-z)^{-\f32}\sin(\Theta(-z))\Big)(1+O(z^{-3})),\\
&\pa_zAi(z)= \f{(-z)^\f14}{\sqrt\pi }\Big(\sin(\Theta(-z))+b_0(-z)^{-\f32}\cos(\Theta(-z))\Big)(1+O(z^{-3})),\\
&Bi(z)\sim \f{1}{\sqrt\pi (-z)^\f14}\Big(-\sin(\Theta(-z))+a_0(-z)^{-\f32}\cos(\Theta(-z))\Big)(1+O(z^{-3})),\\
&\pa_zBi(z)\sim \f{(-z)^\f14}{\sqrt\pi }\Big(\cos(\Theta(-z))-b_0(-z)^{-\f32}\sin(\Theta(-z))\Big)(1+O(z^{-3})).
\end{align*}
\item for  $|z|\leq M$,
\begin{align*}
|Ai(z)|+|Bi(z)|+|\pa_zAi(z)|+|\pa_zBi(z)|\leq C.
\end{align*}

\end{enumerate}

\end{lemma}
We decompose $\mathbb R_+\cup\{0\}=\mathcal N^-\cup\mathcal N\cup\mathcal N^+$, where
\begin{align*}
		&\mathcal N^+=\{Y\geq Y_0:|\kappa\eta(Y)|\geq M\},\quad\mathcal N=\{Y:|\kappa\eta(Y)|\leq M\},\quad
		\mathcal N^-=\{0\leq Y\leq Y_0:|\kappa\eta(Y)|\geq M\},
\end{align*}
and $M>0$ is a large constant given in Lemma \ref{lem: Airy-original}.
Moreover, we shall rewrite intervals $\mathcal{N}^-,\mathcal{N}$ and $\mathcal{N}^+$ as 
\begin{align}\label{def: intervals N}
\mathcal{N}^-=[0,Y^-],\quad\mathcal{N}=[Y^-,Y^+]\quad\mathcal{N}^+=[Y^+,+\infty),
\end{align}
with the relation $1\sim Y^{-}<Y_0<Y^{+}\sim1$.

\begin{remark}
We point out that for $Y\in[Y_0,Y^+]$, $|\al w_0(Y)|\leq C$
and for $Y\geq Y^{+}$,
\begin{align*}
&\f23(\kappa\eta(Y))^\f32=\kappa^\f32\int_{Y_0}^Y\Big(\frac{\tF(Z)}{\partial_Y F(Z_0)}\Big)^\f12dZ=\alpha\int_{Y_0}^Y\tF^\f12(Z)dZ,\\&\Re\Big(\f23(\kappa\eta(Y))^\f32\Big)=\alpha\int_{Y_0}^Y\Re(\tF^\f12)(Z)dZ,\quad \Big|e^{\pm\f23(\kappa\eta(Y))^\f32}\Big|=e^{\pm \alpha\int_{Y_0}^Y\Re(\tF^\f12)(Z)dZ}.
\end{align*}
In particular, for $Y\in[ Y^{+}+\d, \bar Y_2^*]$ with $\d>0$ independent of $\al$, we deduce from \eqref{est: Im tF} that 
\begin{align*}
Ai(\kappa\eta(Y))\sim \alpha^{-\f16}e^{-\alpha\int_{Y_0}^YF_r^\f12(Z)dZ},\quad Bi(\kappa\eta(Y))\sim\alpha^{-\f16}e^{\alpha\int_{Y_0}^YF_r^\f12(Z)dZ}.
\end{align*}

\end{remark}

With asymptotic expansions of $Ai(z)$ and $Bi(z)$ presented in Lemma \ref{lem: Airy-original}, we can deduce the following result.

\begin{lemma}\label{lem: Ai, Bi}
For $L\in(0,1)$ given in Lemma \ref{lem:est-eta}, it holds that 
\begin{enumerate}
\item for $Y\in [0, Y^{-}]$, we have
\begin{align*}
|Ai(\kappa \eta) Bi(\kappa\eta)(\pa_Y\eta)^{-1}|+|Ai(\kappa \eta) ^2(\pa_Y\eta)^{-1}|+|Bi(\kappa\eta)^2(\pa_Y\eta)^{-1}|\leq C\kappa^{-\f12}\max\{L^{-\f12}, |Y-Y_0|^{-\f12}\}.
\end{align*}
 \item for $Y\in [Y^{-}, Y^{+}]$, we have
 \begin{align*}
|Ai(\kappa \eta) Bi(\kappa\eta)(\pa_Y\eta)^{-1}|+|Ai(\kappa \eta) ^2(\pa_Y\eta)^{-1}|+|Bi(\kappa\eta)^2(\pa_Y\eta)^{-1}|\leq C.
\end{align*}
\item for $Y\in [Y^{+},+\infty)$, we have
\begin{align*}
&|Ai(\kappa \eta) Bi(\kappa\eta)(\pa_Y\eta)^{-1}|+|e^{2\al w_0}Ai(\kappa \eta) ^2(\pa_Y\eta)^{-1}|\\
&+|e^{-2\al w_0}Bi(\kappa\eta)^2(\pa_Y\eta)^{-1}|\leq C\kappa^{-\f12}\max\{L^{-\f12}, |Y-Y_0|^{-\f12}\}.
\end{align*}
\item for $Y\in[0, \bar Y_2^*+\d_0]$, we have
\begin{align*}
 \Im (\eta(Y))=0,\quad\mbox{and}\quad  \eta(Y),~Ai(\kappa\eta), ~Bi(\kappa\eta) \quad \mbox{are real functions}.
\end{align*}

\end{enumerate}

\end{lemma}
	
\begin{proof}
The above estimates are direct consequences by using \eqref{eq:eta-i}, Lemma \ref{lem:est-eta}, Lemma \ref{lem: Airy-original} and the fact $\eta(\partial_Y\eta)^2=\tF(\partial_Y \tF(Y_0))^{-1}$.
\end{proof}

\subsection{Solving the mixed type system with a source term $f=a(Y)A(Y)+b(Y)B(Y)$. }

 Let $\d>0$ be a suitable small constant independent of $\al$. The results are given as follows.

\begin{proposition}\label{pro: Airy-1}
	Let $(\alpha,c)\in\mathbb H_0$ and $a(Y)$, $b(Y)$ be two complex-valued functions with $(1+Y)^2(|a|, |e^{2\al w_0 }b|)\in L^\infty$. Then there exists a solution $\tilde P(Y)\in W^{2,\infty}_{w_0}$ to the system \eqref{sys:app-f-1} with the source term $f(Y)=a(Y)A(Y)+b(Y)B(Y)$. More precisely, there exist $\cA_f(Y)$ and $\cB_f(Y)$ with $(|\cA_f|, |e^{2\al w_0} \cB_f|)\in L^\infty$ such that 
\begin{align*}	
	\tilde P(Y)=\mathcal{A}_f(Y)A(Y)+\mathcal{B}_f(Y)B(Y),
\end{align*}
	 and the following estimates hold
	 
\begin{enumerate}
\item for $Y\in[0,+\infty)$, we have
\begin{align*}
&|\mathcal{A}_f(Y)|\leq C\alpha^{-1}\Big(\|(1+Y)^2a\|_{L^\infty([0,Y])}+\|(1+Y)^2b\|_{L^\infty_{2w_0}([0,Y])}\Big),\\
&|e^{2\alpha w_0(Y)}\mathcal{B}_f(Y)|\leq C\al^{-1}\Big(\|a\|_{L^\infty([0, Y^{+}+\d])}+\al^{-1}\|a\|_{L^\infty}+\|b\|_{L^\infty_{2 w_0}}\Big).
\end{align*}

\item  for $Y\in[0,\bar{Y}_2^*+\delta_0]$, we have
\begin{align*}
&|\mathrm{Im}(\mathcal{A}_f(Y))|\leq C\alpha^{-1}\Big(\|(1+Y)^2\Im a\|_{L^\infty([0,Y])}+\|(1+Y)^2\Im b\|_{L^\infty_{2w_0}([0,Y])}\Big),\\
&|e^{2\alpha w_0(Y)}\mathrm{Im}(\mathcal{B}_f(Y))|\leq C\al^{-2}e^{2\al w_0(Y)}e^{-2\al w_0(\bar Y_2^*+\d_0)}\Big(\|a\|_{L^\infty}+\|b\|_{L^\infty_{2 w_0}}\Big)\\
&\qquad\qquad\qquad\qquad\qquad+C\al^{-1}\Big(\|\Im a\|_{L^\infty([0, Y^{+}+\d])}+\al^{-1}\|\Im a\|_{L^\infty}+\|\Im b\|_{L^\infty_{2 w_0}}\Big).
\end{align*}

\end{enumerate}

\end{proposition}

\begin{proof}
By \eqref{def: (A,B)} and \eqref{def: cA_f, cB_f}, $\mathcal{A}_f(Y)$ and $\mathcal{B}_f(Y)$ with $f(Y)=a(Y)A(Y)+b(Y)B(Y)$ are given by 
\begin{align}\label{def: (cA_f, cB_f)}
\begin{split}
\mathcal{A}_f(Y)=-\kappa^{-1}\pi\int_0^YBi(\kappa\eta(Z))\big(a(Z) Ai(\kappa\eta(Z))+b(Z) Bi(\kappa\eta(Z))\big)\partial_Z\eta(Z)^{-1}dZ,\\
\mathcal{B}_f(Y)=-\kappa^{-1}\pi\int_Y^{+\infty}Ai(\kappa\eta(Z))\big(a(Z)Ai(\kappa\eta(Z))+b(Z)Bi(\kappa\eta(Z))\big)\partial_Z\eta(Z)^{-1}dZ.
\end{split}
\end{align}
For convenience, we drop subscript in $\mathcal{A}_f$ and $\mathcal{B}_f$, and write $\mathcal{A}$ and $\mathcal{B}$ for short, and  assume $a(Y)$ and $b(Y)$ are two real functions. For the case when $a(Y)$ and $b(Y)$ are complex functions, we use the decomposition \eqref{relation: (cA, cB)} and relation \eqref{relation: (cA^1, cB^1)-(cA^2, cB^2)} to get the final result.\smallskip

\underline{Estimates of $\mathcal{A}(Y)$. } For $Y\in[0, Y^{-}]$, we use the first statement in Lemma \ref{lem: Ai, Bi} to get
\begin{align}\label{eq:A_r-Y-}
\begin{split}
|\mathcal{A}(Y)|\leq& C\kappa^{-\f32}\int_0^Y\left(|a(Z)Ai(\kappa \eta(Z)) Bi(\kappa\eta(Z))|+|b(Z)Bi(\kappa\eta(Z))^2|\right)|(\pa_Y\eta)^{-1}|dZ\\
\leq& C\kappa^{-\f32}\int_0^Y(|a(Z)|+|b(Z)|)\max\{L^{-\f12}, |Z-Y_0|^{-\f12}\}dZ\\
\leq& C\alpha^{-1}(\|a\|_{L^\infty([0, Y])}+\|b\|_{L^\infty([0, Y])}),
\end{split}
\end{align}
by using the fact $\int_0^Y\max\{L^{-\f12}, |Z-Y_0|^{-\f12}\}dZ\leq C$ and \eqref{relation kappa and alpha} in the last line.

For any $Y\in[Y^{-}, Y^{+}]$, we notice that 
\begin{align*}
\mathcal{A}(Y)
=&\mathcal{A}(Y^{-})-\kappa^{-1}\pi\int_{Y^{-}}^YBi(\kappa\eta(Z))\left(a(Z) Ai(\kappa\eta(Z))+b(Z) Bi(\kappa\eta(Z))\right)\partial_Z\eta(Z)^{-1}dZ\\
=&\mathcal{A}(Y^{-})+I_1.
\end{align*}
From \eqref{eq:A_r-Y-}, we can obtain 
$
|\mathcal{A}(Y^{-})|
\leq C\alpha^{-1}(\|a\|_{L^\infty([0, Y^{-}])}+\|b\|_{L^\infty([0, Y^{-}])}).
$
According to Lemma \eqref{lem:est-eta}, the second statement in Lemma \ref{lem: Ai, Bi} and \eqref{relation kappa and alpha}, we get
\begin{align*}
|I_1|
\leq& C\kappa^{-1}\left(\|a Ai(\kappa \eta) Bi(\kappa\eta)\|_{L^\infty([Y^-,Y])}+\|bBi(\kappa\eta)^2\|_{L^\infty([Y^-,Y])}\right)\int_{Y^{-}}^{Y}1dZ\\
\leq& C\alpha^{-\f43}(\|a\|_{L^\infty([Y^{-},Y])}+\|b\|_{L^\infty([Y^{-},Y])}).
\end{align*}
Therefore, we infer that for any $Y\in[Y^{-}, Y^{+}]$,
\begin{align}\label{eq:A-N+1}
|\mathcal{A}(Y)|\leq C\alpha^{-1}(\|a\|_{L^\infty([0,Y])}+\|b\|_{L^\infty([0,Y])}).
\end{align}

For  $Y\in[Y^{+},+\infty)$, we notice that 
\begin{align*}
\mathcal{A}(Y)=&\mathcal{A}(Y^{+})
-\kappa^{-1}\pi\int_{Y^+}^YBi(\kappa\eta(Z))(a(Z) Ai(\kappa\eta(Z))+b(Z) Bi(\kappa\eta(Z)))\partial_Z\eta(Z)^{-1}dZ\\
=&\mathcal{A}(Y^{+})+I_2.
\end{align*}
According to \eqref{eq:A-N+1}, we first have 
$
|\mathcal{A}(Y^{+})|\leq
C\alpha^{-1}(\|a\|_{L^\infty([0, Y^{+}])}+\|b\|_{L^\infty([0, Y^{+}])}).
$
Using the third statement in Lemma \ref{lem: Ai, Bi}, we have 
\begin{align*}
|I_2|
\leq& C\kappa^{-\f32}\|(1+Y)^2(a, e^{2\al w_0}b)\|_{L^\infty([Y^{+}, Y])}\int_{Y^+}^{Y}(1+Z)^{-2}\max\{L^{-\f12}, |Z-Y_0|^{-\f12}\}dZ\\
\leq& C\alpha^{-1}\|(1+Y)^2(a, e^{2\al w_0}b)\|_{L^\infty([Y^{+}, Y])}.
\end{align*}
Therefore, we obtain that for any $Y\in[Y^+,+\infty)$,
\begin{align}\label{eq:A-N+}
|\mathcal{A}(Y)|\leq C\alpha^{-1}(\|(1+Y)^2a\|_{L^\infty([0, Y])}+\|(1+Y)^2b\|_{L^\infty_{2w_0}([0, Y])}).
\end{align}

We collect \eqref{eq:A_r-Y-}, \eqref{eq:A-N+1} and \eqref{eq:A-N+} together to deduce that for $Y\in[0, +\infty)$,
\begin{align*}
|\mathcal{A}(Y)|\leq C\alpha^{-1}(\|(1+Y)^2a\|_{L^\infty([0, Y])}+\|(1+Y)^2b\|_{L^\infty_{2w_0}([0, Y])}),
\end{align*}
which gives the first estimate in the statement (1).

By statement (4) in Lemma \ref{lem: Ai, Bi} and  real functions $a(Y),~b(Y)$, we obtain  $\Im(\mathcal{A}(Y))=0$ for $Y\in[0,\bar{Y}_2^*+\delta_0]$, {which along with \eqref{relation: (cA^1, cB^1)-(cA^2, cB^2)} implies the first estimate in result (2)}.\smallskip

\underline{Estimates of $\mathcal{B}(Y)$.}
For $Y\in[Y^{+}+\d,+\infty)$, we notice that 
\begin{align*}
\mathcal{B}(Y)=-\kappa^{-1}\pi\int_Y^{+\infty}Ai(\kappa\eta(Z))(a(Z)Ai(\kappa\eta(Z))+b(Z)Bi(\kappa\eta(Z)))\partial_Z\eta(Z)^{-1}dZ.
\end{align*}
By statement (3) in Lemma \ref{lem: Ai, Bi}, we have that for any $Z\in [Y^{+}+\delta,+\infty)$,
\begin{align*}
|Ai(\kappa \eta) Bi(\kappa\eta)(\pa_Z\eta)^{-1}|+|e^{2\al w_0}Ai(\kappa \eta) ^2(\pa_Z\eta)^{-1}|\leq C\kappa^{-\f12}.
\end{align*}
Thus, we can obtain that for any $Y\in [Y^{+}+\delta,+\infty)$,
\begin{align}\label{eq:B-Y2}
\begin{split}
|\mathcal{B}(Y)|\leq& C\kappa^{-\f32}\int_Y^{+\infty}\left(|a(Z)|e^{-2\alpha w_0(Z)}+|b(Z)|\right)dZ\\
\leq&C\kappa^{-\f32}\left(\|a\|_{L^\infty([Y, \infty))}+ \|e^{2\alpha w_0}b\|_{L^\infty([Y, \infty))}\right)\int_Y^{+\infty}e^{-2\alpha w_0(Z)}dZ\\
\leq&\alpha^{-2}e^{-2\alpha w_0(Y)}(\|a\|_{L^\infty}+ \|b\|_{L^\infty_{2 w_0}}).
\end{split}
\end{align}

For $Y\in[Y^{-},Y^{+}+\d]$, we notice that 
\begin{align*}
\mathcal{B}_f(Y)=&\mathcal{B}_f(Y^++\d)-\kappa^{-1}\pi\int^{Y^++\d}_YAi(\kappa\eta(Z))\left(a(Z)Ai(\kappa\eta(Z))+b(Z)Bi(\kappa\eta(Z))\right)\partial_Z\eta(Z)^{-1}dZ\\
=&\mathcal{B}_f(Y^++\d)+I_4.
\end{align*}
By results (2)-(3) in Lemma \ref{lem: Ai, Bi}, we know that for $Z\in[Y^{-},Y^{+}+\d]$,
\begin{align*}
&|Ai(\kappa\eta(Z))^2(\pa_Z\eta)^{-1}|+|Ai(\kappa\eta(Z))Bi(\kappa\eta(Z))(\pa_Z\eta)^{-1}|\leq C,
\end{align*}
which along with the fact $\int_{Y}^{Y^++\d}e^{-\al w_0(Z)}dZ\leq C\kappa^{-1}$ implies 
\begin{align*}
|I_4|
\leq& C\kappa^{-1}\|(a,b)\|_{L^\infty([Y, Y^{+}+\d])}\Big(\int_{Y}^{Y^++\d}e^{-\al w_0(Z)}dZ\Big)
\leq C\alpha^{-\f43}\|(a,b)\|_{L^\infty([Y, Y^{+}+\d_0])}.
\end{align*}
Then by \eqref{eq:B-Y2} and the above estimates, we obtain that for $Y\in[Y^{-}, Y^{+}+\d]$,
\begin{align}\label{eq:B-N}
|\mathcal{B}(Y)|\leq& C\al^{-2}\Big(\|a\|_{L^\infty}+\|b\|_{L^\infty_{2 w_0}}\Big)+C\al^{-\f43}\Big(\|a\|_{L^\infty([Y,Y^{+}+\d ])}+\|b\|_{L^\infty_{2 w_0}([Y,Y^{+}+\d])}\Big).
\end{align}

For $Y\in [0, Y^{-}]$, we  write 
\begin{align*}
\mathcal{B}(Y)=&\mathcal{B}(Y^-)-\kappa^{-1}\pi\int^{Y^-}_YAi(\kappa\eta(Z))(a(Z)Ai(\kappa\eta(Z))+b(Z)Bi(\kappa\eta(Z)))\partial_Z\eta(Z)^{-1}dZ\\
=&\mathcal{B}(Y^-)+I_5.
\end{align*}
By statement (1) in Lemma \ref{lem: Ai, Bi}, we obtain
\begin{align*}
|I_5|
\leq& C\kappa^{-\f32}(\|a\|_{L^\infty([0, Y^{-}])}+\|b\|_{L^\infty([0, Y^{-}])})|\eta(0)|^\f12
\leq C\alpha^{-1}(\|a\|_{L^\infty([0, Y^{-}])}+\|b\|_{L^\infty([0, Y^{-}])}),
\end{align*}
which along with \eqref{eq:B-N} implies that for $Y\in [0, Y^{-}]$,
\begin{align}\label{eq:B-N-}
|\mathcal{B}(Y)|\leq & C\al^{-2}\Big(\|a\|_{L^\infty}+\|b\|_{L^\infty_{2 w_0}}\Big)+C\al^{-1}\Big(\|a\|_{L^\infty([0, Y^{+}+\d])}+\|b\|_{L^\infty_{2 w_0}([0, Y^{+}+\d])}\Big).
\end{align}
Thus, we collect \eqref{eq:B-Y2}, \eqref{eq:B-N} and \eqref{eq:B-N-} to deduce that for $Y\in[0,+\infty)$,
\begin{align*}
|e^{2\alpha w_0(Y)}\mathcal{B}_f(Y)|\leq C\al^{-1}\Big(\|a\|_{L^\infty([0, Y^{+}+\d])}+\al^{-1}\|a\|_{L^\infty}+\|b\|_{L^\infty_{2 w_0}}\Big),
\end{align*}
which gives the second estimate in the statement (1). 

For $Y\in[0, \bar{Y}_2^*+\delta_0]$, we write 
\begin{align*}
\mathcal{B}(Y)=\mathcal{B}(\bar{Y}_2^*+\delta_0)-\kappa^{-1}\pi\int^{\bar{Y}_2^*+\delta_0}_YAi(\kappa\eta(Z))\left(a(Z)Ai(\kappa\eta(Z))+b(Z)Bi(\kappa\eta(Z))\right)\partial_Z\eta(Z)^{-1}dZ.
\end{align*}
According to Lemma \ref{lem: Ai, Bi}, the imaginary part of the second term above is zero, which implies $\Im(\mathcal{B}(Y))=\Im(\mathcal{B}(\bar{Y}_2^*+\delta_0))$. Then we  use \eqref{eq:B-Y2} to obtain that for $Y\in[0,\bar{Y}_2^*+\delta_0]$,
\begin{align*}
|\mathrm{Im}(\mathcal{B}(Y))|\leq C\alpha^{-2}e^{-2\alpha w_0(\bar{Y}_2^*+\delta_0)}(\|a\|_{L^\infty}+ \|b\|_{L^\infty_{2w_0}}),
\end{align*}
which along with \eqref{relation: (cA^1, cB^1)-(cA^2, cB^2)} implies the second estimate in the statement (2).
\end{proof}

\subsection{Solving the mixed type system with special source terms} 
First, we give estimates for the first kind of source term in \eqref{def: special source terms}.

\begin{proposition}\label{pro: Airy-2}
	Let $(\alpha, c)\in\mathbb H_0$ and $g(Y)$ be a complex-valued function with $e^{\al w_0}|g|\in L^\infty$ and  $\mathrm{supp}(g)\subseteq[Y_1, Y_1+\d_0]\cup[Y_2-\d_0, Y_2]$. Then there exists a solution $\tilde P\in W^{2,\infty}_{w_0}$ to the system \eqref{sys:app-f-1} with the source term $K_g(Y)=(1-\chi)\int_Y^{+\infty}g(Z)dZ$. More precisely, 
there exist $\mathcal{A}_{K_g}(Y)$ and $\mathcal{B}_{K_g}(Y)$ with $(|\cA_{Kg}|, |e^{2\al w_0} \cB_{Kg}|)\in L^\infty$ such that 
\begin{align*}
\tilde{P}(Y)=\mathcal{A}_{K_g}(Y)A(Y)+\mathcal{B}_{K_g}(Y)B(Y),
\end{align*}
and it holds that $\mathrm{supp}(\mathcal{B}_{K_g})=[0,Y_1-\delta_0]$ and
\begin{align*}
&|\mathcal{A}_{K_g}(Y)|+|e^{2\alpha w_0(Y)}\mathcal{B}_{K_g}(Y)|\leq C \al^{-\f{17}{6}}e^{-\alpha\int_{Y_1-\d_0}^{Y_1} F_r^\f12(Z)dZ}\|g\|_{L^\infty_{w_0}}, \\
&|\mathrm{Im}(\mathcal{A}_{Kg}(Y))|+|e^{2\alpha w_0(Y)}\mathrm{Im}(\mathcal{B}_{Kg}(Y))|\leq C\alpha^{-\f{17}{6}}e^{-\alpha\int_{Y_1-\d_0}^{Y_1} F_r^\f12(Z)dZ}(c_i\|g\|_{L^\infty_{w_0}}+\|\Im g\|_{L^\infty_{w_0}}).
\end{align*}
Moreover, for $Y\in[Y_1-\d_0,+\infty)$, it holds that $\mathcal{A}_{K_g}(Y)=\mathcal{A}_{K_g}(Y_1-\d_0).$

%
\end{proposition}

\begin{proof}
According to $\mathrm{supp}(g)=[Y_1, Y_1+\d_0]\cup[Y_2-\d_0, Y_2]$ and $\mathrm{supp}(1-\chi)=[0,Y_1-\delta_0]\cup[Y_2+\delta_0,+\infty)$, we have 
\begin{align*}
K_g(Y)=\left\{
\begin{aligned}
&(1-\chi)\int_{Y_1}^{Y_2}g(Z)dZ,\quad Y\in[0,Y_1-\delta_0],\\
&0,\quad Y\geq Y_1-\delta_0 .
\end{aligned}
\right.
\end{align*}
Then we can obtain that for $Y\in[0,Y_1-\delta_0]$,
\begin{align}\label{eq:g-nonlocal}
\begin{split}
\left|K_g(Y)\right|\leq& \int_{Y_1}^{Y_2}|g(Z)|dZ\leq C\|g\|_{L^\infty_{w_0}}\int_{Y_1}^{Y_2}e^{-\alpha w_0(Z)}dZ
\leq C\alpha^{-1}e^{-\alpha w_0(Y_1)}\|g\|_{L^\infty_{w_0}}.
\end{split}
\end{align}
By the definition in \eqref{def: cA_f, cB_f}, we write
\begin{align}\label{def: (cA_Kg, cB_Kg)}
\begin{split}
&\mathcal{A}_{Kg}(Y)=-\kappa^{-1}\pi\int_0^YBi(\kappa\eta(Z))K_g(Z)\bar M^{-1}(Z)\partial_Z\eta(Z)^{-\f12} dZ,\\
&\mathcal{B}_{Kg}(Y)=-\kappa^{-1}\pi\int_Y^{+\infty}Ai(\kappa\eta(Z))K_g(Z)\bar M^{-1}(Z)\partial_Z\eta(Z)^{-\f12}dZ.
\end{split}
\end{align}
For convenience, we drop subscript in $\mathcal{A}_{Kg}$ and $\mathcal{B}_{Kg}$, and write $\mathcal{A}$ and $\mathcal{B}$ for short, we assume $g$ is a real function. For the case when $g$ is a complex function, we use the similar decomposition \eqref{relation: (cA, cB)} and relation \eqref{relation: (cA^1, cB^1)-(cA^2, cB^2)} to get the final result.\smallskip

\underline{Estimates of $\mathcal{A}(Y)$.} For $Y\in[0, Y^{-}]$, we get by Lemma \ref{lem: Ai, Bi}  that
\begin{align}\label{est: Airy-2, Ai-1}
\begin{split}
|\mathcal{A}(Y)|\leq& C\kappa^{-1}\alpha^{-1}e^{-\alpha w_0(Y_1)}\|g\|_{L^\infty_{w_0}}\int_0^Y |Bi(\kappa\eta(Z))(\pa_Z\eta(Z))^{-\f12}|dZ\\
\leq& C\alpha^{-\f{11}{6}}e^{-\alpha w_0(Y_1)}\|g\|_{L^\infty_{w_0}}.
\end{split}
\end{align}
For $Y\in[Y^{-}, Y^{+}+\d]$, we write
\begin{align*}
\mathcal{A}(Y)=\mathcal{A}(Y^{-})-\kappa^{-1}\pi\int_{Y^{-}}^YBi(\kappa\eta(Z))K_g(Z)\bar M^{-1}(Z)\partial_Z\eta(Z)^{-\f12} dZ=\mathcal{A}(Y^{-})+I_0.
\end{align*}
By \eqref{est: Airy-2, Ai-1} and  Lemma \ref{lem: Ai, Bi}, we have
\begin{align*}
|\mathcal{A}(Y^{-})|\leq C\alpha^{-\f{11}{6}}e^{-\alpha w_0(Y_1)}\|g\|_{L^\infty_{w_0}},\quad |I_0|\leq C\alpha^{-\f{7}{3}}e^{-\alpha w_0(Y_1)}e^{\al w_0(Y)}\|g\|_{L^\infty_{w_0}},
\end{align*}
which gives that for $Y\in[Y^{-}, Y^{+}+\d]$,
\begin{align}\label{est: Airy-2, Ai-2}
|\mathcal{A}(Y)|\leq C\alpha^{-\f{7}{3}}e^{-\al \int_{Y^++\d}^{Y_1} F_r^\f12(Z)dZ}\|g\|_{L^\infty_{w_0}}.
\end{align}

For $Y\in[Y^{+}+\d,Y_1-\delta_0]$, we notice that 
\begin{align*}
\mathcal{A}(Y)=\mathcal{A}(Y^++\d)-\kappa^{-1}\pi \int_{Y^++\d}^YBi(\kappa\eta(Z))K_g(Z)\bar M^{-1}(Z)\partial_Z\eta(Z)^{-\f12}dZ=\mathcal{A}(Y^++\d)+I_1.
\end{align*}
By Lemma \ref{lem: Ai, Bi}, we know that for $Z\geq Y^{+}+\d$, $$|e^{-\al w_0(Z)}Bi(\kappa\eta(Z))(\pa_Z\eta)^{-\f12}|\leq C\kappa^{-\f14},$$
which along with \eqref{eq:g-nonlocal} implies 
\begin{align*}
|I_1|\leq&  C\kappa^{-\f54}\alpha^{-1}e^{-\alpha w_0(Y_1)}\|g\|_{L^\infty_{w_0}}\int_{Y^++\d}^Ye^{\alpha w_0(Z)} dZ
\leq C\alpha^{-\f{17}{6}}e^{-\alpha\int_Y^{Y_1}F_r^\f12(Z)dZ}\|g\|_{L^\infty_{w_0}}.
\end{align*}
Thus, we use  \eqref{est: Airy-2, Ai-2} to infer that for $Y\in[Y^{+}+\d,Y_1-\delta_0]$,
\begin{align}\label{est: Airy-2, Ai-3}
|\mathcal{A}(Y)|\leq C\alpha^{-\f{17}{6}}e^{-\alpha\int_{Y_1-\d_0}^{Y_1}F_r^\f12(Z)dZ}\|g\|_{L^\infty_{w_0}}.
\end{align}
Moreover, by \eqref{eq:g-nonlocal}, we know that $\mathcal{A}(Y)=\mathcal{A}(Y_1-\delta_0)$ for any $Y> Y_1-\delta_0$. From which and \eqref{est: Airy-2, Ai-1}-\eqref{est: Airy-2, Ai-3},  we deduce that for any $Y\geq 0$,
\begin{align*}
|\mathcal{A}(Y)|\leq C\alpha^{-\f{17}{6}}e^{-\alpha\int_{Y_1-\d_0}^{Y_1} F_r^\f12(Z)dZ}\|g\|_{L^\infty_{w_0}}.
\end{align*}

\underline{Estimates of $\mathcal{B}(Y)$.} Again by \eqref{eq:g-nonlocal}, we have that $\mathcal{B}(Y)\equiv0$ for any $Y>Y_1-\delta_0$. 
For $Y\in [Y^{+}+\d ,Y_1-\delta_0]$, we first notice that 
\begin{align*}
\mathcal{B}(Y)=-\kappa^{-1}\pi\int_Y^{Y_1-\d_0}Ai(\kappa\eta(Z))K_g(Z)\bar M^{-1}(Z)\partial_Z\eta(Z)^{-\f12}dZ.
\end{align*}
By Lemma \ref{lem: Ai, Bi}, we know that for $Z\in[Y^{+}+\d, Y_1-\d_0]$,
\begin{align*}
|e^{\al w_0}Ai(\kappa \eta(Z)) (\pa_Z\eta)^{-\f12}|\leq C\kappa^{-\f14},
\end{align*}
which along with \eqref{eq:g-nonlocal} implies that for $Y\in[Y^{+}+\d,Y_1-\delta_0]$,
\begin{align*}
|\mathcal{B}(Y)|\leq& C\kappa^{-\f54}\alpha^{-1}e^{-\al w_0(Y_1)}\|g\|_{L^\infty_{w_0}}\int_{Y}^{Y_1-\delta_0}e^{-\al w_0(Z)}dZ\\
\leq& C\alpha^{-\f{17}{6}} e^{-\alpha\int_Y^{Y_1}F_r^\f12(Z)dZ}e^{-2\al w_0(Y)}\|g\|_{L^\infty_{w_0}}.
\end{align*}

For $Y\in[Y^{-}, Y^{+}+\d]$, we can write 
\begin{align*}
\mathcal{B}(Y)=&\mathcal{B}(Y^++\d)-\kappa^{-1}\pi\int_Y^{Y^++\d}Ai(\kappa\eta(Z))K_g(Z)\bar M^{-1}(Z)\partial_Z\eta(Z)^{-\f12}dZ.
\end{align*}
By Lemma \ref{lem: Ai, Bi}  and similar procedure of $\mathcal{A}(Y)$ in $[Y^{-}, Y^{+}+\d]$,  we have for $Y\in[Y^{-}, Y^{+}+\d]$,
\begin{align*}
|\mathcal{B}(Y)|\leq &C\alpha^{-\f{17}{6}} e^{-\alpha\int_{Y^{+}+\d}^{Y_1}F_r^\f12(Z)dZ}e^{-2\al w_0(Y^{+}+\d)}\|g\|_{L^\infty_{w_0}}+C\al^{-\f73}e^{-\al w_0(Y_1)}e^{-\al w_0(Y)}\|g\|_{L^\infty_{w_0}}\\
\leq&C\al^{-\f73}e^{-\al w_0(Y_1)}e^{-\al w_0(Y)}\|g\|_{L^\infty_{w_0}}.
\end{align*}

For $Y\in[0, Y^{-}]$, we write 
\begin{align*}
\mathcal{B}(Y)=\mathcal{B}(Y^-)-\kappa^{-1}\pi\int_Y^{Y^-}Ai(\kappa\eta(Z))K_g(Z)\bar M^{-1}(Z)\partial_Z\eta(Z)^{-\f12}dZ.
\end{align*}
By Lemma \ref{lem: Ai, Bi} and a similar argument as above, we can obtain that for $Y\in[0, Y^{-}]$,
\begin{align*}
|\mathcal{B}(Y)|\leq C\kappa^{-1}\kappa^{-\f14}\al^{-1}e^{-\al w_0(Y_1)}\|g\|_{L^\infty_{w_0}}+
C\al^{-\f73}e^{-\al w_0(Y_1)}\|g\|_{L^\infty_{w_0}}\leq C\al^{-\f{11}{6}}e^{-\al w_0(Y_1)}\|g\|_{L^\infty_{w_0}}.
\end{align*}
Thus, we infer that for any $Y\geq 0$,
\begin{align*}
|\mathcal{B}(Y)|\leq C\al^{-\f{17}{6}}e^{-\alpha\int_{Y_1-\d_0}^{Y_1} F_r^\f12(Z)dZ}e^{-2\al w_0(Y)}\|g\|_{L^\infty_{w_0}}.
\end{align*}

On the other hand, by the definition of $\bar{M},~F(Y)$ and $K_g(Y)=0$ for $Y\geq Y_1-\delta_0$, we know that the imaginary part of $\mathcal{A}$ and $\mathcal{B}$  comes from the function $\bar{M}$. That is,
\begin{align*}
&\mathrm{Im}(\mathcal{A}(Y))=-\kappa^{-1}\pi \int_0^YBi(\kappa\eta(Z))K_g(Y)\mathrm{Im}(\bar M^{-1}(Z))\partial_Z\eta(Z)^{-\f12}dZ,\\
&\mathrm{Im}(\mathcal{B}(Y))=-\kappa^{-1}\pi\int_Y^{+\infty}Ai(\kappa\eta(Z))K_g(Y)\mathrm{Im}(\bar M^{-1}(Z))\partial_Z\eta(Z)^{-\f12}dZ.
\end{align*}
Moreover, for any $Y\in[0,Y_1-\delta_0]$, we have $|\mathrm{Im}(\bar{M}(Y)^{-1})|\leq Cc_i.$
Therefore, by a similar argument as above and the decomposition \eqref{relation: (cA, cB)}-\eqref{relation: (cA^1, cB^1)-(cA^2, cB^2)}, we  obtain the estimates for the imaginary part.
\end{proof}

 For the second kind of source term in \eqref{def: special source terms}, we have the following results.
 
\begin{proposition}\label{pro: Airy-3}
Let $(\alpha, c)\in\mathbb H_0$ and $f(Y)$ be a complex-valued function  with $e^{\al w_0}|f|\in L^\infty$ and $\mathrm{supp}(f)\subseteq[Y_1^*-\d_0, Y_1^*]\cup [Y_2^*, Y_2^*+\d_0]$. There exists a solution $\tilde P\in W^{2,\infty}_{w_0}$ to the system \eqref{sys:app-f-1} with the source term $f(Y)$. More precisely, there exist $\mathcal{A}_f(Y)$ and $\mathcal{B}_f(Y)$ with $(|\cA_f|, |e^{2\al w_0} \cB_f|)\in L^\infty$ such that 
\begin{align*}
\tilde{P}(Y)=\mathcal{A}_f(Y)A(Y)+\mathcal{B}_f(Y)B(Y),
\end{align*}
and the following estimates hold

\begin{enumerate}
\item
for $Y\geq 0$, we have
\begin{align*}
&|\mathcal{A}_f(Y)|\leq C\alpha^{-\f56}\|f\|_{L^\infty_{w_0}},\quad\mathrm{Supp}(\mathcal{A}_f)\subseteq[Y_1^*-\delta_0,+\infty),\\
&|\mathcal{B}_f(Y)|\leq C\alpha^{-\f{11}{6}}e^{-2\alpha w_0(Y)}\|f\|_{L^\infty_{w_0}},\quad\mathrm{Supp}(\mathcal{B}_f)\subseteq[0,Y_2^*+\delta_0].
\end{align*}

\item in particular, for $Y\in[0, Y_2^*]$, we have
\begin{align*}
&|\mathcal{A}_f(Y)|\leq C\alpha^{-\f56}\|f\|_{L^\infty_{w_0}([Y_1^*-\delta_0,Y^*_1])}.
\end{align*}

\item in particular,  for $Y\in[Y_1^*-\d_0,Y_2^*]$, we have
\begin{align*}
&|\mathcal{B}_f(Y)|\leq C\alpha^{-\f{11}{6}}e^{-2\alpha w_0(Y)}\Big(\|f\|_{L^\infty_{w_0}([Y_1^*-\delta_0,Y^*_1])}+e^{-2\al \int_{Y}^{Y_2^*} F_r^\f12(Z)dZ}\|f\|_{L^\infty_{w_0}}\Big),
\end{align*}
and for $Y\in[0, Y_1^*-\d_0]$, we have
\begin{align*}
&|\mathcal{B}_f(Y)|\leq C\alpha^{-\f{11}{6}}\Big(e^{-2\alpha w_0(Y_1^*-\delta_0)}\|f\|_{L^\infty_{w_0}([Y_1^*-\delta_0,Y^*_1])}+e^{-2\alpha w_0(Y_2^*)}\|f\|_{L^\infty_{w_0}}\Big).
\end{align*}
\end{enumerate}

\end{proposition}

\begin{proof}
By\eqref{def: cA_f, cB_f}, we have
\begin{align*}
&\mathcal{A}_f(Y)=-\kappa^{-1}\pi \int_0^YBi(\kappa\eta(Z))f(Z)\bar{M}^{-1}(Z)\partial_Z\eta(Z)^{-\f12} dZ,\\
&\mathcal{B}_f(Y)=-\kappa^{-1}\pi\int_Y^{+\infty}Ai(\kappa\eta(Z))f(Z)\bar{M}^{-1}(Z)\partial_Z\eta(Z)^{-\f12}dZ.
\end{align*}
Again we drop subscript in $\mathcal{A}_f$ and $\mathcal{B}_f$, and write $\mathcal{A}$ and $\mathcal{B}$ for short, and  assume $f$ is a real function.

Recall that $\mathrm{supp}(f)=[Y_1^*-\delta_0,Y_1^*]\cup[Y_2^*,Y_2^*+\delta_0]$. Then we have
\begin{align*}
\mathrm{supp}(\mathcal{A})\subseteq[Y_1^*-\delta_0,+\infty),\quad \mathrm{supp}(\mathcal{B})\subseteq[0,Y_2^*+\delta_0].
\end{align*}

\underline{Estimates of $\mathcal{A}(Y)$.} For $Y\in[Y_1^*-\delta_0,Y_1^*]$, we notice that 
\begin{align*}
\mathcal{A}(Y)=-\kappa^{-1}\pi \int_{Y_1^*-\delta_0}^YBi(\kappa\eta(Z))f(Z)\bar{M}^{-1}(Z)\partial_Z\eta(Z)^{-\f12} dZ.
\end{align*}
By the fact $[Y_1^*-\delta_0,Y_1^*]\subset[Y^{+}, +\infty)$ and Lemma \ref{lem: Ai, Bi}, we derive  that for $Y\in [Y_1^*-\delta_0,Y_1^*]$,
\begin{align*}
|\mathcal{A}(Y)|\leq& C\kappa^{-1}\kappa^{-\f14}\int_{Y_1^*-\delta_0}^Ye^{\alpha w_0(Z)}|f(Z)|dZ
\leq C\alpha^{-\f56}\|f\|_{L^\infty_{w_0}([Y_1^*-\delta_0,Y^*_1])}.
\end{align*}
Notice that $\mathcal{A}(Y)=\mathcal{A}(Y_1^*)$ for $Y\in[Y_1^*,Y_2^*]$. Then we infer that for any $Y\in[0,Y_2^*]$,
\begin{align}\label{eq:A-f-Y1*}
|\mathcal{A}(Y)|
\leq C\alpha^{-\f56}\|f\|_{L^\infty_{w_0}([Y_1^*-\delta_0,Y^*_1])}.
\end{align}

For $Y\in[Y_2^*,+\infty)$, we  write 
\begin{align*}
\mathcal{A}(Y)=\mathcal{A}(Y_2^*)-\kappa^{-1}\pi \int_{Y_2^*}^YBi(\kappa\eta(Z))f(Z)\bar{M}^{-1}(Z)\partial_Z\eta(Z)^{-\f12}dZ=\mathcal{A}(Y_2^*)+I_1.
\end{align*}
Again by Lemma \ref{lem: Ai, Bi}, we obtain that for $Y\in[Y_2^*,+\infty)$,
\begin{align*}
|I_1|
&\leq C\kappa^{-\f54}\int_{Y_2^*}^Ye^{\alpha w_0(Z)}|f(Z)|dZ\leq C\alpha^{-\f56}\|f\|_{L^\infty_{w_0}}.
\end{align*}
Therefore, we obtain that for $Y\in[Y_2^*,+\infty)$,
\begin{align}\label{eq:A-f-Y2*}
|\mathcal{A}(Y)|\leq C\alpha^{-\f56}\|f\|_{L^\infty_{w_0}}.
\end{align}
Then by \eqref{eq:A-f-Y1*} and \eqref{eq:A-f-Y2*}, we obtain that
\begin{align*}
|\mathcal{A}(Y)|\leq 
\left\{
\begin{aligned}
&C\alpha^{-\f56}\|f\|_{L^\infty_{w_0}},\quad Y\geq 0,\\
&C\alpha^{-\f56}\|f\|_{L^\infty_{w_0}([Y_1^*-\delta_0,Y^*_1])},\quad Y\in[0, Y_2^*],
\end{aligned}
\right.
\end{align*}

\underline{Estimates of $\mathcal{B}(Y)$.} For $Y\in[Y_2^*,Y_2^*+\delta_0]$, we have 
\begin{align*}
\mathcal{B}(Y)=-\kappa^{-1}\pi\int_Y^{Y_2^*+\delta_0}Ai(\kappa\eta(Z))f(Z)\bar{M}^{-1}(Z)\partial_Z\eta(Z)^{-\f12}dZ.
\end{align*}
By Lemma \ref{lem: Ai, Bi}, we have that for $Y\in[Y_2^*,Y_2^*+\delta_0]$,
\begin{align*}
|\mathcal{B}(Y)|\leq& C\kappa^{-\f54}\int_Y^{Y_2^*+\delta_0}e^{-\alpha w_0(Z)}|f(Z)|dZ\\
\leq& C\kappa^{-\f54}\|f\|_{L^\infty_{w_0}}\int_Y^{Y_2^*+\delta_0}e^{-2\alpha w_0(Z)}dZ
\leq C\alpha^{-\f{11}{6}}e^{-2\alpha w_0(Y)}\|f\|_{L^\infty_{w_0}}.
\end{align*}
Notice that $\mathcal{B}(Y)=\mathcal{B}(Y_2^*)$ for $Y\in[Y_1^*,Y_2^*]$. Then we have for $Y\in[Y_1^*,Y_2^*]$,
\begin{align*}
|\mathcal{B}(Y)|\leq&C\alpha^{-\f{11}{6}}e^{-2\alpha w_0(Y_2^*)}\|f\|_{L^\infty_{w_0}}
\leq C\alpha^{-\f{11}{6}}e^{-2\alpha w_0(Y)}e^{-2\al \int_{Y}^{Y_2^*} F_r^\f12(Z)dZ}\|f\|_{L^\infty_{w_0}}.
\end{align*}

For $Y\in[Y_1^*-\delta_0,Y_1^*]$, we can write 
\begin{align*}
\mathcal{B}(Y)=\mathcal{B}(Y_1^*)-\kappa^{-1}\pi\int_Y^{Y_1^*}Ai(\kappa\eta(Z))f(Z)\bar{M}^{-1}(Z)\partial_Z\eta(Z)^{-\f12})dZ=\mathcal{B}(Y_1^*)+I_2.
\end{align*}
Again by  Lemma \ref{lem: Ai, Bi}, we have that for $Y\in[Y_1^*-\delta_0,Y_1^*]$,
\begin{align*}
|I_2|
&\leq C\kappa^{-\f54}\|f\|_{L^\infty_{w_0}([Y_1^*-\delta_0,Y^*_1])}\int_{Y}^{Y_1^*}e^{-2\alpha w_0(Z)}dZ\leq C\alpha^{-\f{11}{6}}e^{-2\alpha w_0(Y)}\|f\|_{L^\infty_{w_0}([Y_1^*-\delta_0,Y^*_1])}.
\end{align*}
Then we obtain that for any $Y\in[Y_1^*-\delta_0,Y_1^*]$,
\begin{align*}
|\mathcal{B}(Y)|\leq& C\alpha^{-\f{11}{6}}e^{-2\alpha w_0(Y)}\|f\|_{L^\infty_{w_0}([Y_1^*-\delta_0,Y^*_1])}+C\alpha^{-\f{11}{6}}e^{-2\alpha w_0(Y_2^*)}\|f\|_{L^\infty_{w_0}}\\
\leq&C\alpha^{-\f{11}{6}}e^{-2\alpha w_0(Y)}(\|f\|_{L^\infty_{w_0}([Y_1^*-\delta_0,Y^*_1])}+e^{-2\al \int_{Y}^{Y_2^*} F_r^\f12(Z)dZ}\|f\|_{L^\infty_{w_0}}).
\end{align*}
Notice that $\mathcal{B}(Y)=\mathcal{B}(Y_1^*-\delta_0)$ for $Y\leq Y_1^*-\delta_0$. Thus, we deduce that for $Y\leq Y_1^*-\delta_0$,
\begin{align*}
|\mathcal{B}(Y)| \leq&C\alpha^{-\f{11}{6}}(e^{-2\alpha w_0(Y_1^*-\delta_0)}\|f\|_{L^\infty_{w_0}([Y_1^*-\delta_0,Y^*_1])}+e^{-2\alpha w_0(Y_2^*)}\|f\|_{L^\infty_{w_0}}).
\end{align*}

Gathering the above results about $\mathcal{B}(Y)$, we infer that for any $Y\geq 0$,
\begin{align*}
|\mathcal{B}(Y)|\leq C\alpha^{-\f{11}{6}}e^{-2\alpha w_0(Y)}\|f\|_{L^\infty_{w_0}}.
\end{align*}
In particular, for $Y\in[Y_1^*-\d_0, Y_2^*]$,
\begin{align*}
|\mathcal{B}(Y)|\leq C\alpha^{-\f{11}{6}}e^{-2\alpha w_0(Y)}(\|f\|_{L^\infty_{w_0}([Y_1^*-\delta_0,Y^*_1])}+e^{-2\al \int_{Y}^{Y_2^*} F_r^\f12(Z)dZ}\|f\|_{L^\infty_{w_0}}),
\end{align*}
and for $Y\in[0, Y_1^*-\d_0]$,
\begin{align*}
|\mathcal{B}(Y)| \leq&C\alpha^{-\f{11}{6}}(e^{-2\alpha w_0(Y_1^*-\delta_0)}\|f\|_{L^\infty_{w_0}([Y_1^*-\delta_0,Y^*_1])}+e^{-2\alpha w_0(Y_2^*)}\|f\|_{L^\infty_{w_0}}).
\end{align*}
\end{proof}

For the third kind of terms in \eqref{def: special source terms}, we obtain the following results.

\begin{proposition}\label{pro: Airy-4}
	Let $(\alpha,c)\in\mathbb H_0$ and $f(Y)$ be a complex-valued function with $e^{\al w_0}|f|\in L^\infty$ and $\mathrm{supp}(f)\subseteq[Y_1^*-\d_0, Y_1^*]\cup [Y_2^*, Y_2^*+\d_0]$. There exists a solution $\tilde P\in W^{2,\infty}_{w_0}$ to the system \eqref{sys:app-f-1} with the source term $K_f(Y)=(1-\chi)\int_Y^{+\infty}f(Z)dZ$. More precisely, 
there exist $\mathcal{A}_{K_f}(Y)$ and $\mathcal{B}_{K_f}(Y)$ with $(|\cA_{Kf}|, |e^{2\al w_0} \cB_{Kf}|)\in L^\infty$ such that 
\begin{align*}
\tilde{P}(Y)=\mathcal{A}_{K_f}(Y)A(Y)+\mathcal{B}_{K_f}(Y)B(Y),
\end{align*}
and the following estimates hold

\begin{enumerate}
\item for $Y\in[0, +\infty)$, we have $\mathrm{supp}(\mathcal B_{K_f})\subseteq[0,Y_2^*+\delta_0]$ and 
\begin{align}\label{est: Airy-41}
\begin{split}
&|\mathcal{A}_{K_f}(Y)|+\al e^{2\alpha w_0(Y)}|\mathcal B_{K_f}(Y)|\leq  C\alpha^{-\frac{11}{6}}\|f\|_{L^\infty_{w_0}}.
\end{split}
\end{align}
\item in particular, for $Y\in[0, Y_2^*]$, we have 
\begin{align}\label{est: Airy-42}
\begin{split}
&|\mathcal{A}_{K_f}(Y)|\leq C\al^{-\f{11}{6}}\Big(\|f\|_{L^\infty_{w_0}([Y_1^*-\delta_0,Y_1^*])}+e^{-\al w_0(Y_2^*)}e^{\al w_0(Y)}\|f\|_{L^\infty_{w_0}}\Big).
\end{split}
\end{align}

\item in particular,  for $Y\in[Y_1^*-\d_0,Y_2^*]$, we have
\begin{align*}
&|\mathcal{B}_f(Y)|\leq C\al^{-\f{17}{6}}e^{-2\al w_0(Y)}\Big(\|f\|_{L^\infty_{w_0}([Y_1^*-\d_0, Y_1^*])}+e^{-\al\int_{Y}^{Y_2^*}F_r^\f12(Z)dZ}\|f\|_{L^\infty_{w_0}}\Big),
\end{align*}
and for $Y\in[Y^{+}+\d,Y_1^*-\d_0]$, we have
\begin{align}\label{est: Airy-43}
\begin{split}
&|\mathcal{B}_f(Y)|\leq C\al^{-\f{17}{6}}e^{-\al w_0(Y)}\left(e^{-\al w_0(Y_1^*-\d_0)}\|f\|_{L^\infty_{w_0}([Y_1^*-\delta_0,Y_1^*])}+e^{-\al w_0(Y_2^*)}\|f\|_{L^\infty_{w_0}}\right),\end{split}
\end{align}
and for $Y\in[0, Y^{+}+\d]$, we have
\begin{align*}
	&|\mathcal B_{K_f}(Y)|\leq C\alpha^{-\f{11}{6}}\left(\|f\|_{L^\infty_{w_0}([Y_1^*-\delta_0,Y_1^*])}e^{-\alpha w_0(Y_1^*-\d_0)} +\|f\|_{L^\infty_{w_0}}e^{-\alpha w_0(Y_2^*)}\right).
\end{align*}

\end{enumerate}
\end{proposition}

\begin{proof}
By \eqref{def: cA_f, cB_f}, we have
\begin{align*}
&\mathcal{A}_{K_f}(Y)=-\kappa^{-1}\pi\int_0^YBi(\kappa\eta(Z))K_f(Z)\bar M^{-1}(Z)\partial_Z\eta(Z)^{-\f12} dZ,\\
&\mathcal{B}_{K_f}(Y)=-\kappa^{-1}\pi\int_Y^{+\infty}Ai(\kappa\eta(Z))K_f(Z)\bar M^{-1}(Z)\partial_Z\eta(Z)^{-\f12}dZ.
\end{align*}
Again we drop subscript in $\mathcal{A}_{K_f}$ and $\mathcal{B}_{K_f}$, and write $\mathcal{A}$ and $\mathcal{B}$ for short, and  assume $f$ is a real function.

Recall that $\mathrm{supp}(1-\chi)=[0,Y_1-\delta_0]\cup[Y_2+\delta_0,+\infty)$ and $\mathrm{supp}(f)\subseteq[Y_1^*-\delta_0,Y_1^*]\cup[Y_2^*,Y_2^*+\delta_0]$. Therefore, we have
\begin{align*}
K_f(Y)=\left\{
\begin{aligned}
&\int_{Y_1^*-\delta_0}^{Y_1^*}f(Z)dZ+\int_{Y_2^*}^{Y_2^*+\delta_0}f(Z)dZ,\quad Y\in[0,Y_1^*-\delta_0],\\
&\int_Y^{Y_1^*}f(Z)dZ+\int_{Y_2^*}^{Y_2^*+\delta_0}f(Z)dZ,\quad Y\in[Y_1^*-\delta_0,Y_1^*],\\
&(1-\chi(Y))\int_{Y_2^*}^{Y_2^*+\delta_0}f(Z)dZ,\quad Y\in [Y_1^*,Y_2^*],\\
&\int_Y^{Y_2^*+\delta_0}f(Z)dZ,\quad Y\in[Y_2^*,Y_2^*+\delta_0],\\
&0,\quad Y>Y_2^*+\delta_0,
\end{aligned}
\right.
\end{align*}
which implies
\begin{align}\label{est: K_f}
|K_f(Y)|\leq C\al^{-1} \left\{
\begin{aligned}
&e^{-\alpha w_0(Y_1^*-\delta_0)} \|f\|_{L^\infty_{w_0}([Y_1^*-\delta_0,Y_1^*])}+e^{-\alpha w_0(Y_2^*)}\|f\|_{L^\infty_{w_0}},\quad Y\in[0,Y_1^*-\delta_0],\\
&e^{-\alpha w_0(Y)} \|f\|_{L^\infty_{w_0}([Y_1^*-\delta_0,Y_1^*])}+e^{-\alpha w_0(Y_2^*)}\|f\|_{L^\infty_{w_0}},\quad Y\in[Y_1^*-\delta_0,Y_1^*],\\
&e^{-\alpha w_0(Y_2^*)}\|f\|_{L^\infty_{w_0}},\quad Y\in [Y_1^*,Y_2^*],\\
&e^{-\alpha w_0(Y)}\|f\|_{L^\infty_{w_0}},\quad Y\in[Y_2^*,Y_2^*+\delta_0],\\
&0,\quad Y>Y_2^*+\delta_0.
\end{aligned}
\right.
\end{align}

\underline{Estimates of $\mathcal{A}(Y)$.} For $Y\in[0,Y_1^*-\delta_0]$, we write
\begin{align*}
\mathcal{A}(Y)
=&-\kappa^{-1}\pi\Big(\int_{Y_1^*-\delta_0}^{Y_1^*}f(Z)dZ+\int_{Y_2^*}^{Y_2^*+\delta_0}f(Z)dZ\Big)\int_0^YBi(\kappa\eta(Z))\bar M^{-1}(Z)\partial_Z\eta(Z)^{-\f12} dZ.
\end{align*}
Then we  get by \eqref{est: K_f} and Lemma \ref{lem: Ai, Bi} that
\begin{align}\label{eq:A-int-f1}
\begin{split}
|\mathcal{A}(Y)|\leq& C\kappa^{-\f54}\al^{-1}\Big(e^{-\alpha w_0(Y_1^*-\delta_0)} \|f\|_{L^\infty_{w_0}([Y_1^*-\delta_0,Y_1^*])}+e^{-\alpha w_0(Y_2^*)}\|f\|_{L^\infty_{w_0}}\Big)\int_0^{Y}e^{\alpha w_0(Z)}dZ\\
\leq&C\al^{-\f{11}{6}}e^{\alpha w_0(Y)}\left(e^{-\alpha w_0(Y_1^*-\delta_0)} \|f\|_{L^\infty_{w_0}([Y_1^*-\delta_0,Y_1^*])}+e^{-\alpha w_0(Y_2^*)}\|f\|_{L^\infty_{w_0}}\right)\\
\leq& C\al^{-\f{11}{6}}\left(\|f\|_{L^\infty_{w_0}([Y_1^*-\delta_0,Y_1^*])}+e^{\alpha w_0(Y)}e^{-\alpha w_0(Y_2^*)}\|f\|_{L^\infty_{w_0}}\right).
\end{split}
\end{align}

For $Y\in[Y_1^*-\delta_0,Y_1^*]$, we write 
\begin{align*}
\mathcal{A}(Y)=&\mathcal{A}(Y_1^*-\delta_0)-\kappa^{-1}\pi\int_{Y_1^*-\delta_0}^YBi(\kappa\eta(Z))K_f(Z)\bar M^{-1}(Z)\partial_Z\eta(Z)^{-\f12} dZ\\
=&\mathcal{A}(Y_1^*-\delta_0)+I_1.
\end{align*}
By \eqref{eq:A-int-f1}, we have
\begin{align*}
|\mathcal{A}(Y_1^*-\delta_0)|
\leq&C\al^{-\f{11}{6}}\Big(\|f\|_{L^\infty_{w_0}([Y_1^*-\delta_0,Y_1^*])}+e^{-\al\int_{Y_1^*-\delta_0}^{Y_2^*}F_r^\f12(Z)dZ }\|f\|_{L^\infty_{w_0}}\Big).
\end{align*}
Then by  \eqref{est: K_f} and Lemma \ref{lem: Ai, Bi}, we  have 
\begin{align}\label{eq:A-int-f1-Y11}
\begin{split}
|I_1|
\leq& C\kappa^{-\f54}\alpha^{-1}\|f\|_{L^\infty_{w_0}([Y_1^*-\delta_0,Y_1^*])}\int_{Y_1^*-\delta_0}^Y1dZ\\
&+C\kappa^{-\f54}\alpha^{-1}e^{-\al w_0(Y_2^*)}\|f\|_{L^\infty_{w_0}}\int_{Y_1^*-\delta_0}^{Y}e^{\alpha w_0(Z')}dZ
\leq C\alpha^{-\frac{11}{6}}\|f\|_{L^\infty_{w_0}([Y_1^*-\delta_0,Y_1^*])}.
\end{split}
\end{align}
By \eqref{eq:A-int-f1} and \eqref{eq:A-int-f1-Y11}, we obtain that for $Y\in[Y_1^*-\delta_0,Y^*_1]$,
\begin{align}\label{eq:A-int-f2}
|\mathcal{A}(Y)|\leq  C\al^{-\f{11}{6}}\Big(\|f\|_{L^\infty_{w_0}([Y_1^*-\delta_0,Y_1^*])}+e^{-\al\int_Y^{Y_2^*}F_r^\f12(Z)dZ }f\|_{L^\infty_{w_0}}\Big).
\end{align}

For $Y\in[Y_1^*,Y_2^*]$, we  write 
\begin{align*}
\mathcal{A}(Y)=&\mathcal{A}(Y_1^*)-\kappa^{-1}\pi\int_{Y_1^*}^YBi(\kappa\eta(Z))K_f(Z)\bar M^{-1}(Z)\partial_Z\eta(Z)^{-\f12} dZ
=\mathcal{A}(Y_1^*)+I_2.
\end{align*}
According to \eqref{est: K_f} and Lemma \ref{lem: Ai, Bi}, we infer that for $Y\in[Y_1^*,Y_2^*]$,
\begin{align*}
|I_2|&\leq  C\kappa^{-\f54}\alpha^{-2}e^{-\al w_0(Y_2^*)}e^{\al w_0(Y)}\|f\|_{L^\infty_{w_0}}\leq C\alpha^{-\frac{17}{6}}e^{-\al\int_{Y}^{Y_2^*}F_r^\f12(Z)dZ }\|f\|_{L^\infty_{w_0}}.
\end{align*}
Hence, we obtain that for $Y\in[Y_1^*,Y_2^*]$,
\begin{align}\label{eq:A-int-f3}
|\mathcal{A}(Y)|\leq   C\al^{-\f{11}{6}}\left(\|f\|_{L^\infty_{w_0}([Y_1^*-\delta_0,Y_1^*])}+e^{-\al\int_{Y}^{Y_2^*}F_r^\f12(Z)dZ }\|f\|_{L^\infty_{w_0}}\right).
\end{align}
In particular, for $Y\in[0, Y_2^*]$, we have 
\begin{align}\label{eq:A-int-f4}
|\mathcal{A}(Y)|\leq  C\al^{-\f{11}{6}}\left(\|f\|_{L^\infty_{w_0}([Y_1^*-\delta_0,Y_1^*])}+e^{-\al w_0(Y_2^*)}e^{\al w_0(Y)}\|f\|_{L^\infty_{w_0}}\right).
\end{align}
\medskip

For $Y\in[Y_2^*,Y^*_2+\delta_0]$, we write 
\begin{align*}
\mathcal{A}(Y)=&\mathcal{A}(Y_2^*)-\kappa^{-1}\pi\int_{Y_2^*}^YBi(\kappa\eta(Z))K_f(Z)\bar M^{-1}(Z)\partial_Z\eta(Z)^{-\f12} dZ
=\mathcal{A}(Y_2^*)+I_3.
\end{align*}
We use \eqref{est: K_f} and Lemma \ref{lem: Ai, Bi} to get
\begin{align*}
|I_4|
&\leq C\kappa^{-\f54}\alpha^{-1}\|f\|_{L^\infty_{w_0}}\leq C\alpha^{-\frac{11}{6}}\|f\|_{L^\infty_{w_0}}.
\end{align*}
Therefore, for $Y\in[Y_2^*,Y^*_2+\delta_0]$,
\begin{align*}
|\mathcal{A}(Y)|\leq  C\alpha^{-\frac{11}{6}}\|f\|_{L^\infty_{w_0}},
\end{align*}
which along with the fact that $\mathcal{A}(Y)=\mathcal{A}(Y_2^*+\delta_0)$ for $Y>Y_2^*+\delta_0$ implies that for $Y\geq Y_2$,
\begin{align*}
|\mathcal{A}(Y)|\leq  C\alpha^{-\frac{11}{6}}\|f\|_{L^\infty_{w_0}}.
\end{align*}
The above result combined with \eqref{eq:A-int-f1}, \eqref{eq:A-int-f2}, \eqref{eq:A-int-f3} and \eqref{eq:A-int-f4} implies that for $Y\geq 0$,
\begin{align*}
|\mathcal{A}(Y)|\leq  C\alpha^{-\frac{11}{6}}\|f\|_{L^\infty_{w_0}}.
\end{align*}
In particular, 
for $Y\in[0, Y_2^*]$,  we have 
\begin{align*}
|\mathcal{A}(Y)|\leq  C\al^{-\f{11}{6}}\Big(\|f\|_{L^\infty_{w_0}([Y_1^*-\delta_0,Y_1^*])}+e^{-\al\int_{Y_1^*}^{Y_2^*}F_r^\f12(Z)dZ }\|f\|_{L^\infty_{w_0}}\Big).
\end{align*}
\medskip

\underline{Estimates of $\mathcal{B}(Y)$.} We first notice that $\mathcal{B}(Y)\equiv 0$ for any $Y\in[Y_2^*+\delta_0,+\infty)$. 
For $Y\in[Y_2^*,Y_2^*+\delta_0]$, we have
\begin{align*}
\mathcal{B}(Y)
=&-\kappa^{-1}\pi\int_Y^{Y_2^*+\delta_0}Ai(\kappa\eta(Z))K_f(Z)\bar M^{-1}(Z)\partial_Z\eta(Z)^{-\f12}dZ.
\end{align*}
On the other hand, by Lemma \ref{lem: Ai, Bi}, we know that for $Z\in[Y_1^*-\delta_0,Y_2^*+\delta_0]$,
\begin{align}\label{eq:Ai-Y1*-Y2*}
|Ai(\kappa\eta(Z))|\leq C\kappa^{-\f14}e^{-\alpha w_0(Z)}.
\end{align}
Then we use \eqref{est: K_f} to obtain that for $Y\in[Y_2^*,Y_2^*+\delta_0]$,
\begin{align}\label{est: B(Y)-(Y_2^*, Y_2^*+d_0)}
\begin{split}
|\mathcal{B}(Y)|\leq& C\kappa^{-\f54}\alpha^{-1}\|f\|_{L^\infty_{w_0}}\int_Y^{Y_2^*+\delta_0}e^{-2\alpha w_0(Z)}dZ
\leq  C\alpha^{-\frac{17}{6}}e^{-2\alpha w_0(Y)}\|f\|_{L^\infty_{w_0}}.
\end{split}
\end{align}

For $Y\in[Y_1^*,Y_2^*]$, we  write 
\begin{align*}
\mathcal{B}(Y)=&\mathcal{B}(Y_2^*)-\kappa^{-1}\pi\int_Y^{Y_2^*}Ai(\kappa\eta(Z))K_f(Z)\bar M^{-1}(Z)\partial_Z\eta(Z)^{-\f12}dZ
=\mathcal{B}(Y_2^*)+II_1.
\end{align*}
By \eqref{eq:Ai-Y1*-Y2*} and \eqref{est: K_f}, we deduce that for $Y\in[Y_1^*,Y_2^*]$,
\begin{align*}
\begin{split}
|II_1|\leq& C\kappa^{-1}\al^{-1}\al^{-\f76}e^{-\al w_0(Y_2^*)} e^{-\al w_0(Y)}\|f\|_{L^\infty_{w_0}}
\leq C \al^{-\f{17}{6}}e^{-\al\int_{Y}^{Y_2^*}F_r^\f12(Z)dZ }e^{-2\al w_0(Y)}\|f\|_{L^\infty_{w_0}},
\end{split}
\end{align*}
which along with \eqref{est: B(Y)-(Y_2^*, Y_2^*+d_0)} implies
\begin{align}\label{est: B(Y)-(Y_1^*, Y_2^*)}
|\mathcal{B}(Y)|\leq&C\alpha^{-\frac{17}{6}}e^{-2\alpha w_0(Y_2^*)}\|f\|_{L^\infty_{w_0}}+\al^{-\f{17}{6}}e^{-\al w_0(Y_2^*)} e^{\al w_0(Y)}e^{-2\al w_0(Y)}\|f\|_{L^\infty_{w_0}}\\
\nonumber
\leq &C\al^{-\f{17}{6}}e^{-\al\int_{Y}^{Y_2^*}F_r^\f12(Z)dZ }e^{-2\al w_0(Y)}\|f\|_{L^\infty_{w_0}}.
\end{align}

For $Y\in[Y_1^*-\delta_0,Y_1^*]$, we can write 
\begin{align*}
	\mathcal B(Y)	=& \mathcal B(Y_1^*)- \kappa^{-1}\pi\int_Y^{Y_1^*}Ai(\kappa\eta(Z))K_f(Z)\bar M^{-1}(Z)\partial_Z\eta(Z)^{-\f12}dZ
	=\mathcal B(Y_1^*)+II_2.
\end{align*}
We use  \eqref{est: K_f} and \eqref{eq:Ai-Y1*-Y2*} to imply that $Y\in[Y_1^*-\delta_0,Y_1^*]$,
\begin{align*}
	|II_2|
	&\leq C\kappa^{-\f54}\alpha^{-1}\Big( \|f\|_{L^\infty_{w_0}([Y_1^*-\delta_0,Y_1^*])}\int_{Y}^{Y_1^*} e^{-2\alpha w_0(Z)}dZ+ \|f\|_{L^\infty_{w_0}}e^{-\alpha w_0(Y_2^*)} \int_Y^{Y_1^*}e^{-\alpha w_0(Z)} dZ\Big)\\
	&\leq C\alpha^{-\frac{17}{6}}e^{-2\alpha w_0(Y)}\Big(\|f\|_{L^\infty_{w_0}([Y_1^*-\delta_0,Y_1^*])}+e^{-\al\int_{Y}^{Y_2^*}F_r^\f12(Z)dZ }\|f\|_{L^\infty_{w_0}}\Big).
\end{align*}
Then we use \eqref{est: B(Y)-(Y_1^*, Y_2^*)} to obtain that for $Y\in[Y_1^*-\delta_0,Y_1^*]$,
\begin{align}\label{est: B(Y)-[Y_1^*-d_0, Y_1^*]}
\begin{split}
		|\mathcal B(Y)|
		\leq&C\al^{-\f{17}{6}}e^{-2\alpha w_0(Y)}\Big(\|f\|_{L^\infty_{w_0}([Y_1^*-\delta_0,Y_1^*])}+e^{-\al\int_{Y}^{Y_2^*}F_r^\f12(Z) dZ}\|f\|_{L^\infty_{w_0}}\Big).
\end{split}
\end{align}

For $Y\in[0,Y_1^*-\delta_0]$, we write
\begin{align*}
	\mathcal B(Y)	=& \mathcal B(Y_1^*-\delta_0)-\kappa^{-1}\pi\int_Y^{Y_1^*-\delta_0}Ai(\kappa\eta(Z))K_f(Z)\bar M^{-1}(Z)\partial_Z\eta(Z)^{-\f12}dZ\\
	=&\mathcal B(Y_1^*-\delta_0)+II_3.
\end{align*}
By \eqref{est: B(Y)-[Y_1^*-d_0, Y_1^*]}, we have
\begin{align*}
|\mathcal B(Y_1^*-\delta_0)|\leq&C\al^{-\f{17}{6}}\left(e^{-\alpha w_0(Y_2^*)}e^{-\alpha w_0(Y_1^*-\delta_0)}\|f\|_{L^\infty_{w_0}}+e^{-2\alpha w_0(Y_1^*-\delta_0)}\|f\|_{L^\infty_{w_0}([Y_1^*-\delta_0,Y_1^*])}\right).
\end{align*}

On the other hand, according to Lemma \ref{lem: Ai, Bi}, we know that for $Z\in[0,Y_1^*-\delta_0]$,
\begin{align*}
	|Ai(\kappa\eta(Z))||\pa_Z\eta(Z)|^{-\f12}\leq C\min\{1, |\kappa\eta(Z)|^{-\f14}|\pa_Z\eta(Z)|^{-\f12} \}e^{-\alpha w_0(Z)},
\end{align*}
which along with \eqref{est: K_f} implies that for $Y\in[Y^++\d,Y_1^*-\delta_0]$,
\begin{align*}
\begin{split}
	|II_4|
	\leq& C\kappa^{-1}\alpha^{-1}\left(e^{-\al w_0(Y_1^*-\d_0)}\|f\|_{L^\infty_{w_0}([Y_1^*-\delta_0,Y_1^*])}+e^{-\al w_0(Y_2^*)}\|f\|_{L^\infty_{w_0}}\right)\int_Y^{Y_1^*-\delta_0}e^{-\al w_0(Z)} dZ\\
\leq& C\al^{-\f{17}{6}}e^{-\al w_0(Y)}\left(e^{-\al w_0(Y_1^*-\d_0)}\|f\|_{L^\infty_{w_0}([Y_1^*-\delta_0,Y_1^*])}+e^{-\al w_0(Y_2^*)}\|f\|_{L^\infty_{w_0}}\right),
\end{split}
\end{align*}
and for $Y\in[0,Y^++\d]$, 
\begin{align*}
|II_4|\leq& C\al^{-\f{17}{6}}\left(e^{-\al w_0(Y_1^*-\d_0)}\|f\|_{L^\infty_{w_0}([Y_1^*-\delta_0,Y_1^*])}+e^{-\al w_0(Y_2^*)}\|f\|_{L^\infty_{w_0}}\right)e^{-\al w_0(Y^{+}+\d)}\\
&+C\kappa^{-\f54}\alpha^{-1}\|f\|_{L^\infty_{w_0}([Y_1^*-\delta_0,Y_1^*])}e^{-\alpha w_0(Y_1^*-\d_0)} +C\kappa^{-\f54}\alpha^{-1}\|f\|_{L^\infty_{w_0}}e^{-\alpha w_0(Y_2^*)} \\
\leq& C\al^{-\f{11}{6}}\left(e^{-\al w_0(Y_1^*-\d_0)}\|f\|_{L^\infty_{w_0}([Y_1^*-\delta_0,Y_1^*])}+e^{-\al w_0(Y_2^*)}\|f\|_{L^\infty_{w_0}}\right).
\end{align*}
Therefore, we obtain that for $Y\in[Y^{+}+\d,Y_1^*-\delta_0]$,
\begin{align*}
	|\mathcal B(Y)|\leq& C\al^{-\f{17}{6}}e^{-\al w_0(Y)}\left(e^{-\al w_0(Y_1^*-\d_0)}\|f\|_{L^\infty_{w_0}([Y_1^*-\delta_0,Y_1^*])}+e^{-\al w_0(Y_2^*)}\|f\|_{L^\infty_{w_0}}\right),
\end{align*}
and for $Y\in[0, Y^{+}+\d]$,
\begin{align*}
	|\mathcal B(Y)|\leq C\alpha^{-\f{11}{6}}(\|f\|_{L^\infty_{w_0}([Y_1^*-\delta_0,Y_1^*])}e^{-\alpha w_0(Y_1^*-\d_0)} +\|f\|_{L^\infty_{w_0}}e^{-\alpha w_0(Y_2^*)}).
\end{align*}
Thus, we infer that $\mathrm{supp}(\mathcal B)\subseteq [0,Y_2^*+\delta_0]$ and for $Y\geq0$,
\begin{align*}
	|\mathcal B(Y)|\leq& C\alpha^{-\frac{17}{6}}e^{-2\alpha w_0(Y)}\|f\|_{L^\infty_{w_0}}.
\end{align*}
\end{proof}

The following results are about the fourth kind of source terms in \eqref{def: special source terms}.

\begin{proposition}\label{pro: Airy-5}
Under the assumption on $f(Y)$ in Proposition  \ref{pro: Airy-4},  there exists a solution $\tilde P\in W^{2,\infty}_{w_0}$ to the system \eqref{sys:app-f-1} with the source term $ hK_f(Y)$ for $\mathrm{supp}(h)\subseteq[\bar Y_1^{**}, +\infty)$. More precisely, 
there exist $\mathcal{A}_{hK_f}(Y)$ and $\mathcal{B}_{hK_f}(Y)$ with $(|\cA_{hKf}|, |e^{2\al w_0} \cB_{hKf}|)\in L^\infty$ such that 
\begin{align*}
\tilde{P}(Y)=\mathcal{A}_{hK_f}(Y)A(Y)+\mathcal{B}_{hK_f}(Y)B(Y),
\end{align*}
 $\mathrm{supp}(\mathcal{A}_{hK_f})\subset \mathrm{supp}(h)$, and the estimates \eqref{est: Airy-41}-\eqref{est: Airy-43} hold. Moreover, for $Y\in[0, \bar Y_1^{**}]$,
\begin{align*}
&|\mathcal B_{K_f}(Y)|\leq C\alpha^{-\f{11}{6}}e^{-\al w_0(\bar Y_1^{**})}\left(\|f\|_{L^\infty_{w_0}([Y_1^*-\delta_0,Y_1^*])}e^{-\alpha w_0(Y_1^*-\d_0)} +\|f\|_{L^\infty_{w_0}}e^{-\alpha w_0(Y_2^*)}\right).
\end{align*}

\end{proposition}

\begin{proof}
The proof is quite similar to Proposition  \ref{pro: Airy-4}. We point out that  the support of $h$ does not touch the supersonic regime. It is easy to deduce that $\mathcal{A}_{hK_f}(Y)=0$ for $Y\in[0, \bar Y_1^{**}]$ and $\mathcal{B}_{hK_f}(Y)=\mathcal{B}_{hK_f}(\bar Y_1^{**})$ for  $Y\in[0, \bar Y_1^{**}]$.
\end{proof}

\section{The Rayleigh type equation}

In this section, we consider the approximate system in the regime near the critical layer, which can be described by  the following Rayleigh-type equation in a finite interval $Y\in[\bar Y_1^*,\bar Y_2^*]$:
 \begin{align}\label{eq: Rayleigh-nonhomo}
           	\left\{
           	\begin{aligned}
           		&\bar M\left(\partial_Y\left(\frac{\partial_Y}{F(Y)}\right)-\alpha^2\right)\varphi-\partial_Y\left(\frac{\partial_Y\bar M}{F(Y)}\right)\varphi=f,\quad Y\in(\bar Y_1^*,\bar Y_2^*),\\
           		&\varphi(\bar Y_1^*)=\varphi(\bar Y_2^*)=0,
           	\end{aligned}
           	\right.
           \end{align}
 where $f$ is a suitable function with the support away from the critical layer. Here the definitions of $F(Y)$ and $\bM(Y)$ are given in \eqref{def(bM, T_0)1}. For $\bM=(U_B-c)$ and $F=1$, the system \eqref{eq: Rayleigh-nonhomo} is reduced to the classical Rayleigh equation.

Let $Y_c\in(\bar Y_1^*,\bar Y_2^*)$ be the critical point which satisfies $U_B(Y_c)=c_r$. According to \eqref{def:r-M} and \eqref{def:F_r}, it is easy to see that
\begin{align}\label{BC: (F_r, M_r)}
\begin{split}
&F_r(Y_c)=1,\quad F_r'(Y_c)=0,\quad \bM_r(Y_c)=0,\quad \bM_r'(Y_c)=\f{M_a \pa_Y U_B(Y_c)}{T_0(Y_c)^\f12}.
\end{split}
\end{align}
Since $U_B(Y)$ is a monotone function and $[\bar Y_1^*,\bar Y_2^*]$ is a finite interval with length independent of $\al$,  there exists a constant $c_0>0$ such that 
\begin{align}\label{Low bound: pa_Y U_B}
\pa_Y U_B\geq c_0>0\quad \mbox {for}\quad Y\in[\bar Y_1^*,\bar Y_2^*].
\end{align}
Therefore, there exists a constant $C>1$ such that for $Y\in[\bar Y_1^*,\bar Y_2^*]$
\begin{align}
&C^{-1}\leq |F(Y)|\leq C,\quad C^{-1}\leq F_r(Y)\leq C,\quad |F_i(Y)|\leq C c_i|Y-Y_c|+c_i^2,\label{bound: F(Y)}\\
&C^{-1}(|Y-Y_c|+c_i)\leq|\bM|\leq C(|Y-Y_c|+c_i),\quad C^{-1}|Y-Y_c|\leq|\bM_r|\leq C|Y-Y_c|.\label{bound: bM}
\end{align}

 Let $\d>0$ be a suitable constant independent of $\al$, and  $\bar Y_1$ and $\bar Y_2$ are two points such that $\bar Y_1^*<\bar Y_1-\d<\bar Y_1<Y_c<\bar Y_2<\bar Y_2+\d<\bar Y_2^*$. Moreover, we assume that $|w_c(\bar Y_1^*)|\neq|w_c(\bar Y_2^*)|$, where 
 \begin{align*}
 w_c(Y)=\int_{Y_c}^YF_r^\f12(Z)dZ.
 \end{align*}

The main result of this section is presented as follows.

\begin{proposition}\label{pro: rayleigh-varphi}
Let $(\al,c)\in\mathbb{H}_0$ and $|w_c(\bar Y_1^*)|\neq|w_c(\bar Y_2^*)|$.
Assume $f$ is a complex-valued function with $\mathrm{supp}(f)\subset[\bar Y_1-\d, \bar Y_1 ]\cup[\bar Y_2, \bar Y_2+\d]$ and  $e^{\al w_0}|f|\in L^\infty$. Then there exists a solution $\varphi(Y)\in W^{2,\infty}$ to system \eqref{eq: Rayleigh-nonhomo}. Moreover,  we have 
\begin{align*}
&|\varphi(Y)|\leq C\al^{-1}e^{-\al w_0(Y)}\|f\|_{L^\infty_{w_0}},\quad Y\in[\bar Y_1^*, \bar Y_2^*],\\
&|\Im\varphi(Y)|\leq C\al^{-1}e^{-\al w_0(Y)}(\al^{-1}\|f\|_{L^\infty_{w_0}}+\|\Im f\|_{L^\infty_{w_0}}), \quad Y\in[\bar Y_1^*, \bar Y_1]\cup[\bar Y_2, \bar Y_2^*],
\end{align*}
and for $Y\in[\bar Y_1^*,\bar Y_1-\d]$, we have
\begin{align*}
&|\varphi(Y)|\leq C\al^{-2}e^{-2\al \int_{Y}^{\bar Y_1-\d}F_r^\f12(Y') dY'} e^{-\al w_0(Y)}\|f\|_{L^\infty_{w_0}}, \\
&|\Im\varphi(Y)|\leq C\al^{-2}e^{-2 \al  \int_{Y}^{\bar Y_1-\d} F_r(Y')dY'} e^{-\al w_0(Y)}(\al^{-1}\|f\|_{L^\infty_{w_0}}+\|\Im f\|_{L^\infty_{w_0}}).
\end{align*}
\end{proposition}

For convenience, we introduce the following two kinds of definitions of operators
\begin{align}
\mathcal{L}_{\xi}(\cdot):=\partial_Y\left(\f{\xi(Y)^2}{F(Y)}\partial_Y\left(\f{\cdot}{\xi}\right)\right)\text{ and }\tilde{\mathcal{L}}_{\xi}(\cdot):=\partial_Y\left(\f{\xi(Y)^2}{F(Y)}\partial_Y\right).
\end{align}
Then we can write that $\mathcal{L}_{cr}(\varphi)=\mathcal{L}_{\bar{M}}(\varphi)-\al^2\bar{M}\varphi.$ Moreover. we denote  $\mathcal{L}_1(\cdot)=\partial_Y(F(Y)^{-1}\partial_Y)$.

To solve the system \eqref{eq: Rayleigh-nonhomo}, we first construct a homogenous solution related to  the system \eqref{eq: Rayleigh-nonhomo}, which is expressed as follows
\begin{align}\label{eq: Rayleigh-homo}
           	\left\{
           	\begin{aligned}
           		&\mathcal{L}_{cr}(\phi)=\mathcal{L}_{\bar{M}}(\phi)-\al^2\bar{M}\phi=0,\quad Y\in(\bar Y_1^*,\bar Y_2^*),\\
           		&\phi|_{Y=Y_c}=\bM(Y_c),\quad \partial_Y\phi|_{Y=Y_c}=\pa_Y\bM(Y_c).
           	\end{aligned}
           	\right.
           \end{align} 
   
With $\phi$  at hand, the non-homogenous system \eqref{eq: Rayleigh-nonhomo} is equivalent to 
 \begin{align*}
 \mathcal{L}_{\phi}(\varphi)=\f{\phi f}{\bar{M}},\quad Y\in(\bar Y_1^*,\bar Y_2^*)\text{ and }\varphi(\bar Y_1^*)=\varphi(\bar Y_2^*)=0.
 \end{align*}

\subsection{Solving the homogeneous equation.}

The first step is to construct the solution of the homogeneous Rayleigh type equation in \eqref{eq: Rayleigh-homo}. 
To construct the solution $\phi$ to \eqref{eq: Rayleigh-homo}, we are going to find an accurate profile $\psi(Y)$ of $\phi(Y)$ such that $\phi(Y)=\psi(Y)\tilde{\phi}_0(Y)$, where $\tilde{\phi}_0(Y)$ is a small perturbation around $1$.
On the other hand, we notice for any function $\psi$,  the following identity holds: if $\phi(Y)$ is the solution to \eqref{eq: Rayleigh-homo}, then
\begin{align}\label{formula: (phi, psi)}
\begin{split}
\mathcal{L}_{\psi}(\phi)
=&\psi(\al^2\phi+\bM^{-1}\partial_Y(F^{-1}\partial_Y\bM)\phi)-\phi\mathcal{L}_1(\psi)\\
=&\phi(\al^2\psi-\mathcal{L}_1(\psi))\phi+\bM^{-1}\partial_Y(F^{-1}\partial_Y\bM)\phi\psi.
\end{split}
\end{align}
For the first term on the right-hand side of the above equality, there is a large factor $\al^2$. To eliminate this term, we choose function $\psi$ to approximatively solve the variable coefficient elliptic equation $\mathcal{L}_1(f)=\al^2f$. 
 Taking $\psi(Y)=F_r^\f14(Y)\sinh\al w_c(Y)$ or $\psi(Y)=F_r^\f14(Y)\cosh\al w_c(Y)$, it is easy to verify that
\begin{align*}
\mathcal{L}_1(\psi)=\al^2\psi+F_r^{-\f14}\partial_Y(F_r^{-1}\partial_Y(F_r^\f14))\psi+\partial_Y((F^{-1}-F_r^{-1})\partial_Y\psi),
\end{align*}
which gives an approximated solution to the equation $\mathcal{L}_1(f)=\al^2 f$. To match the boundary condition in \eqref{eq: Rayleigh-homo}, we finally choose 
\begin{align}\label{eq: choose psi}
\psi(Y)=&\al^{-1}\partial_Y\bM(Y_c)F_r^\f14(Y)\sinh\al w_c(Y)+ \bM(Y_c) F_r^\f14(Y)\cosh\al w_c(Y)\\
\nonumber
=&\psi_r(Y)+i\psi_i(Y).
\end{align}
It is easy to check that $\psi(Y_c)=\bar{M}(Y_c)$ and $\partial_Y\psi(Y)=\partial_Y\bar{M}(Y_c)$. 
Therefore, we have
\begin{align}\label{sim: psi}
|\psi(Y)|\sim \al^{-1}\sinh \al |w_c(Y)|+ c_i.
\end{align}

\begin{proposition}\label{pro: bound-(phi, phi_0)}
Let $(\al,c)\in \mathbb{H}_0$ and $\psi$ be defined as \eqref{eq: choose psi}. Then there exists a solution $\phi\in W^{2,\infty}([\bar Y_1^*,\bar Y_2^*])$ with the form of $\phi=\psi\tilde\phi_0$, where $\tilde\phi_0$ satisfies 
\begin{align*}
C^{-1}\leq\|\tilde\phi_0\|_{L^\infty}\leq C,\quad  \|\mathrm{Im}\tilde\phi_0\|_{L^\infty}\leq Cc_i(|\log c_i|+\alpha),\quad |1-\tilde{\phi}_0(Y)|\leq C\min\{\al^{-1}\log\al,|Y-Y_c|\}.
\end{align*}
Moreover, we have
\begin{align}
&C^{-1}\al^{-1}\Big(\sinh \al |w_c(Y)|+ \al c_i\Big)\leq|\phi(Y)|\leq C\al^{-1}\Big(\sinh \al |w_c(Y)|+ \al c_i\Big). \label{est: phi}
\end{align}
More precisely, we have
\begin{align}
&|\phi_r|\leq C(\al^{-1}\sinh \al |w_c(Y)|+c_i^2(|\log c_i|+\al)),\quad |\phi_i|\leq Cc_i e^{\al |w_c(Y)|},
\end{align}
where $\phi=\phi_r+i\phi_i$.
\end{proposition}

\begin{proof}
We choose $\psi(Y)$ defined as \eqref{eq: choose psi}. To construct the solution $\phi$ to \eqref{eq: Rayleigh-nonhomo}, we are left with the construction with $\tilde{\phi}_0=\phi/\psi$, which satisfies  
\begin{align}\label{eq: tilde phi_0}
\tilde{\mathcal{L}}_{\psi}(\tilde{\phi}_0)=\partial_Y\left(\psi^2F^{-1}\partial_Y\tilde{\phi}_0\right) =\tilde\phi_0 g \text{ and }\tilde\phi_0(Y_c)=1,\quad\partial_Y\tilde{\phi}_0(Y_c)=0,
\end{align}
where 
\begin{align*}
g=\Big(\psi^2\bM^{-1}\partial_Y(F^{-1}\partial_Y\bM)-F_r^{-\f14}\partial_Y(F_r^{-1}\partial_Y(F_r^\f14))\psi^2-\psi\partial_Y((F^{-1}-F_r^{-1})\partial_Y\psi)\Big)=g_1+g_2+g_3.
\end{align*}
Then $\tilde{\phi}_0$ can be represented as
\begin{align}\label{form:tphi}
\tilde\phi_0(Y)=&1+\int_{Y_c}^Y (F\psi^{-2})(Y')\int_{Y_c}^{Y'}(g_1+g_2+g_3)(Y'')\tilde\phi_0(Y'')  dY''dY':\eqdef1+T\tilde\phi_0.
\end{align}
For $Y\in[\bar{Y}_1^*,\bar{Y}_2^*]$, we have 
\begin{align}\label{eq:est-g123}
|g_1(Y)|\leq \f{C|\psi|^2}{|\bar{M}|},\quad|g_2(Y)|\leq C|\psi|^2\text{ and }|g_3(Y)|\leq C\al^2c_i|\psi|^2.
\end{align}

To show the existence of $\tilde{\phi}_0$, we only need to show  $T$ is a contracted operator in $\|\cdot\|_{X}$, where 
\begin{align*}
\|f\|_X:=\|f\|_{L^\infty([\bar{Y}_1^*,\bar{Y}_2^*])}+\al^{-1}\|\partial_Y f\|_{L^\infty([\bar{Y}_1^*,\bar{Y}_2^*])}+\al^{-2}\|\bar{M}\partial_Y^2f\|_{L^\infty([\bar{Y}_1^*,\bar{Y}_2^*])}.
\end{align*}

For the first term in $Tf$ with $f\in X$, we use \eqref{bound: F(Y)}-\eqref{bound: bM}, \eqref{sim: psi} and \eqref{eq:est-g123} to  obtain
\begin{align*}
&\Big|\int_{Y_c}^Y (F\psi^{-2})(Y')\int_{Y_c}^{Y'}g_1(Y'')f(Y'')  dY''dY'\Big|\\
\leq&C\int_{Y_c}^Y \f{1}{|\psi(Y')|^2}\int_{Y_c}^{Y'}(|\psi|^2|\bM|^{-1})(Y'')dY''dY'\|f\|_{L^\infty}\\
\leq&C\int_{Y_c}^Y \f{\al^2 c_i|Y'-Y_c|}{\sinh^2 \al w_c(Y')+\al^2c_i^2 }dY'\|f\|_{L^\infty}\\
&+C\int_{Y_c}^Y \f{1}{\sinh^2 \al w_c(Y')+\al^2c_i^2 }\int_{Y_c}^{Y'}\f{\sinh^2 \al w_c(Y'')}{|Y''-Y_c|+c_i}dY''dY'\|f\|_{L^\infty}=I_1+I_2.
\end{align*}
For $I_1$, we have
\begin{align}\label{est: integral-1}
|I_1|\leq&C\al c_i\int_{|Y'-Y_c|\leq \al^{-1}} \f{1}{\al|Y'-Y_c|+\al c_i }dY'\|f\|_{L^\infty}\\
\nonumber
&+C\al c_i\int_{|Y'-Y_c|\geq \al^{-1}} \al|Y'-Y_c|e^{-2C\al|Y'-Y_c|}dY'\|f\|_{L^\infty}\leq Cc_i|\log c_i|\|\tilde\phi_0\|_{L^\infty}.
\end{align}
For $I_2$, we change the order of the integral to get
\begin{align}\label{est: integral-2}
|I_2|\leq&C\al^{-1}\int_{0}^{\al |w_c(Y)|} \f{dx}{\sinh ^2x}\int_{0}^{|x|}\f{\sinh ^2 z}{z}dz\|f\|_{L^\infty}\\
\nonumber
\leq&C\al^{-1}\int_{0}^{C\al } \f{\sinh ^2zdz}{z}\int_{|z|}^{+\infty}\sinh ^{-2} xdx\|\tilde\phi_0\|_{L^\infty}\\
\nonumber
\leq&C\al^{-1}\int_{0}^{C\al }\f{1-e^{-2z}}{z}  dz \|f\|_{L^\infty}
\leq C(\al^{-1}+\al^{-1}\log \al)\|f\|_{L^\infty}.
\end{align}
Collecting \eqref{est: integral-1} and \eqref{est: integral-2}, we obtain
\begin{align}\label{est: T_3 g_1}
\Big|\int_{Y_c}^Y (F\psi^{-2})(Y')\int_{Y_c}^{Y'}g_1(Y'')f(Y'')  dY''dY'\Big|
\leq C(c_i|\log c_i|+\al^{-1}+\al^{-1}\log \al)\|f\|_{L^\infty}.
\end{align}

For the second and the third term in $Tf$, the procedure is much easier than the first term because there is not a singular factor $\bM^{-1}$. As $|g_2|\leq C |\psi(Y)|^2$ and $|g_3|\leq C\al^2 c_i |\psi(Y)|^2$, we use \eqref{sim: psi} and similar argument in \eqref{est: integral-1} and \eqref{est: integral-2} to get
\begin{align}\label{est: T_3 g_2}
&\Big|\int_{Y_c}^Y (F\psi^{-2})(Y')\int_{Y_c}^{Y'}(g_2+g_3)(Y'')f(Y'')  dY''dY'\Big|\\
\nonumber
&\leq C  \int_{Y_c}^Y \f{1}{|\psi^2(Y')|} \int_{Y_c}^{Y'} |\psi(Y'')|^2dY''dY'\|f\|_{L^\infty}
\leq C\al^{-1}\|f\|_{L^\infty}.
\end{align}

Combining \eqref{est: T_3 g_1}-\eqref{est: T_3 g_2} together, we arrive at 
\begin{align}\label{eq:Tf}
\|Tf\|_{L^\infty}\leq C(c_i|\log c_i|+\al^{-1}+\al^{-1}\log \al)\|f\|_{L^\infty}\leq \al^{-1}\log \al \|f\|_{X}
\end{align}
by using $(\al,c)\in \mathbb{H}_0$. 

Now we turn to focus on $\partial_YTf(Y)$. We have 
\begin{align*}
\partial_YTf(Y)=(F\psi^{-2})(Y)\int_{Y_c}^{Y}(g_1+g_2+g_3)(Y')f(Y')  dY'.
\end{align*}
By a similar argument as above, we have 
\begin{align*}
&\Big|(F\psi^{-2})(Y)\int_{Y_c}^{Y}g_1(Y')\tilde\phi_0(Y')  dY'\Big|\\
&\quad\leq C\f{\al^2 c_i|Y-Y_c|}{\sinh^2 \al w_c(Y)+\al^2c_i^2 }\|f\|_{L^\infty}+\f{C}{\sinh^2 \al w_c(Y)+\al^2c_i^2 }\int_{Y_c}^{Y}\f{\sinh^2 \al w_c(Y')}{|Y'-Y_c|+c_i}dY'\|f\|_{L^\infty}\\
&\quad\leq C\|f\|_{L^\infty}\Big(1+\f{\al}{\sinh^2 \al w_c(Y)+\al^2c_i^2 }\int_{Y_c}^{Y}\sinh \al w_c(Y')dY'\Big)\leq  C\|f\|_{L^\infty},
\end{align*}
\begin{align*}
&\Big|(F\psi^{-2})(Y)\int_{Y_c}^{Y}(g_2+g_3)(Y')\tilde\phi_0(Y')  dY'\Big|\leq \f{C}{|\psi^2(Y)|} \int_{Y_c}^{Y} |\psi(Y')|^2dY'\|f\|_{L^\infty}\leq C\|f\|_{L^\infty}.
\end{align*}
Therefore, we obtain 
\begin{align}\label{eq:d1-Tf}
\|\partial_Y Tf\|_{L^\infty}\leq C\|f\|_{L^\infty}\leq C\|f\|_{X}.
\end{align}

Next we consider $\partial_Y^2 Tf$. By the definition of $T$, we have 
\begin{align*}
\bar{M}\partial_Y^2 Tf(Y)=\f{g f\bar{M}F}{\psi^2}-\f{\bar{M}\psi^2\partial_Y Tf\partial_Y(\psi^{-2}F)}{F},
\end{align*}
which along with the fact
\begin{align*}
\Big|\f{gf\bar{M}F}{\psi^2}(Y)\Big|\leq C\|f\|_{L^\infty},\quad\Big|\f{\bar{M}\psi^2\partial_Y Tf\partial_Y(\psi^{-2}F)}{F(Y)}\Big|\leq C\al\|\partial_YT f\|_{L^\infty}\leq C\al\|f\|_{L^\infty},
\end{align*}
implies 
\begin{align}\label{eq:d2-Tf}
\|\bar{M}\partial_Y^2 Tf(Y)\|_{L^\infty}\leq C\al\|f\|_{L^\infty}\leq C\al\|f\|_{X}.
\end{align}
According to \eqref{eq:Tf}, \eqref{eq:d1-Tf} and \eqref{eq:d2-Tf}, we obtain  
\begin{align*}
\|Tf\|_{X}\leq C\al^{-1}\log\al\|f\|_{L^\infty}\leq C\al^{-1}\log\al\|f\|_X,
\end{align*}
which implies that
$I-T_3$ is invertible in X. Then we obtain that $\tilde\phi_0=(I-T)^{-1}1$. Moreover,
\begin{align*}
C^{-1}\leq\|\tilde\phi_0\|_{L^\infty}\leq C,\quad |1-\tilde{\phi}_0(Y)|\leq C\min\{\al^{-1}\log\al,|Y-Y_c|\}.
\end{align*}
Then we use \eqref{sim: psi} to obtain \eqref{est: phi}.

Next we give the estimates for the real part and imaginary part of $\tilde{\phi}_0$. To achieve that, we introduce $\tilde\phi_0=\tilde\phi_{01}\tilde\phi_{02}$, where $\tilde\phi_{01}$ is a real function solving the following system
\begin{align*}
\partial_Y\Big(\f{\psi_r^2}{F_r}\partial_Y\tilde{\phi}_{01}\Big)=\tilde\phi_{01}g_r,\text{ and }\tilde\phi_{01}(Y_c)=1,\quad (\psi_r^2\partial_Y \tilde\phi_{01})(Y_c)=0
\end{align*}
and $\tilde\phi_{02}$ satisfies
 \begin{align}\label{eq: tilde phi02}
\left\{
\begin{aligned}
&\tilde{\mathcal L}_{\psi}(\tilde\phi_{02})+\Big(\f{F_r\psi^2}{F \psi_r^2} g_r-g\Big)\tilde\phi_{01}^2\tilde\phi_{02}-(\f{\psi^2 F_r}{\psi^2_r F}-1)\partial_Y (F_r^{-1}\psi_r^2)\partial_Y \tilde\phi_{01}\tilde\phi_{01}\tilde\phi_{02}=0,\\
&\tilde\phi_{02}(Y_c)=1,\quad (\psi^2\partial_Y\tilde\phi_{02})(Y_c)= \bM^2(Y_c)\tilde\phi'_{01}(Y_c)\sim  c_i^2.
\end{aligned}
\right.
\end{align}
Here
\begin{align*}
g_r=\psi_r^2(\bM^{-1})_r\partial_Y (F^{-1}_r\partial_Y\bM_r)-F_r^{-\f14}\partial_Y(F_r^{-1}\partial_Y(F_r^\f14))\psi_r^2=g_{1r}+g_{2r}.
\end{align*}
Moreover, we have
\begin{align}
	|g_{1r}|\leq \frac{C|Y-Y_c||\psi_r|^2}{|Y-Y_c|^2+c_i^2}\text{ and }|g_{2r}|\leq C|\psi_r|^2.
\end{align}

It is easy to check that
\begin{align*}
\tilde\phi_{01}(Y)=1+\int_{Y_c}^Y F_r\psi_r^{-2}\int_{Y_c}^{Y'} (\tilde\phi_{01} g_r )(Y'')dY'' dY':\eqdef1+T_r\tilde\phi_{01}.
\end{align*}
Using  similar argument in $T$, we can deduce that for any $f\in X$,
\begin{align*}
\|T_rf\|_{X}\leq& C\al^{-1}(1+\log \al)\|f\|_{X},
\end{align*}
by taking $\al$ large. Then $I-T_r$ is invertible in $X$ and we construct $\tilde\phi_{01}$ by $\tilde\phi_{01}=(I-T_r)^{-1}1$, which gives 
\begin{align}\label{est: phi'_01-1}
C^{-1}\leq|\tilde\phi_{01}|\leq C,\quad \|\partial_Y\tilde\phi_{01}\|_{L^\infty}\leq C\text{ and }\|1-\tilde\phi_{01}\|_{L^\infty}\leq C\alpha^{-1}\log\alpha.
\end{align}
Moreover, we can obtain the estimates of $\partial_Y\tilde\phi_{01}$ for $Y$ close to $Y_c$. Indeed, we use the behavior $\psi_r(Y)\sim \alpha^{-1}\sinh \al w_c (Y)$  to have
\begin{align*}
&\Big|F_r\psi_r^{-2}\int_{Y_c}^{Y} (\tilde\phi_{01} g_{1r} )(Y')dY'\Big|
\leq\f{C}{\sinh^2\al w_c(Y)}\int_{Y_c}^{Y}\f{|Y'-Y_c|\sinh^2\al  w_c(Y')}{|Y'-Y_c|^2+c_i^2}dY'\|\tilde\phi_{01}\|_{L^\infty},
\end{align*}
which implies that for  $|Y-Y_c|\leq \al^{-1}$,
\begin{align*}
	\Big|F_r\psi_r^{-2}\int_{Y_c}^{Y} (\tilde\phi_{01} g_{1r} )(Y')dY'\Big|&\leq\f{C\al^2|Y-Y_c|^3}{\al^2|Y-Y_c|^2}\int_{Y_c}^{Y}\f{1}{|Y'-Y_c|^2+c_i^2}d Y'\|\tilde\phi_{01}\|_{L^\infty}\\&\leq C c_i^{-1}|Y-Y_c|\|\tilde\phi_{01}\|_{L^\infty}.
\end{align*}
Thus,  for $|Y-Y_c|\leq \al^{-1}$, we have 
\begin{align*}
\Big|F_r\psi_r^{-2}\int_{Y_c}^{Y} (\tilde\phi_{01} g_{1r} )(Y')dY'\Big|\leq C c_i^{-1}|Y-Y_c|\|\tilde\phi_{01}\|_{L^\infty}.
\end{align*}
We point out that for the estimate in the region $|Y-Y_c|\leq \al^{-1}$, we obtain a good factor $|Y-Y_c|$ at the price of factor $c_i^{-1}$, and {this estimate will be used to deal with $\tilde\phi_{0,2}$.}
Similar process for $g_{2r}$ gives 
\begin{align*}
\Big\|F_r\psi_r^{-2}\int_{Y_c}^{Y} (\tilde\phi_{01} g_{2r} )(Y')dY'\Big\|_{L^\infty}\leq&C|Y-Y_c|\|\tilde\phi_{01}\|_{L^\infty}\leq C\|\tilde\phi_{01}\|_{L^\infty}.
\end{align*}
The above three estimates ensure that for $|Y-Y_c|\leq \al^{-1}$,
\begin{align}\label{est: phi'_01-2}
|\pa_Y\tilde\phi_{01}|\leq c_i^{-1}|Y-Y_c|\|\tilde\phi_{01}\|_{L^\infty}\leq c_i^{-1}|Y-Y_c|.
\end{align}

We turn to give the  estimates for $\tilde\phi_{02}$. First, $\tilde\phi_{02}$ can be represented as
\begin{align*}
\tilde\phi_{02}(Y)=&1+\left(\al^2 \bM^2(Y_c)\tilde\phi'_{01}(Y_c)\right)\int_{Y_c}^Y (F\psi^{-2})(Y') dY'\\
&+\int_{Y_c}^Y (F\psi^{-2})(Y')\int_{Y_c}^{Y'}\left(W_1(Y'')\tilde\phi_{01}^2+W_2(Y'')\partial_Y\tilde\phi_{01}\tilde\phi_{01}\right)\tilde\phi_{02} dY'' dY'\\
\eqdef&1+\left(\al^2 \bM^2(Y_c)\tilde\phi'_{01}(Y_c)\right)\int_{Y_c}^Y (F\psi^{-2})(Y') dY'+T_i(\tilde\phi_{02}),
\end{align*}
where 
\begin{align*}
	W_1(Y)=-\Big(\f{F_r\psi^2}{F \psi_r^2} g_r-g\Big)\text{ and }W_2(Y)=(\f{\psi^2 F_r}{\psi^2_r F}-1)\partial_Y(F_r^{-1}\psi_r^2).
\end{align*}
For $W_1(Y)$, notice that 
\begin{align*}
|W_1(Y)|=&|F^{-1} \psi_r^{-2}(F_r\psi^2 g_{1r}-g_1F\psi_r^2)+F^{-1} \psi_r^{-2}(F_r\psi^2 g_{2r}-g_2F\psi_r^2)-g_3|\\
\leq&|F^{-1} \psi_r^{-2}(F_r\psi^2 g_{1r}-g_1F\psi_r^2)|+|F^{-1} \psi_r^{-2}(F_r\psi^2 g_{2r}|+|g_3|,
\end{align*}
which along with the facts 
\begin{align*}
&|F^{-1} \psi_r^{-2}(F_r\psi^2 g_{1r}-g_1F\psi_r^2)|=\left|F^{-1}\psi^2\left(F_r(\bM^{-1})_r\partial_Y(F_r^{-1}\partial_Y\bM_r)-F \bM^{-1}\partial_Y(F^{-1}\partial_Y\bM)\right)\right|\leq \f{Cc_i|\psi|^2}{|\bM|^2},\\
&|F^{-1} \psi_r^{-2}(F_r\psi^2 g_{2r}-g_2F\psi_r^2)|=|F^{-1}F_r^{-\f14}\partial_Y(F_r^{-1}\partial_Y(F_r^\f14))\psi^2(F_r-F)|\leq C c_i|\psi|^2,\quad |g_3|\leq C\al^2c_i|\psi|^2,
\end{align*}
implies 
\begin{align}\label{est:W1}
	|W_1(Y)|\leq Cc_i(\alpha^2+|\bar M(Y)|^{-2})|\psi|^2.
\end{align}
We also have 
\begin{align*}
|W_2(Y)|=\left|(\f{\psi^2 F_r}{\psi^2_r F}-1)\partial_Y(F_r^{-1}\psi_r^2)\right|\leq&\f{(|\psi_r|+|\psi_i|)|\psi_i||\partial_Y\psi_r|}{|\psi_r|}+C c_i |\psi_r\partial_Y\psi_r|.
\end{align*}
By the similar argument in \eqref{est: integral-2}, we have
\begin{align*}
&\int_{Y_c}^Y\f{1}{\sinh^2 \al w_c(Y')}\int_{Y_c}^{Y'}\f{\sinh^2 \al w_c(Y'')}{|Y-Y_c|^2+c_i^2} dY'' dY'\\
&\leq \int_{0}^{C\al} \f{\sinh^2 z}{z^2+\al ^2 c_i^2}\cdot \f{e^{-2z}}{1-e^{-2z}}dz\leq C(1+\log (\al c_i)),
\end{align*}
which along with \eqref{est: integral-1},\eqref{est: phi'_01-1}, \eqref{est:W1} deduce that 
\begin{align}\label{est: T_5-1}
&\left|\int_{Y_c}^Y (F\psi^{-2})(Y')\int_{Y_c}^{Y'}W_1(Y'')\tilde\phi_{01}^2f dY'' dY'\right|\\
\nonumber
\leq&C(c_i|\log c_i|+\al c_i+c_i\log (\al c_i))\|f\|_{L^\infty}
\leq Cc_i(|\log c_i|+\al)\|f\|_{L^\infty}.
\end{align}

For the second term in $T_if$, when $|Y-Y_c|\leq \al^{-1}$, we use $|\partial_Y\psi_r|\leq C$, $|\psi_i|\leq C c_i$ and \eqref{est: phi'_01-1} to infer
\begin{align*}
 c_i\int_{Y_c}^Y \f{1}{|\psi(Y')|^2}\int_{Y_c}^{Y'}|\partial_Y\psi_r\partial_Y\tilde\phi_{01} | dY'' dY'\leq&C c_i\int_{Y_c}^Y \f{|Y'-Y_c|}{|Y'-Y_c|^2+ c_i^2}dY'\leq C c_i|\log c_i|,
\end{align*}
and use \eqref{est: phi'_01-2} to obtain
\begin{align*}
&\int_{Y_c}^Y \f{1}{|\psi(Y')|^2}\int_{Y_c}^{Y'}\f{ c_i^2|\psi_r'|}{|\psi_r|}|\partial_Y\tilde\phi_{01} f| dY'' dY'\\
\leq&C\int_{Y_c}^Y \f{1}{|Y'-Y_c|^2+ c_i^2}\int_{Y_c}^{Y'}\f{ c_i|Y''-Y_c| }{|Y''-Y_c|} dY'' dY'\|f\|_{L^\infty}\\
\leq&Cc_i\int_{Y_c}^Y \f{1}{|Y'-Y_c|+c_i}dY'\|f\|_{L^\infty}\leq Cc_i |\log c_i|\|f\|_{L^\infty},
\end{align*}
and
\begin{align*}
&\int_{Y_c}^Y \f{1}{|\psi(Y')|^2}\int_{Y_c}^{Y'}c_i |\psi_r\partial_Y\psi_r|\partial_Y\tilde\phi_{01} f| dY'' dY'\\
\leq&C\int_{Y_c}^Y \f{1}{|Y'-Y_c|^2+ c_i^2}\int_{Y_c}^{Y'} c_i(|Y'-Y_c| + c_i)dY'' dY' \|f\|_{L^\infty}\leq C c_i \|f\|_{L^\infty}.
\end{align*}
When $|Y-Y_c|\geq \al^{-1}$,  we use $ |\partial_Y\psi_r(Y'')|\leq C e^{\al |w_c(Y'')|},~\psi_i(Y'')|\leq C c_i e^{\al| w_c(Y'')|}$ for $|Y''-Y_c|\geq \al^{-1}$ and $|\psi_i(Y'')|\leq C c_i $ for $|Y''-Y_c|\leq \al^{-1}$ to have the following estimates
 \begin{align*}
&\int_{Y_c}^Y \f{1}{|\psi(Y')|^2}\int_{Y_c}^{Y'}|\psi_i\partial_Y\psi_r\partial_Y\tilde\phi_{01}f| dY'' dY'\\
\leq&C c_i|\log c_i| \|f\|_{L^\infty}+C\al c_i\int_{|Y'-Y_c|\geq \al^{-1}}e^{-2\al |w_c(Y')|}\int_{Y_c}^{Y'}\al e^{2\al |w_c(Y'')|}dY''dY'\|f\|_{L^\infty}\\
 \leq& C c_i|\log c_i|\|f\|_{L^\infty}+C c_i\al \|f\|_{L^\infty}\leq C c_i(|\log c_i|+\al)\|f\|_{L^\infty},
 \end{align*}
and
\begin{align*}
&\int_{Y_c}^Y \f{1}{|\psi(Y')|^2}\int_{Y_c}^{Y'}\f{|\psi_i^2\partial_Y \psi_r|}{|\psi_r|}|\partial_Y\tilde\phi_{01}f | dY'' dY'\\
\leq&Cc_i|\log c_i| \|f\|_{L^\infty}+C\int_{|Y'-Y_c|\geq \al^{-1}} e^{-2\al |w_c(Y')|}\Big(\al c_i+\int_{|Y''-Y_c|\geq \al^{-1}}\al^3 c_i^2e^{2\al |w_c(Y'')|} dY''\Big) dY'\|f\|_{L^\infty}\\
\leq&C(c_i|\log c_i|+c_i+\al^2 c_i^2)\|f\|_{L^\infty}\leq Cc_i|\log c_i|\|f\|_{L^\infty},
\end{align*}
and
\begin{align*}
&\int_{Y_c}^Y \f{1}{|\psi(Y')|^2}\int_{Y_c}^{Y'}c_i |\psi_r\partial_Y\psi_r|\partial_Y\tilde\phi_{01} f | dY'' dY'\\
\leq&C c_i \|f\|_{L^\infty}+Cc_i\int_{|Y'-Y_c|\geq \al^{-1}} e^{-2\al |w_c(Y')|}\left(1+e^{2\al |w_c(Y')|}\right) dY' \|f\|_{L^\infty}\leq C c_i \|f\|_{L^\infty}.
\end{align*}
Then we obtain
\begin{align}\label{est: T_5-2}
\Big|\int_{Y_c}^Y (F\psi^{-2})(Y')\int_{Y_c}^{Y'}W_2(Y'')\partial_Y\tilde\phi_{01} \tilde\phi_{01}f dY'' dY'\Big|\leq C c_i(|\log c_i| +\al)\|f \|_{L^\infty}.
\end{align}
Combining \eqref{est: T_5-1} and \eqref{est: T_5-2}, we arrive at
\begin{align*}
|T_if(Y)|\leq Cc_i(|\log c_i|+\al)\|f\|_{L^\infty}.
\end{align*}
Thus, the operator $I-T_i$ is invertible in 
$L^\infty$ due to $(\al,c)\in \mathbb{H}_0$ such that $c_i(|\log c_i| +\al)\leq \f12$, 
and it holds that
\begin{align*}
\tilde\phi_{02}=(I-T_i)^{-1}\left(1+\left( \bM^2(Y_c)\partial_Y\tilde\phi_{01}(Y_c)\right)\int_{Y_c}^Y (F\psi^{-2})(Y') dY'\right).
\end{align*}

On the other hand, notice that 
\begin{align*}
\Big|\int_{Y_c}^Y (F\psi^{-2})(Y') dY'\Big|\leq& \int_{|Y'-Y_c|\leq \al^{-1}} \f{C}{|Y'-Y_c|^2+ c_i^2}dY'+C\al^2\int_{|Y'-Y_c|\geq \al^{-1}} e^{-2\al |w_c(Y')|}dY'\\
\leq& C(c_i^{-1}+\al)\leq Cc_i^{-1},
\end{align*}
which implies
\begin{align*}
\Big| \bM^2(Y_c)\partial_Y\tilde\phi_{01}(Y_c)\int_{Y_c}^Y (F\psi^{-2})(Y') dY'\Big|\leq C c_i.
\end{align*}
Thus, we obtain
\begin{align*}
|\tilde\phi_{02}-1|\leq C c_i(|\log c_i|+\al).
\end{align*}
According to $\tilde\phi_0=\tilde\phi_{01} \tilde\phi_{02}$ and $\tilde\phi_{01}$ a real function, we have
\begin{align*}
C^{-1}\leq|\tilde\phi_{0,r}|\leq C,\quad |\tilde\phi_{0,i}|\leq Cc_i(|\log c_i|+\al).
\end{align*}
Using the relation $\phi=\tilde\phi_0\psi=(\tilde\phi_{0,r}+i\tilde\phi_{0,i})(\psi_r+i \psi_i)$, we get 
\begin{align}\label{eq:phir-phii}
\phi_r=\tilde\phi_{0,r}\psi_r-\tilde\phi_{0,i}\psi_i,\quad \phi_i=\tilde\phi_{0,r}\psi_i+\tilde\phi_{0,i}\psi_r,
\end{align}
which gives 
\begin{align*}
|\phi_r|\leq&C\al^{-1}(\sinh \al |w_c(Y)|+c_i^2\al (|\log c_i| +\al)\cosh \al |w_c(Y)|)\\
\leq& C(\al^{-1}\sinh \al |w_c(Y)|+c_i^2(|\log c_i|+\al)), \\
|\phi_i|\leq& C\al^{-1}(\al c_i \cosh \al |w_c(Y)|+c_i(|\log c_i|+\al)\sinh \al |w_c(Y)|)
\leq Cc_i e^{\al |w_c(Y)|}.
\end{align*}

The proof is completed.
\end{proof}

\subsection{Solving the non-homogeneous equation.}

In this section, we first give the representation formula of the solution $\varphi(Y)$ to the non-homogeneous system \eqref{eq: Rayleigh-nonhomo} via the solution $\phi(Y)$ to homogeneous system   \eqref{eq: Rayleigh-homo} obtained in Section 4.1. Based on this formula, we prove Proposition \ref{pro: rayleigh-varphi}. 
For convenience, we let
\begin{align}\label{def: G(Y)}
G_f(Y)=\f{F(Y)}{\phi^2(Y)}\int_{Y_c}^{Y}\f{\phi(Y')}{\bM(Y')}f(Y')dY'=G_{f,r}(Y)+i G_{f,i}(Y).
\end{align}

Here we always assume that $\phi$ is the solution of the system \eqref{eq: Rayleigh-homo}  constructed in Proposition \ref{pro: bound-(phi, phi_0)}.

\begin{proposition}\label{pro: Rayleigh formula}
 We have the following representation formula for the solution of \eqref{eq: Rayleigh-nonhomo}:
\begin{align}\label{formula: varphi}
	\varphi(Y)=\phi(Y)\int_{\bar Y_1^*}^YG_f(Y')dY'+\mu_f(c)\phi(Y)\int_{\bar Y_1^*}^Y\f{F(Y')}{\phi^2(Y')}dY'
\end{align}
where $\mu_f(c)$ is a constant given by
\begin{align}\label{def: mu(c)}
	\mu_f(c)=-\Big(\int_{\bar Y_1^*}^{\bar Y_2^*}G_f(Y') dY'\Big)\Big(\int_{\bar Y_1^*}^{\bar Y_2^*}\f{F(Y)}{\phi^2(Y)}dY\Big)^{-1}.
\end{align}

\end{proposition}

\begin{proof}
The nonhomogeneous Rayleigh-type equation
\begin{align*}
           	\left\{
           	\begin{aligned}
      		&\bM\left(\partial_Y\left(\frac{\partial_Y\varphi}{F(Y)}\right)-\alpha^2\varphi\right)-\partial_Y\left(\frac{\partial_Y\bar M}{F(Y)}\right)\varphi=f,\quad Y\in(\bar Y_1^*,\bar Y_2^*),\\
           		&\varphi(\bar Y_1^*)=\varphi(\bar Y_2^*)=0.
           	\end{aligned}
           	\right.
           \end{align*}
            is equivalent to
 \begin{align*}
 \mathcal L_{\phi}(\varphi)=\f{\phi f}{\bM},\quad Y\in(\bar Y_1^*,\bar Y_2^*)\text{ and }\varphi(\bar Y_1^*)=\varphi(\bar Y_2^*)=0.	
 \end{align*}
This gives our result by the integration and noting that $\mu_f(c)$ is well-defined for $c_i>0$.
\end{proof}

\begin{remark}\label{rem: Rayleigh formula}
There are two formulas for $\varphi$. 
For $Y\in[\bar Y_1^*, Y_c]$, we use the formula \eqref{formula: varphi}.
For $Y\in[Y_c, \bar Y_2^* ]$, we use the formula
\begin{align}\label{formula: phi-2}
\varphi(Y)=-\phi(Y)\int_Y^{\bar Y_2^*} G_f(Y')dY'-\mu_f(c)\phi(Y)\int^{\bar Y_2^*}_Y\f{F(Y')}{\phi^2(Y')}dY'.
\end{align}
\end{remark}

We first give estimates for the first term in \eqref{formula: varphi} and \eqref{formula: phi-2}.

\begin{lemma}\label{lem: est of G}
Assume that $f$ is a real function and  $\mathrm{supp}(f)\subset[\bar Y_1-\d, \bar Y_1 ]\cup[\bar Y_2, \bar Y_2+\d]$ with $e^{\al w_0}f\in L^\infty$.
It holds that
\begin{align*}
&\Big|\phi(Y)\int_{\bar Y_1^*}^Y G_f(Y') dY'\Big|\leq
\left\{
\begin{aligned}
&C\al^{-2}e^{-\al w_0(Y)}  e^{-2 \al  \int_{Y}^{\bar Y_1-\d} F_r(Y')dY'}  \|e^{\al w_0}f\|_{L^\infty},\quad Y\in[\bar Y_1^*, \bar Y_1-\d],\\
&C\al^{-1}e^{-\al w_0(Y)}  \|e^{\al w_0}f\|_{L^\infty},\quad Y\in[ \bar Y_1-\d, Y_c],
\end{aligned}
\right.
\\
&\Big|\Im\Big(\phi(Y)\int_{\bar Y_1^*}^Y G_f(Y') dY'\Big)\Big|\leq
\left\{
\begin{aligned}
&C\al^{-1}c_ie^{-\al w_0(Y)}  e^{-2 \al  \int_{Y}^{\bar Y_1-\d} F_r(Y')dY'}\|e^{\al w_0}f\|_{L^\infty},\quad Y\in[\bar Y_1^*, \bar Y_1-\d],\\
&C c_ie^{-\al w_0(Y)}  \|e^{\al w_0}f\|_{L^\infty},\quad Y\in[ \bar Y_1-\d, Y_c],
\end{aligned}
\right.
\end{align*}
and for $Y\in[Y_c, \bar Y_2^*]$,
\begin{align*}
	\alpha\Big|\phi(Y)\int_Y^{\bar Y_2^*} G_f(Y') dY'\Big|+c_i^{-1}\Big|\mathrm{Im}\Big(\phi(Y)\int_Y^{\bar Y_2^*} G_f(Y') dY'\Big)\Big|\leq Ce^{-\al w_0(Y)}  \|e^{\al w_0}f\|_{L^\infty}.
\end{align*}
Moreover, we have
\begin{align*}
	\Big|\int^{\bar Y_2^*}_{\bar Y_1^*} G_f(Y') dY'\Big|+(\alpha c_i)^{-1}\Big|\int^{\bar Y_2^*}_{\bar Y_1^*} G_{f,i}(Y') dY'\Big|\leq C e^{-\al w_0(Y_c)}  \|e^{\al w_0}f\|_{L^\infty}.
\end{align*}

\end{lemma}

\begin{proof}
In view of the definition of $G_f(Y)$ in \eqref{def: G(Y)} and  $\mathrm{supp} (f)\subset[\bar Y_1-\d, \bar Y_1 ]\cup[\bar Y_2, \bar Y_2+\d]$, we know that $\mathrm{supp}(G_f)\subset [\bar Y_1^*, \bar Y_1 ]\cup[\bar Y_2, \bar Y_2^*]$  and $G_f(Y)$ has no singularity.
 The following inequalities are used frequently:
  for $Y\in [\bar Y_1^*, \bar Y_1]\cup[\bar Y_2, \bar Y_2^*]$, 
  \begin{align}\label{est: phi, phi_i- non singular}
  \begin{split}
&|Y-Y_c|\gtrsim 1,\quad  C^{-1}\leq|F(Y)|\leq C,\\
&C^{-1}\al^{-1}e^{\al | w_c(Y)|}\leq |\phi(Y)|, |\phi_r(Y)|\leq C\al^{-1}e^{\al  |w_c(Y)|},\quad|\phi_i(Y)|\leq C c_ie^{\al  |w_c(Y)|},
\end{split}
\end{align}
which follow from Proposition \ref{pro: bound-(phi, phi_0)}. From \eqref{est: phi, phi_i- non singular}, we know that for $|Y-Y_c|\gtrsim 1$,
\begin{align}\label{est: (G, G_i)}
\begin{split}
&\left|\f{\phi}{\bM}\right|+\left|\Re\left(\f{\phi}{\bM}\right)\right|\leq C|\phi|\leq C\al^{-1}e^{\al  |w_c(Y)|},\quad
\left|\Im\left(\f{\phi}{\bM}\right)\right|\leq C c_i e^{\al  |w_c(Y)|},\\
&|\phi^{-2}|+|\Re(\phi^{-2})|\leq C\al^2 e^{-2\al  |w_c(Y)|},\quad |\Im(\phi^{-2})|\leq C\al^3c_i e^{-2\al  |w_c(Y)|}.
\end{split}
\end{align}

\smallskip

\underline{Estimates of $G_f(Y)$ and $G_{f,i}(Y)$}: \smallskip

{\bf{Case 1.}} $Y\in[\bar Y_1-\d,\bar Y_1]$. By using \eqref{est: (G, G_i)} and $|w_c(Y)|=-w_c(Y)$ for $Y\in[\bar Y_1^*, Y_c]$, we  have
\begin{align}\label{est: G(Y)-1}
\begin{split}
|G_f(Y)|
\leq&C\al e^{-2\al |w_c(Y)|}\int^{\bar Y_1}_{Y}e^{ -\al  w_c(Y')}e^{-\al  w_0(Y')}dY'\|e^{\al w_0}f\|_{L^\infty}\\
\leq&C\al e^{2\al w_c(Y)}e^{\al w_0(Y_c)}\int^{\bar Y_1}_{Y}e^{-2 \al  w_0(Y')}dY'\|e^{\al w_0}f\|_{L^\infty}\\
\leq& C e^{-\al w_0(Y_c)}\|e^{\al w_0}f\|_{L^\infty},
\end{split}
\end{align}
and 
\begin{align}\label{est: G_i(Y)-1}
\begin{split}
|G_{f,i}(Y)|
\leq&\Big|\Re\Big(\f{F(Y)}{\phi^2(Y)}\Big)\int_{Y_c}^{Y}\Im\Big(\f{\phi(Y')}{\bM(Y')}\Big)f(Y')dY'\Big|\\
&+\Big|\Im\Big(\f{F(Y)}{\phi^2(Y)}\Big)\int_{Y_c}^{Y}\Re\Big(\f{\phi(Y')}{\bM(Y')}\Big)f(Y')dY'\Big|\\
\leq&C\al^2 c_i e^{-2\al |w_c(Y)|}\int^{\bar Y_1}_{Y}e^{ -\al  w_c(Y')}e^{-\al  w_0(Y')}dY'\|e^{\al w_0}f\|_{L^\infty}\\
\leq& C\al c_i e^{-\al w_0(Y_c)}\|e^{\al w_0}f\|_{L^\infty}.
\end{split}
\end{align}

{\bf{Case 2.}} $Y\in[\bar Y_1^*,\bar Y_1-\d]$. By using \eqref{est: (G, G_i)} again and similar argument  as above, and $\mathrm{supp} (f)\subset[\bar Y_1-\d, \bar Y_1]\cup[\bar Y_2, \bar Y_2+\d]$, we infer 
\begin{align}\label{est: G(Y)-2}
	\begin{split}
	|G_f(Y)|+(\al c_i)^{-1}|G_{f,i}(Y)|\leq&C\al e^{-2\al |w_c(Y)|}e^{\al w_0(Y_c)}\int^{\bar Y_1}_{\bar Y_1-\d}e^{-2 \al  w_0(Y')}dY'\|e^{\al w_0}f\|_{L^\infty}\\
\leq& C e^{-\al w_0(Y_c)}e^{-2 \al  \int_{Y}^{\bar Y_1-\d} F_r(Y')dY'}\|e^{\al w_0}f\|_{L^\infty}.	
	\end{split}
\end{align}

{\bf{Case 3.}} $Y\in[\bar Y_2,\bar Y_2+\d]$. Similar to Case 1, we have
\begin{align}\label{est: G(Y)-3}
	\begin{split}
	|G_f(Y)|+(\alpha c_i)^{-1}|G_{f,i}(Y)|	\leq&C\al e^{-2\al w_c(Y)}\int_{\bar Y_2}^{Y}e^{ \al  w_c(Y')}e^{-\al  w_0(Y')}dY'\|e^{\al w_0}f\|_{L^\infty}\\
\leq& C\al e^{-\al w_0(Y_c)}e^{-2\al w_c(Y)}|Y-\bar Y_2|\|e^{\al w_0}f\|_{L^\infty}.
	\end{split}
\end{align}

{\bf{Case 4.}} $Y\in[\bar Y_2+\d, \bar Y_2^*]$. Similar to Case 2, we have
\begin{align}\label{est: G(Y)-4}
	\begin{split}
	|G_f(Y)|+(\alpha c_i)^{-1}|G_{f,i}(Y)|\leq&
C\al e^{-2\al w_c(Y)}\int_{\bar Y_2}^{\bar Y_2+\d}e^{ \al  w_c(Y')}e^{-\al  w_0(Y')}dY'\|e^{\al w_0}f\|_{L^\infty}\\
\leq&C\al e^{-\al w_0(Y_c)}e^{-2\al w_c(Y)}|\bar Y_2+\d-\bar Y_2|\|e^{\al w_0}f\|_{L^\infty}.		
	\end{split}
\end{align}

\smallskip

\underline{Estimtes of $\phi(Y)\int_{\bar Y_1^*}^Y G_f(Y') dY'$ and $\Im\left(\phi(Y)\int_{\bar Y_1^*}^Y G_f(Y') dY'\right)$ for $Y\in[\bar Y_1^*, Y_c]$:} \smallskip

{\bf{Case 1.}} $Y\in[\bar Y_1^*, \bar Y_1-\d]$. By \eqref{est: G(Y)-2}, we have
\begin{align*}
\Big|\int_{\bar Y_1^*}^Y G_f(Y') dY'\Big|+(\alpha c_i)^{-1}\Big|\int_{\bar Y_1^*}^Y G_{f,i}(Y') dY'\Big|\leq&Ce^{-\al w_0(Y_c)}\int_{\bar Y_1^*}^{Y}  e^{-2 \al  \int_{Y'}^{\bar Y_1-\d} F_r(Y'')dY''}dY'\|e^{\al w_0}f\|_{L^\infty}\\
\leq&C\al^{-1}e^{-\al w_0(Y_c)}  e^{-2 \al  \int_{Y}^{\bar Y_1-\d} F_r(Y')dY'}\|e^{\al w_0}f\|_{L^\infty},	
\end{align*}
which along with Proposition \ref{pro: bound-(phi, phi_0)} and \eqref{est: phi, phi_i- non singular} implies
\begin{align*}
	&\Big|\phi(Y)\int_{\bar Y_1^*}^Y G_f(Y') dY'\Big|+(\alpha c_i)^{-1}\Big|\Im\Big(\phi(Y)\int_{\bar Y_1^*}^Y G_f(Y') dY'\Big)\Big|\\
	&\quad\leq C\al^{-2}e^{\al |w_c(Y)|}e^{-\al w_0(Y_c)}  e^{-2 \al  \int_{Y}^{\bar Y_1-\d} F_r(Y')dY'}\|e^{\al w_0}f\|_{L^\infty}\\
    &\quad \leq C\al^{-2}e^{-\al w_0(Y)}  e^{-2 \al  \int_{Y}^{\bar Y_1-\d} F_r(Y')dY'}\|e^{\al w_0}f\|_{L^\infty}.
\end{align*}

{\bf{Case 2.}} $Y\in[ \bar Y_1-\d, \bar Y_1]$. By \eqref{est: G(Y)-1}-\eqref{est: G(Y)-2} and the above result, we have
\begin{align}\label{est: int-G, case 2}
	\begin{split}
		&\Big|\int_{\bar Y_1^*}^Y G_f(Y') dY'\Big|+(\alpha c_i)^{-1}\Big|\int_{\bar Y_1^*}^Y G_{f,i}(Y') dY'\Big|\\
		&\quad\leq Ce^{-\al w_0(Y_c)}\|e^{\al w_0}f\|_{L^\infty}\Big(\int_{\bar Y_1^*}^{\bar Y_1-\d}   e^{-2 \al  \int_{Y'}^{\bar Y_1-\d} F_r(Y'')dY''}dY'+Ce^{-\al w_0(Y_c)}\int_{\bar Y_1-\d}^{Y}  dY'\Big)\\
&\quad \leq Ce^{-\al w_0(Y_c)}(\al^{-1}+|Y-(\bar Y_1-\d)|)  \|e^{\al w_0}f\|_{L^\infty}\leq Ce^{-\al w_0(Y_c)} \|e^{\al w_0}f\|_{L^\infty},
	\end{split}
\end{align}
which along with \eqref{est: phi, phi_i- non singular} implies
\begin{align*}
	&\Big|\phi(Y)\int_{\bar Y_1^*}^Y G_f(Y') dY'\Big|+(\alpha c_i)^{-1}\Big|\Im(\phi(Y)\int_{\bar Y_1^*}^Y G_f(Y') dY')\Big|\\
	&\quad\leq C\al^{-1}e^{\al |w_c(Y)|}e^{-\al w_0(Y_c)} \|e^{\al w_0}f\|_{L^\infty}
\leq C\al^{-1}e^{-\al w_0(Y)}  \|e^{\al w_0}f\|_{L^\infty}.
\end{align*}

{\bf{Case 3.}} $Y\in[\bar Y_1, Y_c]$. For $Y\in[\bar Y_1, Y_c-\al^{-1}]$, we have $\phi(Y)\sim \al^{-1}e^{\al |w_c(Y)|}=\al^{-1}e^{-\al w_c(Y)}$ and for $Y\in[Y_c-\al^{-1}, Y_c]$, we have $\phi(Y)\sim |Y-Y_c|+c_i$, which along with \eqref{est: int-G, case 2} for $Y=\bar Y_1$  gives 
\begin{align*}
\Big|\phi(Y)\int_{\bar Y_1^*}^Y G_f(Y') dY'\Big|\leq  C \al^{-1}e^{-\al w_c(Y)}\Big|\int_{\bar Y_1^*}^{\bar Y_1}G_f(Y') dY'\Big|
\leq C\al^{-1}e^{-\al w_0(Y)}  \|e^{\al w_0}f\|_{L^\infty},
\end{align*}
and 
\begin{align*}
\Big|\Im\Big(\phi(Y)\int_{\bar Y_1^*}^Y G_f(Y') dY'\Big)\Big|\leq &Cc_ie^{-\al w_0(Y)}  \|e^{\al w_0}f\|_{L^\infty}.
\end{align*}
\medskip

\underline{Estimtes of $\phi(Y)\int^{\bar Y_2^*}_Y G_f(Y') dY'$ and $\Im(\phi(Y)\int^{\bar Y_2^*}_Y G_f(Y') dY')$ for $Y\in[Y_c, \bar Y_2^*]$:}\smallskip  

{\bf{Case 1.}} $Y\in[\bar Y_2+\d, \bar Y_2^*]$. By \eqref{est: G(Y)-4}, we have
\begin{align*}
	&\Big|\int^{\bar Y_2^*}_Y G_f(Y') dY'\Big|+(\alpha c_i)^{-1}\Big|\int^{\bar Y_2^*}_Y G_{f,i}(Y') dY'\Big|\\
	&\quad\leq C\al e^{-\al w_0(Y_c)}|\bar Y_2+\d-\bar Y_2|\int^{\bar Y_2^*}_Y e^{-2\al w_c(Y')} dY' \|e^{\al w_0}f\|_{L^\infty}\\
&\quad\leq Ce^{-\al w_0(Y_c)}e^{-2\al w_c(Y)}\|e^{\al w_0}f\|_{L^\infty},
\end{align*}
which along with \eqref{est: phi, phi_i- non singular} deduces
\begin{align*}
	\Big|\phi(Y)\int^{\bar Y_2^*}_Y G_f(Y') dY'\Big|+(\alpha c_i)^{-1}\Big|\mathrm{Im}\Big(\phi(Y)\int^{\bar Y_2^*}_Y G_f(Y') dY'\Big)\Big|\leq& C\al^{-1}e^{-\al w_0(Y)}  \|e^{\al w_0}f\|_{L^\infty}.
\end{align*}

{\bf{Case 2.}} $Y\in[\bar Y_2,\bar Y_2+\d]$. By using the above result and \eqref{est: G(Y)-3}, we have
\begin{align}\label{est: int-G, case2-2}
\begin{split}
&\Big|\int^{\bar Y_2^*}_Y G_f(Y') dY'\Big|+(\alpha c_i)^{-1}\Big|\int^{\bar Y_2^*}_Y G_{f,i}(Y') dY'\Big|\\
&\quad\leq \Big|\int^{\bar Y_2^*}_{\bar Y_2+\d} G_f(Y') dY'\Big|+\Big|\int^{\bar Y_2+\d}_Y G_f(Y') dY'\Big|\\
&\quad\leq Ce^{-\al w_0(Y_c)}e^{-2\al w_c(\bar Y_2+\d)}\|e^{\al w_0}f\|_{L^\infty}\\
&\quad\quad+C\al e^{-\al w_0(Y_c)}\int^{\bar Y_2+\d}_Y e^{-2\al w_c(Y')}|Y'-\bar Y_2| dY'  \|e^{\al w_0}f\|_{L^\infty}\\
&\quad\leq Ce^{-\al w_0(Y_c)}e^{-2\al w_c(Y)}\|e^{\al w_0}f\|_{L^\infty},
\end{split}
\end{align}
which along with \eqref{est: phi, phi_i- non singular} implies
\begin{align*}
	\Big|\phi(Y)\int^{\bar Y_2^*}_Y G_f(Y') dY'\Big|+(\alpha c_i)^{-1}\Big|\Im\Big(\phi(Y)\int^{\bar Y_2^*}_Y G_f(Y') dY'\Big)\Big|\leq& C\al^{-1}e^{-\al w_0(Y)}  \|e^{\al w_0}f\|_{L^\infty}.
\end{align*}

{\bf{Case 3.}} $Y\in[Y_c,\bar Y_2]$. For $Y\in[Y_c+\al^{-1}, \bar Y_2]$, we have $\phi(Y)\sim \al^{-1}e^{\al w_c(Y)}$ and for $Y\in[Y_c, Y_c+\al^{-1}]$, we have $\phi(Y)\sim |Y-Y_c|+c_i$, which along with \eqref{est: int-G, case2-2} for $Y=\bar Y_2$  gives 
\begin{align*}
\Big|\phi(Y)\int^{\bar Y_2^*}_Y G_f(Y') dY'\Big|\leq  C \al^{-1}e^{\al w_c(Y)}\Big|\int^{\bar Y_2^*}_{\bar Y_2}G_f(Y') dY'\Big|
\leq C\al^{-1}e^{-\al w_0(Y)}  \|e^{\al w_0}f\|_{L^\infty},
\end{align*}
and 
\begin{align*}
\Big|\Im\Big(\phi(Y)\int^{\bar Y_2^*}_Y G_f(Y') dY'\Big)\Big|\leq &Cc_ie^{-\al w_0(Y)}  \|e^{\al w_0}f\|_{L^\infty}.
\end{align*}

From the above estimates and the fact $G_f(Y)=0$ for $Y\in[\bar Y_1,\bar Y_2]$, we deduce that 
\begin{align*}
\int^{\bar Y_2^*}_{\bar Y_1^*} G_f(Y') dY'=&\int^{\bar Y_2^*}_{\bar Y_2} G_f(Y') dY'+\int_{\bar Y_1^*}^{\bar Y_1} G_f(Y') dY'.
\end{align*}
Then we have
\begin{align*}
	\Big|\int^{\bar Y_2^*}_{\bar Y_1^*} G_f(Y') dY'\Big|+(\alpha c_i)^{-1}\Big|\int^{\bar Y_2^*}_{\bar Y_1^*} G_{f,i}(Y') dY'\Big|\leq&Ce^{-\al w_0(Y_c)}\|e^{\al w_0}f\|_{L^\infty}.
\end{align*}
\end{proof}

Next we present the estimates for the numerator of the second term in \eqref{formula: varphi} and \eqref{formula: phi-2}.

\begin{lemma}\label{lem: int-phi}
It holds that
\begin{align}
&\Big|\phi(Y)\int_{\bar Y_1^*}^Y \f{F(Y')}{\phi^{2}(Y')}dY'\Big|\leq Ce^{-\al |w_c(Y)|},\quad Y\in[\bar Y_1^*, Y_c ],\label{est: int-phi-1}\\
&\Big|\phi(Y)\int^{\bar Y_2^*}_Y \f{F(Y')}{\phi^{2}(Y')}dY'\Big|\leq Ce^{-\al |w_c(Y)|},\quad Y\in[Y_c, \bar Y_2^*]\label{est: int-phi-2}.
\end{align}
Moreover, for the imaginary part, we have
\begin{align}
&\Big|\Im\Big(\phi(Y)\int_{\bar Y_1^*}^Y\f{F(Y')}{\phi^{2}(Y')}dY'\Big)\Big|\leq C c_i\al e^{-\al |w_c(Y)|}, \quad Y\in[\bar Y_1^*, Y_c-\al^{-1} ],\label{est: int-Im phi-1}\\
&\Big|\Im\Big(\phi(Y)\int^{\bar Y_2^*}_Y \f{F(Y')}{\phi^{2}(Y')}dY'\Big)\Big|\leq C c_i\al e^{-\al |w_c(Y)|}, \quad Y\in[Y_c+\al^{-1}, \bar Y_2^*]\label{est: int-Im phi-2}.
\end{align}
\end{lemma}

\begin{proof}
We only give the proof for \eqref{est: int-phi-1} and \eqref{est: int-Im phi-1}. The proof for \eqref{est: int-phi-2} and \eqref{est: int-Im phi-2} are similar and we left details to readers.

For $Y\in[\bar Y_1^*, Y_c-\al^{-1} ]$, since the function $\phi^{-2}$ has no singularity, we get by \eqref{est: phi, phi_i- non singular} that
\begin{align*}
\Big|\phi(Y)\int_{\bar Y_1^*}^Y \f{F(Y')}{\phi^{2}(Y')}dY'\Big|\leq& C\al^{-1}e^{\al |w_c(Y)|}\int_{\bar Y_1^*}^Y \al^2e^{-2\al |w_c(Y')|}dY'
\leq Ce^{-\al |w_c(Y)|},
\end{align*}
and 
\begin{align*}
\Big|\Im\Big(\phi(Y)\int_{\bar Y_1^*}^Y \f{F(Y')}{\phi^{2}(Y')}dY'\Big)\Big|\leq& Cc_ie^{\al |w_c(Y)|}\int_{\bar Y_1^*}^Y \al^2e^{-2\al |w_c(Y')|}dY'
\\
&+C\al^{-1}e^{\al |w_c(Y)|}\int_{\bar Y_1^*}^Y c_ie^{\al |w_c(Y")|}\al^3e^{-3\al |w_c(Y')|}dY'\\
\leq&C\al c_ie^{-\al |w_c(Y)|}.
\end{align*}

For $Y\in[Y_c-\al^{-1}, Y_c]$, we get by Proposition \ref{pro: bound-(phi, phi_0)} that 
\begin{align*}
\Big|\phi(Y)\int_{\bar Y_1^*}^Y \f{F(Y')}{\phi^{2}(Y')}dY'\Big|\leq& C(|Y-Y_c|+ c_i)\Big(\int_{\bar Y_1^*}^{Y_c-\al^{-1}} \al^2 e^{-2\al |w_c(Y')|}dY'+\int_{Y_c-\al^{-1}}^Y \f{1}{|Y'-Y_c|^2+c_i^2}dY'\Big)\\
\leq&C\al (|Y-Y_c|+ c_i)+|Y-Y_c||Y-Y_c|^{-1}+ c_i c_i^{-1}\leq C\leq C e^{-\al |w_c(Y)|}.
\end{align*}
\end{proof}

The following results are related to $\mu_f(c)$.

\begin{lemma}\label{lem: int-phi^(-2)}
 Assume that $(\al,c)\in\mathbb{H}_0$ and $|w_c(\bar Y_1^*)|\neq|w_c(\bar Y_2^*)|$. Then there exists $C>1$ such that 
\begin{align*}
&C^{-1}\al\leq \Big|\int_{\bar Y_1^*}^{\bar Y_2^*}\f{F(Y)}{\phi^{2}(Y)}dY\Big|\leq C\al,\\
&\Im\Big(\int_{\bar Y_1^*}^{\bar Y_2^*}\f{F(Y)}{\phi^{2}(Y)}dY\Big)=-\f{\pa_Y^2 \bM_r(Y_c)\pi}{\pa_Y \bM_r(Y_c)^3}+O(\al^2c_i|\log c_i|).
\end{align*}

\end{lemma}

\begin{proof}
Let $I=\int_{\bar Y_1^*}^{\bar Y_2^*}\f{F(Y)}{\phi^{2}(Y)}dY$. We can write 
\begin{align*}
	I
	=-\int_{\bar Y_1^*}^{\bar Y_2^*}\frac{\bar M^2(Y)}{\phi^2(Y)}dY+\int_{\bar Y_1^*}^{\bar Y_2^*}\phi^{-2}(Y)dY=I_1+I_2.
\end{align*}
According to Proposition \ref{pro: bound-(phi, phi_0)}, we first have 
\begin{align}
	|I_1|=\Big|\int_{\bar Y_1^*}^{\bar Y_2^*}\frac{\bar M^2(Y)}{\phi^2(Y)}dY\Big|\leq C \int_{\bar Y_1^*}^{\bar Y_2^*}\frac{|Y-Y_c|^2}{(\alpha^{-1}\sinh (\alpha|Y-Y_c|)+c_i)^2}\leq C.
\end{align}
 Notice that 
\begin{align*}
	|\mathrm{Im}(\bar M^2/\phi^2)|\leq|\phi|^{-4}(|(U_B-c_r)^2-c_i^2||\phi_r\phi_i|+c_i(U_B-c)|\phi_r^2-\phi_i^2|),
\end{align*}
which implies 
\begin{align}
	|\mathrm {Im}(I_1)|\leq Cc_i\int_{\bar Y_1^*}^{\bar Y_2^*}\frac{1}{|Y-Y_c|+c_i}dY+Cc_i^3\int_{\bar Y_1^*}^{\bar Y_2^*}\frac{1}{(|Y-Y_c|+c_i)^3}dY\leq Cc_i|\log c_i|.
\end{align}

Now we turn to focus on $I_2$. We write 
\begin{align*}
I_2=\f{1}{ \pa_Y\phi(Y_c)}\int_{\bar Y_1^*}^{\bar Y_2^*}\f{ \pa_Y\phi(Y)}{\phi^2(Y)}dY+\f{1}{ \pa_Y\phi(Y_c)}\int_{\bar Y_1^*}^{\bar Y_2^*}\f{ \pa_Y\phi(Y_c)-\pa_Y\phi(Y)}{\phi^2(Y)}dY=I_{21}+I_{22}.	
\end{align*}
For $I_{21}$, we have
\begin{align*}
I_{21}=\f{1}{  \pa_Y\bM(Y_c)}\left(\f{1}{\phi(\bar Y_1^*)}-\f{1}{\phi(\bar Y_2^*)}\right),
\end{align*}
from which and \eqref{est: phi},  we deduce that $|I_{21}|\leq C$ and $ |\Im I_{21}|\leq C c_i.$
We introduce
\begin{align*}
g(Y)=\pa_Y\phi(Y)-\pa_Y\phi(Y_c)-\f{\pa_Y^2\phi(Y_c)}{\pa_Y\phi(Y_c)\pa_Y\phi_r(Y_c)}\pa_Y\phi(Y)\phi_r(Y),
\end{align*}
 which holds 
\begin{align*}
g(Y_c)=g'(Y_c)=0,\quad |g(Y)|\leq C|Y-Y_c|^2,\quad  |\Im g(Y)|\leq Cc_i|Y-Y_c|^2.
\end{align*}
Then we have
\begin{align*}
I_{22}=&-\f{1}{ \pa_Y\phi(Y_c)}\int_{\bar Y_1^*}^{\bar Y_2^*}\f{g(Y')}{\phi(Y)^2}dY-\f{\pa_Y^2\phi(Y_c)}{ (\pa_Y\phi(Y_c))^2\pa_Y\phi_r(Y_c)}\int_{\bar Y_1^*}^{\bar Y_2^*}\f{ \pa_Y\phi(Y)\phi_r(Y)}{\phi^2(Y)}dY\\
=&-\f{1}{ \pa_Y\phi(Y_c)}\int_{\bar Y_1^*}^{\bar Y_2^*}\f{g(Y')}{\phi(Y)^2}dY+\f{\pa_Y^2\phi(Y_c)}{ (\pa_Y\phi(Y_c))^2\pa_Y\phi_r(Y_c)}\log\f{\phi(\bar Y_2^*)}{\phi(\bar Y_1^*)}\\
&+i\f{\pa_Y^2\phi(Y_c)}{ (\pa_Y\phi(Y_c))^2\pa_Y\phi_r(Y_c)}\int_{\bar Y_1^*}^{\bar Y_2^*}\f{\phi_i(Y) \pa_Y \phi(Y)}{\phi^2(Y)}dY\\
=&I_{22}^1+I_{22}^2+I_{22}^3.
\end{align*}
For $I_{22}^1$, we get
\begin{align*}
|I_{22}^1|\leq C\int_{\bar Y_1^*}^{\bar Y_2^*}\f{|Y-Y_c|^2}{|Y-Y_c|^2+c_i^2}dY\leq C,
\end{align*}
and
\begin{align*}
|\Im I_{22}^1|\leq& C c_i\int_{\bar Y_1^*}^{\bar Y_2^*}\f{|g(Y')| |Y-Y_c|}{(|Y-Y_c|^2+c_i^2)^2}dY'+C \int_{\bar Y_1^*}^{\bar Y_2^*}\f{|\Im g(Y)| }{|Y-Y_c|^2+c_i^2}dY\\
\leq& Cc_i\int_{\bar Y_1^*}^{\bar Y_2^*}\f{|Y-Y_c|^3}{(|Y'-Y_c|^2+c_i^2)^2}dY+Cc_i\int_{\bar Y_1^*}^{\bar Y_2^*}\f{|Y-Y_c|^2}{|Y'-Y_c|^2+c_i^2}dY
\leq  C c_i |\log c_i|.
\end{align*}
For $I_{22}^2$, by \eqref{eq: Rayleigh-homo} and Proposition \ref{pro: bound-(phi, phi_0)}, $\partial_Y\phi_r(Y_c)=\partial_Y\bar M_r(Y_c)+\mathcal O(c_i(\al+|\log c_i|))$ and $\partial_Y^2\phi(Y_c)=\partial_Y^2\bar M(Y_c)+\mathcal O(\al^2c_i)$. Then we have
\begin{align}\label{eq:phi-Yc}
	\f{\pa_Y^2\phi(Y_c)}{ (\pa_Y\phi(Y_c))^2\pa_Y\phi_r(Y_c)}=\f{\pa_Y^2\bM(Y_c)}{ (\pa_Y\bM(Y_c))^2\pa_Y\bM_r(Y_c)}+\mathcal O(\al^2c_i).
\end{align}
Again by Proposition \ref{pro: bound-(phi, phi_0)}, we have 
\begin{align}\label{eq: fr-phi}
	\frac{\phi(\bar Y_1^*)}{\phi(\bar Y_2^*)}=\frac{\psi(\bar Y_1^*)\tilde\phi_0(\bar Y_1^*)}{\psi(\bar Y_2^*)\tilde\phi_0(\bar Y_2^*)}=\frac{\psi(\bar Y_1^*)}{\psi(\bar Y_2^*)}(1+(\frac{\tilde\phi_0(\bar Y_1^*)}{\tilde\phi_0(\bar Y_2^*)}-1)).
\end{align}
According to \eqref{sim: psi}, we have 
\begin{align}\label{eq:fr-psi}
	\Big|\frac{\psi(\bar Y_1^*)}{\psi(\bar Y_2^*)}\Big|\sim e^{\alpha |w_c(\bar Y_2^*)|}e^{-\alpha |w_c(\bar Y_1^*)|}\text{ and }\arg\Big(\frac{\psi(\bar Y_1^*)}{\psi(\bar Y_2^*)}\Big)=\pi+\mathcal O(c_i).
\end{align}
On the other hand, by Proposition \ref{pro: bound-(phi, phi_0)}, we get
\begin{align*}
	\Big|\frac{\tilde\phi_0(\bar Y_1^*)}{\tilde\phi_0(\bar Y_2^*)}-1\Big|\leq C\alpha^{-1}\log\alpha\text{ and }\Big|\Im\Big(\frac{\tilde\phi_0(\bar Y_1^*)}{\tilde\phi_0(\bar Y_2^*)}-1\Big)\Big|\leq Cc_i|\log c_i|,
\end{align*}
which implies 
\begin{align}\label{eq:fr-tphi}
	\Big|\frac{\tilde\phi_0(\bar Y_1^*)}{\tilde\phi_0(\bar Y_2^*)}\Big|\sim 1\text{ and }\arg\Big(\frac{\tilde\phi_0(\bar Y_1^*)}{\tilde\phi_0(\bar Y_2^*)}\Big)\sim c_i|\log c_i|.
\end{align}
By gathering \eqref{eq: fr-phi}, \eqref{eq:fr-psi} and \eqref{eq:fr-tphi}, we arrive at
\begin{align*}
	\Big|\frac{\phi(\bar Y_1^*)}{\phi(\bar Y_2^*)}\Big|\sim e^{\alpha |w_c(\bar Y_2^*)|}e^{-\alpha |w_c(\bar Y_1^*)|}\text{ and }\arg\Big(\frac{\phi(\bar Y_1^*)}{\phi(\bar Y_2^*)}\Big)= \pi+\mathcal O(c_i),
\end{align*}
which along with \eqref{eq:phi-Yc} implies 
\begin{align*}
|I_{22}^2|\sim \al||w_c(\bar Y_1^*)|-|w_c(\bar Y_2^*)||\sim\al,\quad \Im I_{22}^2=-\f{\pa_Y^2 \bM_r(Y_c)}{\pa_Y \bM_r(Y_c)^3}\pi+O(c_i).
\end{align*}

For $I_{22}^3$,  by Proposition \ref{pro: bound-(phi, phi_0)}, we have
\begin{align*}
\Big|\int_{\bar Y_1^*}^{\bar Y_2^*}\f{\phi_i(Y) \pa_Y \phi(Y)}{\phi^2(Y)}dY\Big|\leq& \Big|\int_{\bar Y_1^*}^{\bar Y_2^*}\f{\tilde\phi_{0,r}\psi_i\partial_Y\phi(Y)}{\phi^2(Y)}dY\Big|+\Big|\int_{\bar Y_1^*}^{\bar Y_2^*}\f{\tilde\phi_{0,i}\psi_r\partial_Y\phi(Y)}{\phi^2(Y)}dY\Big|\\
\leq& Cc_i|\log c_i|+\Big|\int_{\bar Y_1^*}^{\bar Y_2^*}\f{\partial_Y(\tilde{\phi}_{0,r}\psi_i)}{\phi(Y)}dY\Big|+ c_i\Big|\int_{\bar Y_1^*}^{\bar Y_2^*}\f{1}{|\phi(Y)|}dY\Big|\\
\leq& c_i(|\log c_i|+\al),
\end{align*}
which along with \eqref{eq:phi-Yc} implies that  $|I_{22}^3|\leq C c_i\al|\log c_i|$. Putting $I_{22}^1-I_{22}^3$ together, we get
\begin{align*}
|\Re I_{22}|\sim \alpha ,\quad \Im I_{22}=-\f{\pa_Y^2 \bM_r(Y_c)\pi}{\pa_Y \bM_r(Y_c)^3}+O(c_i\al^2 |\log c_i|).
\end{align*}

Finally,  we conclude that 
\begin{align*}
|\Re I_2|\sim \al,\quad \Im I_{2}=-\f{\pa_Y^2 \bM_r(Y_c)\pi}{\pa_Y \bM_r(Y_c)^3}+O(c_i\al^2 |\log c_i|).
\end{align*}
\end{proof}

\subsection{Proof of Proposition \ref{pro: rayleigh-varphi}}
With results in Lemma \ref{lem: est of G}-Lemma \ref{lem: int-phi^(-2)}, we are in a position to prove Proposition \ref{pro: rayleigh-varphi}.

\begin{proof}

For $Y\in[\bar Y_1^*, Y_c]$, we use the formula \eqref{formula: varphi} for $\varphi(Y)$. 
By Lemma \ref{lem: est of G}-\ref{lem: int-phi^(-2)} , we have
\begin{align*}
&\Big|\phi(Y)\int_{\bar Y_1^*}^Y G_f(Y')dY'\Big|\leq C\al^{-1}e^{-\al w_0(Y)}  \|e^{\al w_0}f\|_{L^\infty},\quad \Big|\phi(Y)\int_{\bar Y_1^*}^Y \f{F(Y')}{\phi^2(Y')}dY'\Big|\leq Ce^{-\al |w_c(Y)|},\\
&\Big|\int^{\bar Y_2^*}_{\bar Y_1^*} G_f(Y') dY'\Big|\leq Ce^{-\al w_0(Y_c)}\|e^{\al w_0}f\|_{L^\infty},\quad \Big|\int_{\bar Y_1^*}^{\bar Y_2^*}\f{F(Y')}{\phi^2(Y')}dY'\Big|\geq C^{-1}\al,
\end{align*}
and  for $Y\in[\bar Y_1^*, \bar Y_1-\d]$,
\begin{align*}
\Big|\phi(Y)\int_{\bar Y_1^*}^Y G_f(Y')dY'\Big|\leq&C\al^{-2}e^{-\al w_0(Y)} e^{-2 \al  \int_{Y}^{\bar Y_1-\d} F_r(Y')dY'}  \|e^{\al w_0}f\|_{L^\infty},
\end{align*}
from which, we infer that  for $Y\in[\bar Y_1^*, Y_c]$,
\begin{align*}
|\varphi(Y)|\leq C\al^{-1}e^{-\al w_0(Y)}  \|e^{\al w_0}f\|_{L^\infty}.
\end{align*}
Moreover, for $Y\in[\bar Y_1^*, \bar Y_1-\d]$, we have 
\begin{align*}
|\varphi(Y)|\leq& C\al^{-2}e^{-\al w_0(Y)} e^{-2 \al  \int_{Y}^{\bar Y_1-\d} F_r(Y')dY'}  \|e^{\al w_0}f\|_{L^\infty}\\
&+C\al^{-1}  e^{-\al |w_c(Y)|}e^{-\al w_0(Y_c)}\|e^{\al w_0}f\|_{L^\infty}\\
\leq&C\al^{-2}e^{-\al w_0(Y)}  e^{-2\al \int_{Y}^{\bar Y_1-\d}F_r^\f12(Y') dY'}\|e^{\al w_0}f\|_{L^\infty}.
\end{align*}
Similar procedure gives that for $Y\in[Y_c, \bar Y_2^*]$,
\begin{align*}
|\varphi(Y)|\leq&C\al^{-1}e^{-\al w_0(Y)}  \|e^{\al w_0}f\|_{L^\infty}.
\end{align*}

Next we give the estimates for $\Im \varphi$. Since $\varphi$ satisfies  $\mathcal{L}_{cr}[\varphi]=f$ with $f=\Re f+i\Im f$, we divide $\varphi=\varphi_1+i\varphi_2$, where $\mathcal{L}_{cr}[\varphi_1]=\Re f,~\mathcal{L}_{cr}[\varphi_2]=\Im f$. Then $\Im \varphi=\Im \varphi_1-\Re \varphi_2$. To estimate  $\Re \varphi_2$, we use the above estimates to obtain
\begin{align*}
&|\Re \varphi_2(Y)|\leq |\varphi_2(Y)|\leq C\al^{-1}e^{-\al w_0(Y)}  \|e^{\al w_0}\Im f\|_{L^\infty},\quad Y\in[\bar Y_1^*, \bar Y_2^*],\\
&|\Re \varphi_2(Y)|\leq |\varphi_2(Y)|\leq C\al^{-2}e^{-\al w_0(Y)}  e^{-2\al \int_{Y}^{\bar Y_1-\d}F_r^\f12(Y') dY'}\|e^{\al w_0}\Im f\|_{L^\infty},\quad Y\in[\bar Y_1^*, \bar Y_1-\d].
\end{align*}
Then we focus on the estimates for $\Im \varphi_1$. By Lemma \ref{lem: est of G}, for $Y\in[\bar Y_1^*, \bar Y_1-\d]$, we have
\begin{align}\label{est: Im varphi-1}
\Big|\Im\Big(\phi(Y)\int_{\bar Y_1^*}^Y G_{\Re f}(Y') dY'\Big)\Big|\leq&C\al^{-1}c_ie^{-\al w_0(Y)}  e^{-2 \al  \int_{Y}^{Y_1-2\d_0} F_r(Y')dY'}\|e^{\al w_0} \Re f\|_{L^\infty},
\end{align}
and for $Y\in [ \bar Y_1-\d, \bar Y_1]\cup [\bar Y_2, \bar Y_2^*]$,
\begin{align}\label{est: Im varphi-2}
\Big|\Im\Big(\phi(Y)\int_Y^{\bar Y_2^*} G_{\Re f}(Y') dY'\Big)\Big|\leq
&C c_ie^{-\al w_0(Y)}  \|e^{\al w_0} \Re f\|_{L^\infty},
\end{align}
and 
\begin{align*}
&\Big|\int^{\bar Y_2^*}_{\bar Y_1^*} G_{\Re f}(Y') dY'\Big|+(\al c_i)^{-1}\Big|\int^{\bar Y_2^*}_{\bar Y_1^*} G_{\Re f,i}(Y') dY'\Big|\leq Ce^{-\al w_0(Y_c)}  \|e^{\al w_0}\Re f\|_{L^\infty}.
\end{align*}
By Lemma \ref{lem: int-phi^(-2)}, we have
\begin{align*}
\Big|\Big(\int_{\bar Y_1^*}^{\bar Y_2^*}\f{F(Y')}{\phi^2(Y')}dY'\Big)^{-1}\Big|\leq C\al^{-1},\quad \Big|\Im\Big(\Big(\int_{\bar Y_1^*}^{\bar Y_2^*}\f{F(Y')}{\phi^2(Y')}dY'\Big)^{-1}\Big)\Big|\leq C\al^{-2}.
\end{align*}
The above estimates along with Lemma \ref{lem: int-phi} and \eqref{def: mu(c)} ensure that for $Y\in[\bar Y_1^*, \bar Y_1]$,
\begin{align}\label{est: Im varphi-3}
\begin{split}
\Big|\Im\Big(\mu_{Re f}(c)\phi(Y)\int_{\bar Y_1^*}^Y\f{F(Y')}{\phi^2(Y')}dY'\Big)\Big|\leq& C\al^{-2} e^{-2\al|w_c(Y)|}e^{-\al w_0(Y)}\|e^{\al w_0}\Re f\|_{L^\infty}\\
\leq&C\al^{-3}e^{-\al w_0(Y)}  e^{-2 \al  \int_{Y}^{\bar Y_1-\d} F_r(Y')dY'}\|e^{\al w_0} \Re f\|_{L^\infty}.
\end{split}
\end{align}
Similarly, for $Y\in[\bar Y_2, \bar Y_2^*]$, we have
\begin{align}\label{est: Im varphi-4}
&\Big|\Im\Big(\mu_{Re f}(c)\phi(Y)\int_{\bar Y_1^*}^Y\f{F(Y')}{\phi^2(Y')}dY'\Big)\Big|\leq C\al^{-2}e^{-\al w_0(Y)}\|e^{\al w_0}\Re f\|_{L^\infty},
\end{align}
which gives 
\begin{align*}
&|\Im \varphi_1(Y)| \leq C\al^{-2}e^{-\al w_0(Y)}  \|e^{\al w_0}\Re f\|_{L^\infty},\quad Y\in[\bar Y_1^*, \bar Y_1]\cup[\bar Y_2, \bar Y_2^*],\\
&|\Im \varphi_1(Y)|\leq C\al^{-3}e^{-\al w_0(Y)}  e^{-2 \al  \int_{Y}^{\bar Y_1-\d} F_r(Y')dY'}\|e^{\al w_0} \Re f\|_{L^\infty},\quad Y\in[\bar Y_1^*, \bar Y_1-\d].
\end{align*}

In the end, we combine the estimates of $\Re \varphi_2$ to obtain the results of $\Im\varphi$.
\end{proof}

\section{Inner-outer gluing iteration scheme} 

The goal of this section is to construct  a homogeneous solution $P(Y)$ to the following system 
\begin{align}\label{eq:ray-F-nbv}
	\left\{
	\begin{aligned}
		&\mathcal{L}[P]=\partial_Y^2P-\frac{2\partial_Y\bar M}{\bar M}\partial_Y P-\alpha^2(1-\bar M^2) P=0,\\
		&\lim_{Y\to +\infty}P(Y)=0,
	\end{aligned}
	\right.
\end{align}
such that $P(Y)$ contains the interaction between the supersonic regime and the critical layer in the subsonic regime. We point out that the solution to \eqref{eq:ray-F-nbv} doesn't give a restriction at $Y=0$. 

\subsection{Roadmap of our construction} We introduce  the inner-outer gluing iteration.\smallskip

 \textbf{$\bullet$ The initial step.} In this step, we give the leading order $P^{(0)}(Y)$ of the solution $P(Y)$. That is, $P^{(0)}=P_{0,0}+P_{0,1}+P_{0,2}$.
 The operators $\mathcal{G}_1$, $\mathcal G_{2,nloc}$ and $\mathcal{G}_{2,loc}$ related to the inner-outer gluing iteration are defined in Section \ref{sec:IN}. 
 We define that for any $Y\geq 0$,
	      \begin{align*}
	      	P_{0,0}(Y)=\mathcal{G}_{1}(A).
	      \end{align*}
	      We split $\mathcal L[P_{0,0}]=H_1[A]+H_2[A]$, where $H_1[A]$ and $H_2[A]$ are the errors in the regime near and far away from critical layer respectively. The exact formulas of $H_1[\cdot]$ and $H_2[\cdot]$ are given in \eqref{def: H_1-H_4}.
	      
We define 
	     \begin{align*}
	     	P_{0,1}(Y)=\mathcal{G}_{2,nloc}\circ\varphi_0,\quad \varphi_0=\mathcal{L}_{cr}^{-1}(-H_1[A]).
	     \end{align*}
	  Then we have
 \begin{align*}
\mathcal{L}[P_{0,1}]=-H_1[A] + H_3[\varphi_0] + H_4[\varphi_0],
\end{align*}
where $H_4[\varphi_0]$ and $H_3[\varphi_0]$ are the errors related to $\varphi_0$ in the regime near and far away from the critical layer respectively. The exact formulas of $H_3[\cdot]$ and $H_4[\cdot]$ are given in \eqref{def: H_1-H_4}.     Therefore, we obtain  
	     \begin{align*}
	     	\mathcal L[P_{0,0}+P_{0,1}]=H_2[A]+H_3[\varphi_0]+H_4[\varphi_0],
	     \end{align*}
We write  
	\begin{align*}
		H_3[\varphi_0]=H_3^{ncr}[\varphi_0]+H_3^{sup}[\varphi_0].
	\end{align*}

We define 
\begin{align}\label{eq: tilde P_(0,2)^(1)}
P_{0,2}(Y)=\mathcal{G}_{1}\circ\tilde{P}_{0,2}^{(1)}+\bar\chi_1  \tilde P_{0,2}^{(2)}(Y),\quad \tilde{P}_{0,2}^{(1)}=\mathcal L_{app}^{-1}(-H_2[A]-H_3^{ncr}[\varphi_0]),
\end{align}
where $\bar\chi_1$ is a smooth cut-off function supported in the supersonic regime, and the reverse wave $\tilde P_{0,2}^{(2)}$ satisfies 
\begin{align}\label{eq: tilde P_(0,2)^(2)}
\mathcal L[\tilde P_{0,2}^{(2)}]=-H_{3}^{sup}[\varphi_0]+(Q_1+Q_2)\tilde P_{0,2}^{(2)},
\end{align}
and is given by 
{\small
 \begin{align}\label{formula: tilde P_(0,2)^(2)}
\begin{split}
\tilde P_{0,2}^{(2)}(Y)=&-\kappa^{-1}\pi A(Y)\int_Y^{+\infty} B(Z)\bar M^{-2}H_{3}^{sup}[\varphi_0](Z)dZ-\kappa^{-1}\pi B(Y)\int_0^{Y}A(Z)\bar M^{-2}H_{3}^{sup}[\varphi_0](Z)dZ.
\end{split}
\end{align}}
We emphasize that $\tilde P_{0,2}^{(2)}(Y)$ is the key part to obtain the instability, which carries the effect of the critical layer and oscillates in the supersonic regime.
Then we have
 \begin{align*}
	     	\mathcal L[P_{0,2}]=-H_2[A]-H_3[\varphi_0]+H_1[\tilde P_{0,2}^{(1)}]+H_2[\tilde P_{0,2}^{(1)}]+\tilde H_{21}[\tilde P_{0,2}^{(2)}],
	     \end{align*}
which implies 
 \begin{align*}
	     	\mathcal L[P^{(0)}]=\underbrace{H_4[\varphi_0]+H_1[\tilde P_{0,2}^{(1)}]}_{\text{near critical layer}}+\underbrace{H_2[\tilde P_{0,2}^{(1)}]+\tilde H_{21}[\tilde P_{0,2}^{(2)}]}_{\text{away from critical layer}}.
	     \end{align*}

\textbf{$\bullet$ Iteration scheme.} For $j=1$, we define 
\begin{align*}
P_{1,1}(Y)=\mathcal{G}_{2,loc}\circ \varphi_1,\quad \varphi_1=\mathcal L_{cr}^{-1}(-H_1[\tilde P_{0,2}^{(1)}]-H_4[\varphi_{0}]).
\end{align*}
Then we have 
\begin{align*}
\mathcal L[P_{1,1}]
=&-H_4[\varphi_{0}]-H_1[\tilde P_{0,2}^{(1)}]+\tilde H_3[\varphi_{1}]+H_4[\varphi_{1}],
\end{align*}
where $\tilde H_3[\varphi_{1}]$ is the error related to $\varphi_1$ and the support of $\tilde H_3[\varphi_{1}]$ lies in the regime away from the critical layer and the supersonic regime.  Next we define  
	     \begin{align*}
	     	P_{1,2}(Y)=\mathcal{G}_{1}\circ\tilde P_{1,2},\quad \tilde P_{1,2}=\mathcal L_{app}^{-1}(-H_2[\tilde P_{0,2}^{(1)}]-\tilde H_{21}[\tilde P_{0,2}^{(2)}]-\tilde H_3[\varphi_{1}]).
	     \end{align*}
Then we denote $P^{(1)}=P_{1,1}+P_{1,2}$ and we obtain
\begin{align*}
\mathcal L[P^{(0)}+P^{(1)}]=H_1[\tilde P_{1,2}]+H_2[\tilde P_{1,2}]+H_4[\varphi_{1}]:=Err^{(1)}(Y).
\end{align*}
For any $j\geq 2$, we assume that $Err^{(j-1)}(Y)$ is the error from the $(j-1)-$th iteration with the form of $Err^{(j-1)}=H_4[\varphi_{j-1}]+H_1[\tilde P_{j-1,2}]+H_2[\tilde P_{j-1,2}].$
	     We define 
	     \begin{align}
	     	&P_{j,1}(Y)=\mathcal{G}_{2,loc}\circ\varphi_j,\quad \varphi_j=\mathcal L_{cr}^{-1}(-H_1[\tilde P_{j-1,2}]-H_4[\varphi_{j-1}]),\label{eq: Rayleigh varphi_j}\\
		&P_{j,2}(Y)=\mathcal{G}_{1}\circ\tilde P_{j,2},\quad \tilde P_{j,2}=\mathcal L_{app}^{-1}(-H_2[\tilde P_{j-1,2}]-\tilde H_3[\varphi_{j}])\label{eq: tilde P_(j,2)}.
	     \end{align}
	 	     Therefore, we obtain that $P^{(j)}=P_{j,1}+P_{j,2}$ satisfies
	     \begin{align*}
	     	\mathcal L[P^{(j)}]=&-H_4[\varphi_{j-1}]-H_1[\tilde P_{j-1,2}](Y)-H_2[\tilde P_{j-1,2}](Y)\\
	     	&+H_4[\varphi_{j}]+H_1[\tilde P_{j,2}](Y)+H_2[\tilde P_{j,2}](Y)
	     	=-Err^{(j-1)}(Y)+Err^{(j)}(Y).
	     \end{align*}
	     We can obtain iteratively that $\mathcal L[\sum_{j=0}^N P^{(j)}]=Err^{(N)}(Y)$. Then we know $P=\sum_{j=0}^{\infty} P^{(j)}$ solves \eqref{eq:ray-F-nbv} formally.

\begin{proposition}\label{pro: exist}
	Let $(\alpha,c)\in \mathbb H_0$. Then $\sum_{j=0}^{\infty} P^{(j)}$ converges in $W^{2,\infty}_{w_0}$ and $P=\sum_{j=0}^{\infty} P^{(j)}$ is a solution to \eqref{eq:ray-F-nbv}. Moreover, we have
	\begin{align*}
		&\partial_Y P(0)=\partial_Y A(0)+\partial_Y\tilde P_{0,2}^{(1)}(0)+\partial_Y\tilde P_{0,2}^{(2)}(0)+ \sum_{j=1}^{\infty}\partial_Y\tilde P_{j,2}(0).
	\end{align*}
\end{proposition}

We give the definitions related to the errors in the process of iteration: 
\begin{align}\label{def: H_1-H_4}
\begin{split}
	      	H_1[\tilde P]=&\chi\left((Q_1+Q_2)\chi_1 \tilde P-\alpha^2F\int_Y^{+\infty}\partial_{Z}\chi_1\tilde PdZ\right)+\partial_Y\chi_1\partial_Y\tilde P ,\\
H_2[\tilde P]=&(1-\chi)\left((Q_1+Q_2)\chi_1 \tilde P-\alpha^2 F\int_Y^{+\infty}\partial_{Z}\chi_1\tilde PdZ\right)=:H_{21}[\tilde P]+H_{22}[\tilde P],\\
\tilde H_{21}[\tilde P]=&\left(\bar\chi_1(Q_1+Q_2)+\pa_Y^2\bar\chi_1-\f{2\pa_Y\bM}{\bM}\pa_Y \bar\chi_1\right)\tilde P+2\pa_Y\bar\chi_1\pa_Y\tilde P,\\
&\quad\mbox{with} \quad \mathrm{supp}(\tilde H_{21}[\tilde P])\subset[0, \f32 \bar Y]\subset[0, Y^{-}],\\
H_3[\varphi]=&2\partial_Y\chi_2\bar M\varphi+(1-\chi)(1-\bar\chi)(1-\bar M^2)\mathcal{K}_B[\varphi]+\bar\chi(1-\bar M^2)\mathcal{K}_B[\varphi]\\
=:&\underbrace{H_{31}[\varphi]+H_{32}^{(1)}[\varphi]}_{H_3^{ncr}[\varphi]} +\underbrace{H_{32}^{(2)}[\varphi]}_{H_3^{sup}[\varphi]} ,\\
\tilde H_3[\varphi]=&2\partial_Y\chi_2\bar M\varphi+\Big((1-\bM^2)\chi_3(1-\chi)+\al^{-2}(\f{2\pa_Y \bM}{\bM}\pa_Y \chi_3-\pa_Y^2 \chi_3)\Big)\mathcal K_B[\varphi]\\
=:&H_{31}[\varphi]+\tilde H_{32}[\varphi],\\
H_4[\varphi]=&\chi(1-\bar M^2)\mathcal{K}_B[\varphi].
\end{split}
\end{align}
Here 
\begin{align}\label{def:KB}
Q_2(Y)=-\al^2 F_i(Y)\chi_0(Y),\quad \mathcal K_B[f](Y)=\int_{Y}^{+\infty}\pa_Z\Big(\f{\pa_Z \chi_2 \bM^2}{1-\bM^2}\Big)\bM^{-1}f dZ.
\end{align}
\begin{remark}\label{rem: cut-off functions}
	We gather the relation of the support of these cut-off functions as follows. Let $\delta_0\sim1$ be a positive constant dependent of distance $Y_c-Y_0$, but independent of $\alpha$. Let $\bar Y$ lie in the supersonic regime such that $0<\f32\bar Y\leq \f{9}{10} Y_0<Y^{-}$.

\begin{enumerate}
	\item The smooth cut-off function $\chi_1(Y)$ satisfies $\chi_1(Y)= 1$  for $Y\in[0,Y_1]\cup[Y_2,+\infty)$ and   $\chi_1(Y)=0$ for  $Y\in[Y_1+\delta_0, Y_2-\delta_0]$, where 
\begin{align*}
Y_1+2\delta_0<Y_c< Y_2-2\delta_0.
\end{align*} 
In section 4, we take $\bar Y_1=Y_1+\d_0,~\bar Y_2=Y_2-\d_0$ and $\d=3\d_0$.  Thus we write  $\bar Y_1-\d=Y_1-2\d_0$.   
	   
\item The smooth cut-off function $\bar \chi_1(Y)$ satisfies $\bar\chi_1(Y)=1$ for $Y\in[0, \bar Y]$ and $\bar\chi_1(Y)=0$ for $Y\in[\f32 \bar Y,+\infty)$. The smooth cut-off function $\bar\chi(Y)$ satisifes $\mathrm{supp}( \bar\chi)=[0, \bar Y]$, which holds that $\bar\chi \bar\chi_1=\bar\chi$. 
	\item The smooth cut-off function $\chi(Y)$ satisfies $\chi(Y)=1$ for $Y\in[Y_1-\delta_0,Y_2+\delta_0]$ and $\chi(Y)=0$  for $Y\in[0,Y_1-2\delta_0]\cup[Y_2+2\delta_0,+\infty)$. 
		\item The smooth cut-off function $\chi_2(Y)$ satisfies $\chi_2(Y)= 1$ for $Y\in[Y_1^*,Y_2^*]$ and $\chi_2(Y)=0$ for $Y\in[0,Y_1^*-\delta_0]\cup[Y_2^*+\delta_0,+\infty)$, where 
	      \begin{align*}
	      	\bar Y_1^*+2\delta_0< Y_1^*<Y_1-3\delta_0\text{ and }Y_2+3\delta_0<Y_2^*<\bar Y_2^*-2\delta_0.
	      \end{align*}
	 \item The smooth cut-off function $\chi_3(Y)$ satisfies $\chi_3(Y)=1$ for $Y\geq \bar Y_1^{*}$ and $\chi_3(Y)=0$ for $Y\leq \bar Y_1^{**}$, where 
\begin{align*}
Y^{+}+2\d_0<\bar Y_1^{**}<\bar Y_1^{*}-2\d_0.
\end{align*}
	 \item The smooth cut-off function $\chi_0(Y)$ satisfies $\chi_0(Y)=1$  for $Y\in[0,\bar{Y}_2^*+\delta_0]$ and   $\chi_0(Y)=0$ for $Y\in[\bar{Y}_2^*+2\delta_0,+\infty)$. 
\end{enumerate}
\end{remark}

\subsection{Proof of Proposition \ref{pro: exist} }	   
By the construction of $P(Y)$ introduced above, we know that for $j=0$, $P^{(0)}(Y)=P_{0,0}(Y)+P_{0,1}(Y)+P_{0,2}(Y),$
where
 \begin{align}\label{def: P_0}
 \begin{split}
&P_{0,0}(Y)=\mathcal{G}_{1}(A),
	      	\quad P_{0,1}(Y)=\mathcal{G}_{2,nloc}\circ \varphi_0,\quad P_{0,2}(Y)=\mathcal{G}_{1}\circ \tilde{P}_{0,2}^{(1)}+\bar\chi_1  \tilde P_{0,2}^{(2)}(Y).
		\end{split}
	      \end{align}
For $j\geq 1$, we have $P^{(j)}(Y)=P_{j,1}(Y)+P_{j,2}(Y),$
where
\begin{align}\label{def: P_j}
 \begin{split}
P_{j,1}(Y)=\mathcal{G}_{2,loc}\circ \varphi_j
,\quad P_{j,2}(Y)=\mathcal{G}_{1}\circ\tilde P_{j,2}.
\end{split}
\end{align}

We introduce the following notations related to $\tilde P_{j,2}$.
For $j=0$, we denote 
\begin{align}\label{eq:notation-tilde P02}
	\tilde P_{0,2}(Y):=\tilde P_{0,2}^{(1)}(Y)=\mathcal{A}_0(Y) A(Y)+\mathcal{B}_0(Y) B(Y),\quad \tilde P_{0,2}^{(2)}(Y)=\tilde{\mathcal{A}_0}(Y) A(Y)+\tilde{\mathcal{B}_0}(Y) B(Y).
\end{align}
For $j\geq 1$, we denote
\begin{align}\label{eq:notation-tilde Pj2}
\tilde P_{j,2}(Y)=\mathcal{A}_j (Y)A(Y)+\mathcal{B}_j(Y) B(Y).
\end{align}
To avoid heavy notations, we put the definitions of $(\mathcal{A}_j(Y), \mathcal{B}_j(Y) ),~j\geq 0$ and $(\tilde{\mathcal{A}_0}(Y),\tilde{\mathcal{B}_0}(Y))$ in the next subsection (see the precise definitions in \eqref{eq:notation-A0-B0}-\eqref{eq:notation-A1-B1}).

We introduce the following functional space 
\begin{align*}
	\mathcal X:=\left\{u\in W^{2,\infty}_{w_0}:\|u\|_{\mathcal X}<+\infty\right\},
\end{align*}
where
\begin{align*}
	\|u\|_{\mathcal X}:=\alpha\|u\|_{L^\infty_{w_0}}+\|\partial_Y u\|_{L^\infty_{w_0}}+\alpha^{-1}\|\partial_Y^2 u\|_{L^\infty_{w_0}}.
\end{align*}

 Our goal is to prove  the convergence of  the series $\sum_{j=0}^\infty P^{(j)}$ in $\mathcal X$. For this purpose, we introduce that for $j\geq 1$,
 \begin{align*}
 \mathcal{E}_j=&\|\mathcal{A}_j\|_{L^\infty}+\al\|\mathcal{A}_{j}\|_{L^\infty([0, Y_2^*])}+\al\|\mathcal{B}_{j}\|_{L^\infty_{2w_0}}\\
 &+\al^{-\f56}\|\varphi_j\|_{L^\infty_{w_0}}+\al^{\f16}e^{2\al \int_{Y_1^*}^{Y_1-2\d_0}F_r^\f12(Z) dZ}\|\varphi_j\|_{L^\infty_{w_0}([\bar Y_1^*, Y_1^*])},
 \end{align*}
and for $j=0$,
\begin{align*}
\mathcal{E}_0=&\|\mathcal{A}_0\|_{L^\infty}+\al\|\mathcal{A}_{0}\|_{L^\infty([0, Y_2^*])}+\al\|\mathcal{B}_{0}\|_{L^\infty_{2w_0}}+\al e^{\al w_0(Y_1^*-\d_0)}(\|\tilde{\mathcal{A}}_0\|_{L^\infty([0, Y^{-}])}+\|\tilde{\mathcal{B}}_0\|_{L^\infty([0, Y^{-}])})\\
&+\al^{-\f56}\|\varphi_0\|_{L^\infty_{w_0}}+\al^{\f16}e^{2\al \int_{Y_1^*}^{Y_1-2\d_0}F_r^\f12(Z) dZ}\|\varphi_0\|_{L^\infty_{w_0}([\bar Y_1^*, Y_1^*])}.
\end{align*}

The following two propositions will be used frequently. The proof is left to subsequent subsections. 

\begin{proposition}\label{pro: iteration-varphi}
Let $(\al, c)\in  \mathbb H_0$. It holds that for $j=0$,
\begin{align*}
\|\varphi_0\|_{L^\infty_{w_0}}+\al e^{-2\al \int_{Y_1^*}^{Y_1-2\d_0}F_r^\f12(Z) dZ}\|\varphi_0\|_{L^\infty_{w_0}([\bar Y_1^*,  Y_1^*])} \leq C\al^{\f56},
\end{align*}
 and for $j\geq 1$,
 \begin{align*}
\|\varphi_j\|_{L^\infty_{w_0}}+\al e^{-2\al \int_{Y_1^*}^{Y_1-2\d_0}F_r^\f12(Z) dZ}\|\varphi_j\|_{L^\infty_{w_0}([\bar Y_1^*,  Y_1^*])}\leq& C\al^{-\f{1}{6}}\mathcal{E}_{j-1}.
\end{align*}
\end{proposition}

\begin{proposition}\label{pro: interation-Airy} 
Let $(\al, c)\in  \mathbb H_0$. It holds that for $j=0$, 
\begin{align*}
&\|\mathcal{A}_0\|_{L^\infty}+\al\|\mathcal{A}_{0}\|_{L^\infty([0, Y_2^*])}+\al\|\mathcal{B}_{0}\|_{L^\infty_{2w_0}}+\al e^{\al w_0(Y_1^*-\d_0)}(\|\tilde{\mathcal{A}}_0\|_{L^\infty([0, Y^{-}])}+\|\tilde{\mathcal{B}}_0\|_{L^\infty([0, Y^{-}])})\leq C,
\end{align*}
and for $j\geq 1$, 
\begin{align*}
&\|\mathcal{A}_{j}\|_{L^\infty}+\al\|\mathcal{A}_{j}\|_{L^\infty([0, Y_2^*])}+\al\|\mathcal{B}_{j}\|_{L^\infty_{2w_0}}
\leq \al^{-1} \mathcal{E}_{j-1}.
\end{align*}

\end{proposition}

 \subsubsection{Estimates of $P_{j,1}$}
Recall that $\varphi_j$ is the solution of the Rayleigh type equation on interval $[\bar Y_1^*, \bar Y_2^*]$:
\begin{align}\label{eq: Rayleigh varphi_j}
\mathcal L_{cr}[\varphi_j]=\pa_Y\Big(\f{g_{j}}{1-\bM^2}\Big),
\end{align}
where we denote
\begin{align}\label{def: (f_j, g_j)}
f_j=\pa_Y\Big(\f{g_{j}}{1-\bM^2}\Big),\quad j\geq 0, \quad  g_{0}=-H_1[A],\quad g_j=-(H_1[\tilde P_{j-1,2}]+H_4[\varphi_{j-1}]),\quad   j\geq 1,
\end{align}
and we extend $\varphi_j=0$ for $Y\in \mathbb{R}_{+}/[\bar Y_1^*, \bar Y_2^*]$. Then $\|\varphi_j\|_{L^\infty_{w_0}}=\|\varphi_j\|_{L^\infty_{w_0}([\bar Y_1^*, \bar Y_2^*])}$. It is easy to verify that $\tilde P_{j,1}$ satisfies the following equation on $[\bar Y_1^*, \bar Y_2^*]$:
\begin{align*}
\mathcal{L}[\tilde P_{j,1}]=g_j,\quad \pa_Y\tilde P_{j,1}=\bM \varphi_j.
\end{align*}
By the definition of $P_{j,1}$ in \eqref{def: P_0}, we have 
\begin{align}\label{def: pa_YP_0,1 }
\begin{split}
\pa_YP_{0,1}=&\chi_2\pa_Y \tilde P_{j,1}=\chi_2\bM \varphi_0,\\
\pa_Y^2P_{0,1}=&\chi_2(2\pa_Y\bM \varphi_0+\al^2(1-\bM^2)\tilde P_{0,1}+g_0)+\pa_Y\chi_2\bM \varphi_0,
\end{split}
\end{align}
and for $j\geq1$,
\begin{align}\label{def: pa_YP_j,1 }
\begin{split}
\pa_YP_{j,1}=&\chi_2\pa_Y \tilde P_{j,1}-\pa_Y\chi_3\int_Y^{+\infty}\chi_2 \bM \varphi_jdZ,\\
\pa_Y^2P_{j,1}=&\chi_2(2\pa_Y\bM \varphi_j+\al^2(1-\bM^2)\tilde P_{j,1}+g_j)+\pa_Y\chi_2\bM \varphi_j-\pa_Y^2\chi_3\int_Y^{+\infty}\chi_2 \bM \varphi_jdZ.
\end{split}
\end{align}
\begin{lemma}\label{lem: Pj1}
Let $(\al,c)\in\mathbb{H}_0$. It holds that for $j\geq 0$,
\begin{align*}
\|P_{j,1}\|_{\mathcal X}\leq C\al^{-\f16}\mathcal{E}_{j-1},
\end{align*}
where we define $\mathcal{E}_{-1}=\al.$
\end{lemma}
\begin{proof}
For $j=0$,
according to the definition of $P_{0,1}$ in \eqref{def: P_0} and  \eqref{def: pa_YP_0,1 } , we have
\begin{align*}
&\|P_{0,1}\|_{L^\infty_{w_0}}\leq C\|\varphi_0\|_{L^\infty_{w_0}}\int_{Y}^{+\infty}\chi_2|\bM| e^{-\al w_0(Z)} dZ\leq C\al^{-1}\|\varphi_0\|_{L^\infty_{w_0}},\\
&\| \pa_Y P_{0,1}\|_{L^\infty_{w_0}}\leq \|\chi_2\pa_Y \tilde P_{0,1}\|_{L^\infty_{w_0}}\leq C\|\varphi_0\|_{L^\infty_{w_0}},\\
&\|\pa_Y^2P_{0,1}\|_{L^\infty_{w_0}}\leq C\|\varphi_0\|_{L^\infty_{w_0}}+C\al^2\|\tilde P_{0,1}\|_{L^\infty_{w_0}}+C\|\chi_2 f_0\|_{L^\infty_{w_0}}+C\|\pa_Y\chi_2\bM \varphi_0\|_{L^\infty_{w_0}}\\
&\qquad \qquad\qquad \leq C\al \|\varphi_0\|_{L^\infty_{w_0}}+C\|\chi_2 g_0\|_{L^\infty_{w_0}}.
\end{align*}
For $j\geq1$, we use \eqref{def: P_0},  \eqref{def: pa_YP_j,1 } and the fact $|\int_Y^{+\infty}\chi_2 \bM \varphi_jdZ|\leq C\al^{-1}\|\varphi_j\|_{L^\infty_{w_0}}$ to imply
\begin{align*}
&\|P_{j,1}\|_{L^\infty_{w_0}}\leq  C\al^{-1}\|\varphi_j\|_{L^\infty_{w_0}},\\
&\| \pa_Y P_{j,1}\|_{L^\infty_{w_0}}\leq  C\|\varphi_j\|_{L^\infty_{w_0}}+C\al^{-1}\|\varphi_j\|_{L^\infty_{w_0}}\leq C\|\varphi_j\|_{L^\infty_{w_0}},\\
&\|\pa_Y^2P_{j,1}\|_{L^\infty_{w_0}}
\leq C\al \|\varphi_j\|_{L^\infty_{w_0}}+C\|\chi_2 g_j\|_{L^\infty_{w_0}}.
\end{align*}
Then we are left with the estimates of $\|\chi_2 g_j\|_{L^\infty_{w_0}}$. For $j=0$, we use $\chi_2 H_1[A]= H_1[A]$ with the definition in \eqref{def: H_1-H_4} to  get
\begin{align*}
\|\chi_2g_0\|_{L^\infty_{w_0}}=&\|H_1[A]\|_{L^\infty_{w_0}}
\leq C\al\|A\|_{L^\infty_{w_0}([Y_1^*, Y_2^*])}+C\|\pa_YA\|_{L^\infty_{w_0}([Y_1^*, Y_2^*])}\leq C\al^{\f56},
\end{align*}
where we have used  $|A(Y)|+\al^{-1}|\pa_Y A(Y)|\leq C\al^{-\f16}e^{-\al w_0(Y)}$ for $Y\in [Y_1^*, Y_2^*]$ in the last step. Then we get by Proposition \ref{pro: iteration-varphi} that
 \begin{align*}
\|P_{0,1}\|_{\mathcal X}\leq C\|\varphi_0\|_{L^\infty_{w_0}}+C\al^{-1}\|\chi_2g_0\|_{L^\infty_{w_0}}\leq C\|\varphi_0\|_{L^\infty_{w_0}}+C\al^{-\f16}\leq C\al^{\f56}.
 \end{align*}

For $j\geq1$, we have
\begin{align*}
\chi_2 g_j=-\chi_2(H_1[\tilde P_{j-1,2}]+H_4[\varphi_{j-1}])=-(H_1[\tilde P_{j-1,2}]+H_4[\varphi_{j-1}]).
\end{align*}
Notice that
\begin{align*}
\|H_1[\tilde P_{j-1,2}]\|_{L^\infty_{w_0}}\leq& \|\chi(Q_1+Q_2)\chi_1  \tilde P_{j-1,2}\|_{L^\infty_{w_0}}+\al^2\Big\|\chi F(Y)\int_Y^{+\infty}\partial_{Z}\chi_1(Z)\tilde P_{j-1,2}(Z)dZ\Big\|_{L^\infty_{w_0}}\\
&+\|\pa_Y \chi_1 \pa_YP_{j-1,2} \|_{L^\infty_{w_0}}\\
\leq& C\al\| \tilde P_{j-1,2}\|_{L^\infty_{w_0}([Y_1^*, Y_2^*])}+C\| \pa_Y\tilde P_{j-1,2}\|_{L^\infty_{w_0}([Y_1^*, Y_2^*])}\\
\leq& C\al^{-\f16}(\al\|\mathcal{A}_{j-1}\|_{L^\infty([0, Y_2^*])}+\al\|\mathcal{B}_{j-1}\|_{L^\infty_{2w_0}})\leq C\al^{-\f16}\mathcal{E}_{j-1},
\end{align*}
and by Lemma \ref{lem: Im-int}, we have
\begin{align*}
\|H_4[\varphi_{j-1}]\|_{L^\infty_{w_0}}\leq& \left\|\chi(1-\bar M^2)\mathcal{K}_{B}[\varphi_{j-1}]\right\|_{L^\infty_{w_0}}
\leq C\al^{-1}\| \varphi_{j-1}\|_{L^\infty_{w_0}}\leq C\al^{-\f16}\mathcal{E}_{j-1},
\end{align*}
which give
\begin{align*}
\|\chi_2 g_j\|_{L^\infty_{w_0}}\leq&C\al^{-\f16}\mathcal{E}_{j-1}.
\end{align*}
Thus, by Proposition \ref{pro: iteration-varphi}, we infer that for $j\geq 1$,
\begin{align*}
\|P_{j,1}\|_{\mathcal X}
\leq& C\|\varphi_j\|_{L^\infty_{w_0}}+C\al^{-1}\|\chi_2g_j\|_{L^\infty_{w_0}}\leq C\al^{-\f16}\mathcal{E}_{j-1}.
\end{align*}
\end{proof}

 \subsubsection{Estimates about $P_{j,2}$}
According to the definition of $P_{j,2}$, we have
\begin{align}\label{eq: P_(0,2)}
\begin{split}
&P_{0,2}=\chi_1 \tilde P_{0,2}+\int_Y^{+\infty}\pa_Z\chi_1\tilde P_{0,2}(Z)dZ+\bar\chi_1 \tilde P_{0,2}^{(2)},\\
&\pa_YP_{0,2}=\chi_1\pa_Y\tilde P_{0,2}+\bar\chi_1 \pa_Y\tilde P_{0,2}^{(2)}+\pa_Y\bar\chi_1 \tilde P_{0,2}^{(2)},\\
&\pa_Y^2P_{0,2}=\chi_1\left(2\f{\pa_Y\bM}{\bM}\pa_Y\tilde P_{0,2} +\al^2(1-\bM^2)\tilde P_{0,2}+h_0^{(1)}\right)+\pa_Y\chi_1\pa_Y\tilde P_{0,2}\\
&\qquad\qquad+\bar\chi_1\left(2\f{\pa_Y\bM}{\bM}\pa_Y\tilde P_{0,2}^{(2)} +\al^2(1-\bM^2)\tilde P_{j,2}^{(2)}+h_0^{(2)}\right)+2\pa_Y\bar\chi_1 \pa_Y\tilde P_{0,2}^{(2)}+\pa_Y^2\bar\chi_1\tilde P_{0,2}^{(2)} .
\end{split}
\end{align}
and
for $j\geq1$,
\begin{align}\label{eq: P_(j,2)}
\begin{split}
P_{j,2}=&\chi_1 \tilde P_{j,2}+\int_Y^{+\infty}\pa_Z\chi_1\tilde P_{j,2}(Z)dZ,\quad \pa_YP_{j,2}=\chi_1 
\pa_Y\tilde P_{j,2},\\
\pa_Y^2P_{j,2}=&\chi_1\left(2\f{\pa_Y\bM}{\bM}\pa_Y\tilde P_{j,2} +\al^2(1-\bM^2)\tilde P_{j,2}+h_j\right)+\pa_Y\chi_1\pa_Y \tilde P_{j,2},
\end{split}
\end{align}
where 
\begin{align}\label{def: (h_0^(1), h_0^(2))}
\begin{split}
&h_0^{(1)}=-H_2[A]-H_{31}[\varphi_0]-H_{32}^{(1)}[\varphi_0],\quad h_0^{(2)}=-H_{32}^{(2)}[\varphi_0],\\
& h_1=-H_2[\tilde P_{0,2}^{(1)}]-\tilde H_{21}[\tilde P_{0,2}^{(2)}]-\tilde H_3[\varphi_{1}],\quad h_j=-H_2[\tilde P_{j-1,2}]-\tilde H_3[\varphi_{j}],\quad j\geq2.
\end{split}
\end{align}

\begin{lemma}\label{lem: Pj2}
Let $(\al,c)\in\mathbb{H}_0$. It holds that for $j\geq 0$,
\begin{align*}
\|P_{j,2}\|_{\mathcal X}\leq C\al^{-\f16}\mathcal{E}_{j-1},
\end{align*}
where we define $\mathcal{E}_{-1}=\al.$

\end{lemma}

\begin{proof}
For $j=0$,  by  \eqref{eq: P_(0,2)} and Proposition \ref{pro: interation-Airy}, we get 
\begin{align*}
\|P_{0,2}\|_{L^\infty_{w_0}}\leq&C\|\chi_1 \tilde P_{0,2}\|_{L^\infty_{w_0}}+C\left\|\int_Y^{+\infty}\pa_Z\chi_1\tilde P_{0,2}(Z)dZ\right\|_{L^\infty_{w_0}}+\|\bar\chi_1 \tilde P_{0,2}^{(2)}\|_{L^\infty_{w_0}}\\
\leq&C(\|\mathcal{A}_{0}A\|_{L^\infty_{w_0}([0, Y_2^*])}+\|\mathcal{A}_{0}A\|_{L^\infty_{w_0}([Y_2^*,\infty))}+\|\mathcal{B}_{0}B\|_{L^\infty_{w_0}})\\
&+C(\|\tilde{\mathcal{A}_0} A\|_{L^\infty([0, \f32 \bar Y])}+\|\tilde{\mathcal{B}_0} B\|_{L^\infty([0, \f32 \bar Y])})\\
\leq&C(\|\mathcal{A}_{0}\|_{L^\infty_{w_0}([0, Y_2^*])}+\al^{-\f16}\|\mathcal{A}_{0}\|_{L^\infty_{w_0}}+\|\mathcal{B}_{0}\|_{L^\infty_{w_0}})\\
&+C\al^{-\f16}(\|\tilde{\mathcal{A}_0}\|_{L^\infty([0, Y^{-}])}+\|\tilde{\mathcal{B}_0}\|_{L^\infty([0, Y^{-}])})\leq C\al^{-\f16},
\end{align*}
and
\begin{align*}
\|\pa_YP_{0,2}\|_{L^\infty_{w_0}}\leq&C\|\chi_1 \pa_Y\tilde P_{0,2}\|_{L^\infty_{w_0}}+\|\bar\chi_1 \pa_Y\tilde P_{0,2}^{(2)}\|_{L^\infty_{w_0}}+\|\pa_Y\bar\chi_1 \tilde P_{0,2}^{(2)}\|_{L^\infty_{w_0}}\\
\leq&C\al(\|\mathcal{A}_{0}\|_{L^\infty_{w_0}([0, Y_2^*])}+\al^{-\f16}\|\mathcal{A}_{0}\|_{L^\infty_{w_0}}+\|\mathcal{B}_{0}\|_{L^\infty_{w_0}})\\&+C\al^{\f56}(\|\tilde{\mathcal{A}_0}\|_{L^\infty([0, Y^{-}])}+\|\tilde{\mathcal{B}_0}\|_{L^\infty([0, Y^{-}])})
\leq C\al^{\f56},
\end{align*}
and 
\begin{align*}
\|\pa_Y^2 P_{0,2}\|_{L^\infty_{w_0}}\leq& C(\|\chi_1\pa_Y\tilde P_{0,2}\|_{L^\infty_{w_0}}+\al^2\|\tilde P_{0,2}\|_{L^\infty_{w_0}}+\|\chi_1 g_0^{(1)}\|_{L^\infty_{w_0}})+C\|\pa_Y\chi_1\pa_Y \tilde P_{0,2}\|_{L^\infty_{w_0}}\\
&+ C(\|\bar\chi_1\pa_Y\tilde P_{0,2}^{(2)}\|_{L^\infty_{w_0}}+\al^2\|\bar\chi_1\tilde P_{0,2}^{(2)}\|_{L^\infty_{w_0}}+\|\bar\chi_1 f_0^{(2)}\|_{L^\infty_{w_0}})\\
&+C(\|\pa_Y\bar\chi_1\pa_Y \tilde P_{0,2}^{(2)}\|_{L^\infty_{w_0}}+\|\pa_Y^2\bar\chi_1 \tilde P_{0,2}^{(2)}\|_{L^\infty_{w_0}})\\
\leq&C\al^{\f{11}6}+C\|\chi_1 h_0^{(1)}\|_{L^\infty_{w_0}}+C\|\bar\chi_1 h_0^{(2)}\|_{L^\infty_{w_0}}.
\end{align*}

Using the similar argument as above and \eqref{eq: P_(j,2)}, we infer that for $j\geq1$,
\begin{align*}
&\|P_{j,2}\|_{L^\infty_{w_0}}\leq C\al^{-\f16}(\al\|\mathcal{A}_{j}\|_{L^\infty_{w_0}([0, Y_2^*])}+\|\mathcal{A}_{j}\|_{L^\infty_{w_0}}+\al\|\mathcal{B}_{j}\|_{L^\infty_{w_0}})
\leq C\al^{-\f76}\mathcal{E}_{j-1},\\
&\|\pa_YP_{j,2}\|_{L^\infty_{w_0}}\leq C\al^{\f56}(\al\|\mathcal{A}_{j}\|_{L^\infty_{w_0}([0, Y_2^*])}+\|\mathcal{A}_{j}\|_{L^\infty_{w_0}}+\al\|\mathcal{B}_{j}\|_{L^\infty_{w_0}})
\leq C\al^{-\f16}\mathcal{E}_{j-1},\\
&\|\pa_Y^2 P_{j,2}\|_{L^\infty_{w_0}}\leq C\al^{\f{11}{6}}(\al\|\mathcal{A}_{j}\|_{L^\infty_{w_0}([0, Y_2^*])}+\|\mathcal{A}_{j}\|_{L^\infty_{w_0}}+\al\|\mathcal{B}_{j}\|_{L^\infty_{w_0}})+C\|\chi_1 h_j\|_{L^\infty_{w_0}}\\
&\qquad\qquad\qquad\leq C\al^{\f{5}6}\mathcal{E}_{j-1}+C\|\chi_1 h_j\|_{L^\infty_{w_0}}.
\end{align*}

Next we focus on $\|\chi_1 h_0^{(1)}\|_{L^\infty_{w_0}}$, $\|\bar\chi_1 h_0^{(2)}\|_{L^\infty_{w_0}}$ and $\|\chi_1 h_j\|_{L^\infty_{w_0}}$ for $j\geq 1$, where $h_0^{(1)}$, $ h_0^{(2)}$ and $h_j~(j\geq 1)$ are given in \eqref{def: (h_0^(1), h_0^(2))}. For $j=0$, we infer from Proposition \ref{pro: iteration-varphi} that
\begin{align*}
\|H_2[A]\|_{L^\infty_{w_0}}\leq& C\|(1-\chi)(Q_1+Q_2)\chi_1 A\|_{L^\infty_{w_0}}+C\alpha^2\left\|F(Y)\int_Y^{+\infty}\partial_{Z}\chi_1(Z)A(Z)dZ\right\|_{L^\infty_{w_0}}\\
\leq& C+C\al^{\f56}\leq C\al^{\f56},\\
\|H_{31}[\varphi_0]\|_{L^\infty_{w_0}}\leq&C\|\partial_Y\chi_2\bar M\varphi_0\|_{L^\infty_{w_0}}\leq C\al^{\f56},\\
\|H_{32}^{(1)}[\varphi_0]\|_{L^\infty_{w_0}}\leq&C\|(1-\chi)(1-\bar\chi_1)(1-\bar M^2)\mathcal{K}_B[\varphi_0]\|_{L^\infty_{w_0}}\leq C\al^{-1}\|\varphi_0\|_{L^\infty_{w_0}}\leq C\al^{-\f16},\\
\|H_{32}^{(2)}[\varphi_0]\|_{L^\infty_{w_0}}\leq&C\|\bar\chi_1(1-\bar M^2)\mathcal{K}_B[\varphi_0]\|_{L^\infty_{w_0}}\leq C\al^{-1}\|\varphi_0\|_{L^\infty_{w_0}}\leq C\al^{-\f16},
\end{align*}
which imply 
\begin{align*}
\|\chi_1 h_0^{(1)}\|_{L^\infty_{w_0}}\leq C\al^\f56,\quad \|\bar\chi_1 h_0^{(2)}\|_{L^\infty_{w_0}}\leq C\al^{-\f16}.
\end{align*}
Thus, we obtain 
\begin{align*}
\|P_{0,2}\|_{\mathcal X}\leq C\al^\f56.
\end{align*}

For $j\geq1$, we get by Proposition \ref{pro: iteration-varphi}  that
\begin{align*}
\|H_2[\tilde P_{j-1,2}]\|_{L^\infty_{w_0}}\leq&C\|(1-\chi)(Q_1+Q_2)\chi_1 \tilde P_{j-1,2}\|_{L^\infty_{w_0}}\\&+C\alpha^2\|F(Y)\int_Y^{+\infty}\partial_{Z}\chi_1(Z)\tilde P_{j-1,2}(Z)dZ\|_{L^\infty_{w_0}}\\
\leq&C(\al^{\f56}\|\mathcal{A}_{j-1}\|_{L^\infty([0, Y_2^*])}+\al^{-\f16}\|\mathcal{A}_{j-1}\|_{L^\infty}+\al^{\f56}\|\mathcal{B}_{j-1}\|_{L^\infty_{2w_0}})
\leq C\al^{-\f16}\mathcal{E}_{j-1},
\end{align*}
and
\begin{align*}
\|\tilde H_3[\varphi_j]\|_{L^\infty_{w_0}}\leq&C\|\partial_Y\chi_2\bar M\varphi_j\|_{L^\infty_{w_0}}+C\|\mathcal{K}_B[\varphi_j]\|_{L^\infty_{w_0}}
\leq C\|\varphi_j\|_{L^\infty_{w_0}}\leq C\al^{-\f16}\mathcal{E}_{j-1},
\end{align*}
and 
\begin{align*}
\|\tilde H_{21}[\tilde P_{0,2}^{(2)}]\|_{L^\infty_{w_0}}\leq& C(\|\tilde P_{0,2}^{(2)}\|_{L^\infty([0, \f32 \bar Y])}+\|\pa_Y\tilde P_{0,2}^{(2)}\|_{L^\infty([0, \f32 \bar Y])})\\
\leq& C\al^{\f56}(\|\tilde{\mathcal{A}_0}\|_{L^\infty([0, Y^{-}])}+\|\tilde{\mathcal{B}_0}\|_{L^\infty([0, Y^{-}])})\leq C\al^{-\f16} \mathcal{E}_{0},
\end{align*}
which imply  
\begin{align*}
&\|\chi_1 h_1\|_{L^\infty_{w_0}}\leq C\al^{-\f16}\mathcal{E}_{0},\quad 
\|\chi_1 h_j\|_{L^\infty_{w_0}}\leq C\al^{-\f16}\mathcal{E}_{j-1},\quad j\geq2.
\end{align*}
Thus, we conclude that  for $j\geq 1$,
\begin{align*}
\|P_{j,2}\|_{\mathcal X}\leq& C\al^{-\f16}\mathcal{E}_{j-1}+C\al^{-1}\|\chi_1 h_j\|_{L^\infty_{w_0}}
\leq C\al^{-\f16}\mathcal{E}_{j-1}.
\end{align*}

\end{proof}

\subsubsection{Convergence of the iteration}
In this part, we give the proof of the convergence of  the series $\sum_{j=0}^\infty P^{(j)}$ in $\mathcal X$. To achieve this, we need the following lemma related to the iteration scheme.
\begin{lemma}\label{lem: cE_j}
	Let $(\al,c)\in\mathbb{H}_0$,  then we have for $j\geq 0$,
	\begin{align*}
	\mathcal{E}_j\leq C\al^{-1}\mathcal{E}_{j-1}\leq C\al^{-j},
	\end{align*}
	where we define $\mathcal{E}_{-1}=\al$.
\end{lemma}

\begin{proof}
	
	By Proposition \ref{pro: iteration-varphi} and Proposition \ref{pro: interation-Airy}, we have $\mathcal{E}_0\leq C$, 
	and for $j\geq 1$,
	\begin{align*}
	&\|\mathcal{A}_j\|_{L^\infty}+\al\|\mathcal{A}_{j}\|_{L^\infty([0, Y_2^*])}+\al\|\mathcal{B}_{j}\|_{L^\infty_{2w_0}}\leq C\al^{-1}\mathcal{E}_{j-1},\\
	&\al^{-\f56}\|\varphi_j\|_{L^\infty_{w_0}}+\al^{\f16}e^{2\al \int_{Y_1^*}^{Y_1-2\d_0}F_r^\f12(Z) dZ}\|\varphi_0\|_{L^\infty_{w_0}([\bar Y_1^*, Y_1^*])}\leq C\al^{-1}\mathcal{E}_{j-1}.
	\end{align*}
	Hence, for $j\geq 0$, we have
	\begin{align*}
	\mathcal{E}_{j}\leq C\al^{-1}\mathcal{E}_{j-1}\leq C\al^{-j}.
	\end{align*}
	\end{proof}
	
Now we are in a position to prove Proposition \ref{pro: exist}.

\begin{proof}[Proof of Proposition \ref{pro: exist} ]
	Recall that \begin{align*}
	P^{(0)}(Y)=P_{0,0}(Y)+P_{0,1}(Y)+P_{0,2}(Y),
\end{align*}
and it is easy to get
$\|P_{0,0}\|_{\mathcal X}\leq C \alpha^{\f56}.$
This along with Lemma  \ref{lem: Pj1} and  Lemma  \ref{lem: Pj2} deduces that for $j=0$,
\begin{align*}
	\|P^{(0)}\|_{\mathcal X}\leq C \alpha^{\f56}+C \al^{-\f16}\mathcal{E}_{-1}\leq C\alpha^{\f56}.
\end{align*}
For $j\geq 1$, recall that $P^{(j)}=P_{j,1}+P_{j,2}$, which along with Lemma  \ref{lem: Pj1}-\ref{lem: cE_j} implies 
\begin{align*}
	\|P^{(j)}\|_{\mathcal X}\leq C \al^{-\f16}\mathcal{E}_{j-1}\leq C\alpha^{-j+\f56}.
\end{align*}	
Therefore, we have 
\begin{align*}
	\|\sum_{j=0}^\infty P^{(j)}\|_{\mathcal X}\leq \sum_{j=0}^{\infty}\|P^{(j)}\|_{\mathcal X}\leq C\alpha^\f56+\sum_{j=1}^\infty\alpha^{-j+\f56}\leq C\alpha^\f56.
\end{align*}
 By the definition of $P^{(j)}(Y)$, we know that 
\begin{align*}
	\partial_YP^{(0)}(0)=\partial_YA(0)+\partial_Y\tilde P_{0,2}^{(1)}(0)+\partial_Y\tilde P_{0,2}^{(2)}(0),\quad \partial_Y P^{(j)}(0)=\partial_Y \tilde P_{j,2}(0),
\end{align*}
which implies 
\begin{align*}
		\partial_Y P(0)=\partial_Y A(0)+\partial_Y\tilde P_{0,2}^{(1)}(0)+\partial_Y\tilde P_{0,2}^{(2)}(0)+ \sum_{j=1}^{\infty}\partial_Y\tilde P_{j,2}(0).
	\end{align*}

\end{proof}

\subsection{Proof of Proposition \ref{pro: iteration-varphi}}

Since $\varphi_j,~j\geq 0$ satisfies Rayleigh type equation \eqref{eq: Rayleigh varphi_j}, we finish the proof of Proposition \ref{pro: iteration-varphi} by using Proposition \ref{pro: rayleigh-varphi} and Lemma \ref{lem: rayleigh-source term} directly.

\begin{lemma}\label{lem: rayleigh-source term}
For $(\al, c)\in  \mathbb H_0$, it holds that
\begin{align*}
&\left\|f_0\right\|_{L^\infty_{w_0}}\leq C\al^{\f{11}{6}},\quad \left\|f_j\right\|_{L^\infty_{w_0}}\leq C\al^{\f56}\mathcal{E}_{j-1},\quad j\geq 1,
\end{align*}
with
$ \mathrm{supp}(f_j)\subset  [Y_1-2\d_0, Y_1+\d_0]\cup[ Y_2-\d_0,Y_2+2\d_0]$ for $ j\geq 0.
$
Here $f_j$ is given in \eqref{def: (f_j, g_j)}.

\end{lemma}
\begin{proof}
We first estimate $f_0$. Recalling the definition of $H_1(A)$, a direct calculation gives 
\begin{align*}
f_0
=&\pa_Y\chi_1\left(\al^2 A+\f{\pa_Y^2A}{1-\bM^2}\right)+\pa_Y\left(\f{\pa_Y \chi_1}{1-\bM^2}\right)\pa_Y A+\pa_Y\left(\f{\chi \chi_1 (Q_1+Q_2)A}{1-\bM^2}\right)\\
&+\al^2 \pa_Y\chi \int_Y^{+\infty}\pa_Z \chi_1 AdZ.
\end{align*}
According to  $\mathcal{L}[A]=(Q_1+Q_2)A$, we get
\begin{align*}
\pa_Y^2 A=2\f{\pa_Y\bM}{\bM}\pa_Y A+\al^2(1-\bM^2)A+(Q_1+Q_2)A,
\end{align*}
which gives 
\begin{align}\label{eq: f_0}
\begin{split}
f_0=&2\al^2\pa_Y\chi_1 A+\f{\pa_Y\chi_1}{1-\bM^2}\left(\f{2\pa_Y\bM}{\bM}\pa_Y A+(Q_1+Q_2)A\right)\\
&+\pa_Y\left(\f{\pa_Y \chi_1}{1-\bM^2}\right)\pa_Y A+\pa_Y\left(\f{\chi \chi_1 (Q_1+Q_2)A}{1-\bM^2}\right)+\al^2 \pa_Y\chi \int_Y^{+\infty}\pa_Z \chi_1 AdZ.
\end{split}
\end{align}
By Lemma \ref{lem: Q_1}  and $A(Y)\sim \al^{-\f16} e^{-\al w_0(Y)},~\pa_Y A(Y)\sim \al^{\f56} e^{-\al w_0(Y)}$ for $Y\in [Y_1-2\d_0,Y_2+2\d_0],$ and the facts that 
\begin{align}
&|(1+Y)^2Q_2|\leq C\al^2|F_i|\leq C\al^2 c_i,\label{est: Q_2}\\
\nonumber
&\Big\| \pa_Y\chi \int_Y^{+\infty}\pa_Z \chi_1 AdZ\Big\|_{L^\infty_{w_0}}\leq C\al^{-\f16}e^{-\al \int_{Y_1-\d_0}^{Y_1} F_r^\f12(Y) dY},
\end{align}
we infer that
\begin{align*}
\left\|f_0\right\|_{L^\infty_{w_0}}\leq& C\al^{-\f16}(\al^2+\al^2 c_i+\al^2 e^{-\al \int_{Y_1-\d_0}^{Y_1} F_r^\f12(Y) dY})\leq C\al^{\f{11}{6}}.
\end{align*}

Recalling the definition of $H_4(\varphi_{j-1})$ for $j\geq 1$, we have
\begin{align*}
\pa_Y\Big(\f{H_4(\varphi_{j-1})}{1-\bM^2}\Big)=&\pa_Y\Big(\chi\mathcal{K}_{B}[\varphi_{j-1}]\Big)
=\pa_Y\chi ~\mathcal{K}_{B}[\varphi_{j-1}](Y)=\pa_Y\chi~\mathcal{K_{B}}[\varphi_{j-1}](Y_2^*),
\end{align*}
by using $\chi\pa_Y \chi_2=0$ and $ \mathrm{supp}(\pa_Y\chi)\subset[Y_1-2\d_0, Y_1-\d_0]\cup[Y_2+\d_0, Y_2+2\d_0]$. 
Using a similar argument of $f_0$, we get for $j\geq 1$,
\begin{align*}
f_j
=&\pa_Y\chi_1\Big(\al^2 \tilde P_{j-1,2}+\f{\pa_Y^2\tilde P_{j-1,2}}{1-\bM^2}\Big)+\pa_Y\Big(\f{\pa_Y \chi_1}{1-\bM^2}\Big)\pa_Y \tilde P_{j-1,2}\\
&+\pa_Y\Big(\f{\chi \chi_1 (Q_1+Q_2)\tilde P_{j-1,2}}{1-\bM^2}\Big)+\al^2 \pa_Y\chi \int_Y^{+\infty}\pa_Z \chi_1 \tilde P_{j-1,2}dZ+\pa_Y\chi~\mathcal{K_{B}}[\varphi_{j-1}](Y_2^*).
\end{align*}
By the facts $\pa_Y\chi_1 (1-\chi)=0$ and $\pa_Y\chi_1\pa_Y\chi_2=0$, we have $\pa_Y \chi_1H_2[\cdot]=\pa_Y \chi_1H_{31}[\cdot]=\pa_Y \chi_1H_{32}^{(1)}[\cdot]=\pa_Y \chi_1\tilde H_{32}[\cdot]=\pa_Y \chi_1\tilde H_{21}[\cdot]=0,$ which combines the equations of $\tilde P_{0,2}:=\tilde P_{0,2}^{(1)}$ and $\tilde P_{j,2} (j\geq 1)$ in \eqref{eq: tilde P_(0,2)^(1)} and \eqref{eq: tilde P_(j,2)} to get
\begin{align*}
\pa_Y\chi_1\pa_Y^2\tilde P_{j-1,2}=\pa_Y\chi_1\Big(\f{2\pa_Y\bM}{\bM}\pa_Y\tilde P_{j-1,2}+\al^2(1-\bM^2)\tilde P_{j-1,2}+(Q_1+Q_2)\tilde P_{j-1,2}\Big).
\end{align*}
Hence, we deduce
{\small
\begin{align}\label{eq: f_j}
\begin{split}
f_j=&2\al^2\pa_Y\chi_1 \tilde P_{j-1,2}+\f{\pa_Y\chi_1}{1-\bM^2}\Big(\f{2\pa_Y\bM}{\bM}\pa_Y \tilde P_{j-1,2}+(Q_1+Q_2)\tilde P_{j-1,2}\Big)+\pa_Y\chi~\mathcal{K_{B}}[\varphi_{j-1}](Y_2^*)\\
&+\pa_Y\Big(\f{\pa_Y \chi_1}{1-\bM^2}\Big)\pa_Y \tilde P_{j-1,2}+\pa_Y\Big(\f{\chi \chi_1 (Q_1+Q_2)\tilde P_{j-1,2}}{1-\bM^2}\Big)+\al^2 \pa_Y\chi \int_Y^{+\infty}\pa_Z \chi_1 \tilde P_{j-1,2}dZ.
\end{split}
\end{align}}
For $Y\in [Y_1-2\d_0,Y_2+2\d_0]\subset[Y_1^*, Y_2^*]$ and $\tilde P_{j,2}=\mathcal{A}_jA+\mathcal{B}_j B$, we have
\begin{align*}
&|\tilde P_{j-1,2}|\leq |\mathcal{A}_{j-1} | |A|+|\mathcal{B}_{j-1} | |B|\leq C\al^{-\f16}e^{-\al w_0(Y)}(|\mathcal{A}_{j-1}|+e^{2\al w_0(Y)}|\mathcal{B}_{j-1}|),\\
&|\pa_Y\tilde P_{j-1,2}|\leq |\mathcal{A}_{j-1} | |\pa_YA|+|\mathcal{B}_{j-1} | |\pa_YB|\leq C\al^{\f56}e^{-\al  w_0(Y)}(|\mathcal{A}_{j-1}|+e^{2\al w_0(Y)}|\mathcal{B}_{j-1}|),
\end{align*}
which along with $\|\pa_Y\chi~\mathcal{K_{B}}[\varphi_{j-1}](Y_2^*)\|_{L^\infty_{w_0}}\leq C\al^{-1}\|\varphi_{j-1}\|_{L^\infty_{w_0}}$ show 
\begin{align*}
\left\|f_j\right\|_{L^\infty_{w_0}}\leq&C\al^{\f{11}{6}}\left(\|\mathcal{A}_{j-1}\|_{L^\infty([Y_1^*, Y_2^*])}+\|\mathcal{B}_{j-1}\|_{L^\infty_{2w_0}([Y_1^*, Y_2^*])} \right)+\al^{-1}\|\varphi_{j-1}\|_{L^\infty_{w_0}}\leq C\al^{\f56}\mathcal{E}_{j-1}.
\end{align*}

From \eqref{eq: f_0} and \eqref{eq: f_j} and the definition of $\chi,~\chi_1$ in Remark \ref{rem: cut-off functions}, we infer that
\begin{align*}
& \mathrm{Supp}(f_j)\subset [Y_1-2\d_0, Y_1+\d_0]\cup[ Y_2-\d_0,Y_2+2\d_0],\quad j\geq 0.
\end{align*}
\end{proof}

\subsection{Proof of Proposition \ref{pro: interation-Airy} }

Let's first  introduce the following notations. 
\begin{align}\label{eq:notation-A0-B0}
\begin{split}
&\mathcal{A}_0=\sum_{k=1}^4\mathcal{A}_0^k=-\mathcal{A}_{H_{21}[A]}-\mathcal{A}_{H_{22}[A]}-\mathcal{A}_{H_{31}[\varphi_0]}-\mathcal{A}_{H_{32}^{(1)}[\varphi_0]},\quad \tilde{\mathcal{A}_0}=-\tilde{\mathcal{A}}_{H_{32}^{(2)}[\varphi_0]},\\
& \mathcal{B}_0=\sum_{k=1}^4\mathcal{B}_0^k=-\mathcal{B}_{H_{21}[A]}-\mathcal{B}_{H_{22}[A]}-\mathcal{B}_{H_{31}[\varphi_0]}-\mathcal{B}_{H_{32}^{(1)}[\varphi_0]},\quad \tilde{\mathcal{B}_0}=-\tilde{\mathcal{B}}_{H_{32}^{(2)}[\varphi_0]}.
\end{split}
\end{align}
and for $j\geq 1$,
\begin{align}\label{eq:notation-A1-B1}
\begin{split}
&\mathcal{A}_j=\sum_{k=1}^5 \mathcal{A}_j^k=-\mathcal{A}_{H_{21}[\tilde P_{j-1,2}]}-\mathcal{A}_{H_{22}[\tilde P_{j-1,2}]}-\mathcal{A}_{H_{31}[\varphi_j]}-\mathcal{A}_{\tilde H_{32}[\varphi_j]}-\mathbf{ 1}_{j=1}\mathcal{A}_{\tilde H_{21}[\tilde P_{0,2}^{(2)}]},\\
&\mathcal{B}_j=\sum_{k=1}^5 \mathcal{B}_j^k=-\mathcal{B}_{H_{21}[\tilde P_{j-1,2}]}-\mathcal{B}_{H_{22}[\tilde P_{j-1,2}]}-\mathcal{B}_{H_{31}[\varphi_j]}-\mathcal{B}_{\tilde H_{32}[\varphi_j]}-\mathbf{ 1}_{j=1}\mathcal{B}_{\tilde H_{21}[\tilde P_{0,2}^{(2)}]}.
\end{split}
\end{align}
where $\cA_j^5=\cB_j^5=0$ for $j\geq 2$ and 
\begin{align*}
&\mathcal{A}_f(Y)=-\kappa^{-1}\pi \int_0^YB(Z)f(Z)\bar{M}^{-2}(Z)dZ,\quad \tilde{\mathcal{A}}_f(Y)=\kappa^{-1}\pi \int_Y^{+\infty}B(Z)f(Z)\bar{M}^{-2}(Z)dZ,\\
& \mathcal{B}_f(Y)=-\kappa^{-1}\pi \int_Y^{+\infty}A(Z)f(Z)\bar{M}^{-2}(Z)dZ,\quad  \tilde{\mathcal{B}}_f(Y)=\kappa^{-1}\pi \int_0^YA(Z)f(Z)\bar{M}^{-2}(Z)dZ.
\end{align*}

\smallskip

\begin{proof}[Proof of Proposition \ref{pro: interation-Airy}]
We first prove the case $j=0$. 
For $\mathcal{A}_0^1$ and  $\mathcal{B}_0^1$,
we use Proposition \ref{pro: Airy-1} and take $a=(1-\chi)(Q_1+Q_2)\chi_1$ and $b=0$. By Lemma \ref{lem: Q_1} and \eqref{est: Q_2}, we have
\begin{align*}
\|(1+Y)^2a\|_{L^\infty}\leq C(1+\al^2 c_i)\leq C,
\end{align*}
which implies
\begin{align*}
&\|\mathcal{A}_0^1\|_{L^\infty}+\|\mathcal{B}_0^1\|_{L^\infty_{2w_0}}\leq C\al^{-1}\|(1+Y)^2a\|_{L^\infty}\leq C\al^{-1}.
\end{align*}

For
 $\mathcal{A}_0^2$ and  $\mathcal{B}_0^2$,
we apply Proposition \ref{pro: Airy-2} by taking $g=-\al^2 \pa_Y \chi_1 A$ with $\|g\|_{L^\infty_{w_0}}\leq C\al^{\f{11}{6}},$ 
 to obtain
\begin{align*}
&\|\mathcal{A}_0^2\|_{L^\infty}+\|\mathcal{B}_0^2\|_{L^\infty_{2w_0}}\leq C \al^{-\f{17}6}e^{-\alpha\int_{Y_1-\d_0}^{Y_1} F_r^\f12(Z)dZ}\|g\|_{L^\infty_{w_0}}\leq C\alpha^{-1}e^{-\alpha\int_{Y_1-\d_0}^{Y_1} F_r^\f12(Z)dZ}.
\end{align*}

For  $\mathcal{A}_0^3$ and  $\mathcal{B}_0^3$, we take $f=\pa_Y\chi_2 \bM \varphi_0$ in Proposition \ref{pro: Airy-3}  and use Proposition \ref{pro: iteration-varphi} to get
\begin{align*}
\|\mathcal{A}_0^3\|_{L^\infty}+\al\|\mathcal{B}_0^3\|_{L^\infty_{2w_0}}\leq& C\al^{-\f56}\|f\|_{L^\infty_{w_0}}\leq C\al^{-\f56}\|\varphi_0\|_{L^\infty_{w_0}}\leq C,
\end{align*}
and for $Y\in[0, Y_2^*]$,
\begin{align*}
&|\mathcal{A}_0^3(Y)|\leq C\alpha^{-\f56}\|f\|_{L^\infty_{w_0}([Y_1^*-\delta_0,Y^*_1])}\leq C\alpha^{-\f56}\| \varphi_0\|_{L^\infty_{w_0}([Y_1^*-\delta_0,Y^*_1])}
\leq Ce^{-2\al \int_{Y_1^*}^{Y_1-2\d_0}F_r^\f12(Z) dZ}.
\end{align*}

For  $\mathcal{A}_0^4$ and  $\mathcal{B}_0^4$, we take $f=\partial_Y\left(\frac{\partial_Y\chi_2(Y)\bar M^2(Y)}{1-\bar M^2(Y)} \right)\bar M^{-1}\varphi_0(Y)$ in Proposition \ref{pro: Airy-4} and use Proposition \ref{pro: iteration-varphi} to obtain
\begin{align*}
&\|\mathcal{A}_0^4\|_{L^\infty}+\al\|\mathcal{B}_0^4\|_{L^\infty_{2w_0}}\leq C\alpha^{-\frac{11}{6}}\|f\|_{L^\infty_{w_0}}\leq C\alpha^{-\frac{11}{6}}\|\varphi_0\|_{L^\infty_{w_0}}\leq C\al^{-1}.
\end{align*}

Gathering the above estimates, we arrive at
\begin{align*}
\|\mathcal{A}_0\|_{L^\infty}+\al\|\mathcal{A}_0\|_{L^\infty([0, Y_2^*])}+\al \|\mathcal{B}_0\|_{L^\infty_{2w_0}}\leq C.
\end{align*}

For  $\tilde{\mathcal{A}}_0$ and  $\tilde{\mathcal{B}}_0$,  by Lemma \ref{lem: Ai, Bi} and the fact $|\mathcal{K}_{B}[\varphi_0](Y_1^*-\d_0)|\leq C\al^{-1}e^{-\al w_0(Y_1^*-\d_0)}$, we infer that for $Y\in[0, Y^{-}]$,
\begin{align*}
|\tilde{\mathcal{A}}_0(Y)|=&|\tilde{\mathcal{A}}_{H_{32}^{(2)}[\varphi_0]}(Y)|\leq\kappa^{-1}\pi \int_Y^{+\infty}|B\bar{M}^{-2}\bar\chi(1-\bar M^2)\mathcal K_B[\varphi_0](Z)|dZ\\
\leq&C\kappa^{-1}|\mathcal{K}_{B}[\varphi_0](Y_1^*-\d_0)|\int_0^{\f32 \bar Y}|B(Z)|dZ\\
\leq&C\kappa^{-\f54}\al^{-1} e^{-\al w_0(Y_1^*-\d_0)}\|\varphi_0\|_{L^\infty_{w_0}}\leq C\al^{-1}e^{-\al w_0(Y_1^*-\d_0)},
\end{align*}
and
 \begin{align*}
|\tilde{\mathcal{B}}_0(Y)|=&|\tilde{\mathcal{B}}_{H_{32}^{(2)}[\varphi_0]}(Y)|\leq \kappa^{-1}\pi \int_0^{Y}|A\bar{M}^{-2}\bar\chi(1-\bar M^2)\mathcal K_B[\varphi_0](Z)|dZ\\
\leq&C\kappa^{-1}|\mathcal{K}_{B}[\varphi_0](Y_1^*-\d_0)|\int_0^{\f32 \bar Y}|A(Z)|dZ\leq C\al^{-1}e^{-\al w_0(Y_1^*-\d_0)},
\end{align*}
which gives 
\begin{align*}
&\|\tilde{\mathcal{A}}_0\|_{L^\infty([0, Y^{-}])}+\|\tilde{\mathcal{B}}_0\|_{L^\infty([0, Y^{-}])}\leq C\al^{-1}e^{-\al w_0(Y_1^*-\d_0)}.
\end{align*}
Thus, we conclude that
\begin{align*}
\|\mathcal{A}_0\|_{L^\infty}+\al\|\mathcal{A}_{0}\|_{L^\infty([0, Y_2^*])}+\al\|\mathcal{B}_{0}\|_{L^\infty_{2w_0}}+\al e^{\al w_0(Y_1^*-\d_0)}(\|\tilde{\mathcal{A}}_0\|_{L^\infty([0, Y^{-}])}+\|\tilde{\mathcal{B}}_0\|_{L^\infty([0, Y^{-}])})\leq C. 
\end{align*}

Next we consider the case $j\geq 1$. For $\mathcal{A}_j^{1}$ and $\mathcal{B}_j^{1}$, we use Proposition \ref{pro: Airy-1} and take $a=(1-\chi)(Q_1+Q_2)\chi_1 \mathcal{A}_{j-1}$ and $b=(1-\chi)(Q_1+Q_2)\chi_1 \mathcal{B}_{j-1}$.
By Lemma \ref{lem: Q_1} and \eqref{est: Q_2}, we get
\begin{align*}
&\|(1+Y)^2 a\|_{L^\infty([0,Y])}=\|(1+Y)^2 (1-\chi)(Q_1+Q_2)\chi_1 \mathcal{A}_{j-1}\|_{L^\infty([0,Y])}\leq C\|\mathcal{A}_{j-1}\|_{L^\infty([0,Y])},\\
& \|(1+Y)^2b\|_{L^\infty_{2w_0}}=\|(1+Y)^2(1-\chi)(Q_1+Q_2)\chi_1 \mathcal{B}_{j-1}\|_{L^\infty_{2w_0}}\leq C\|\mathcal{B}_{j-1}\|_{L^\infty_{w_0}}.
\end{align*}
Then we  obtain 
\begin{align*}
&\|\mathcal{A}_{j}^{1}\|_{L^\infty}\leq C\al^{-1}(\|\mathcal{A}_{j-1}\|_{L^\infty}+\|\mathcal{B}_{j-1}\|_{L^\infty_{2w_0}})\leq \al^{-1} \mathcal{E}_{j-1},\\
&\|\mathcal{B}_{j}^1\|_{L^\infty_{2w_0}}\leq C\al^{-1}(\|\mathcal{A}_{j-1}\|_{L^\infty([0, Y_2^*])}+\al^{-1}\|\mathcal{A}_{j-1}\|_{L^\infty}+\|\mathcal{B}_{j-1}\|_{L^\infty_{2w_0}})\leq \al^{-2} \mathcal{E}_{j-1}.
\end{align*}
Thanks to $Y_2^*>Y^{+}+\d$, we have
\begin{align*}
&\|\mathcal{A}_{j}^{1}\|_{L^\infty([0, Y_2^*])}\leq C\al^{-1}(\|\mathcal{A}_{j-1}\|_{L^\infty([0, Y_2^*])}+\|\mathcal{B}_{j-1}\|_{L^\infty_{2w_0}})\leq \al^{-2} \mathcal{E}_{j-1}.
\end{align*}

For $\mathcal{A}_j^2$ and  $\mathcal{B}_j^2$,
by applying Proposition \ref{pro: Airy-2} with  $g=-\al^2\pa_Y\chi_1 \tilde P_{j-1,2}$ and the fact 
\begin{align*}
&\|g\|_{L^\infty_{w_0}}\leq C\al^2\|\pa_Y\chi_1\tilde P_{j-1,2}\|_{L^\infty_{w_0}}\leq C\al^{\f{11}{6}}(\|\mathcal{A}_{j-1}\|_{L^\infty}+\|\mathcal{B}_{j-1}\|_{L^\infty_{2w_0}}),
\end{align*}
we obtain 
\begin{align*}
\|\mathcal{A}_{j}^{2}\|_{L^\infty}+\|\mathcal{B}_{j}^{2}\|_{L^\infty_{2w_0}}\leq& C\alpha^{-\f{17}6}e^{-\alpha\int_{Y_1-\d_0}^{Y_1} F_r^\f12(Z)dZ}\|g\|_{L^\infty_{w_0}}\\
\leq& C\alpha^{-1}e^{-\alpha\int_{Y_1-\d_0}^{Y_1} F_r^\f12(Z)dZ}(\|\mathcal{A}_{j-1}\|_{L^\infty}+\|\mathcal{B}_{j-1}\|_{L^\infty_{2w_0}})\leq \al^{-2} \mathcal{E}_{j-1}.
\end{align*}

For $\mathcal{A}_j^3$ and $\mathcal{B}_j^3$, we use Proposition \ref{pro: Airy-3}, Proposition \ref{pro: iteration-varphi} and take $f=\pa_Y\chi_2 \bM \varphi_j$ to have 
\begin{align*}
\|\mathcal{A}_{j}^{3}\|_{L^\infty}+\al\|\mathcal{B}_{j}^{3}\|_{L^\infty_{2w_0}}\leq& C\alpha^{-\f56}\|f\|_{L^\infty_{w_0}}\leq C\alpha^{-\f56}\|\varphi_j\|_{L^\infty_{w_0}}\leq C\al^{-1}\mathcal{E}_{j-1},
\end{align*}
and 
\begin{align*}
\|\mathcal{A}_j^3\|_{L^\infty([0, Y_2^*])}\leq& C\alpha^{-\f56}\|f\|_{L^\infty_{w_0}([Y_1^*-\delta_0,Y^*_1])}\leq C\alpha^{-\f56}\|\varphi_j\|_{L^\infty_{w_0}([Y_1^*-\delta_0,Y^*_1])}\leq C\al^{-2}\mathcal{E}_{j-1}.
\end{align*}

For $\mathcal{A}_j^4$ and $\mathcal{B}_j^4$, we use Proposition \ref{pro: Airy-5}, Proposition \ref{pro: iteration-varphi} and take $f=\partial_Y\left(\frac{\partial_Y\chi_2(Y)\bar M^2(Y)}{1-\bar M^2(Y)} \right)\bar M^{-1}\varphi_j(Y)$ to infer that for $Y\geq 0$,
\begin{align*}
\|\mathcal{A}_{j}^{4}\|_{L^\infty}+\al\|\mathcal{B}_{j}^{4}\|_{L^\infty_{2w_0}}\leq& C\alpha^{-\frac{11}{6}}\|f\|_{L^\infty_{w_0}}\leq C\alpha^{-\frac{11}{6}}\|\varphi_j\|_{L^\infty_{w_0}}\leq C\al^{-2}\mathcal{E}_{j-1}.
\end{align*}
In particular, for $j=1$, we have 
\begin{align}\label{est: iteration j=1-1}
\begin{split}
&\sum_{k=1}^4(\|\mathcal{A}_1^k\|_{L^\infty}+\al\|\mathcal{A}_1^k\|_{L^\infty([0, Y_2^*])}+\al\|\mathcal{B}_1^k\|_{L^\infty_{2w_0}})\leq \al^{-1}\mathcal{E}_0,
\end{split}
\end{align}
and for $j\geq 2$, we have
\begin{align*}
&\|\mathcal{A}_{j}\|_{L^\infty}+\al\|\mathcal{A}_{j}\|_{L^\infty([0, Y_2^*])}+\al\|\mathcal{B}_{j}\|_{L^\infty_{2w_0}}\leq \al^{-1}\mathcal{E}_{j-1}.
\end{align*}

It remains to estimate $\mathcal{A}_{1}^5$ and $\mathcal{B}_{1}^5$. Recall that 
\begin{align*}
&\tilde H_{21}[\tilde P_{1,2}^{(2)}]=(\bar\chi_1(Q_1+Q_2)+\pa_Y^2\bar\chi_1-\f{2\pa_Y\bM}{\bM}\pa_Y \bar\chi_1)\tilde P_{1,2}^{(2)}+2\pa_Y\bar\chi_1\pa_Y\tilde P_{1,2}^{(2)},\\
&\tilde P_{1,2}^{(2)}=\tilde{\mathcal{A}_0} A+\tilde{\mathcal{B}_0} B,\quad \pa_Y\tilde P_{1,2}^{(2)}=\tilde{\mathcal{A}_0} \pa_YA+\tilde{\mathcal{B}_0} \pa_YB.
\end{align*}
By Lemma \ref{lem: Ai, Bi}, we have 
\begin{align*}
&\|\tilde P_{0,2}^{(2)}\|_{L^\infty([0, Y^{-}])}\leq C\al^{-\f16}(\|\tilde{\mathcal{A}}_0\|_{L^\infty([0, Y^{-}])}+\|\tilde{\mathcal{B}}_0\|_{L^\infty([0, Y^{-}])}),\\
&\|\pa_Y\tilde P_{0,2}^{(2)}\|_{L^\infty([0, Y^{-}])}\leq C\al^{\f56}(\|\tilde{\mathcal{A}}_0\|_{L^\infty([0, Y^{-}])}+\|\tilde{\mathcal{B}}_0\|_{L^\infty([0, Y^{-}])}),
\end{align*}
which gives 
\begin{align*}
\|\tilde H_{21}[\tilde P_{1,2}^{(2)}]\|_{L^\infty([0, Y^{-}])}\leq C \al^{\f56}(\|\tilde{\mathcal{A}}_0\|_{L^\infty([0, Y^{-}])}+\|\tilde{\mathcal{B}}_0\|_{L^\infty([0, Y^{-}])}).
\end{align*}

By using the argument in Proposition \ref{pro: Airy-1}, we have
\begin{align*}
|\mathcal{A}_{1}^5|\leq& C\kappa^{-1}\|\tilde H_{21}[\tilde P_{1,2}^{(2)}]\|_{L^\infty([0, Y^{-}])}\int_{0}^{\f32 \bar Y}|B(Z)|dY\leq C\kappa^{-\f54}\|\tilde H_{21}[\tilde P_{1,2}^{(2)}]\|_{L^\infty([0, Y^{-}])}\\
\leq& C(\|\tilde{\mathcal{A}}_0\|_{L^\infty([0, Y^{-}])}+\|\tilde{\mathcal{B}}_0\|_{L^\infty([0, Y^{-}])}).
\end{align*}
Similarly, we obtain
\begin{align*}
|\mathcal{B}_{1}^5|\leq& C(\|\tilde{\mathcal{A}}_0\|_{L^\infty([0, Y^{-}])}+\|\tilde{\mathcal{B}}_0\|_{L^\infty([0, Y^{-}])}).
\end{align*}
Along with \eqref{est: iteration j=1-1}, we show that
\begin{align*}
\|\mathcal{A}_{1}\|_{L^\infty}+\al\|\mathcal{A}_{1}\|_{L^\infty([0, Y_2^*])}+\al\|\mathcal{B}_{1}\|_{L^\infty_{2w_0}}\leq \al^{-1}\mathcal{E}_0+C(\|\tilde{\mathcal{A}}_0\|_{L^\infty([0, Y^{-}])}+\|\tilde{\mathcal{B}}_0\|_{L^\infty([0, Y^{-}])})\leq C\al^{-1}\mathcal{E}_0.
\end{align*}
As a result, we conclude that for $j\geq 1$,
\begin{align*}
\|\mathcal{A}_{j}\|_{L^\infty}+\al\|\mathcal{A}_{j}\|_{L^\infty([0, Y_2^*])}+\al\|\mathcal{B}_{j}\|_{L^\infty_{2w_0}}\leq \al^{-1}\mathcal{E}_{j-1}.
\end{align*}

\end{proof}

\section{Asymptotic expansion of boundary value}
  At the boundary $Y=0$, by the notations in \eqref{eq:notation-tilde P02} and \eqref{eq:notation-tilde Pj2} and the fact $\cA_j(0)=\tilde\cB_0(0)=0$ for any $j\geq 0$, we get
\begin{align*}
\partial_Y\tilde P_{0,2}^{(1)}(0)=\mathcal B_0(0)\partial_Y B(0),\quad \partial_Y\tilde P_{0,2}^{(2)}(0)=\tilde{\mathcal A}_0(0)\partial_Y A(0),\quad \partial_Y\tilde P_{j,2}(0)=\mathcal B_j(0)\partial_Y B(0),\quad j\geq 1,
\end{align*}
which along with the results in Proposition \ref{pro: exist} implies
\begin{align}\label{expansion: pa_Y P(0)}
\partial_YP(0)=(1+\tilde{\mathcal{A}}_0(0))\partial_YA(0)+\Big(\mathcal B_0(0)+\sum_{j=1}^{+\infty}\mathcal B_j(0)\Big)\pa_Y B(0).
\end{align}

To state our main result in this section, we introduce the following definitions and notations.
First, we introduce a set of $(\al,c)$
\begin{align*}
\mathbb{H}=\mathbb{H}_0\cap\{(\al,c)\in \mathbb{R}_{+}\times \mathbb{C}: \kappa^{\f32}\tan \Theta(-\kappa\eta(0))\in(C_1, C_2),~0<c_i\ll \al^n e^{-\al w_0( \tilde {Y}) }\},
\end{align*}
for some $n>0$, where
$\kappa^3=\alpha^2\partial_Y\tF(Y_0)=\alpha^2\partial_YF_r(Y_0)$ and $\tilde{Y}$ is given in Lemma \ref{lem:real-phi(0)}.

For a fixed Mach number $M_a>1$, we introduce a function on interval $[T_0^\f12(0) M_a^{-1}, 1+M_a^{-1}]$
{\small
\begin{align}\label{def: J(c_r)}
J(c_r; M_a)=\int_0^{Y_0(c_r)}(-\tilde F(Z))^{-\f12}\tilde Q_1(Z; c_r, M_a)dZ,\quad \mbox{with}\quad \tilde Q_1(Z; c_r, M_a)=\tilde Q_1(Z; Y_0(c_r), c_r, M_a),
\end{align}
}
where 
\begin{align}\label{def: tilde Q_1}
\tilde Q_1(Y; Y_0(c_r), c_r, M_a)=\f{\pa_Y^2 \bM_r}{\bM_r}-\f{2(\pa_Y\bM_r)^2}{\bM_r^2}-\f34\f{(\pa_Y^2\eta)^2}{(\pa_Y\eta)^2}-\f{\pa_Y^3 \eta}{2\pa_Y\eta}
\end{align}
is a real function independent of $c_i$ for $Y\leq Y_0$, and  $Y_0(c_r)$ is determined by the relation $M_a (c_r-U_B(Y_0))=T_0^\f12(Y_0)$  involved in Langer tansform $\eta(Y)$.  It is easy to see $|Q_1-\tilde Q_1|\leq C c_i$.  

Moreover, we introduce
 a set $\mathcal{H}(U_B, M_a)$, which is determined by the background flow $U_B$ and Mach number $M_a>1$
\begin{align}\label{set: H}
\mathcal{H}(U_B, M_a)=\Big\{c_r\in[T_0^\f12(0) M_a^{-1}, 1+M_a^{-1}]: J(c_r; M_a)\neq 0\Big\}.
\end{align}

In addition, we denote 
\begin{align}\label{def: D_0}
\frak{D}_0:=\frak{D}_0(\al,c)\eqdef\kappa^{-1}\al^{-2}\pi \mathcal K_B[\varphi_0](\bar Y_1^*),
\end{align}
which is an important quantity in obtaining the precise expansion of boundary value. Here  $\varphi_0=\mathcal{L}_{cr}^{-1}(-H_1[A])$.
\medskip

The main result of  this section is presented as follows.

\begin{proposition}\label{pro: expan-dispersion}
Let $(\alpha,c)\in\mathbb H$ and $Y_1^*, Y^*_2$ be chosen in Lemma \ref{lem:real-phi(0)}.  Suppose that $P(Y)$ is the solution to \eqref{eq:ray-F-nbv} constructed in Proposition \ref{pro: exist}. Then  we have
\begin{align*}
\partial_YP(0)=(1+\tilde{\mathcal{A}}_0(0))\partial_YA(0)+\Big(\mathcal B_0(0)+\sum_{j=1}^{+\infty}\mathcal B_j(0)\Big)\pa_Y B(0).
\end{align*}
Moreover,
\begin{itemize}
	\item it holds that
	\begin{align}
	&\mathcal B_0(0)= d_0^1(\al ,c )~\al^{-1}+i d_0^2(\al, c)~\al c_i+\frak{D}_0(\al,c)~d_0^3(\al, c),\label{expansion: B_0(0)}\\
	&\tilde{\mathcal{A}}_0(0)=\frak{D}_0(\al,c)\Big(\partial_YB(0)\bar M(0)^{-2}+d_0^4(\al, c)\Big), \label{expansion: tA_0(0)}\\
	&\sum_{j=1}^{+\infty}\mathcal B_j(0)=\al^{-1}\left(d_1^1(\al ,c )~\al^{-1}+i d_1^2(\al, c)~\al c_i\right)\label{expansion: tB_1(0)}\\
	\nonumber
	&\qquad\qquad\qquad+\frak{D}_0(\al,c)\left(d_1^3(\al, c)+\partial_YA(0)\bar M^{-2}(0)(C_{P,1}+d_1^4(\al, c)\right),
	\end{align}
	where 
	\begin{align}\label{est: d_0^1-1}
	|d_0^1(\al,c)|\lesssim 1+\al^{-1},
	\end{align}
	and
	\begin{align}\label{est: d_0^2-d_1^4}
	\begin{split}
	& |d_0^2(\al,c)|\sim 1+\al^{-1},\\
	& |d_1^j(\al,c)|\lesssim 1+\al^{-1},\quad |\pa_{c_i}d_0^j(\al,c)+|\pa_{c_i}d_1^j(\al,c)|\lesssim1+\al^{-1},\quad j=1,2,\\
	&|d_0^j(\al,c)|+|d_1^j(\al,c)|\lesssim \al^{-\f23},\quad |\pa_{c_i}d_0^j(\al,c)|+|\pa_{c_i}d_1^j(\al,c)|\lesssim \al^{-\f23},\quad j=3, 4,\\
	&C_{P,1}=\kappa^{-1}\pi\int_0^{\bar Y_1^*}A(Z)\bar M^{-2}(Z)\partial_Z\bar\chi_1(Z)\partial_ZB(Z)dZ,\quad |C_{P,1}|+c_i^{-1}|\mathrm {Im}C_{P,1}|\lesssim 1 .
	\end{split}
	\end{align}
	\item
	in particular, for $(\alpha,c)\in\mathbb H$ with $c_r\in\mathcal{H}(U_B, M_a)\cap(1-M_a^{-1},1)$, \eqref{expansion: tA_0(0)}-\eqref{expansion: tB_1(0)} and \eqref{est: d_0^2-d_1^4} still hold with \eqref{est: d_0^1-1} replaced by 
	\begin{align}\label{est: d_0^1-2}
	|d_0^1(\al,c)|\sim 1+\al^{-1}.
	\end{align}
\end{itemize}

\end{proposition}

We first present the bound for $\mathcal K_B[\varphi_0](Y)$, which is used frequently in the following arguments.
\begin{lemma}\label{lem: Im-int}
	Let $(\alpha, c)\in \mathbb H_0$ and $Y_1^*, Y^*_2$ be given in Lemma \ref{lem:real-phi(0)} and $\varphi_0=\mathcal{L}_{cr}^{-1}(-H_1[A])$. Then we have
	\begin{align*}
	&\mathcal K_B[\varphi_0](Y)\equiv 0,\quad Y\geq Y_2^*+\delta_0,\quad \mathcal K_B[\varphi_0](Y)\equiv\mathcal K_B[\varphi_0](Y^*_1-\delta_0),\quad Y\in[0,Y_1^*-\delta_0],\\
	&|\mathcal K_B[\varphi_0](Y)|\leq C\alpha^{\f56}e^{-\alpha w_0(Y)},\quad Y\in[Y_1^*-\delta_0,Y_2^*+\delta_0],\\
	&|\mathcal K_B[\varphi_0](Y)|\leq C\al^{-\f{1}{6}}\Big(e^{-\al w_0(Y_1^*-\d_0)}e^{-2\al \int_{Y_1^*}^{Y_1-2\d_0}F_r^\f12(Y') dY'}+\al e^{-\al w_0(Y_2^*)}\Big),\quad Y\in[0, Y_1^*-\d_0].
	\end{align*}
	
	In addition, if  $c_i\ll\al^n e^{-\al w_0(\tilde Y)}$ for some $n>0$, then  there exists $\mathring{Y}\in[Y_c,  \tilde Y]$ such that 
	\begin{align*}
	|\mathrm{Im}(\mathcal K_B[\varphi_0])(Y_1^*-\delta_0)|\sim \al^{\f56}e^{-\al w_0(\mathring Y)},\quad \mathrm{Sign}(\mathrm{Im}(\mathcal K_B[\varphi]))=\mathrm{Sign}(\partial_Y^2\bar M(Y_c))\mathrm{Sign}\mathcal M,
	\end{align*}
	where $\tilde Y$ is determined in Lemma \ref{lem:real-phi(0)}.
	
\end{lemma}
The proof of Lemma \ref{lem: Im-int} is based on Lemma \ref{lem: G-sim}-\ref{lem:real-phi(0)}. For the readability of the article, we leave the proof  in Appendix \ref{appendix: CL}.

\subsection{Expansion of boundary value for $j=0$.}We adopt notations in \eqref{eq:notation-A0-B0} with the form of $(\mathcal{A}_0(Y),\mathcal{B}_{0}(Y))=\sum_{k=1}^4(\mathcal{A}_0^k(Y),\mathcal{B}_0^k(Y))$. In details, we can write
\begin{align*}
\cB_0^1(0)=\kappa^{-1}\pi\int_0^{+\infty}A(Y)^2\bar M^{-2}(Y)(1-\chi)(Q_1(Y)+Q_2(Y))dY=I_{Q_1}+I_{Q_2}. 
\end{align*}
where $Q_1(Y)$ is given in \eqref{def: Q_1} and $Q_2(Y)=-\al^2 F_i(Y)\chi_0(Y)$. To obtain the precise dispersion relation, we need to get the upper and lower bound of  $\Re(\cB_0^1(0)) $ and $\Im (\mathcal B_0^1(0))$.

\begin{lemma}\label{lem: estimate B_0^1(0)}
The following three results hold.
\begin{enumerate}
\item Let $(\alpha,c)\in\mathbb H_0$. It holds that
\begin{align*}
		 |\Re(\cB_0^1(0)) |\lesssim\al^{-1},\quad\Im(\mathcal B_0^1(0))\sim \al c_i.
	\end{align*}
\item In particular, let $(\alpha,c)\in\mathbb H_0$ with $c_r\in\mathcal{H}(U_B, M_a)\cap(1-M_a^{-1},1)$.  It holds that
\begin{align*}
		 |\Re(\cB_0^1(0)) |\sim \al^{-1}.
	\end{align*}
\end{enumerate}
\end{lemma}

\begin{proof}
By Lemma \ref{lem: Q_1}, we know that 
     \begin{align*}
     	\|(1-\chi) Q_1\|_{L^\infty}\leq C\text{ and }\|\mathrm{Im}((1-\chi)Q_1)\|_{L^\infty}\leq Cc_i.
     \end{align*}
Then we use the fact $|Q_2|\leq C\al^2 c_i$ and Proposition \ref{pro: Airy-1} to deduce
\begin{align}\label{eq:I-Q1}
|\mathrm{Re}(I_{Q_1})|\leq C\al^{-1},\quad |\mathrm{Im}(I_{Q_1})|\leq C\alpha^{-1}c_i,\quad |I_{Q_2}|\leq C\al c_i.
	\end{align}

\underline{Proof of (1).} From \eqref{eq:I-Q1}, it is easy to deduce 
\begin{align*}
		\left|I_{Q_1}\right|\lesssim\al^{-1},\quad |\Re(\cB_0^1(0)) |\lesssim\al^{-1}.
	\end{align*}
About $I_{Q_2}$, we can  write
\begin{align*}
I_{Q_2}=&- \alpha^2\kappa^{-1}\pi\int_0^{+\infty}A(Y)^2\bar M^{-2}\chi_0(1-\chi)F_i(Y)dY\\
=&-2\mathrm iM_a^2\alpha^2\kappa^{-1}\pi c_i\Big(\int_0^{Y^-}+\int_{Y^{-}}^{Y^{+}}+\int_{Y^+}^{+\infty}\Big)\f{Ai(\kappa\eta(Y))^2(U_B(Y)-c_r)\chi_0(1-\chi)}{\partial_Y\eta(Y)T_0(Y)}dY\\
&-M_a^2\alpha^2\kappa^{-1}\pi c_i^2\int_0^{+\infty}\f{Ai(\kappa\eta(Y))^2\chi_0(1-\chi)}{\partial_Y\eta(Y)T_0(Y)}dY\\
=&-2\mathrm iM_a^2\alpha^2\kappa^{-1}\pi c_i(\tilde I_1+\tilde I_2+\tilde I_3)-M_a^2\alpha^2\kappa^{-1}\pi c_i^2 \tilde I_4.
\end{align*}
For $\tilde I_2,~\tilde I_3 $	and $\tilde I_4$, we get by Lemma \ref{lem: Ai, Bi} that
\begin{align}\label{est: tilde I_1-I_2}
|\tilde I_2|+|\tilde I_3|\leq&C\kappa^{-1}\leq C\al^{-\f23},\quad |\tilde I_4|\leq C\kappa^{-\f12}\leq C\al^{-\f13}.
\end{align}
Then we focus on $\tilde I_1$. We first notice that $\mathrm {Im}(\tilde I_1)=0$ and $\tilde I_1<0$. Moreover, we have
\begin{align*}
|\tilde I_1|>\int_0^{Y_0/2}Ai(\kappa\eta(Y))^2\partial_Y\eta(Y)^{-1}|U_B(Y)-c_r|T_0^{-1}dY\geq C^{-1}\int_0^{Y_0/2}(-\eta(Y))^{\f12}Ai(\kappa\eta(Y))^2dY,
\end{align*}
due to the fact $(-\eta(Y))^{-\f12}\partial_Y\eta(Y)^{-1}|U_B(Y)-c_r|T_0^{-1}\geq C^{-1}$ for $Y\in[0, Y_0/2]$.
By Lemma \ref{lem: Airy-original}, we have
\begin{align*}
\int_0^{Y_0/2}(-\eta(Y))^{\f12}Ai(\kappa\eta(Y))^2dY=&\frac{1}{8\pi\kappa^\f12}\int_0^{Y_0/2}1dY+\frac{1}{8\pi\kappa^\f12}\int_0^{Y_0/2}(\cos(2\Theta(-\kappa\eta(Y)))dY+\mathcal O(\kappa^{-2})\\
\geq&C^{-1}\alpha^{-\f13}-C\al^{-\f43}\geq C^{-1}\alpha^{-\f13},
\end{align*}
which implies 
\begin{align}\label{eq:til-I_1}
|\tilde I_1|\geq  C^{-1}\alpha^{-\f13}.
\end{align}
This along with \eqref{est: tilde I_1-I_2} shows that $\Im (I_{Q_2})>0$ and
\begin{align*}
\Im (I_{Q_2})\geq C^{-1}c_i\al^2\kappa^{-1}(|\tilde I_1|-|\tilde I_2|-|\tilde I_3|)-C^{-1}\al^2\kappa^{-1}c_i^2|\tilde I_4|\geq C^{-1}c_i\al ,
\end{align*}
from which and \eqref{eq:I-Q1}, we deduce $\Im(I_{Q_2})\sim \al c_i$. Then due to $|\mathrm{Im}(I_{Q_1})|\leq C\alpha^{-1}c_i$ in \eqref{eq:I-Q1},   we get $\Im(\mathcal B_0^1(0))\sim \al c_i.$

\underline{Proof of (2).} In this case, we shall prove that for 	$(\alpha,c)\in\mathbb H_0$ with $c_r\in\mathcal{H}(U_B, M_a)\cap(1-M_a^{-1},1)$, we can get improved estimates (lower bound estimates) compared to (1).
Using the fact $A(Y)^2\bar M^{-2}(Y)=Ai(\kappa\eta(Y))^2(\pa_Y\eta)^{-1}$, we write 
\begin{align*}
I_{Q_1}
=&\kappa^{-1}\pi\int_0^{Y^{-}}Ai(\kappa\eta(Y))^2(\pa_Y\eta)^{-1}Q_1(Y)dY+\mathcal{R}_1,
\end{align*}
where
\begin{align*}
\mathcal{R}_1=\kappa^{-1}\pi\int_{Y^{-}}^{+\infty}Ai(\kappa\eta(Y))^2(\pa_Y\eta)^{-1}(1-\chi)\chi_1 Q_1(Y)dY.
\end{align*}
By Lemma \ref{lem: Ai, Bi}, it is easy to see $|\mathcal{R}_1|\leq C\kappa^{-2}\leq C\al^{-\f43}$. So, we focus on the first term and apply Lemma \ref{lem: Airy-original} and \eqref{eq: eta} to write
\begin{align*}
&\int_0^{Y^{-}}Ai(\kappa\eta(Y))^2(\pa_Y\eta)^{-1}Q_1(Y)dZ\\
&=\kappa^{-\f12}( \pa_Y\tilde F(Y_0))^{\f12}\int_0^{Y^{-}}(- \tilde F(Y))^{-\f12}\left(\cos(\Theta(-\kappa\eta(Y)))+a_0(-\kappa \eta(Y))^{-\f32}\sin(\Theta(-\kappa\eta(Y)))\right)^2\\
&\qquad\qquad\qquad\qquad\qquad\qquad\cdot(1+\mathcal O((-\kappa\eta(Y))^{-3}))Q_1(Y)dZ\\
&=\kappa^{-\f12}( \pa_Y \tilde F(Y_0))^{\f12}\int_0^{Y^{-}}(-\tilde F(Y))^{-\f12}\cos(\Theta(-\kappa\eta(Y)))^2Q_1(Y)dZ+\mathcal{R}_2.
\end{align*}
By $|Q_1(Y)|\leq C$ for $Y\in[0, Y^{-}]$ and Lemma \ref{lem:est-eta}, we have
\begin{align*}
|\mathcal{R}_2|\leq C\kappa^{-2}\int_0^{Y^{-}}\f1{|Y-Y_c|^2}dY\leq C\kappa^{-1}.
\end{align*}
We notice that 
\begin{align*}
&\kappa^{-\f12}( \pa_Y \tilde F(Y_0))^{\f12}\int_0^{Y^{-}}(-\tilde F(Y))^{-\f12}\cos(\Theta(-\kappa\eta(Y)))^2Q_1(Y)dY\\
&=\f12\kappa^{-\f12}( \pa_Y \tilde F(Y_0))^{\f12}\int_0^{Y^{-}}(-\tilde F(Y))^{-\f12}Q_1(Y)dY+\mathcal{R}_3.
\end{align*}
About the remainder $\mathcal{R}_3$, we get by integration by parts that 
\begin{align*}
|\mathcal{R}_3|
\leq&C\kappa^{-2}\Big|(-\tilde F(Y))^{-1}\sin(2\Theta(-\kappa\eta(Y))) Q_1(Y)\Big|_{Y=0}^{Y=Y^{-}}\Big|\\
&+C\kappa^{-2}\Big|\int_0^{Y^{-}}\pa_Y((-\tilde F(Y))^{-1}Q_1)\sin(2\Theta(-\kappa\eta(Y))) dY\Big|\leq C\kappa^{-1}.
\end{align*}
Thus, we infer from the above analysis and notations \eqref{def: J(c_r)}-\eqref{def: tilde Q_1} that
\begin{align}\label{equiv: I_Q1}
I_{Q_1}=&\f{\pi \kappa^{-\f32}}2( \pa_Y \tilde F(Y_0))^{\f12}\int_0^{Y^{-}}(-\tilde F(Y))^{-\f12}Q_1(Y)dY+\mathcal O(\al^{-\f43})\\
\nonumber
=&\f{\pi \kappa^{-\f32}}2( \pa_Y \tilde F(Y_0))^{\f12}\int_0^{Y_0}(-\tilde F(Y))^{-\f12}\tilde Q_1(Y)dY+\mathcal O(\al^{-\f43}+\al^{-1}c_i)\\
\nonumber
=&\f{\pi \kappa^{-\f32}}2( \pa_Y \tilde F(Y_0))^{\f12} J(c_r; M_a)+\mathcal O(\al^{-\f43}),
\end{align}
by using $(\al,c)\in \mathbb{H}_0$ in the last line. 
For  any $c_r\in\mathcal{H}(U_B, M_a)\cap(1-M_a^{-1},1)$, it holds $J(c_r; M_a)\neq 0$. According to \eqref{equiv: I_Q1} , we get $|I_{Q_1}|\sim \al^{-1}$.  By \eqref{eq:I-Q1} and $(\al,c)\in \mathbb{H}_0$, we obtain estimates in (2).
\end{proof}

Then we give the pointwise upper bound for $\mathcal{A}_0^k$ and $\mathcal{B}_0^k$, $k=1,2,3$ in the interval $[0, Y_2^*]$, which shall be used in the iteration scheme for $j\geq 1$.

\begin{lemma}\label{lem:A01B01}
 Let $(\alpha, c)\in \mathbb H_0$. Then it holds that for any $Y\in[0, Y_2^*]$,
	\begin{align*}
		&\left(|\mathcal A_0^1(Y)|+e^{2\alpha w_0(Y)}|\mathcal B_0^1(Y)|\right)+e^{\alpha\int_{Y_1-\d_0}^{Y_1} F_r^\f12(Z)dZ}\left(|\mathcal A_0^2(Y)|+e^{2\alpha w_0(Y)}|\mathcal B_0^2(Y)|\right)\leq C\alpha^{-1},\\
		&\left(|\mathrm{Im}(\mathcal A_0^1(Y)|+e^{2\al w_0(Y)}|\mathrm{Im}(\mathcal B_0^1(Y)|\right)\\
		&\qquad\qquad +\al^{2}e^{\alpha\int_{Y_1-\d_0}^{Y_1} F_r^\f12(Z)dZ}\left(|\mathrm{Im}(\mathcal A_0^2(Y))|+e^{2\al w_0(Y)}|\mathrm{Im}(\mathcal B_0^2(Y))|\right)\leq C\alpha c_i,\\
		&|\mathcal A_0^3(Y)|\leq C\al^{-1}e^{-2\al \int_{Y_1^*}^{Y_1-2\d_0}F_r^\f12(Y') dY'}\mathbf 1_{[Y_1^*-\delta_0,Y_2^*]},\\
		&|\mathcal B_0^3(Y)|\leq
		\left\{
		\begin{aligned}
		&C\al^{-2}e^{-2\alpha w_0(Y_1^*-\delta_0)}e^{-2\al \int_{Y_1^*}^{Y_1-2\d_0}F_r^\f12(Z)dZ}+\al^{-1}e^{-2\al w_0(Y_2^*)},\quad Y\in[0, Y_1^*-\d_0],\\
		&C\al^{-2}e^{-2\al w_0(Y)}e^{-2\al \int_{Y_1^*}^{Y_1-2\d_0}F_r^\f12(Z)dZ}\mathbf 1_{[Y_1^*-\delta_0,Y_1^*]}+\al^{-1}e^{-2\al w_0(Y_2^*)},\quad Y\in[Y_1^*-\d_0, Y_2^*].
		\end{aligned}
		\right.
	\end{align*}

		\end{lemma}

\begin{proof}
The first estimate is a direct result of the procedure in Proposition \ref{pro: interation-Airy}. So, we focus on the estimates for the last three ones.\smallskip

\underline{Estimates of $\Im \mathcal{A}_0^1$ and $\Im \mathcal{B}_0^1$.}	We use Proposition \ref{pro: Airy-1} and take $a(Y)=(1-\chi)(Q_1+Q_2)\chi_1$ and $b(Y)=0$. By Lemma \ref{lem: Q_1}, we have
	\begin{align*}
		|(1+Y)^2a(Y)|\leq C \text{ and } |(1+Y)^2\mathrm{Im}(a(Y))|\leq C\alpha^2 c_i,
	\end{align*}
	which implies that for any $Y\in[0, Y_2^*]$,
	\begin{align*}
		|\mathrm{Im}(\mathcal A_0^1(Y))|+|e^{2\al w_0(Y)}\mathrm{Im}(\mathcal B_0^1(Y))|\leq C\alpha c_i.
	\end{align*}
	
\underline{Estimates of $\Im \mathcal{A}_0^2$ and $\Im \mathcal{B}_0^2$.}	We use Proposition \ref{pro: Airy-2} and take $g(Y)=-\al^2 \pa_Y \chi_1 A(Y)$ with $\|g\|_{L^\infty_{w_0}}+c_i^{-1}\|\Im g\|_{L^\infty_{w_0}}\leq C\al^{\f{11}{6}},$ 
 to obtain 
 \begin{align*}
 \|\Im\mathcal{A}_0^2\|_{L^\infty}+\|\Im \mathcal{B}_0^2\|_{L^\infty_{2w_0}}\leq& C \al^{-\f{17}{6}}e^{-\alpha\int_{Y_1-\d_0}^{Y_1} F_r^\f12(Z)dZ}(c_i\|g\|_{L^\infty_{w_0}}+\|\Im g\|_{L^\infty_{w_0}})
\\
\leq& Cc_i\al^{-1}e^{-\alpha\int_{Y_1-\d_0}^{Y_1} F_r^\f12(Z)dZ}\ll C\alpha c_i.
 \end{align*}

\underline{Estimates of $\mathcal{A}_0^3$ and $\mathcal{B}_0^3$.}	By the definition, we know 
	\begin{align*}
		\mathcal A_0^3(Y)=&2\kappa^{-1}\pi\int_0^YB(Z)\bar M(Z)^{-2}\partial_Z\chi_2\bar M(Z)\varphi_0(Z)dZ,\\
		\mathcal{B}_0^3(Y)=&	2\kappa^{-1}\pi\int_Y^{+\infty}A(Z)\bar M(Z)^{-2}\partial_Z\chi_2\bar M\varphi_0(Z)dZ,
	\end{align*}
	which implies 
	\begin{align*}
		&\mathcal A_0^3(Y)\equiv 0 \text{ on }[0,Y_1^*-\delta_0] \text{ and }\mathcal A_0^3(Y)\equiv\mathcal A_0^3( Y_1^*)\text{ on }[Y_1^*,Y_2^*],\\
		&\mathcal B_0^3(Y)\equiv\mathcal B_0^3( Y_2^*)\text{ on }[ Y_1^*, Y_2^*],\quad\mathcal B_0^3(Y)\equiv\mathcal B_0^3( Y_1^*-\delta_0)\text{ on }[0,  Y_1^*-\delta_0].
	\end{align*}
	On the other hand,  notice that 
	\begin{align*}
		&|\varphi_0(Y)|\leq C\al^{-\f16}e^{-2\al \int_{Y}^{Y_1-2\d_0}F_r^\f12(Y') dY'} e^{-\al w_0(Y)},\quad Y\in[Y_1^*-\delta_0,Y_1^*],\\
		&|\varphi_0(Y)|\leq C\al^{\f56}e^{-\al w_0(Y)},\quad Y\in[Y_2^*, Y_2^*+\d_0]
	\end{align*}
	which along with Proposition \ref{pro: Airy-3}  implies
		\begin{align*}
		|\mathcal A_0^3(Y)|\leq& C\al^{-1}e^{-2\al \int_{Y_1^*}^{Y_1-2\d_0}F_r^\f12(Y') dY'}\mathbf 1_{[Y_1^*-\delta_0,Y_2^*]},\quad Y\in[0, Y_2^*],\\
		|\mathcal B_0^3(Y)|\leq&
		\left\{
		\begin{aligned}
		&C\al^{-2}e^{-2\alpha w_0(Y_1^*-\delta_0)}e^{-2\al \int_{Y_1^*}^{Y_1-2\d_0}F_r^\f12(Z)dZ}+\al^{-1}e^{-2\al w_0(Y_2^*)},\quad Y\in[0, Y_1^*-\d_0],\\
		&C\al^{-2}e^{-2\al w_0(Y)}e^{-2\al \int_{Y_1^*}^{Y_1-2\d_0}F_r^\f12(Z)dZ}\mathbf 1_{[Y_1^*-\delta_0,Y_1^*]}+\al^{-1}e^{-2\al w_0(Y_2^*)},\quad Y\in[Y_1^*-\d_0, Y_2^*].
				\end{aligned}
		\right.
	\end{align*}
	
\end{proof}

We introduce an inequality which is used frequently in the following part. Let $\bar Y_1^{**}\in (Y^{+},\bar Y_1^*)$ with
$2w_0(\bar Y_1^{**})>w_0(Y_1+\d)+w_0(Y^{+}+\d)$ for some suitable small $\d>0$  but independent of $\al$. Then by Lemma \ref{lem:real-phi(0)} and Lemma \ref{lem: Im-int}, it holds that
\begin{align}\label{est: bdd by Im K_B}
\begin{split}
&e^{\al w_0(Y^{+}+\d_0)}e^{-\alpha w_0(\bar Y_1^{**})}\Big(e^{-\al w_0(Y_1^*-\d_0)}e^{-2\al \int_{Y_1^*}^{Y_1-2\d_0}F_r^\f12(Y') dY'}+\al e^{-\al w_0(Y_2^*)}\Big)\\
&\ll \al^{-\f23}\cdot \al^{-\f{11}{6}}e^{-\al w_0(\tilde Y)}\leq \al^{-\f23}|\Im(\frak{D}_0(\al,c))|.
\end{split}
\end{align}

Next we give pointwise estimates for $\mathcal{A}_0^4(Y)$ and $\mathcal{B}_0^4(Y)$ in the interval $Y\in[0, Y_2^*]$.

\begin{lemma}\label{lem:A04B04}
Let $(\alpha, c)\in\mathbb H$ and $Y_1^*, Y^*_2$ be given as in Lemma \ref{lem:real-phi(0)}.	Then there holds  for any $Y\in[0,Y^{+}+\d]$,
\begin{align}
		\mathcal A_0^4(Y)=&\frak{D}_0(\al,c)(\partial_YB(Y)\bar M(Y)^{-2}(1-\bar\chi)\label{est: A_0^4-near}+B(Y) M(Y)^{-2}\pa_Y\bar \chi+\mathcal R_{\mathcal{A}_0^4}(Y)),\\
    	\mathcal B_0^4(Y)=&\frak{D}_0(\al,c)(-\partial_Y A(Y)\bar M^{-2}(Y)(1-\bar\chi)\label{est: B_0^4-near}-A(Y) M(Y)^{-2}\pa_Y\bar \chi+\mathcal{R}_{\mathcal{B}_0^4}(Y)),
    \end{align}
    where 
    \begin{align*}
		&|\mathcal{R}_{\mathcal{A}_0^4}(Y)|+c_i^{-1}|\Im (\mathcal{R}_{\mathcal{A}_0^4}(Y))|\leq C\al^{-\f23}e^{\al w_0(Y)},\\
		&|\mathcal{R}_{\mathcal{B}_0^4}(Y)|+c_i^{-1}|\Im(\mathcal{R}_{\mathcal{B}_0^4}(Y))|\leq C\al^{-\f23}e^{-\al w_0(Y)}.
	\end{align*}
For any $Y\in[Y^{+}+\d,Y_2^*]$, we have
    	\begin{align}
		|\mathcal A_0^4(Y)|\leq& C\alpha^{-2}\Big(e^{-2\al \int_{Y_1^*}^{Y_1-2\d_0}F_r^\f12(Y') dY'}+\al e^{-\al \int_{Y}^{Y_2^*}F_r^\f12(Y') dY'}\Big),\label{est: A_0^4-far}\\
		|\mathcal{B}_0^4(Y)|
    	\leq&C\alpha^{-3}e^{-2\alpha w_0(Y)}\Big(e^{-2\al \int_{Y_1^*}^{Y_1-2\d_0}F_r^\f12(Y') dY'}+\al e^{-\al \int_{Y}^{Y_2^*}F_r^\f12(Y') dY'}\Big). \label{est: B_0^4-far}
	\end{align}
In particular, we have
\begin{align}
\mathcal B_0^4(0)=&\frak{D}_0(\al,c)\mathcal{R}_{\mathcal{B}_0^4}(0)\label{est: B_0^4(0)}.
\end{align}
	
	\end{lemma}
\begin{proof}
	By the definition, we write for any $Y\leq  Y^{+}+\d$,
	\begin{align*}
		&\mathcal A_0^4(Y)=\kappa^{-1}\pi\int_0^{Y}B(Z)\bar M(Z)^{-2}(1-\bar M(Z)^2)(1-\bar\chi)\mathcal K_B[\varphi_0](Z)dZ, \\
		&\mathcal B_0^4(Y)=\kappa^{-1}\pi\int_Y^{+\infty}A(Z)\bar M(Z)^{-2}(1-\bar M(Z)^2)(1-\chi)(1-\bar\chi)\mathcal K_B[\varphi_0](Z)dY.
	\end{align*}
	
	\underline{Estimates of $\mathcal A_0^4(Y)$ for any $Y\in[0, Y^{+}+\d]$.} By the fact  $\mathcal K_B[\varphi_0](Y)=\mathcal K_B[\varphi_0](\bar Y_1^*)$	for $Y\in[0, Y_1^*-\d_0]$, we first  write 
	\begin{align}\label{eq:A04(Y)}
		\begin{split}
				\mathcal A_0^4(Y)		=&\frak{D}_0(\al,c)\int_0^Y\alpha^2F(Z)B(Z)\bar M(Z)^{-2}(1-\bar\chi)dZ.
		\end{split}
	\end{align}
We use the  equation $\mathcal{L}[B]=(Q_1+Q_2)B$ and integration by parts to write
	\begin{align*}
		&\int_0^Y\alpha^2F(Z)B(Z)\bar M(Z)^{-2}(1-\bar\chi)dZ\\
		=&\int_0^{Y}(\partial_Z^2B-\frac{2\pa_Y\bar M}{M}\partial_Z B-(Q_1+Q_2)B)\bar M(Z)^{-2}(1-\bar\chi)dZ\\
		=&\partial_YB(Y)\bar M(Y)^{-2}(1-\bar\chi)+B(Y) M(Y)^{-2}\pa_Y\bar \chi-\int_0^YQ_1(Z)B(Z)\bar M(Z)^{-2}(1-\bar\chi)dZ\\
		&\quad-\int_0^YB(Z)\partial_Z(\bar M^{-2}\partial_Z\bar\chi)dZ-\int_0^YQ_2(Z)B(Z)\bar M(Z)^{-2}(1-\bar\chi)dZ\\
		=&\partial_YB(Y)\bar M(Y)^{-2}(1-\bar\chi)+B(Y) M(Y)^{-2}\pa_Y\bar \chi+I_1+I_2+I_3.
	\end{align*}
By Lemma \ref{lem: Ai, Bi} and \eqref{est: Q_2}, it is easy to see 
\begin{align*}
|I_3|\leq C\kappa^{-\f14} \al^2 c_i e^{\al w_0(Y)}\leq C\al^{\f{11}{6}}c_i e^{\al w_0(Y)}.
\end{align*}

Next, we give an estimate of $I_1$. For $Y\in [0, Y^{-}]$, we use the similar argument above to get
\begin{align*}
	I_1=&-\alpha^{-2}\partial_YBQ_1\bar M^{-2}F^{-1}(1-\bar\chi)\Big|_{Y=0}^{Y}+\alpha^{-2}\int_{0}^Y\partial_ZB(Z)\bar M^{-2}\partial_Z(Q_1F^{-1}(1-\bar\chi))dZ\\
	&\quad+\alpha^{-2}\int_{0}^Y(Q_1+Q_2)B(Z)Q_1\bar M^{-2}F^{-1}(1-\bar\chi)dZ=I_1^1+I_1^2+I_1^3.
	\end{align*}
According to Lemma \ref{lem: Ai, Bi}, we have
\begin{align*}
\sum_{k=1}^3(|I_1^k|+c_i^{-1}|\Im I_1^k|)\leq&C\al^{-2}\al^{\f56}|Y-Y_0|^{-\f34}\leq C\al^{-\f23},
\end{align*}
which implies that for $Y\in[0, Y^{-}]$, it holds that $|I_1|+c_i^{-1}|\Im I_1|\leq C\al^{-\f23}.$
For $Y\in[Y^{-},Y^{+}+\d]$, we have
\begin{align*}
I_1=-\int_0^{Y^{-}}Q_1(Z)B(Z)\bar M(Z)^{-2}(1-\bar\chi)dZ-\int_{Y^-}^YQ_1(Z)B(Z)\bar M(Z)^{-2}(1-\bar\chi)dZ=I_1^4+I_1^5,
\end{align*}
Similarly, we have
\begin{align*}
\sum_{k=4}^5(|I_1^k|+c_i^{-1}|\Im I_1^k|)\leq C\al^{-\f23} e^{\al w_0(Y)}.
\end{align*}
Then  for $Y\in[Y^{-},Y^{+}+\d]$, we have $|I_1|+c_i^{-1}|\Im I_1|\leq C\al^{-\f23}e^{\al w_0(Y)}.$
 Thus, for $Y\in[0, Y^{+}+\d]$, we arrive at
 \begin{align*}
 |I_1|+c_i^{-1}|\Im I_1|\leq C\al^{-\f23}e^{\al w_0(Y)}.
 \end{align*}

Using a similar argument in $I_1$, we also get
\begin{align*}
|I_2|+c_i^{-1}|\Im I_2|\leq &C\al^{-\f23}e^{\al w_0(Y)}.
\end{align*}

Gathering estimates about $I_1$, $I_2$ and $I_3$, we obtain
\begin{align*}
\int_0^Y\alpha^2F(Z)B(Z)\bar M(Z)^{-2}(1-\bar\chi)dZ=\partial_YB(Y)\bar M(Y)^{-2}(1-\bar\chi)+B(Y) M(Y)^{-2}\pa_Y\bar \chi+\mathcal{R}_{\mathcal{A}_0^4}(Y),
\end{align*}
where 
\begin{align*}
|\mathcal{R}_{\mathcal{A}_0^4}(Y)|+c_i^{-1}|\Im (\mathcal{R}_{\mathcal{A}_0^4}(Y))|\leq C\al^{-\f23}e^{\al w_0(Y)}.
\end{align*}
 Then \eqref{est: A_0^4-near} follows by putting the above estimate into \eqref{eq:A04(Y)}.\smallskip

\underline{Estimates of $\mathcal A_0^4(Y)$ for any $Y\in[Y^{+}+\d,Y_2^*]$.} 		
		 \eqref{est: A_0^4-far} follows by Proposition \ref{pro: Airy-4} and Proposition \ref{pro: rayleigh-varphi}.\smallskip
			
\underline{Estimates of $\mathcal B_0^4(Y)$ for any $Y\in[Y^{+}+\d,Y_2^*]$.} 	Again by the proof of Proposition \ref{pro: Airy-4} and Proposition \ref{pro: rayleigh-varphi}, we know that for any $Y\in[Y^{+}+\d,Y_2^*]$,
    \begin{align*}
    	|\mathcal{B}_0^4(Y)|
    	\leq&\alpha^{-3}e^{-2\alpha w_0(Y)}\Big(e^{-2\al \int_{Y_1^*}^{Y_1-2\d_0}F_r^\f12(Y') dY'}+\al e^{-\al \int_{Y}^{Y_2^*}F_r^\f12(Y') dY'}\Big),
    \end{align*}
which gives \eqref{est: B_0^4-far}.\smallskip

\underline{Estimates of $\mathcal B_0^4(Y)$ for any $Y\in[0, Y^{+}+\d]$.} 		We first  write for any $Y\in[0,Y^{+}+\d]$,
	\begin{align}\label{est: B_0^4-1}
		\mathcal B_0^4(Y)=&\mathcal B_0^4( Y_1^*-\d_0)+\frak{D}_0(\al,c)\int_Y^{  Y_1^*-\d_0}\al^2A(Z)\bar M(Z)^{-2}(1-\bar M(Z)^2)(1-\bar\chi)dY,
	\end{align}
where we used the fact $\mathcal K_B[\varphi_0](Y)=\mathcal K_B[\varphi_0](\bar Y_1^*)$	for $Y\in[0, Y_1^*-\d_0]$.
For the first term, we use  the proof of Proposition \ref{pro: Airy-4} and Proposition \ref{pro: rayleigh-varphi} to get
\begin{align*}
|\mathcal B_0^4( Y_1^*-\d_0)|\leq &C\alpha^{-3}e^{-\alpha w_0(Y_1^*-\d_0)}\Big(e^{-\al w_0(Y_1^*-\d_0)}e^{-2\al \int_{Y_1^*}^{Y_1-2\d_0}F_r^\f12(Y') dY'}+\al e^{-\al w_0(Y_2^*)}\Big)\\
\ll &\al^{-\f23}|\Im(\frak{D}_0(\al,c))| e^{-\al w_0(Y)},
\end{align*}
here we used \eqref{est: bdd by Im K_B} in the last line.

For the second term, we use a similar process in $\mathcal{A}_0^4(Y)$ for $Y\in[0, Y^{+}+\d]$. Indeed, by integration by parts, we have
	\begin{align*}
	&\int_Y^{Y_1^*-\d_0}\alpha^2 A(Z)F(Z)\bar M(Z)^{-2}(1-\bar\chi)dZ\\
	&=-\partial_Y A(Y)\bar M^{-2}(Y)(1-\bar\chi)-A(Y)\bar M^{-2}(Y)\pa_Y \bar \chi-\int_Y^{Y_1^*-\d_0}Q_1A(Z)\bar M^{-2}(Z)(1-\bar\chi)dZ\\
	&\quad-\int_Y^{ Y_1^*-\d_0}A(Z)\partial_Z(\bar M^{-2}\partial_Z\bar\chi)dZ-\int_Y^{ Y_1^*-\d_0}Q_2A(Z)\bar M^{-2}(Z)(1-\bar\chi)dZ\\
	&\quad+\partial_YA( Y_1^*-\d_0)\bar M( Y_1^*-\d_0)^{-2}\\
	&=-\partial_Y A(Y)\bar M^{-2}(Y)(1-\bar\chi)-A(Y)\bar M^{-2}(Y)\pa_Y \bar \chi+II_1+II_2+II_3+II_4.
	\end{align*}
Using a similar argument in $I_1$, we have
\begin{align*}
&|II_1|+|II_2|+c_i^{-1}(|\Im II_1|+|\Im II_2|)\leq C\al^{-\f23}e^{-\al w_0(Y)},\quad |II_3|\leq C\al^{\f{11}{6}}c_i e^{-\al w_0(Y)},\\
& |II_4|+c_i^{-1}|\Im II_4|\leq C\al^{\f56}e^{-\al w_0(Y_1^*-\d_0)}\leq C\al^{-\f23}e^{-\al w_0(Y)}.
\end{align*}
Thus, we obtain
\begin{align*}
\int_Y^{Y_1^*-\delta_0}\alpha^2 A(Z)F(Z)\bar M(Z)^{-2}(1-\bar\chi)dZ=-\partial_Y A(Y)\bar M^{-2}(Y)(1-\bar\chi)-A(Y)\bar M^{-2}(Y)\pa_Y \bar \chi+\mathcal{R}_{\mathcal{B}_0^4}(Y),
\end{align*}
where 
\begin{align*}
|\mathcal{R}_{\mathcal{B}_0^4}(Y)|+c_i^{-1}|\Im(\mathcal{R}_{\mathcal{B}_0^4}(Y))|\leq C\al^{-\f23}e^{-\al w_0(Y)}.
\end{align*}
This along with \eqref{est: B_0^4-1} shows that for $Y\in[0,Y^{+}+\d]$,
\begin{align*}
\mathcal B_0^4(Y)=&\frak{D}_0(\al,c)(-\partial_Y A(Y)\bar M^{-2}(Y)(1-\bar\chi)-A(Y)\bar M^{-2}(Y)\pa_Y \bar \chi+\mathcal{R}_{\mathcal{B}_0^4}(Y)),
\end{align*}
which gives \eqref{est: B_0^4-near}. In particular, taking $Y=0$, we obtain \eqref{est: B_0^4(0)}. 
\end{proof}

In the following, we shall give the estimates for $\tilde{\mathcal A}_0(Y)$ and $\tilde{\mathcal B}_0(Y)$ when $Y\in [0, Y^{-}]$.
\begin{lemma}\label{lem: tildeA0B0}
	Let $(\alpha, c)\in\mathbb H$ and $Y_1^*, Y^*_2$ be given as in Lemma \ref{lem:real-phi(0)}.	There holds for any $Y\in[0, Y^{-}]$,
\begin{align}\label{est: tilde A_0}
\tilde{\mathcal A}_0(Y)=&-\frak{D}_0(\al,c)\Big(-\partial_YB(Y)\bar M(Y)^{-2}\bar\chi(Y)+B(Y)\bar M(Y)^{-2}\partial_Y\bar\chi(Y)+\mathcal{R}_{\tilde{\mathcal{A}_0}}(Y)\Big),
\end{align}
and 
\begin{align}\label{est: tilde B_0}
\tilde{\mathcal B}_0(Y)=&-\frak{D}_0(\al,c)\Big(\partial_Y A(Y)\bar M^{-2}(Y)\bar\chi(Y)-\partial_Y A(0)\bar M^{-2}(0)-A(Y)\bar M(Y)^{-2}\partial_Y\bar\chi(Y)+\mathcal{R}_{\tilde{\mathcal{B}_0}}(Y)\Big),
\end{align}
	where 
	\begin{align*}
		&|\mathcal{R}_{\tilde{\mathcal{A}_0}}(Y)|+|\mathcal{R}_{\tilde{\mathcal{B}_0}}(Y)|+c_i^{-1}\left(|\Im(\mathcal{R}_{\tilde{\mathcal{A}_0}}(Y))+|\Im(\mathcal{R}_{\tilde{\mathcal{B}_0}}(Y))|\right)|\leq C\al^{-\f23}.
	\end{align*}
In particular, we have 
\begin{align}\label{est: tilde A_0(0)}
\tilde{\mathcal A}_0(0)=-\frak{D}_0(\al,c)\Big(-\partial_YB(0)\bar M(0)^{-2}+\mathcal{R}_{\tilde{\mathcal{A}_0}}(0)\Big).
\end{align}

\end{lemma}
\begin{proof}
	According to the definition of $\tilde {\mathcal A}_0(Y)$ and $\tilde{\mathcal B}_0(Y)$, we write for any $Y\in[0, Y^{-}]$,
	\begin{align*}
		&\tilde{\mathcal A}_0(Y)=-\kappa^{-1}\pi\int_Y^{+\infty}B(Z)\bar M(Z)^{-2}(1-\bar M(Z)^2)\bar\chi(Z)\mathcal K_B[\varphi_0](Z)dZ, \\
		&\tilde{\mathcal B}_0(Y)=-\kappa^{-1}\pi\int_0^{Y}A(Z)\bar M(Z)^{-2}(1-\bar M(Z)^2)\bar\chi(Z)\mathcal K_B[\varphi_0](Z)dZ.
	\end{align*}
	
\underline{Estimates of $\tilde{\mathcal A}_0(Y)$ for $Y\in[0, Y^{-}]$.}
	As $\mathrm{supp}(\bar\chi)=[0, \bar Y]\subset[0, Y^{-}]$, we know that for any $Y\in[0, Y^{-}]$,
	\begin{align}\label{def: tilde A_0}
		\tilde{\mathcal A}_0(Y)=-\frak{D}_0(\al,c)\int_Y^{Y^{-}}\alpha^2F(Z)B(Z)\bar M(Z)^{-2}\bar\chi(Z)dZ.
	\end{align}
Using the equation $\mathcal{L}[B]=(Q_1+Q_2)B$, we integrate by parts and write
		\begin{align*}
		&\int_Y^{Y^{-}}\alpha^2F(Z)B(Z)\bar M(Z)^{-2}\bar\chi(Z) dZ=\int_Y^{Y^{-}}(\partial_Z^2B-\frac{2\bar M}{M}\partial_Z B-(Q_1+Q_2)B)\bar M(Z)^{-2}\bar\chi(Z) dZ\\
		&=-\partial_YB(Y)\bar M(Y)^{-2}\bar\chi(Y)+B(Y)\bar M(Y)^{-2}\partial_Y\bar\chi(Y) -\int_Y^{Y^{-}}Q_1(Z)B(Z)\bar M(Z)^{-2}\bar\chi(Z) dZ\\
		&\quad+\int_Y^{Y^{-}}B(Z)\partial_Z(\bar M^{-2}\partial_Z\bar\chi)dZ-\int_Y^{Y^{-}}Q_2(Z)B(Z)\bar M(Z)^{-2}\bar\chi(Z) dZ\\
		&=-\partial_YB(Y)\bar M(Y)^{-2}\bar\chi(Y)+B(Y)\bar M(Y)^{-2}\partial_Y\bar\chi(Y)+\mathcal{R}_{\tilde{\mathcal{A}_0}}(Y).
	\end{align*}
	Using a similar argument in Lemma \ref{lem:A04B04}, we dedude
\begin{align*}
|\mathcal{R}_{\tilde{\mathcal{A}_0}}(Y)|+c_i^{-1}|\Im(\mathcal{R}_{\tilde{\mathcal{A}_0}}(Y))|\leq C\al^{-\f23}.
\end{align*}
Then  \eqref{est: tilde A_0} follows by plugging the above equality into \eqref{def: tilde A_0}. Taking $Y=0$ in \eqref{est: tilde A_0}, we obtain \eqref{est: tilde A_0(0)}.\smallskip

\underline{Estimates of $\tilde{\mathcal B}_0(Y)$ for $Y\in[0, Y^{-}]$.}	We first notice that 
	\begin{align}\label{def: tilde B_0}
		\tilde{\mathcal B}_0(Y)=-\frak{D}_0(\al,c)\int_0^{Y}\alpha^2F(Z)A(Z)\bar M(Z)^{-2}\bar\chi(Z)dZ,
	\end{align}
where
	\begin{align*}
	&\int_0^{ Y}\alpha^2 A(Z)F(Z)\bar M(Z)^{-2}\bar\chi(Z) dZ\\
	&=\partial_Y A(Y)\bar M^{-2}(Y)\bar\chi(Y)-\partial_Y A(0)\bar M^{-2}(0)-A(Y)\bar M(Y)^{-2}\partial_Y\bar\chi(Y) \\
	&\quad-\int_0^{Y}Q_1(Z)A(Z)\bar M^{-2}(Z)\bar\chi dZ+\int_0^{Y}A(Z)\partial_Z(\bar M^{-2}\partial_Z\bar\chi)dZ\\
	&\quad -\int_0^{Y}Q_2(Z)A(Z)\bar M^{-2}(Z)\bar\chi dZ\\
	&=\partial_Y A(Y)\bar M^{-2}(Y)\bar\chi(Y)-\partial_Y A(0)\bar M^{-2}(0)-A(Y)\bar M(Y)^{-2}\partial_Y\bar\chi(Y)+\mathcal{R}_{\tilde{\mathcal{B}_0}}(Y).
	\end{align*}
	By a similar argument, we know that 
	\begin{align*}
|\mathcal{R}_{\tilde{\mathcal{B}_0}}(Y)|+c_i^{-1}|\Im(\mathcal{R}_{\tilde{\mathcal{B}_0}}(Y))|\leq C\al^{-\f23}.
\end{align*}
Then \eqref{est: tilde B_0} follows by plugging the above equality into \eqref{def: tilde B_0}.
\end{proof}

For convenience, we denote
\begin{align*}
\mathcal A_{0}^{123}=\sum_{k=1}^3\mathcal A_{0}^{k},\quad \mathcal B_{0}^{123}=\sum_{k=1}^3\mathcal B_{0}^{k}.
\end{align*}

Finally, we present the pointwise estimates of $\tilde{P}_{0,2}^{(1)}$ and $\tilde{P}_{0,2}^{(2)}$ on the internal $[0, Y_2^*]$. Here $\tilde{P}_{0,2}^{(1)}$ and $\tilde{P}_{0,2}^{(2)}$ are defined in \eqref{eq:notation-tilde P02}.

\begin{proposition}\label{Prop: P01P02}
Let $(\alpha,c )\in\mathbb H$ and  $Y_1^*, Y^*_2$ be given as in Lemma \ref{lem:real-phi(0)}. Then 
\begin{enumerate}
\item  for  $Y\in[0,Y^{+}+\d]$, we have
\begin{align*}
	\tilde{P}_{0,2}^{(1)}(Y)=&\mathcal A_{0}^{123}(Y)A(Y)+\mathcal B_{0}^{123}(Y)B(Y)+\frak{D}_0(\al,c)\Big(\pi^{-1}\kappa(1-\bar\chi(Y))+\mathcal{R}_{\tilde{P}_{0,2}^{(1)}}(Y)\Big),
\end{align*}
where 
\begin{align*}
	&\left|\mathcal A_{0}^{123}(Y)\right|+e^{2\alpha w_0(Y)}|\mathcal B_{0}^{123}(Y)|+(\al^2 c_i)^{-1}|\mathrm{Im}(\mathcal A_{0}^{123}(Y))|\leq C\alpha^{-1},\\
	& |\mathrm{Im}(\mathcal B_{0}^{123}(Y))|\leq C\Big(\alpha c_i e^{-2\al w_0(Y)}+\al^{-2}e^{-2\alpha w_0(Y_1^*-\delta_0)}e^{-2\al \int_{Y_1^*}^{Y_1-2\d_0}F_r(Z)dZ}+\al^{-1}e^{-2\al w_0(Y_2^*)}\Big),\\
	& \Big|\mathcal{R}_{\tilde{P}_{0,2}^{(1)}}(Y)\Big|+c_i^{-1}\Big|\mathrm{Im}(\mathcal{R}_{\tilde{P}_{0,2}^{(1)}}(Y))\Big|\leq C\alpha^{-\f23}.
\end{align*}
\item for $Y\in[Y^{+}+\d, Y_2^*]$, we have
\begin{align*}
\tilde{P}_{0,2}^{(1)}(Y)=\mathcal A_0(Y)A(Y)+\mathcal B_0(Y)B(Y),
\end{align*} 
where
\begin{align*}
		&|\mathcal A_0(Y)|+e^{2\alpha w_0(Y)}|\mathcal B_0(Y)|\leq C \alpha^{-1},\\
		&|\mathrm{Im}(\mathcal A_0(Y))|+e^{2\alpha w_0(Y)}|\mathrm{Im}(\mathcal B_0(Y))|\\
		&\quad\leq C\Big(\alpha c_i+\al^{-1}e^{-2\al \int_{Y^*_1}^{Y_1-2\d_0}F_r^\f12(Y') dY'}+\al^{-1}e^{-\al \int_{Y}^{Y_2^*}F_r^\f12(Y') dY'}\Big).
	\end{align*}
\item for $Y\in[0,Y^-]$, we have
\begin{align*}
		\tilde{P}_{0,2}^{(2)}(Y)
	=&\frak{D}_0(\al,c)\Big(\pi^{-1}\kappa\bar\chi+\partial_YA(0)\bar M(0)^{-2}B(Y) +\mathcal {R}_{\tilde P_{0,2}^{(2)}}(Y)\Big),\\
	\partial_Y\tilde{P}_{0,2}^{(2)}(Y)
	=&\frak{D}_0(\al,c)\Big(\pi^{-1}\kappa \partial_Y\bar\chi(Y)+\partial_Y B(Y)\partial_Y A(0)\bar M(0)^{-2}-\mathcal{R}_{\pa_Y\tilde{P}_{0,2}^{(2)}}(Y)\Big),
	\end{align*}
	where 
	\begin{align*}
		&\Big|\mathcal{R}_{\tilde{P}_{0,2}^{(2)}}(Y)\Big|+c_i^{-1}\Big|\mathrm{Im}(\mathcal{R}_{\tilde{P}_{0,2}^{(2)}}(Y))\Big|\leq C\alpha^{-\f23},\quad \Big|\mathcal{R}_{\pa_Y\tilde{P}_{0,2}^{(2)}}(Y)\Big|+c_i^{-1}\Big|\mathrm{Im}(\mathcal{R}_{\pa_Y\tilde{P}_{0,2}^{(2)}}(Y))\Big|\leq C\alpha^{\f16}.
	\end{align*}

\end{enumerate}

 \end{proposition}
\begin{proof}
According to the definition of $\tilde P_{0,2}^{(1)}(Y)$, we have
\begin{align*}
	\tilde P_{0,2}^{(1)}(Y)
	=&\mathcal A_0(Y)A(Y)+\mathcal B_0(Y)B(Y)=(\mathcal A_0^{123}(Y)+\mathcal A_0^4(Y))A(Y)+(\mathcal B_0^{123}(Y)+\mathcal B_0^4(Y))B(Y).
\end{align*}

\underline{Estimates of $\tilde P_{0,2}^{(1)}(Y)$ for $Y\in[0, Y^{+}+\d]$.}
By Lemma \ref{lem:A01B01}, we know that for any $Y\in[0,Y_2^*]$,
\begin{align}\label{est: P_02^(1)-1}
	|\mathcal A_0^{123}(Y)|+e^{2\alpha w_0(Y)}|\mathcal B_0^{123}(Y)|\leq C\alpha^{-1},
\end{align}
and for $Y\in[0,Y^{+}+\d]$,
\begin{align}
	& |\mathrm{Im}(\mathcal A_{0}^{123}(Y))|\leq C\al c_i,\label{est: P_02^(1)-2}\\
	& |\mathrm{Im}(\mathcal B_{0}^{123}(Y))|\leq C\Big(\alpha c_i e^{-2\al w_0(Y)}\label{est: P_02^(1)-3}\\
	\nonumber
	&\qquad\qquad\qquad\qquad\qquad+\al^{-2}e^{-2\alpha w_0(Y_1^*-\delta_0)}e^{-2\al \int_{Y_1^*}^{Y_1-2\d_0}F_r(Z)dZ}+\al^{-1}e^{-2\al w_0(Y_2^*)}\Big).
\end{align}
By Lemma \ref{lem:A04B04}, we  deduce that for $Y\in[0,Y^{+}+\d]$
\begin{align*}
	&\mathcal A_0^4(Y)A(Y)+\mathcal B_0^4(Y)B(Y)\\
	&=\frak{D}_0(\al,c)\Big(A(Y)\partial_YB(Y)-B(Y)\partial_YA(Y)\Big)\bar M(Y)^{-2}(1-\bar\chi)+\frak{D}_0(\al,c)\mathcal{R}_{\tilde{P}_{0,2}^{(1)}}(Y)          \\
	&=\frak{D}_0(\al,c)\Big(\pi^{-1}\kappa(1-\bar\chi(Y))+\mathcal{R}_{\tilde{P}_{0,2}^{(1)}}(Y)\Big),
\end{align*}
where 
\begin{align*}
\mathcal{R}_{\tilde{P}_{0,2}^{(1)}}(Y)=\mathcal R_{\mathcal{A}_0^4}(Y)A(Y)+\mathcal R_{\mathcal{B}_0^4}(Y)B(Y),
\end{align*}
with
\begin{align}\label{est: P_02^(1)-3}
 \Big|\mathcal{R}_{\tilde{P}_{0,2}^{(1)}}(Y)\Big|+c_i^{-1}\Big|\mathrm{Im}(\mathcal{R}_{\tilde{P}_{0,2}^{(1)}}(Y))\Big|\leq C\alpha^{-\f23}.
\end{align}

Thus, we obtain that for any $Y\in[0,Y^{+}+\d]$,
\begin{align*}
	\tilde{P}_{0,2}^{(1)}(Y)=&\mathcal A_{0}^{123}(Y)A(Y)+\mathcal B_{0}^{123}(Y)(Y)B(Y)+\frak{D}_0(\al,c)\Big(\pi^{-1}\kappa(1-\bar\chi(Y))+\mathcal{R}_{\tilde{P}_{0,2}^{(1)}}(Y)\Big),
\end{align*}
and \eqref{est: P_02^(1)-1}-\eqref{est: P_02^(1)-3} hold. This gives the result in (1).\smallskip

\underline{Estimates of $\tilde P_{0,2}^{(1)}(Y)$ for $Y\in[Y^{+}+\d, Y_2^*]$.}
We know by Lemma \ref{lem:A01B01} that 
\begin{align*}
	&|\mathrm{Im}(\mathcal A_0^{123}(Y)|+e^{2\alpha w_0(Y)}|\mathrm{Im}(\mathcal B_0^{123}(Y))|\\
	&\leq C\Big(\alpha c_i+\al^{-1}e^{-2\al \int_{Y^*_1}^{Y_1-2\d_0}F_r^\f12(Y') dY'}+\al^{-1}e^{-2\al \int_{Y}^{Y_2^*}F_r^\f12(Y') dY'}\Big).
\end{align*}
By Lemma \ref{lem:A04B04}, we know that for any $Y\in[Y^{+}+\d,Y_2^*]$,
\begin{align*}
		|\mathcal A_0^4(Y)|+\al e^{2\al w_0(Y)}|\mathcal{B}_0^4(Y)|\leq& C\alpha^{-2}\Big(e^{-2\al \int_{Y_1^*}^{Y_1-2\d_0}F_r^\f12(Y') dY'}+\al e^{-\al \int_{Y}^{Y_2^*}F_r^\f12(Y') dY'}\Big),
	\end{align*}
	which along with \eqref{est: P_02^(1)-1} gives that for $Y\in[Y^{+}+\d,Y_2^*]$,
	\begin{align*}
		&|\mathcal A_0(Y)|+e^{2\alpha w_0(Y)}|\mathcal B_0(Y)|\lesssim \alpha^{-1},\\
		&|\mathrm{Im}(\mathcal A_0(Y))|+e^{2\alpha w_0(Y)}|\mathrm{Im}(\mathcal B_0(Y))|\leq C \Big(\alpha c_i+\al^{-1}e^{-2\al \int_{Y^*_1}^{Y_1-2\d_0}F_r^\f12(Y') dY'}+\al^{-1}e^{-\al \int_{Y}^{Y_2^*}F_r^\f12(Y') dY'}\Big).
	\end{align*}
	This proves the result in (2).\smallskip

\underline{Estimates of $\tilde P_{0,2}^{(2)}(Y)$ and $\pa_Y\tilde P_{0,2}^{(2)}(Y)$for $Y\in[0, Y^{-}]$.}	
	By Lemma \ref{lem: tildeA0B0}, we obtain that for any $Y\in[0,Y^-]$,
	\begin{align*}
		\tilde{P}_{0,2}^{(2)}(Y)=&\tilde{\mathcal A}_0(Y)A(Y)+\tilde{\mathcal B}_0(Y)B(Y)\\
	=&\frak{D}_0(\al,c)\Big(\pi^{-1}\kappa\bar\chi+\partial_YA(0)\bar M(0)^{-2}B(Y) +\mathcal {R}_{\tilde P_{0,2}^{(2)}}(Y)\Big),\\
		\partial_Y\tilde{P}_{0,2}^{(2)}(Y)=&\tilde{\mathcal A}_0(Y)\partial_Y A(Y)+\tilde{\mathcal B}_0(Y)\partial_Y B(Y)\\
	=&\frak{D}_0(\al,c)\Big(\pi^{-1}\kappa \partial_Y\bar\chi(Y)+\partial_Y B(Y)\partial_Y A(0)\bar M(0)^{-2}-\mathcal{R}_{\pa_Y\tilde{P}_{0,2}^{(2)}}(Y)\Big),
	\end{align*}  
	where 
\begin{align*}
&\mathcal {R}_{\tilde P_{0,2}^{(2)}}(Y)=-A(Y)	\mathcal{R}_{\tilde{\mathcal{A}_0}}(Y)-B(Y)	\mathcal{R}_{\tilde{\mathcal{B}_0}}(Y),\quad
\mathcal {R}_{\pa_Y\tilde P_{0,2}^{(2)}}(Y)=-\pa_YA(Y)	\mathcal{R}_{\tilde{\mathcal{A}_0}}(Y)-\pa_YB(Y)	\mathcal{R}_{\tilde{\mathcal{B}_0}}(Y),\end{align*}
with
\begin{align*}
&\Big|\mathcal{R}_{\tilde{P}_{0,2}^{(2)}}(Y)\Big|+c_i^{-1}\Big|\mathrm{Im}(\mathcal{R}_{\tilde{P}_{0,2}^{(2)}}(Y))\Big|\leq C\alpha^{-\f23},\quad \Big|\mathcal{R}_{\pa_Y\tilde{P}_{0,2}^{(2)}}(Y)\Big|+c_i^{-1}\Big|\mathrm{Im}(\mathcal{R}_{\pa_Y\tilde{P}_{0,2}^{(2)}}(Y))\Big|\leq C\alpha^{\f16},
\end{align*}
which gives (3). 
\end{proof}

\subsection{Expansion of boundary values  for $j\geq 1$}
This subsection is devoted to giving the pointwise estimates for the imaginary part of $\mathcal A_j(Y)$ and $\mathcal B_j(Y)$ in the interval $[0, Y_2^*]$, which are used to deduce \eqref{expansion: tB_1(0)}. We first introduce  a function on the interval $Y\in[0, Y^{+}+\d]$
\begin{align}\label{def: cE_near}
\begin{split}
\mathcal{E}_{near}(Y)=c_i&+\al^{-\f23}e^{\alpha w_0(Y)}
	\left|\mathrm{Im}\left(\frak{D}_0(\al,c)\right)\right|+\alpha^{-1}\left|\mathrm{Im}\left(\frak{D}_0(\al,c)\partial_YA(0)\bar M^{-2}(0)\right)\right|,
	\end{split}
\end{align}
and a function on  the interval $Y\in[ Y^{+}+\d, Y_2^*]$
\begin{align}\label{def: cE_far}
\begin{split}
\mathcal{E}_{far}(Y)=&\mathcal{E}_{near}(Y^{+}+\d_0)+\al^{-2}e^{-2\al \int_{Y_1^*}^{Y_1-2\d_0}F_r^\f12(Y') dY'}\\
&+\al^{-2}e^{-\al \int_{Y}^{Y_2^*}F_r^\f12(Y') dY'}+\al ^{\f16}e^{-\al \int_{\min\{Y, Y_1-\d_0\}}^{Y_2^*}F_r^\f12(Y') dY'}.
\end{split}
\end{align}
In particular, for $Y=0$, we have
\begin{align*}
\mathcal{E}_{near}(0)=c_i&+\al^{-\f23}
	\left|\mathrm{Im}\left(\frak{D}_0(\al,c)\right)\right|+\alpha^{-1}\left|\mathrm{Im}\left(\frak{D}_0(\al,c)\partial_YA(0)\bar M^{-2}(0)\right)\right|.
\end{align*}

Moreover, we introduce a function defined in $[0, Y^{-}]$
\begin{align*}
\tilde{\mathcal{E}}_{near}(Y)=\left|\mathrm{Im}\left(\frak{D}_0(\al,c)\partial_YA(0)\bar M^{-2}(0)\right)\right|\mathbf 1_{[0, Y^{-}]}(Y),
\end{align*}
which appears in the estimate of  $\cB_1(Y)$ in the interval $Y\in[0, Y^{+}+\d]$. Besides,  $\tilde{\mathcal{E}}_{near}(Y)$ is larger than the third term in $\mathcal{E}_{near}(Y)$ when $Y\in [0, Y^{-}]$, and is also an essential term  in the estimate of $\cB_1(Y)$.

Thus, we need to focus on the estimates for the imaginary part. The goal  is to obtain the following estimates:
 
 \begin{enumerate}
\item 
 For $Y\in[0,Y^{+}+\d]$, we have and for $j\geq1$
\begin{align}\label{assume: induction-1}
&\left|\Im \left(\mathcal{A}_j(Y)\right)\right|+e^{2\al w_0(Y)}\left|\Im \left(\mathcal{B}_j(Y)\right)\right|\lesssim \al^{-j+1}(\mathcal{E}_{near}(Y)+\tilde{\mathcal{E}}_{near}(Y)), 
\end{align}
\item For $Y\in[Y^{+}+\d, Y_2^*]$ and for $j\geq1$
\begin{align}\label{assume: induction-2}
&\left|\Im \left(\mathcal{A}_j(Y)\right)\right|+e^{2\al w_0(Y)}\left|\Im \left(\mathcal{B}_j(Y)\right)\right|\lesssim \al^{-j+1}\mathcal{E}_{far}(Y).
\end{align}
\end{enumerate}

We will use the induction to prove the above estimates \eqref{assume: induction-1}-\eqref{assume: induction-2}. We present the main result in this subsection.
\begin{proposition}\label{pro: BC B_j(0)}
Let $(\alpha,c)\in \mathbb H$ and $Y_1^*, Y^*_2$ be given as in Lemma \ref{lem:real-phi(0)}. Then the
estimates \eqref{assume: induction-1} and \eqref{assume: induction-2}  hold. Moreover, we have
\begin{align}\label{est: BC B(0)}
\begin{split}
&\Big|\mathrm{Im}\Big(\Big(\sum_{j=1}^{+\infty}\mathcal{B}_j(0)\Big)-\frak{D}_0(\al,c)\partial_YA(0)\bar M^{-2}(0)C_{P,1}\Big)\Big|\lesssim  \mathcal{E}_{near}(0),
\end{split}
\end{align}
where
\begin{align*}
C_{P,1}=\kappa^{-1}\pi\int_0^{\bar Y_1^*}A(Z)\bar M^{-2}(Z)\partial_Z\bar\chi_1(Z)\partial_ZB(Z)dZ,\quad |C_{P,1}|+c_i^{-1}|\mathrm {Im}C_{P,1}|\lesssim 1 .
\end{align*} 
\end{proposition}

In order to obtain the results in Proposition \ref{pro: BC B_j(0)}, we should deal with the imaginary part of  $\mathcal{A}_1^1(Y)+\mathcal A_1^5(Y)$ and $\mathcal{B}_1^1(Y)+\mathcal B_1^5(Y)$ in $[0, Y_2^*]$ firstly. {We emphasize that $\mathcal A_1^5(Y)$ and $\mathcal B_1^5(Y)$  related to $\tilde P_{0,2}^{(2)}$  only appear  for $j=1$.}

\begin{lemma}\label{lem: A11A15}
Let $(\alpha,c)\in \mathbb H$ and $Y_1^*, Y^*_2$ be given as in Lemma \ref{lem:real-phi(0)}. The following results hold

\begin{enumerate}
\item For any $Y\in[0,Y^{+}+\d]$, we have
\begin{align*}
|\mathrm{Im}(\mathcal{A}_1^1(Y)+\mathcal A_1^5(Y))|
	+e^{2\al w_0(Y)}|\mathrm{Im}( \mathcal{B}_1^1(Y)+\mathcal B_1^5(Y))|&\lesssim \mathcal{E}_{near}(Y)+\tilde{\mathcal{E}}_{near}(Y).
\end{align*}

\item For any $Y\in[Y^{+}+\d, Y_2^*]$, we have
\begin{align*}
|\mathrm{Im}(\mathcal{A}_1^1(Y)+\mathcal A_1^5(Y))|
	+e^{2\al w_0(Y)}|\mathrm{Im}( \mathcal{B}_1^1(Y)+\mathcal B_1^5(Y))|&\lesssim \mathcal{E}_{far}(Y).
\end{align*}

\item Moreover, we obtain 
\begin{align*}
	&\left|\mathrm{Im}\left(\mathcal{B}_1^1(0)+\mathcal B_1^5(0)-\kappa^{-1}\pi\alpha^{-2}\mathcal K_B[\varphi_0](\bar Y_1^*)\partial_YA(0)\bar M^{-2}(0)C_{P,1}\right)\right|\lesssim \mathcal{E}_{near}(0),
\end{align*}
where $C_{P,1}$ is given in Proposition \ref{pro: BC B_j(0)}.

\end{enumerate}

\end{lemma}
\begin{proof}
By Proposition \ref{Prop: P01P02}, for  $Y\in[0,Y^{+}+\d]$, we write
\begin{align*}
H_{21}[\tilde P_{0,2}^{(1)}](Y)+\tilde H_{21}[\tilde P_{0,2}^{(2)}](Y)
&=(Q_1+Q_2)\left(\mathcal A_{0}^{123}A+\mathcal B_{0}^{123}B+\alpha^{-2}\mathcal K_B[\varphi_0](\bar Y_1^*)\right)\\
&\quad+(Q_1+Q_2)\frak{D}_0(\al,c)\partial_YA(0)\bar M(0)^{-2}B\bar\chi_1  \\
&\quad+(Q_1+Q_2)\frak{D}_0(\al,c)\left(\mathcal{R}_{\tilde{P}_{0,2}^{(1)}}+\mathcal{R}_{\tilde{P}_{0,2}^{(2)}}\bar\chi_1\right)\\
&\quad+\left(\left(\pa_Y^2\bar\chi_1-\f{2\pa_Y\bM}{\bM}\pa_Y \bar\chi_1\right)\tilde P_{0,2}^{(2)}+2\pa_Y\bar\chi_1\pa_Y\tilde P_{0,2}^{(2)}\right)\\
&=J_1(Y)+\cdots+J_4(Y),
\end{align*}
and for any $Y\in[Y^{+}+\d,Y_2^*]$, we write
\begin{align*}
	&H_{21}[\tilde P_{0,2}^{(1)}](Y)+\tilde H_{21}[\tilde P_{0,2}^{(2)}](Y)=H_{21}[\tilde P_{0,2}^{(1)}](Y)=(1-\chi)\chi_1(Q_1+Q_2)(\mathcal A_0A+\mathcal B_0B)(Y).
\end{align*}

{\bf{\underline{Estimates of $\mathcal A_1^1(Y)+\mathcal A_1^5(Y)$. }}} We first give the estimates  for $Y\in[0,Y^{+}+\d]$.
We write for any $Y\in[0,Y^{+}+\d]$,
\begin{align}\label{eq:A11+A15}
	\mathcal A_1^1(Y)+\mathcal A_1^5(Y)=\sum_{j=1}^4\kappa^{-1}\pi\int_0^YB(Z)\bar M^{-2}(Z) J_j(Z)dZ=\sum_{j=1}^4I_{J_j}.
\end{align}
In the following, we give the estimates of $I_{J_j},~j=1,\cdots,4$ term by term.\smallskip

\underline{Estimates of $I_{J_1}$.} We write for  $Y\in[0,Y^{+}+\d]$,
\begin{align}\label{eq:A11A15-1}
	\begin{split}
		I_{J_1}
&=\kappa^{-1}\pi\int_0^YB(Z)\bar M^{-2}(Z)(Q_1+Q_2)(\mathcal A_{0}^{123}A+\mathcal B_{0}^{123}B)(Z)dZ\\
&\quad+\frak{D}_0(\al,c)\int_0^{Y}(Q_1+Q_2)B(Z)\bar M^{-2}(Z)dZ=I_{J_1}^1+\frak{D}_0(\al,c) I_{J_1}^2.	
	\end{split}
\end{align}
We notice that by Proposition \ref{Prop: P01P02}, Lemma \ref{lem: Q_1}  and $|Q_2|\leq C\al^2 c_i$ for $Y\in[0,Y^{+}+\d]$,
\begin{align}\label{est:Q1Q2A02B02}
\begin{split}
			&|(Q_1+Q_2)\mathcal A_{0}^{123}(Y)|+e^{2\alpha w_0(Y)}|(Q_1+Q_2)\mathcal B_{0}^{123}(Y)|\lesssim \alpha^{-1},\\
	&|\mathrm{Im}((Q_1+Q_2)\mathcal A_{0}^{123}(Y))|\lesssim \al c_i,\\
	& |\mathrm{Im}((Q_1+Q_2)\mathcal B_{0}^{123}(Y))|\\
	&\lesssim\Big(\alpha c_i e^{-2\al w_0(Y)}+\al^{-2}e^{-2\alpha w_0(Y_1^*-\delta_0)}e^{-2\al \int_{Y_1^*}^{Y_1-2\d_0}F_r(Z)dZ}+\al^{-1}e^{-2\al w_0(Y_2^*)}\Big),
	\end{split}
	\end{align}
which along with Proposition \ref{pro: Airy-1}  imply  that for any $Y\in[0,Y^{+}+\d]$,
\begin{align}\label{eq:A11A15-1}
	\begin{split}
		\Big|\mathrm{Im}I_{J_1}^1\Big|\lesssim &c_i+e^{2\al w_0(Y)}\Big(\al^{-3}e^{-2\alpha w_0(Y_1^*-\delta_0)}e^{-2\al \int_{Y_1^*}^{Y_1-2\d_0}F_r^\f12(Y') dY'}+\al^{-2}e^{-2\al w_0(Y_2^*)}\Big)\\
		\lesssim&\mathcal{E}_{near}(Y).
	\end{split}
\end{align}
Here we used \eqref{est: bdd by Im K_B} in the last line.

We also have that for  $Y\in[0,Y^{+}+\d]$,
\begin{align}\label{eq:A11A15-2}
	|I_{J_1}^2|+c_i^{-1}|\mathrm{Im}I_{J_1}^2|\leq C\alpha^{-\f23} e^{\alpha w_0(Y)},
\end{align}
which along with \eqref{eq:A11A15-1} and \eqref{eq:A11A15-2} implies that for any $Y\in[0,Y^{+}+\d]$,
\begin{align}\label{est:A11A15-3}
	\begin{split}
			\left|\mathrm{Im}I_{J_1}\right|\lesssim& \mathcal{E}_{near}(Y)+\alpha^{-\f23}e^{\alpha w_0(Y)}|\mathrm{Im}(\frak{D}_0(\al,c))|
			\lesssim \mathcal{E}_{near}(Y).
	\end{split}
\end{align}

\underline{Estimates of $I_{J_2}$.} We  write  for  $Y\in[0,Y^{+}+\d]$,
\begin{align*}
	I_{J_2}=\frak{D}_0(\al,c)\partial_YA(0)\bar M^{-2}(0)C_B(Y),
\end{align*}
where
\begin{align*}
	C_B(Y)=\kappa^{-1}\pi\int_0^Y(Q_1+Q_2)B(Z)^2\bar M^{-2}(Z)\bar\chi_1(Z)dZ,
\end{align*}
with $|C_B(Y)|\leq C\alpha^{-1}\text{ and }|\mathrm{Im}(C_B(Y))|\leq C\alpha c_i. $
Therefore, we  obtain  
\begin{align}\label{est:A11A15-4}
	\begin{split}
		&\left|\mathrm{Im}I_{J_2}\right|\lesssim \alpha^{-1}|\mathrm{Im}\left(\frak{D}_0(\al,c)\partial_YA(0)\bar M^{-2}(0)\right)|+c_i\lesssim \mathcal{E}_{near}(Y).
			\end{split}
\end{align}

\underline{Estimates of $I_{J_3}$.}
By a similar argument as \eqref{est:A11A15-3} , we obtain 
\begin{align}\label{est:A11A15-5}
	\begin{split}
		&\left|\mathrm{Im}I_{J_3}\right|\lesssim \alpha^{-\f32}e^{\alpha w_0(Y)}|\mathrm{Im}(\frak{D}_0(\al,c))|+c_i\lesssim \mathcal{E}_{near}(Y).
	\end{split}
\end{align}

\underline{Estimates of $I_{J_4}$.}
We first notice that 
\begin{align*}
	I_{J_4}=&\kappa^{-1}\pi\int_0^YB(Z)\bar M^{-2}(Z)\left(\pa_Z^2\bar\chi_1-\f{2\pa_Z\bM}{\bM}\pa_Z \bar\chi_1\right)\tilde P_{0,2}^{(2)}(Z)dZ\\
	&+2\kappa^{-1}\pi\int_0^YB(Z)\bar M^{-2}(Z)\partial_Z\bar\chi_1(Z)\partial_Z\tilde P_{0,2}^{(2)}(Z)dZ=I_{J_4}^1+I_{J_4}^2.
\end{align*}
On the other hand, we use Proposition \ref{Prop: P01P02} and the fact $\pa_Z\bar\chi_1\bar\chi=\pa_Z^2\bar\chi_1 \bar\chi=\pa_Z\bar\chi_1\pa_Z\bar\chi=0$ to write that 
\begin{align}\label{eq: tilde P_02^2-1}
\begin{split}
&\Big(\pa_Z^2\bar\chi_1-\f{2\pa_Z\bM}{\bM}\pa_Z \bar\chi_1\Big)\tilde P_{0,2}^{(2)}(Z)\\
&=\frak{D}_0(\al,c)\Big(\pa_Z^2\bar\chi_1-\f{2\pa_Z\bM}{\bM}\pa_Z \bar\chi_1\Big)\Big(\partial_Y A(0)\bar M(0)^{-2}B(Z)+\mathcal {R}_{\tilde P_{0,2}^{(2)}}(Z)	\Big),
\end{split}
\end{align}
and
\begin{align}\label{eq: tilde P_02^2-2}
\partial_Z\bar\chi_1(Z)\partial_Z\tilde P_{0,2}^{(2)}(Z)=\frak{D}_0(\al,c)\partial_Z\bar\chi_1\Big(\partial_Y A(0)\bar M(0)^{-2}\partial_Z B(Z)-\mathcal {R}_{\pa_Y\tilde P_{0,2}^{(2)}}(Z)	\Big).	
\end{align}
Therefore, by a similar argument as \eqref{est:A11A15-4}, we obtain
\begin{align*}
	\left|\mathrm{Im}I_{J_4}^1\right|+\left|\mathrm{Im}I_{J_4}^2\right|\lesssim& \alpha^{-1}|\mathrm{Im}\left(\frak{D}_0(\al,c)\partial_YA(0)\bar M^{-2}(0)\right)|+\alpha^{-\f32}|\mathrm{Im}\left(\frak{D}_0(\al,c)\right)|+c_i\lesssim \mathcal{E}_{near}(Y),
\end{align*}
where we used the estimate
\begin{align*}
&\Big|\kappa^{-1}\pi\int_0^Y\bar M^{-2}(Z)\pa_Z\bar \chi_1(Z)B(Z)\pa_ZB(Z)dZ\Big|\\
&\leq \f12\Big|\kappa^{-1}\pi\int_0^Y\pa_Z(\bar M^{-2}(Z)\pa_Z\bar \chi_1(Z))B^2(Z)dZ\Big|\\
&\qquad+\f12\Big|\kappa^{-1}\pi\bar M^{-2}(Z)\pa_Z\bar \chi_1(Z)B^2(Z)\Big|_{Z=0}^{Z=Y}\Big|\leq C\al^{-1}.
\end{align*}

Therefore, we obtain  
\begin{align}\label{est:A11A15-6}
	\begin{split}
		\left|\mathrm{Im}I_{J_4}(Y)\right|\lesssim& \mathcal{E}_{near}(Y).
	\end{split}
\end{align}

By gathering \eqref{eq:A11+A15}, \eqref{est:A11A15-3}-\eqref{est:A11A15-6}, we obtain that for any $Y\in[0,Y^{+}+\d]$,
\begin{align}\label{est: A11A15 near}
\begin{split}
		&|\mathrm{Im}(\mathcal{A}_1^1(Y)+\mathcal A_1^5(Y))|\leq \sum_{j=1}^4|\mathrm{Im}I_{J_j}(Y)|\lesssim \mathcal{E}_{near}(Y).
	\end{split}
\end{align}

Next we estimate $\mathcal A_1^1(Y)+\mathcal A_1^5(Y)$ for $Y\in[Y^{+}+\d, Y_2^*]$. We write
\begin{align*}
\mathcal A_1^1(Y)+\mathcal A_1^5(Y)=&\mathcal A_1^1(Y^{+}+\d)+\mathcal A_1^5(Y^{+}+\d)+\kappa^{-1}\pi\int_{Y^{+}+\d}^YB(Z)\bar M^{-2}(Z)H_{21}[\tilde P_{0,2}^{(1)}](Z)dZ\\
=&\mathcal A_1^1(Y^{+}+\d)+\mathcal A_1^5(Y^{+}+\d)+I_1,
\end{align*}
where
\begin{align*}
	H_{21}[\tilde P_{0,2}^{(1)}](Y)=(1-\chi)\chi_1(Q_1+Q_2)(\mathcal A_0A+\mathcal B_0B)(Y).
\end{align*}
By  Proposition \ref{pro: Airy-1} and Proposition \ref{Prop: P01P02}, we  obtain that for any $Y\in[Y^{+}+\d,Y_2^*]$,
\begin{align*}
		&\left|\mathrm{Im}I_1(Y)\right|\lesssim \Big( c_i+\al^{-2}e^{-2\al \int_{Y^*_1}^{Y_1-\d_0}F_r^\f12(Y') dY'}+\al^{-2}e^{-\al \int_{Y}^{Y_2^*}F_r^\f12(Y') dY'}\Big)\lesssim \mathcal{E}_{far}(Y),
\end{align*}
which along with \eqref{est: A11A15 near} for $Y=Y^{+}+\d$  shows that for any $Y\in[Y^{+}+\d,Y_2^*]$,
\begin{align}\label{est: A11A15 far}
\begin{split}
	|\mathrm{Im}(\mathcal{A}_1^1(Y)+\mathcal A_1^5(Y))|\lesssim \mathcal{E}_{near}(\bar Y_1^*)+\mathcal{E}_{far}(Y)\lesssim\mathcal{E}_{far}(Y).
	\end{split}
\end{align}

{\bf{\underline{Estimates of $\mathcal B_1^1(Y)+\mathcal B_1^5(Y)$. } }}For $Y\in[Y^{+}+\d,+\infty)$, it holds that
\begin{align*}
|H_{21}[\tilde P_{0,2}^{(1)}](Y)|\leq C\al^{-\f16}e^{-\al w_0(Y)}(\|\mathcal A_0\|_{L^\infty_{w_0}}+\|\mathcal B_0\|_{L^\infty_{2w_0}})\leq C\al^{-\f16}e^{-\al w_0(Y)},
\end{align*}
where
\begin{align*}
	H_{21}[\tilde P_{0,2}^{(1)}](Y)=(1-\chi)\chi_1(Q_1+Q_2)(\mathcal A_0A+\mathcal B_0B)(Y),\quad \tilde H_{21}[\tilde P_{0,2}^{(2)}](Y)=0.
\end{align*}
Then by Proposition \ref{pro: Airy-1}, we get
\begin{align*}
|\mathcal B_1^1(Y_2^*)|\leq&C\al^{-2}e^{-2\al w_0(Y_2^*)}= C\al^{-2}e^{-2\al \int_{Y}^{Y_2^*}F_r^\f12(Y') dY'} e^{-2\al w_0(Y)},\quad \mathcal B_1^5(Y)=0,\quad Y\in[Y^{-},+\infty).
\end{align*}
Thus, for $Y\in[0, Y_2^*]$, we write
\begin{align*}
\mathcal B_1^1(Y)+\mathcal B_1^5(Y)=\mathcal B_1^1(Y_2^*)+\kappa^{-1}\pi\int_{Y}^{Y_2^*}A(Z)\bar M^{-2}(Z)\left(H_{21}[\tilde P_{0,2}^{(1)}]+\tilde H_{21}[\tilde P_{0,2}^{(2)}]\right)dZ=\mathcal B_1^1(Y_2^*)+II.
\end{align*}
By Proposition \ref{pro: Airy-1} and Proposition \ref{Prop: P01P02}, we  obtain that for  $Y\in[Y^{+}+\d,Y_2^*]$,
\begin{align*}
		|\Im II|\lesssim &\Big( c_i+\al^{-2}e^{-2\al \int_{Y^*_1}^{Y_1-2\d_0}F_r^\f12(Y') dY'}+\al^{-2}e^{-\al \int_{Y}^{Y_2^*}F_r^\f12(Y') dY'}\Big)e^{-2\al w_0(Y)}\\
		\lesssim&\mathcal{E}_{far}(Y)e^{-2\al w_0(Y)},
\end{align*}
which gives that for $Y\in[Y^{+}+\d,Y_2^*]$,
\begin{align}\label{est: B11B15 far}
\begin{split}
&|\Im(\mathcal B_1^1(Y)+\mathcal B_1^5(Y))|\lesssim \mathcal{E}_{far}(Y)e^{-2\al w_0(Y)}.
\end{split}
\end{align}

For  $Y\in[0,Y^{+}+\d]$, we use $\cB_1^5(Y_1^*-\d_0)=0$ to write 
\begin{align*}
	\mathcal B_1^1(Y)+\mathcal B_1^5(Y)=&\kappa^{-1}\pi\sum_{j=1}^4\int_Y^{Y_1^*-\d_0}A(Z)\bar M^{-2}(Z)J_j(Z)dZ+\mathcal B_1^1(Y_1^*-\d_0)
	=\sum_{j=1}^4II_{J_j}+\mathcal B_1^1(Y_1^*-\d_0).
\end{align*}
By \eqref{est: B11B15 far}, we have
\begin{align}\label{eq:B11B15-1}
\begin{split}
&|\mathcal B_1^1(Y_1^*-\d_0)|\lesssim \mathcal{E}_{far}(Y_1^*-\d_0)e^{-2\al w_0(Y_1^*-\d_0)}\\
\lesssim&\Big( \mathcal{E}_{near}(Y^{+}+\d_0)+\al^{-2}e^{-2\al \int_{Y_1^*}^{Y_1-2\d_0}F_r^\f12(Y') dY'}+\al ^{\f16}e^{-\al \int_{Y_1^*-\d_0}^{Y_2^*}F_r^\f12(Y') dY'}\Big)e^{-2\al w_0(Y_1^*-\d_0)}\\
\lesssim& \Big(\al^{-2}e^{-\al w_0(Y_1^*-\d_0)}e^{-2\al \int_{Y_1^*}^{Y_1-2\d_0}F_r^\f12(Y') dY'}+\al ^{\f16}e^{-\al w_0(Y_2^*)}\Big)e^{-\al w_0(Y_1^*-\d_0)}\\
&\qquad+\mathcal{E}_{near}(Y)e^{-2\al w_0(Y)}
\lesssim \mathcal{E}_{near}(Y)e^{-2\al w_0(Y)},
\end{split}
\end{align}
here we used \eqref{est: bdd by Im K_B} in the last line.\smallskip

\underline{Estimates of $II_{J_1}$.}
We write 
\begin{align*}
	\begin{split}
II_{J_1}=&\kappa^{-1}\pi\int_Y^{Y_1^*-\d_0}A(Z)\bar M^{-2}(Z)(Q_1+Q_2)(\mathcal A_{0}^{123}(Z)A(Z)+\mathcal B_{0}^{123}(Z)B(Z))dZ\\
&\quad+\frak{D}_0(\al,c)\int_{Y}^{Y_1^*-\d_0}A(Z)\bar M^{-2}(Z)dZ=II_{J_1}^1+\frak{D}_0(\al,c)II_{J_1}^2.
	\end{split}
\end{align*}
By Proposition \ref{pro: Airy-1}, Lemma \ref{lem:A01B01} , \eqref{est: bdd by Im K_B} and \eqref{est:Q1Q2A02B02}, we obtain 
\begin{align*}
	\begin{split}
		\left|\mathrm{Im}II_{J_1}^1\right|\lesssim &\Big(c_i e^{-2\al w_0(Y)}+\al^{-3}e^{-2\alpha w_0(Y_1^*-\delta_0)}e^{-2\al \int_{Y_1^*}^{Y_1-2\d_0}F_r(Z)dZ}+\al^{-2}e^{-2\al w_0(Y_2^*)}\Big)\\
		\lesssim&  c_ie^{-2\al w_0(Y)}+\alpha^{-\f23}e^{-\alpha w_0(Y)}|\mathrm{Im}(\frak{D}_0(\al,c))|\lesssim \mathcal{E}_{near}(Y)e^{-2\al w_0(Y)}.
	\end{split}
\end{align*}
We also notice that 
\begin{align*}
\left|II_{J_1}^2\right|\lesssim \alpha^{-\f23} e^{-\alpha w_0(Y)},\quad \left|\mathrm{Im}II_{J_1}^2\right|\lesssim \alpha^{-\f23} c_ie^{-\alpha w_0(Y)},
\end{align*}
which implies that  for $Y\in[0,Y^{+}+\d]$,
\begin{align}\label{est:B11B15-J_1}
\begin{split}
\left|\mathrm{Im}II_{J_1}\right|
\lesssim&  c_i+\alpha^{-\f23}e^{-\alpha w_0(Y)}|\mathrm{Im}(\frak{D}_0(\al,c))|+ \mathcal{E}_{near}(Y)e^{-2\al w_0(Y)}
\lesssim  \mathcal{E}_{near}(Y)e^{-2\al w_0(Y)}.
\end{split}
\end{align}

\underline{Estimates of $II_{J_2}$.}
 For $Y\in[0,Y^{+}+\d]$, we notice that 
\begin{align*}
II_{J_2}=-\frak{D}_0(\al,c)\partial_YA(0)\bar M^{-2}(0)C_A(Y),
\end{align*}
where
\begin{align*}
C_A(Y)=\kappa^{-1}\pi\int_Y^{Y_1^*-\d_0}(Q_1+Q_2)A(Z)B(Z)\bar M^{-2}(Z)\bar\chi_1(Z)dZ,
\end{align*}
with $|C_A(Y)|\lesssim\alpha^{-1},~|\mathrm{Im}(C_A(Y))|\lesssim\alpha c_i. $
Then we obtain 
\begin{align}\label{est:B11B15-J_2}
\begin{split}
\left|\mathrm{Im}II_{J_2}\right|\lesssim& \left(c_i+\alpha^{-1}\left|\mathrm{Im}\left(\frak{D}_0(\al,c)\partial_YA(0)\bar M^{-2}(0)\right)\right|\right)\mathbf 1_{[0,Y^-]}
\lesssim \mathcal{E}_{near}(Y)e^{-2\al w_0(Y)}.
\end{split}
\end{align}

\underline{Estimates of $II_{J_3}$.}
By a similar argument as \eqref{est:B11B15-J_1} , we have
\begin{align}\label{est:B11B15-J_3}
\begin{split}
\left|\mathrm{Im}II_{J_3}\right|\lesssim& c_ie^{-2\al w_0(Y)}+\alpha^{-\f32}e^{-\alpha w_0(Y)}|\mathrm{Im}(\frak{D}_0(\al,c))|
\lesssim  \mathcal{E}_{near}(Y)e^{-2\al w_0(Y)}.
\end{split}
\end{align}

\underline{Estimates of $II_{J_4}$.}
 We  write 
\begin{align*}
II_{J_4}
&=\kappa^{-1}\pi\int_Y^{Y_1^*-\d_0}A(Z)\bar M^{-2}(Z)\Big(\pa_Z^2\bar\chi_1-\f{2\pa_Z\bM}{\bM}\pa_Z \bar\chi_1\Big)\tilde P_{0,2}^{(2)}(Z)dZ\\
&+2\kappa^{-1}\pi\int_Y^{Y_1^*-\d_0}A(Z)\bar M^{-2}(Z)\partial_Z\bar\chi_1(Z)\partial_Z\tilde P_{0,2}^{(2)}(Z)dZ=II_{J_4}^1+II_{J_4}^2.
\end{align*}
According to \eqref{eq: tilde P_02^2-1} and \eqref{eq: tilde P_02^2-2}, we use the similar argument in $I_{J_4}$ to deduce
\begin{align*}
\left|\mathrm{Im}II_{J_4}^1\right|
\lesssim&\alpha^{-1}\left|\mathrm{Im}\left(\frak{D}_0(\al,c)\partial_YA(0)\bar M^{-2}(0)\right)\right|\mathbf 1_{[0, Y^{-}]} \\
&\quad+\alpha^{-\f32}e^{-\alpha w_0(Y)}|\mathrm{Im}\left(\frak{D}_0(\al,c)\right)|\lesssim  \mathcal{E}_{near}(Y)e^{-2\al w_0(Y)},
\end{align*}
and 
\begin{align*}
\left|\mathrm{Im}II_{J_4}^2\right|
\lesssim &\left|\mathrm{Im}\left(\frak{D}_0(\al,c)\partial_YA(0)\bar M^{-2}(0)\right)\right|\mathbf 1_{[0, Y^{-}]}\\
&+\alpha^{-\f23}e^{-\alpha w_0(Y)}|\mathrm{Im}\left(\frak{D}_0(\al,c)\right)|\lesssim (\mathcal{E}_{near}(Y)+\tilde{\mathcal{E}}_{near}(Y))e^{-2\al w_0(Y)},
\end{align*}
where we used the estimate
\begin{align*}
|\kappa^{-1}\pi\int_0^Y\bar M^{-2}(Z)\pa_Z\bar \chi_1(Z)A(Z)\pa_ZB(Z)dZ|\leq C.
\end{align*}
We emphasize that $\tilde{\mathcal{E}}_{near}(Y)$ appears in $\mathrm{Im}II_{J_4}^2$. 
Then we obtain
\begin{align}\label{est:B11B15-J_4-1}
\begin{split}
\left|\mathrm{Im}II_{J_4}\right|
\lesssim& (\mathcal{E}_{near}(Y)+\tilde{\mathcal{E}}_{near}(Y))e^{-2\al w_0(Y)}.
\end{split}
\end{align}
In particular, we have
\begin{align*}
	II_{J_4}^2(0)
	=&2\frak{D}_0(\al,c)(\partial_YA(0)\bar M^{-2}(0)C_{P,1}+C_{P,2}),
\end{align*}
where
\begin{align*}
	&C_{P,1}=\kappa^{-1}\pi\int_0^{\bar Y_1^*}A(Z)\bar M^{-2}(Z)\partial_Z\bar\chi_1(Z)\partial_ZB(Z)dZ,\\
	 &C_{P,2}=\kappa^{-1}\pi\int_0^{\bar Y_1^*}A(Z)\bar M^{-2}(Z)\partial_Z\bar\chi_1(Z)\mathcal {R}_{\pa_Y\tilde P_{0,2}^{(2)}}(Z)dZ,
\end{align*}
with $|C_{P,1}|+c_i^{-1}|\mathrm {Im}C_{P,1}|\lesssim 1,\quad |C_{P,2}|+c_i^{-1}|\mathrm {Im}C_{P,2}|\lesssim \al^{-\f23}.
$
Then we  obtain  
\begin{align}\label{est:B11B15-J_4-3}
\begin{split}
	&\left|\mathrm{Im}\left(II_{J_4}(0)-2\frak{D}_0(\al,c)\partial_YA(0)\bar M^{-2}(0)C_{P,1}\right)\right|\lesssim\mathcal{E}_{near}(0).
	\end{split}
\end{align}
Therefore, we gather \eqref{eq:B11B15-1}, \eqref{est:B11B15-J_1}, \eqref{est:B11B15-J_2}, \eqref{est:B11B15-J_3} and \eqref{est:B11B15-J_4-1} to deduce that for $Y\in[0,Y^{+}+\d]$,
\begin{align}\label{est:B11B15 near}
	\begin{split}
	|\mathrm{Im}( \mathcal{B}_1^1(Y)+\mathcal B_1^5(Y))|\lesssim&  (\mathcal{E}_{near}(Y)+\tilde{\mathcal{E}}_{near}(Y))e^{-2\al w_0(Y)}.
		\end{split}
\end{align}
Moreover,  by \eqref{est:B11B15-J_4-3}, we have 
\begin{align*}
	&\left|\mathrm{Im}(\mathcal{B}_1^1(0)+\mathcal B_1^5(0)-\frak{D}_0(\al,c)\partial_YA(0)\bar M^{-2}(0)C_{p,1})\right|\lesssim \mathcal{E}_{near}(0).
\end{align*}
\end{proof}

\begin{lemma}\label{lem: Aj Bj}
Let $(\alpha, c)\in \mathbb H$ and $Y_1^*, Y^*_2$ be given as in Lemma \ref{lem:real-phi(0)}. 
 \begin{enumerate}
 \item Assume that \eqref{assume: induction-2} holds for the j-th ($j\geq 0$) step. Then we have 
 \begin{align*}
 |\Im(\mathcal A_{j+1}^2(Y))|&+e^{2\al w_0(Y)} |\Im(\mathcal B_{j+1}^2(Y))|\lesssim
\left\{
\begin{aligned}
&\al^{-j}\mathcal{E}_{near}(Y),\quad Y\in[0, Y^{+}+\d_0]\\
&\al^{-j}\mathcal{E}_{far}(Y),\quad Y\in[ Y^{+}+\d_0, Y_2^*].
\end{aligned}
\right.
 \end{align*}
 
 \item  For $j\geq 0$, we have
 {\small
 \begin{align*}
& (|\mathcal A_{j+1}^3(Y)|+|\mathcal A_{j+1}^4(Y)|)+e^{2\al w_0(Y)}|(\mathcal B_{j+1}^3(Y)|+|\mathcal B_{j+1}^4(Y)|)\lesssim
\left\{
\begin{aligned}
&\al^{-j}\mathcal{E}_{near}(Y),\quad Y\in[0, Y^{+}+\d_0]\\
&\al^{-j}\mathcal{E}_{far}(Y),\quad Y\in[ Y^{+}+\d_0, Y_2^*].
\end{aligned}
\right.
 \end{align*}
 }
\item  Assume that \eqref{assume: induction-1} -\eqref{assume: induction-2} hold for the j-th ($j\geq 1$) step. Then we have
 \begin{align*}
  |\Im(\mathcal A_{j+1}^1(Y))|&+e^{2\al w_0(Y)} |\Im(\mathcal B_{j+1}^1(Y))|\lesssim
\left\{
\begin{aligned}
&\al^{-j}(\mathcal{E}_{near}(Y)+\tilde{\mathcal{E}}_{near}(Y)),\quad Y\in[0, Y^{+}+\d_0],\\
&\al^{-j}\mathcal{E}_{far}(Y),\quad Y\in[ Y^{+}+\d_0, Y_2^*].
\end{aligned}
\right.
 \end{align*}

 \end{enumerate}

\end{lemma}

\begin{proof}
We first show the results in (1).

\underline{Estimates of $\mathcal{A}_{j+1}^2$}, By the definition, we  write 
\begin{align*}
	\mathcal A_{j+1}^2(Y)=\alpha^2\kappa^{-1}\pi\int_0^YB(Z)\bar M^{-2}(1-\chi)F(Z)\int_Z^{+\infty}\partial_{Y'}\chi_1\tilde P_{j,2}(Y')dY'dZ,
\end{align*}
where
$\tilde P_{j,2}(Y)=\mathcal{A}_j(Y) A(Y)+\mathcal{B}_j(Y) B(Y),$  $j\geq 0$ and $ \tilde P_{0,2}(Y)=\tilde P_{0,2}^{(1)}(Y)$.
By the definitions of $\chi$ and $\chi_1$, we know $\mathcal A_{j+1}^2(Y)=\mathcal A_{j+1}^2(Y_1-\d_0)$ for $Y\in[Y_1-\d_0, Y_2^*]$. For $Y\in[Y_1, Y_2]$, we use  \eqref{assume: induction-2}, Lemma \ref{lem: cE_j}, Proposition \ref{Prop: P01P02} and the fact
\begin{align*}
\al ^{\f16}e^{-\al \int_{\min\{Y, Y_1-\d_0\}}^{Y_2^*}F_r^\f12(Y') dY'}\lesssim\al^{-2}e^{-\al \int_{Y}^{Y_2^*}F_r^\f12(Y') dY'},
\end{align*}
 to obtain
\begin{align*}
&|\tilde P_{j,2}(Y)|\lesssim  \al^{-\f16}e^{-\al w_0(Y)}\left(|\mathcal A_j(Y)|+e^{2\alpha w_0(Y)}|\mathcal B_j(Y)|\right)\lesssim \al^{-j-\f76}e^{-\al w_0(Y)},\\
&|\Im (\tilde P_{j,2}(Y))|\lesssim \al^{-\f16}e^{-\al w_0(Y)}\left(c_i|\mathcal A_j(Y)|+c_ie^{2\alpha w_0(Y)}|\mathcal B_j(Y)|+|\Im(\mathcal A_j(Y))|+e^{2\alpha w_0(Y)}|\Im(\mathcal B_j(Y))|\right)\\
&\qquad\qquad\quad\lesssim e^{-\al w_0(Y)}\al^{-j+\f56}\Big(\mathcal{E}_{near}(Y^{+}+\d_0)+\al^{-2}e^{-2\al \int_{Y_1^*}^{Y_1-2\d_0}F_r^\f12(Y') dY'}\Big)+\al^{-j-\f{7}{6}}e^{-\al w_0(Y_2^*)},
\end{align*}
which show
\begin{align*}
&\int_{Y_1}^{Y_2} |\tilde P_{j,2}(Z)|dZ\lesssim \al^{-j-\f{13}{6}}e^{-\al w_0(Y_1)},\\
&\int_{Y_1}^{Y_2} |\Im \tilde P_{j,2}(Z)|dZ\lesssim e^{-\al w_0(Y_1)}\al^{-j-\f16}\Big(\mathcal{E}_{near}(Y^{+}+\d_0)+\al^{-2}e^{-2\al \int_{Y_1^*}^{Y_1-2\d_0}F_r^\f12(Y') dY'}\Big)\\
&\qquad\qquad\qquad\qquad\qquad+\al^{-j-\f76}e^{-\al w_0(Y_2^*)}.
\end{align*}
For $Y\in[0, Y_1-\d_0]$, by  Lemma \ref{lem:real-phi(0)} and  Lemma \ref{lem: Im-int}, we have
\begin{align}\label{est: Aj2 near}
\begin{split}
&|\Im(\mathcal A_{j+1}^2(Y))|\
\lesssim \al^2\kappa^{-1}\int_0^Y |B(Z)|\int_{Y_1}^{Y_2}\left(c_i|\tilde P_{j,2}(Y')+|\Im\tilde P_{j,2}(Y')|\right)dY' dZ\\
\lesssim&\al^2\kappa^{-1} c_i\int_0^Y e^{\al w_0(Z)}\al^{-j-\f{13}{6}}e^{-\al w_0(Y_1)}dZ+\al^2\kappa^{-1}\al^{-j-\f76}e^{-\al w_0(Y_2^*)}\int_0^Y e^{\al w_0(Z)}dZ\\
&+\al^2\kappa^{-1}\int_0^Y e^{\al w_0(Z)} e^{-\al w_0(Y_1)}\al^{-j-\f16}\Big(\mathcal{E}_{near}(Y^{+}+\d_0)+\al^{-2}e^{-2\al \int_{Y_1^*}^{Y_1-2\d_0}F_r^\f12(Y') dY'}\Big)dZ\\
\lesssim& \al^{-j}c_i+\al^{-j}e^{\al w_0(Y)}\Big(e^{-\al w_0(Y_1)}\al^{\f76}\Big(\mathcal{E}_{near}(Y^{+}+\d_0)+\al^{-2}e^{-2\al \int_{Y_1^*}^{Y_1-2\d_0}F_r^\f12(Y') dY'}\Big)+\al^{\f16}e^{-\al w_0(Y_2^*)}\Big)\\
\lesssim& \al^{-j}(c_i+\al^{-\f23}e^{\alpha w_0(Y)}|\mathrm{Im}(\frak{D}_0(\al,c))|)
\lesssim \al^{-j}\mathcal{E}_{near}(Y),
\end{split}
\end{align}

For $Y\in[Y^{+}+\d, Y_2^*]$, we have
\begin{align}\label{est: Aj2 far}
\begin{split}
|\Im(\mathcal A_{j+1}^2(Y))|\lesssim&\al^{-j}c_i+\al^{-j}\Big(\mathcal{E}_{near}(Y^{+}+\d_0)+\al^{-2}e^{-2\al \int_{Y_1^*}^{Y_1-2\d_0}F_r^\f12(Y') dY'}\\
&\qquad\qquad\qquad+\al^{\f16}e^{-\al \int_{\min\{Y, Y_1-\d_0\}}^{Y_2^*}F_r^\f12(Y') dY'}\Big)
\lesssim \al^{-j} \mathcal{E}_{far}(Y).
\end{split}
\end{align}

\underline{Estimates of $\mathcal{B}_{j+1}^2$}. The process is similar to $\mathcal A_{j+1}^2(Y)$. We have $\mathcal{B}_{j+1}^2(Y)=0$ for $Y\in[Y_1-\delta_0, Y_2^*] $ and for $Y\in[0, Y_1-\d_0]$, we have
 \begin{align*}
e^{2\al w_0(Y)}|\Im(\mathcal B_{j+1}^2(Y))|\lesssim  \al^{-j}\mathcal{E}_{near}(Y).
\end{align*}
Thus, we infer that for $Y\in[Y^{+}+\d, Y_2^*]$
\begin{align}\label{est: Bj2 far}
\begin{split}
|\Im(\mathcal B_{j+1}^2(Y))|
\lesssim\al^{-j} \mathcal{E}_{far}(Y)e^{-2\al w_0(Y)},
\end{split}
\end{align}
and for $Y\in[0, Y^{+}+\d]$, 
\begin{align}\label{est: Bj2 near}
|\Im(\mathcal B_{j+1}^2(Y))|\lesssim& \al^{-j}\mathcal{E}_{near}(Y)e^{-2\al w_0(Y)}.
\end{align}

Collecting \eqref{est: Aj2 near}, \eqref{est: Aj2 far}, \eqref{est: Bj2 far} and \eqref{est: Bj2 near}, we obtain the results in (1).\smallskip

Next we turn to prove the results in (2).\smallskip

\underline{Estimates of $\mathcal{A}_{j+1}^3$}. By Proposition \ref{pro: Airy-3},  \ref{pro: rayleigh-varphi}, \ref{pro: iteration-varphi} and Lemma \ref{lem: cE_j}, we have 
\begin{align}\label{est: Aj3 near}
	\mathcal A_1^3(Y)\equiv 0,\quad Y\in[0, Y^{+}+\d],
\end{align}
and for $Y\in[Y^{+}+\d, Y_2^*]$,
\begin{align}\label{est: Aj3 far}
\begin{split}
|\mathcal{A}_{j+1}^3(Y)|\lesssim& \al^{-\f56}\|\varphi_{j+1}\|_{L^\infty_{w_0}([Y_1^*-\d_0, Y_1^*])}
\lesssim \al^{-j}\al^{-2}e^{-2\al \int_{Y_1^*}^{Y_1-2\d_0}F_r^\f12(Z) dZ}.
\end{split}
\end{align}

\underline{Estimates of $\mathcal{A}_{j+1}^4$}. By the definition, we  write 
\begin{align*}
	\mathcal A_{j+1}^4(Y)=\kappa^{-1}\pi\int_0^YB(Z)\bar M^{-2}(Z)\tilde{H}_{32}[\varphi_{j+1}](Z)dZ,
\end{align*}
where
\begin{align}\label{formula: tilde H_32}
\begin{split}
	&\tilde H_{32}[\varphi_{j+1}](Z)=\Big((1-\bM^2)\chi_3+\al^{-2}(\f{2\pa_Z \bM}{\bM}\pa_Z \chi_3-\pa_Z^2 \chi_3)\Big)\mathcal K_B[\varphi_{j+1}](Z)\\
	&=\left\{
\begin{aligned}
	&0,\quad Z\in[0,\bar Y_1^{**}]\cup[Y_2^*+\d_0,+\infty)\\
	&\Big((1-\bM^2)\chi_3+\al^{-2}(\f{2\pa_Z \bM}{\bM}\pa_Z \chi_3-\pa_Z^2 \chi_3)\Big)\mathcal K_B[\varphi_{j+1}](\bar Y_1^*),\quad Z\in[\bar \bar Y_1^{**},\bar Y_1^*],\\
	&(1-\bM^2)\mathcal K_B[\varphi_{j+1}](Z),\quad Z\in[\bar Y_1^*,Y_1^*]\cup[Y_2^*, Y_2^*+\d_0],\\
	&(1-\bM^2)\mathcal K_B[\varphi_{j+1}](Y_2^*),\quad Z\in[Y_1^*,Y_2^*].
\end{aligned}
\right.
\end{split}
\end{align}
Hence, by  Proposition \ref{pro: Airy-5}, \eqref{est: bdd by Im K_B} and  Lemma \ref{lem: cE_j}, we infer that for
$Y\in[0,Y^{+}+\d]$
\begin{align}\label{est: Aj4 near}
\begin{split}
	|\mathcal A_{j+1}^4(Y)|=0,
	\end{split}
\end{align}
and for  $Y\in[Y^{+}+\d,Y_2^*]$,
\begin{align}\label{est: Aj4 far}
\begin{split}
	|\mathcal A_{j+1}^4(Y)|\lesssim& \alpha^{-j}\Big(\al^{-3}e^{-2\al \int_{Y_1^*}^{Y_1-2\d_0}F_r^\f12(Y') dY'}+\al^{-2} e^{-\al \int_{Y}^{Y_2^*}F_r^\f12(Y') dY'}\Big)\lesssim\al^{-j}\mathcal{E}_{far}(Y).
	\end{split}
\end{align}

\underline{Estimates of $\mathcal{B}_{j+1}^3$}. Notice that 
\begin{align*}
\mathcal B_{j+1}^3(Y)=\mathcal B_{j+1}^3(Y^*_1-\delta_0),\quad Y\in[0,Y^*_1-\delta_0]\text{ and } \mathcal B_{j+1}^3(Y)=0,\quad Y\geq Y_2^*+\d_0.
\end{align*}
By Proposition \ref{pro: Airy-3} and Lemma \ref{lem: cE_j}, we deduce that for  $Y\in[Y^{+}+\d, Y_2^*]$,
\begin{align}\label{est: Bj3 far}
\begin{split}
|\mathcal{B}_{j+1}^3(Y)|\lesssim&\alpha^{-\f{11}{6}}e^{-2\alpha w_0(Y)}\Big(\|\varphi_{j+1}\|_{L^\infty_{w_0}([Y_1^*-\delta_0,Y^*_1])}+e^{-2\al \int_{Y}^{Y_2^*} F_r^\f12(Z)dZ}\|\varphi_{j+1}\|_{L^\infty_{w_0}}\Big)\\
\lesssim&\al^{-j-3}e^{-2\alpha w_0(Y)}\Big(e^{-2\al \int_{Y_1^*}^{Y_1-2\d_0}F_r(Z)dZ}+\al e^{-2\al\int_Y^{Y_2^*}F_r^\f12(Z)dZ }\Big)\\
\lesssim& \al^{-j}\mathcal{E}_{far}(Y),
\end{split}
\end{align}
and for $Y\in[0, Y^{+}+\d]$
\begin{align}\label{est: Bj3 near}
\begin{split}
|\mathcal B_{j+1}^3(Y)|\lesssim& \al^{-j-3}e^{-2\al w_0(Y_1^*-\d_0)}e^{-2\al \int_{Y_1^*}^{Y_1-2\d_0}F_r(Y')dY'}+\al^{-j-2}e^{-2\alpha w_0(Y_2^*)}\\
\ll&\al^{-j}\al^{-\f23}e^{-\alpha w_0(Y)}|\mathrm{Im}\left(\kappa^{-1}\pi\alpha^{-2}\mathcal K_B[\varphi_0](\bar Y_1^*)\right)|\lesssim \al^{-j}\mathcal{E}_{near}(Y)e^{-2\al w_0(Y)},
\end{split}
\end{align}
by using \eqref{est: bdd by Im K_B}.

 \underline{Estimates of $\mathcal{B}_{j+1}^4$}.
We recall $\tilde H_{32}[\varphi_{j+1}](Z)$ in \eqref{formula: tilde H_32} and we know $\mathcal B_{j+1}^4(Y)=\mathcal B_{j+1}^4(\bar Y_1^{**})$ for $Y\in[0, \bar Y_1^{**}]$. By Proposition \ref{pro: Airy-5} and  Lemma \ref{lem: cE_j}, we have for  $Y\in[Y^{+}+\d,Y_2^*]$,
\begin{align}\label{est: Bj4 far}
	|\mathcal B_{j+1}^4(Y)|\lesssim&\alpha^{-j-3}\Big(e^{-2\al \int_{Y_1^*}^{Y_1-2\d_0}F_r^\f12(Y') dY'}+\al e^{-\al \int_{Y}^{Y_2^*}F_r^\f12(Y') dY'}\Big)e^{-2\alpha w_0(Y)}\\
	\nonumber
	\lesssim&\al^{-j}\mathcal{E}_{far}(Y)e^{-2\alpha w_0(Y)},
\end{align}
and for  $Y\in[ 0, Y^{+}+\d]$, we have
\begin{align}\label{est: Bj4 near}
	&|\mathcal B_{j+1}^4(Y)|=|\mathcal B_{j+1}^4(\bar Y_1^{**})|\\
	\nonumber
	\leq& C\alpha^{-3-j}\Big(e^{-\al w_0(Y_1^*-\d_0)}e^{-2\al \int_{Y_1^*}^{Y_1-2\d_0}F_r^\f12(Y') dY'}+\al e^{-\al w_0(Y_2^*)}\Big)e^{-\alpha w_0(\bar Y_1^{**})}\\
	\nonumber
	\ll&\alpha^{-\f23}|e^{-\alpha w_0(Y)}\mathrm{Im}(\frak{D}_0(\al,c))|\lesssim \al^{-j}\mathcal{E}_{near}(Y)e^{-2\al w_0(Y)},
\end{align}
and by a similar argument in \eqref{est: bdd by Im K_B}, we have
\begin{align}\label{est: bdd by Im K_B-2}
\begin{split}
&e^{\al w_0(Y^{+}+\d_0)}e^{-\alpha w_0(\bar Y_1^{**})}\Big(e^{-\al w_0(Y_1^*-\d_0)}e^{-2\al \int_{Y_1^*}^{Y_1-2\d_0}F_r^\f12(Y') dY'}+\al e^{-\al w_0(Y_2^*)}\Big)\\
&\ll e^{-\al w_0(\tilde Y)}\leq |\Im(\frak{D}_0(\al,c))|.
\end{split}
\end{align}

Collecting \eqref{est: Aj3 near}-\eqref{est: Bj4 near} , we obtain the results in (2).\smallskip

Finally, we prove the results in (3). Here $j\geq 1$.\smallskip

\underline{Estimates of $\mathcal{A}_{j+1}^1$}. By Proposition \ref{pro: Airy-1}, we have for $Y\in[0, Y_2^*]$,
\begin{align*}
\left|\Im\left(\mathcal{A}_{j+1}^1(Y)\right)\right|\lesssim \al^{-1}\left(\|\Im \mathcal{A}_{j}\|_{L^\infty([0, Y])}+\|\Im \mathcal{B}_{j}\|_{L^\infty_{2w_0}([0, Y])}\right).
\end{align*}
Hence, we use  \eqref{assume: induction-1} and \eqref{assume: induction-2} to deduce that for $Y\in[0, Y^{+}+\d]$,
\begin{align}\label{est: Aj1 near}
\left|\Im\left(\mathcal{A}_{j+1}^1(Y)\right)\right|\lesssim \al^{-j}\mathcal{E}_{near}(Y),
\end{align}
and for $Y\in[Y^{+}+\d, Y_2^*]$,
\begin{align}\label{est: Aj1 far}
|\Im(\mathcal{A}_{j+1}^1(Y))|\lesssim& \al^{-j}\Big(\mathcal{E}_{near}(Y^{+}+\d_0)+\al^{-2}e^{-2\al \int_{Y_1^*}^{Y_1-2\d_0}F_r^\f12(Y') dY'}\\
\nonumber
&\qquad\qquad+\al^{-2}e^{-\al \int_{Y}^{Y_2^*}F_r^\f12(Y') dY'}+\al ^{\f16}e^{-\al \int_{\min\{Y, Y_1-\d_0\}}^{Y_2^*}F_r^\f12(Y') dY'}\Big)\\
\nonumber
\lesssim&\al^{-j}\mathcal{E}_{far}(Y).
\end{align}

\underline{Estimates of $\mathcal{B}_{j+1}^1$}. By definition, we have
\begin{align*}
\mathcal{B}_{j+1}^1(Y)=\kappa^{-1}\pi\int_Y^\infty A(Z)\bM^{-2}(1-\chi)(Q_1+Q_2)\tilde P_{j, 2}(Z)dZ,
\end{align*}
where $\tilde P_{j, 2}=\mathcal{A}_{j} A+\mathcal{B}_{j} B$. 
For $Y\in [Y^{+}+\d, Y_2^*]$, we write
\begin{align*}
\mathcal{B}_{j+1}^1(Y)=\mathcal{B}_{j+1}^1(Y_2^*)+\kappa^{-1}\pi\int_Y^{Y_2^*} A(Z)\bM^{-2}(1-\chi)(Q_1+Q_2)\tilde P_{j, 2}(Z)dZ=\mathcal{B}_{j+1}^1(Y_2^*)+I_1.
\end{align*}
By Proposition \ref{pro: Airy-1} and Lemma \ref{lem: cE_j}, we have
\begin{align*}
|\mathcal{B}_{j+1}^1(Y_2^*)|\lesssim& \al^{-2}e^{-2\al w_0(Y_2^*)}(\|\mathcal{A}_{j}\|_{L^\infty}+\|\mathcal{B}_{j}\|_{L^\infty_{2w_0}})\\
\lesssim& \al^{-2-j}e^{-2\al w_0(Y_2^*)}\lesssim\al^{-j}e^{-2\al w_0(Y)}\al^{-2}e^{-2\al \int_Y^{Y_2^*}F_r^\f12(Y')dY'}\\
\lesssim& \al^{-j}\mathcal{E}_{far}(Y)e^{-2\al w_0(Y)}.
\end{align*}
By \eqref{assume: induction-2}, we have
\begin{align*}
|\Im(\tilde P_{j, 2}(Z))|\lesssim& \al^{-j+1}\al^{-\f16}\Big(e^{-\al w_0(Y)}\Big(\mathcal{E}_{far}(Y)-\al^{-2}e^{-\al \int_{Y}^{Y_2^*}F_r^\f12(Y') dY'}\Big)+\al^{-2} e^{-\al w_0(Y_2^*)}\Big)\\
&+\al^{-j-\f16} c_i,
\end{align*}
which implies
\begin{align*}
e^{2w_0(Y)}\left|\Im I_1\right|\lesssim& \al^{-j}c_i +\al^{-j-1}\mathcal{E}_{far}(Y)
\lesssim\al^{-j}\mathcal{E}_{far}(Y)e^{-2\al w_0(Y)}.
\end{align*}
Thus, we infer that for  $Y\in [Y^{+}+\d, Y_2^*]$
\begin{align}\label{est: Bj1 far}
|\Im(\mathcal{B}_{j+1}^1(Y))|\lesssim&\al^{-j}\mathcal{E}_{far}(Y)e^{-2\al w_0(Y)}.
\end{align}

For $Y\in[0, Y^{+}+\d]$, we write
\begin{align*}
\mathcal{B}_{j+1}^1(Y)=&\mathcal{B}_{j+1}^1(Y_1^*-\d_0)+\kappa^{-1}\pi\int_Y^{Y_1^*-\d_0} A(Z)\bM^{-2}(1-\chi)(Q_1+Q_2)\tilde P_{j, 2}(Z)dZ\\
=&\mathcal{B}_{j+1}^1(Y_1^*-\d_0)+I_2.
\end{align*}
By  taking $Y=Y_1^*-\d_0$ in \eqref{est: Bj1 far}, we get
\begin{align*}
|\Im(\mathcal{B}_{j+1}^1(Y_1^*-\d_0)|\lesssim& \al^{-j}\mathcal{E}_{far}(Y_1^*-\d_0)e^{-2\al w_0(Y_1^*-\d_0)}
\lesssim \al^{-j}\mathcal{E}_{near}(Y)e^{-2\al w_0(Y)},
\end{align*}
where we used \eqref{est: bdd by Im K_B-2} in the last line.

For $I_2$, we shall use \eqref{assume: induction-1} and \eqref{assume: induction-2} for the $j^{th}$ step. Indeed, we get by Proposition \ref{pro: Airy-1}  that
\begin{align*}
|\Im I_2|\lesssim& c_i\al(\|\mathcal{A}_j\|_{L^\infty[0, Y_1^*-\d_0]}+\al^{-1}\|\mathcal{A}_j\|_{L^\infty}+\|\mathcal{B}_j\|_{L^\infty_{2w_0}})e^{-2\al w_0(Y)}\\
&+ \kappa^{-1}\int_{Y^{+}+\d}^{Y_1^*-\d_0}|Ai(\kappa\eta(Z))|\left(|Ai(\kappa\eta(Z))|+|e^{-2\al w_0(Z)}Bi(\kappa\eta(Z))|\right)\al^{-j+1}\mathcal{E}_{far}(Z)dZ\\
&+ \kappa^{-1}\int_{Y}^{Y^{+}+\d}|Ai(\kappa\eta(Z))|\left(|Ai(\kappa\eta(Z))|+|e^{-2\al w_0(Z)}Bi(\kappa\eta(Z))|\right)\al^{-j+1}\mathcal{E}_{near}(Z)dZ\\
\lesssim&c_i\al^{-j}e^{-2\al w_0(Y)}+\al^{-j}\mathcal{E}_{near}(Y)e^{-2\al w_0(Y)}+\al^{-j-1} e^{-\al w_0(Y^{+}+\d)}e^{-\al w_0(Y_2^*)}\\
\lesssim&\al^{-j}\mathcal{E}_{near}(Y)e^{-2\al w_0(Y)},
\end{align*}
 which shows that for $Y\in[0, Y^{+}+\d]$,
 \begin{align}\label{est: Bj1 near}
 |\Im (\mathcal{B}_{j+1}^1(Y))|\lesssim&\al^{-j}\mathcal{E}_{near}(Y)e^{-2\al w_0(Y)}.
 \end{align}
 
 Collecting  \eqref{est: Aj1 near}-\eqref{est: Bj1 near} together, we get the results in (3).
 \end{proof}

We are in a  position to prove Proposition \ref{pro: BC B_j(0)}. 

\begin{proof}
By (2) in Proposition \ref{Prop: P01P02} and \eqref{est: bdd by Im K_B-2}, we have for $Y\in[Y^{+}+\d, Y_2^*]$,
\begin{align*}
&|\mathrm{Im}(\mathcal A_0(Y))|+e^{2\alpha w_0(Y)}|\mathrm{Im}(\mathcal B_0(Y))|\lesssim \al \mathcal{E}_{far}(Y),
\end{align*}
which deduces that \eqref{assume: induction-2} is valid for $j=0$. Then we use (1) in Lemma \ref{lem: Aj Bj} to get for $j\geq 1$,
\begin{align*}
 |\Im(\mathcal A_{j}^2(Y))|&+e^{2\al w_0(Y)} |\Im(\mathcal B_{j}^2(Y))|\lesssim
\left\{
\begin{aligned}
&\al^{-j+1}\mathcal{E}_{near}(Y),\quad Y\in[0, Y^{+}+\d],\\
&\al^{-j+1}\mathcal{E}_{far}(Y),\quad Y\in[ Y^{+}+\d, Y_2^*].
\end{aligned}
\right.
 \end{align*}
 By (2) in Lemma \ref{lem: Aj Bj}, we  get for $j\geq 1$,
  \begin{align*}
& (|\mathcal A_{j}^3(Y)|+|\mathcal A_{j}^4(Y)|)+e^{2\al w_0(Y)}|(\mathcal B_{j}^3(Y)|+|\mathcal B_{j}^4(Y)|)\lesssim
\left\{
\begin{aligned}
&\al^{-j+1}\mathcal{E}_{near}(Y),\quad Y\in[0, Y^{+}+\d],\\
&\al^{-j+1}\mathcal{E}_{far}(Y),\quad Y\in[ Y^{+}+\d, Y_2^*].
\end{aligned}
\right.
 \end{align*}
 Afterwards, we use Lemma \ref{lem: A11A15}, Lemma \ref{lem: Aj Bj} and the above results to obtain \eqref{assume: induction-1} and \eqref{assume: induction-2}.
 Taking $j\geq 2$ in \eqref{assume: induction-1}, we have
 \begin{align*}
 \Big|\Im\Big(\sum_{j=2}^{+\infty}\mathcal{B}_j(0)\Big)\Big|\lesssim \mathcal{E}_{near}(0)\sum_{j=2}^{+\infty} \al^{-j+1}\lesssim\al^{-1}(\mathcal{E}_{near}(0)+\tilde{\mathcal{E}}_{near}(0)) \lesssim \mathcal{E}_{near}(0),
 \end{align*}
which along with (3) in Lemma \ref{lem: A11A15} gives \eqref{est: BC B(0)}. 
\end{proof} 

\subsection{Proof of Proposition \ref{pro: expan-dispersion}.} With estimates obtained in the previous two subsections, we are in a position to prove Proposition \ref{pro: expan-dispersion}.
\begin{proof}
We write $\cB_0(0)=\sum_{k=1}^4\cB_0^k(0)$ and recall $\frak{D}_0(\al,c)=\kappa^{-1}\al^{-2}\pi \mathcal K_B[\varphi_0](\bar Y_1^*)$. 
By Lemma \ref{lem: estimate B_0^1(0)}-\ref{lem:A04B04}, for $(\alpha,c)\in\mathbb H$, we have 
\begin{align*}
&|\cB_0^1(0)|\lesssim \al^{-1}+i \al c_i,\quad |\cB_0^2(0)|\ll \al^{-1},\quad |\Im(\cB_0^2(0))|\ll\al c_i,\\
& |\cB_0^3(0)|\ll \al^{-\f23}|\Im \frak{D}_0(\al,c)|,\quad \mathcal B_0^4(0)=\frak{D}_0(\al,c)\mathcal{R}_{\mathcal{B}_0^4}(0),
\end{align*}
where we have used Lemma \ref{lem:real-phi(0)}, Lemma \ref{lem: Im-int} and \eqref{est: bdd by Im K_B} for $\cB_0^3(0)$. Thus, we get \eqref{expansion: B_0(0)} with bounds in \eqref{est: d_0^1-1} and \eqref{est: d_0^2-d_1^4}.

By Lemma \ref{lem: tildeA0B0}, we have
\begin{align*}
\tilde{\mathcal A}_0(0)=-\frak{D}_0(\al,c)\Big(-\partial_YB(0)\bar M(0)^{-2}+\mathcal{R}_{\tilde{\mathcal{A}_0}}(0)\Big),
\end{align*}
which gives \eqref{expansion: tA_0(0)}.
By Lemma \ref{lem: cE_j}, we have for $j\geq 1$, $|\cB_j(0)|\leq C\al^{-j-1}$, which gives $\sum_{j=1}^{+\infty}|\mathcal{B}_j(0)|\leq C\al^{-2}.$
 By Proposition \ref{pro: BC B_j(0)}, we have
\begin{align*}
&\Big|\mathrm{Im}\Big(\Big(\sum_{j=1}^{+\infty}\mathcal{B}_j(0)\Big)-\frak{D}_0(\al,c)\partial_YA(0)\bar M^{-2}(0)C_{P,1}\Big)\Big|\\
&\lesssim c_i+\al^{-\f23}
	\left|\Im \frak{D}_0(\al,c)\right|+\alpha^{-1}\left|\mathrm{Im}\left(\frak{D}_0(\al,c)\partial_YA(0)\bar M^{-2}(0)\right)\right|,
\end{align*}
which gives \eqref{expansion: tB_1(0)}. The above estimates show the results in (1).

In particular, for any $c_r\in\mathcal{H}(U_B, M_a)\cap(1-M_a^{-1},1)$, we use (2) in Lemma \ref{lem: estimate B_0^1(0)} to get \eqref{expansion: B_0(0)} with bound in \eqref{est: d_0^1-2} instead of \eqref{est: d_0^1-1}. This gives the result in (2).
\end{proof}

\section{The dispersion relation}

In Section 6, we have obtained the asymptotic expansion of boundary value in Proposition \ref{pro: expan-dispersion}. Then for any $(\alpha,c)\in\mathbb H$ with $c_r\in\mathcal{H}(U_B, M_a)\cap(1-M_a^{-1},1)$, according to Proposition \ref{pro: expan-dispersion}, the boundary condition $\pa_Y P(0)=0$ is equivalent to the following  dispersion relation
\begin{align}\label{expansion: pa_YP(0)-1}
\begin{split}
 & \pa_Y A(0)\left(1+\frak{D}_0(\al,c)d_0(\al, c)\right)\\
 &+\pa_Y B(0)\left(d_1(\al, c)~\al^{-1}+i d_2(\al,c)~\al c_i+\frak{D}_0(\al,c)\frak{D}_1(\al,c)\right)=0,
 \end{split}
\end{align}
where $\frak{D}_1(\al,c)=\pa_Y A(0)\bM^{-2}(0)(1+C_{P,1}+d_3(\al,c))+d_4(\al,c)$ and
\begin{align}\label{est: d_0-d_4}
\begin{split}
&|d_j(\al,c)|\sim 1+\al^{-1},\quad |\pa_{c_i}d_j(\al,c)|\lesssim1+\al^{-1},\quad j=1,2,\\
&|d_j(\al,c)|\lesssim \al^{-\f23},\quad |\pa_{c_i}d_j(\al,c)|\lesssim \al^{-\f23},\quad j=0, 3,4.
\end{split}
\end{align}
Here we take a suitable cut-off function $\bar\chi_1$ such that $|1+C_{P,1}|\neq 0$, where $C_{P,1}$ is given in  \eqref{est: d_0^2-d_1^4}.

The main result of the pressure system \eqref{eq:ray-P1} is presented as follows.

\begin{theorem}\label{thm: solve dispersion relation}
Assume that $U_B(Y)$ satisfies the structure assumption \eqref{assume: U_B} and $\mathcal{H}(U_B, M_a)\cap(1-M_a^{-1},1)\neq \emptyset$ for some Mach number $M_a>1$. Then for any $c_r\in\mathcal{H}(U_B, M_a)\cap(1-M_a^{-1},1)$,
 there exists  a sequence $\{(\al_k, c_{i,k})\}_{k=1}^{\infty}$ with $c_{i,k}=c_i(\al_k)$ such that the system \eqref{eq:ray-P1} contains a sequence of non-trivial solutions $P_k(Y)$ with $e^{\alpha_k w_0(Y)}P(Y)\in W^{2,\infty}(\mathbb R_+)$ for the pair  $(\alpha_k, c_k)$ with $c_k=c_r+\mathrm ic_{i,k}$. Moreover, we have $(\al_k, c_k)\in \mathbb{H}$, and there exists $\al_0\gg1$ such that for any $k\in \mathbb{N}_+$
\begin{align}\label{bound: c_i}
\al_k\geq \al_0,\quad \f32 \pi<C_1(c_r)(\al_{k+1}-\al_k)<\f52 \pi,\quad 0<c_{i,k}\sim \al_k^{-3}e^{-\al_k w_0(\mathring Y(c_r))},
\end{align}
where $C_1(c_r)$ and $\mathring Y=\mathring Y(c_r)$  are  positive constants dependent on $c_r$, but independent of $\alpha_k$.  Here $\mathring Y$ is given in Lemma \ref{lem: Im-int} 
\end{theorem}

We shall illustrate that $\mathcal{H}(U_B, M_a)\cap(1-M_a^{-1},1)$ is not an empty set.
\begin{proposition}\label{rem: Non-empty condition}
$\mathcal{H}(U_B, M_a)\cap(1-M_a^{-1},1)\neq \emptyset$ holds when $U_B(Y)$ is an analytic function. More precisely,

\begin{enumerate}
\item if $U_B(Y)$ is an analytic function, then there exists $M_0>1$ such that for  any $M_a\geq M_0$, it holds that $\mathcal{H}(U_B, M_a)\cap(1-M_a^{-1},1)\neq \emptyset$.

\item Furthermore, if $U_B(Y)$ is the Blasius profile, then for any $M_a\geq 3$, it holds that $\mathcal{H}(U_B, M_a)\cap(1-M_a^{-1},1)\neq \emptyset$.
\end{enumerate}
Moreover,  there exist $C_1(M_a),\cdots, C_N(M_a)$ with $1-M_a^{-1}<C_1<\cdots<C_N<1$, such that 
\begin{align*}
(C_1, C_2)\cup (C_3, C_4)\cup \cdots (C_{N-1}, C_N)\subset \mathcal{H}(U_B, M_a)\cap(1-M_a^{-1},1).
\end{align*}

\end{proposition}

\subsection{Proof of Theorem \ref{thm: solve dispersion relation}. } 
Before giving the proof of Theorem \ref{thm: solve dispersion relation}, we solve the dispersion relation \eqref{expansion: pa_YP(0)-1} under the estimates in \eqref{est: d_0-d_4} firstly.
  
The following lemma is the equivalence of dispersion relation \eqref{expansion: pa_YP(0)-1}.
\begin{lemma}\label{lem: (c_r, c_i)}
Let $(\al,c)\in\mathbb{H}$ and $c_r\in \mathcal{H}(U_B, M_a)\cap(1-M_a^{-1},1)\neq \emptyset$. Then the dispersion relation \eqref{expansion: pa_YP(0)-1} is equivalent to 
\begin{align}
&c_r=\f{D_3(\al, c)}{ D_1(\al,c)\kappa^{\f32} \tan \Theta(-\kappa\eta(0))+D_2(\al,c)},\label{formula: c_r-1}\\
&c_i=\al^{-1}c_r\Im(\frak{D}_0(\al,c))\Re((1+C_{P,1})\pa_Y A(0))(D_4+\mathcal{R}(\al,c)),\label{formula: c_i-1}
\end{align}
where $\kappa^3=\alpha^2\partial_Y\tF(Y_0)=\alpha^2\partial_YF_r(Y_0)$ and for $j=1,2,3$,
\begin{align}\label{est: D_j}
\begin{split}
&D_j(\al,c)\sim 1+\al^{-1}+\al c_i^2+\Re (\frak{D}_0(\al,c))+c_i \Im (\frak{D}_0(\al,c)),\\
&|\pa_{c_i}D_j(\al,c)|\lesssim \al c_i+(1+\al^2 |\log c_i|)\Re (\frak{D}_0(\al,c))+\Im (\frak{D}_0(\al,c))\lesssim e^{-\al w_0(Y_1^*-\d_0)},
\end{split}
\end{align}
and $D_4$ is a constant independent of $\al$ and $c$, and 
\begin{align*}
&\mathcal{R}(\al,c)\sim \al^{-1}+\al^2 c_i^2+\Re (\frak{D}_0(\al,c))+c_i \Im (\frak{D}_0(\al,c)),\\
&|\pa_{c_i}\mathcal{R}(\al,c)|\lesssim \al^2c_i+(1+\al^2 |\log c_i|)\Re (\frak{D}_0(\al,c))+\Im (\frak{D}_0(\al,c))\lesssim e^{-\al w_0(Y_1^*-\d_0)}.
\end{align*}

\end{lemma}

\begin{proof}
According to Lemma \ref{lem: Ai, Bi}, a direct calculation gives
\begin{align}\label{expansion: pa_YA(0)}
\begin{split}
\pa_Y A(0)=&\pa_Y E(0)Ai(\kappa \eta(0))+ E(0)Ai'(\kappa \eta(0))\\
=&\Big(e_1E(0)\kappa^\f54\sin\Theta(-\kappa\eta(0)))\\
&\qquad+\kappa^{-\f14}\left(e_2\pa_YE(0)+e_3E(0))\cos\Theta(-\kappa\eta(0))\right)\Big)(1+O(\kappa^{-3})),
\end{split}
\end{align}
and
\begin{align}\label{expansion: pa_YB(0)}
\begin{split}
\pa_Y B(0)=&\pa_Y E(0)Ai(\kappa \eta(0))+ E(0)Ai'(\kappa \eta(0))\\
=&\Big(\tilde e_1E(0)\kappa^\f54\cos\Theta(-\kappa\eta(0)))\\
&\qquad+\kappa^{-\f14}(\tilde e_2\pa_YE(0)+\tilde e_3E(0))\sin\Theta(-\kappa\eta(0)))\Big)(1+O(\kappa^{-3})).
\end{split}
\end{align}
Here $E(0)=-\f{M_a c}{T_0(0)^\f12\pa_Y\eta(0)^\f12}$ and $e_j~\tilde e_j, j=1,2.3$ are non-zero constants independent of $\al$ and $c$.
By Lemma \ref{lem: int-phi^(-2)},  Lemma \ref{lem: Im-int} and $\frak{D}_0(\al,c)=\kappa^{-1}\al^{-2}\pi \mathcal K_B[\varphi_0](\bar Y_1^*)$, we have
\begin{align*}
&|\frak{D}_0(\al,c)|\lesssim \al^{-\f {17}6}e^{-\al w_0(Y_1^*-\d_0)},\quad |\Im\frak{D}_0(\al,c)|\sim \al ^{-\f{11}6}e^{-\al w_0(\mathring{Y})},\\
& |\pa_{c_i}\frak{D}_0(\al,c)|\leq C\al^2 |\log c_i||\frak{D}_0(\al,c)|,\quad |\frak{D}_1(\al,c)|+c_i^{-1}|\Im \frak{D}_1(\al,c)|\lesssim |\Re(\pa_YA(0))|+\al^{-\f23}.
\end{align*}
By taking the real part of \eqref{expansion: pa_YP(0)-1}, we use \eqref{expansion: pa_YA(0)}-\eqref{expansion: pa_YB(0)} and the above estimates to get
\begin{align*}
&D_1(\al,c)c_r\kappa^\f54\sin\Theta(-\kappa\eta(0)))+\kappa^{-\f14}(D_2(\al, c)c_r+ D_3(\al,c))\cos\Theta(-\kappa\eta(0))= 0,
\end{align*}
which gives 
\begin{align}
&c_r=\f{D_3(\al, c)}{ D_1(\al,c)\kappa^{\f32}\tan \Theta(-\kappa\eta(0))+D_2(\al,c)},\label{formula: c_r}
\end{align}
where for $j=1,2,3$,
\begin{align}\label{est: D_j}
\begin{split}
&D_j(\al,c)\sim 1+\al^{-1}+\al c_i^2+\Re (\frak{D}_0(\al,c))+c_i \Im (\frak{D}_0(\al,c)),\\
&|\pa_{c_i}D_j(\al,c)|\lesssim \al c_i+(1+\al^2 \log c_i)\Re (\frak{D}_0(\al,c))+\Im (\frak{D}_0(\al,c))\lesssim e^{-\al w_0(Y_1^*-\d_0)}.
\end{split}
\end{align}

Similarly, we use the above process to take the imaginary part of \eqref{expansion: pa_YP(0)-1} to derive
\begin{align}\label{formula: c_i}
c_i=\al^{-1}c_r\Im(\frak{D}_0(\al,c))\Re((1+C_{P,1})\pa_Y A(0))(D_4+\mathcal{R}(\al,c)),
\end{align}
where $D_4$ is a constant independent of $\al$ and $c$, and 
\begin{align*}
&\mathcal{R}(\al,c)\sim \al^{-1}+\al^2 c_i^2+\Re (\frak{D}_0(\al,c))+c_i \Im (\frak{D}_0(\al,c)),\\
&|\pa_{c_i}\mathcal{R}(\al,c)|\lesssim \al^2c_i+(1+\al^2 \log c_i)\Re (\frak{D}_0(\al,c))+\Im (\frak{D}_0(\al,c))\lesssim e^{-\al w_0(Y_1^*-\d_0)}.
\end{align*}
From the above process, the dispersion relation \eqref{expansion: pa_YP(0)-1} is equivalent to \eqref{formula: c_r} and \eqref{formula: c_i}. 
\end{proof}

Now we are in a position to prove Theorem \ref{thm: solve dispersion relation}.
\begin{proof}

We first prove the existence of sequence $\{(\al_k,c_{i,k})\}_{k=1}^{+\infty}$ to the dispersion \eqref{expansion: pa_YP(0)-1}. For any $\al\geq \al_0$ and a fixed $c_r\in\mathcal{H}(U_B, M_a)\cap(1-M_a^{-1},1)$, we define
\begin{align*}
c_i^{(j)}=\al^{-\f{17}6}c_r e^{-\al w_0(\tilde Y)}\mathrm{sign}(\Im(\frak{D}_0(\al,c_r+i c_i^{(0)})))\Re((1+C_{P,1})\pa_Y A(0))(D_4+\mathcal{R}(\al, c_r+i c_i^{(j-1)})),
\end{align*}
with $c_i^{(0)}=\al^{-3}e^{-\al w_0(\mathring Y) },$
by using Lemma \ref{lem: Im-int} to ensure $|\Im (\frak{D}_0(\al,c))|\sim \al^{-\f{11}6}e^{-\al w_0(\mathring Y)}(1+\tilde{\mathcal{R}}(\al,c))$ and putting $\tilde{\mathcal{R}}$ into $\mathcal{R}$ for convenience. Since  $|\pa_{c_i}\mathcal{R}(\al,c)|\leq C\al^2 |\log c_i|e^{-\al w_0(Y_1^*-\d_0)}+Cc_i\leq C\al^{-1}$, we have
\begin{align*}
|c_i^{(j+1)}-c_i^{(j)}|\leq&\al^{-\f{17}{6}}c_r|\Re((1+C_{P,1})\pa_Y A(0))||\mathcal{R}(\al, c_r+i c_i^{(j)})-\mathcal{R}(\al, c_r+i c_i^{(j-1)})|\\
\leq&C\al^{-\f{17}{6}-1}|c_i^{(j)}-c_i^{(j-1)}|\leq \f12 |c_i^{(j)}-c_i^{(j-1)}|,
\end{align*}
by taking $\al\geq \al_0$ large. Then we define $c_i=\lim_{j\to \infty}c_i^{(j)}$, which solves \eqref{formula: c_i-1}. Thus, we view $c_i$ a function of $\al$ and $c_r$. Since $c_r$ is a given number, we write $c_i=c_i(\al)$.

With $c_i=c_i(\al)$ solved, we write $D_j(\al,c)=D_j(\al),~j=1,2.3$, which is a function of  $\al$, with bound $|D_j(\al)|\sim 1,~j=1,2,3$ .  We define 
\begin{align*}
\mathcal{F}(\al)=c_r-\f{D_3(\al)}{ D_1(\al)\kappa^{\f32}\tan \Theta(-\kappa\eta(0))+D_2(\al)}=c_r-\f{D_3(1+O(\al^{-1}))}{ D_1\kappa^{\f32}\tan \Theta(-\kappa\eta(0))+D_2},
\end{align*}
where $D_j,~j=1,2,3$ are constants independent of $\al$ with $|D_j(\al)-D_j|\lesssim\al^{-1}$. To solve
\eqref{formula: c_r-1} is equivalent to finding a zero point of $\mathcal{F}(\al)$. We choose $C_1$, $C_2$  and take $\al\geq \al_0$ large  to hold  
\begin{align*} 
\f{D_3(1+O(\al^{-1}))}{ D_1C_1+D_2}<c_r,\quad \f{D_3(1+O(\al^{-1}))}{ D_1C_2+D_2}>c_r.
\end{align*}
For the equation $\kappa^\f32\tan \Theta(-\kappa\eta(0))=C_j,~j=1,2$, there exists a sequence $\{\kappa_{j,k}\}_{k=1}^{\infty},~j=1,2$ with $\kappa_{j,k}^{\f32}=\kappa(\al_{j,k})=\al_{j,k} (\partial_YF_r(Y_0))^\f12\gtrsim \al_0$ such that 
\begin{align*}
\kappa_{1,k}^\f32\tan \Theta(-\kappa_{1,k}\eta(0))=C_1,\quad \kappa_{2,k}^\f32\tan \Theta(-\kappa_{2,k}\eta(0))=C_2,
\end{align*}
which is equivalent to 
\begin{align*}
&\al_{1,k}\tan \Theta\left(\al_{1,k}w(Y_0)-\f{\pi}{4}\right)=C_1(\pa_Y F_r(Y_0))^{-\f12},\quad\al_{2,k}\tan \Theta\left(\al_{2,k}w(Y_0)-\f{\pi}{4}\right)=C_2(\pa_Y F_r(Y_0))^{-\f12},
\end{align*}
where we denote $w(Y_0)=\int_0^{Y_0} (-F_r(Z))^\f12dZ>0$ and  for any $k\in\mathbb{N}_{+}$,
\begin{align*}
w(Y_0)|\al_{1,k}-\al_{2,k}|\leq \pi,\quad \f{\pi}{2}< w(Y_0)(\al_{j,k+1}-\al_{j,k})<\f32\pi,\quad j=1,2.
\end{align*}
As a result, we deduce 
\begin{align*}
\mathcal{F}(\al_{1,k})>0,\quad \mathcal{F}(\al_{2,k})<0,\quad \kappa_{j,k}=\al_{j,k}^{\f23}(\pa_YF_r(Y_0))^\f13,\quad j=1,2.
\end{align*}
As the function $\mathcal{F}(\al)$ is continuous on $\al$, we find $\al _{k}\in(\al_{1,k}, \al_{2,k})$ ( or $\al_{k}\in(\al_{2,k}, \al_{1,k})$) such that $\mathcal{F}(\al_k)=0$. Meanwhile, we obtain a sequence $\{c_{i,k}\}_{k=1}^{+\infty}$ with $c_{i,k}=c_i(\al_k)$.
In conclusion,  given $c_r\in \mathcal{H}(U_B, M_a)\cap(1-M_a^{-1},1)$, we find a sequence $\{(\al_k, c_{i,k})\}_{k=1}^{+\infty}$, which solves \eqref{formula: c_r-1} and \eqref{formula: c_i-1}. Then by Lemma \ref{lem: (c_r, c_i)}, we obtain $\{(\al_k, c_{i,k})\}_{k=1}^{+\infty}$ solves \eqref{expansion: pa_YP(0)-1}.

Next, we  show  $(\al_k,c_k )\in \mathbb{H}$ and \eqref{bound: c_i} with $c_k=c(\al_k)=c_r+c_{i,k}$. According to  \eqref{formula: c_r} and \eqref{est: D_j}, we write 
\begin{align*}
\kappa^\f32\tan \Theta(-\kappa\eta(0))=\f{D_3(\al,c)}{c_r D_1(\al,c)}-\f{D_2(\al,c)}{D_1(\al,c)}=\f{D_3}{c_r D_1}-\f{D_2}{D_1}+O(\al^{-1})\in(C_1, C_2),
\end{align*}
by taking $\al$ large. Then we can get a sequence of $\{\kappa_k\}$ with   $\kappa_k=\al_k^{\f23}(\pa_YF_r(Y_0))^\f13$  such that $\kappa_k^\f32\tan \Theta(-\kappa_k\eta(0))=\bar C\in (C_1, C_2)$ and $\f{\pi}{2}< w(Y_0)(\al_{j,k+1}-\al_{j,k})<\f32\pi$. Moreover, we have $|\sin \Theta(-\kappa_k\eta(0))|\sim \kappa_{k}^{-\f32}$ and $|\cos \Theta(-\kappa_k\eta(0))|\sim 1$.  By using \eqref{formula: c_i} and  denoting $\frak{D}_0(\al_k)=\frak{D}_0(\al_k,c_r+c_i(\al_k)),~\mathcal{R}(\al_k)=\mathcal{R}(\al_k,c_r+c_i(\al_k))$, we have
\begin{align*}
c_{i,k}=c_i(\al_k)=&\al_k^{-1}c_r\Im(\frak{D}_0(\al_k)\Re((1+C_{P,1})\pa_Y A(0, \al_k))(D_4+\mathcal{R}(\al_k))\\
=&D_4\al_k^{-1}c_r\Im(\frak{D}_0(\al_k)\Re((1+C_{P,1}))\Re(\pa_Y A(0,\al_k))\\
=&D_4 \al_k^{-2}c_r\Im(\frak{D}_0(\al_k)\Re((1+C_{P,1}))\Re(\pa_Y B(0,\al_k))\\
=&  D_4\al_k^{-\f{7}{6}}\Im(\frak{D}_0(\al_k)\Re((1+C_{P,1}))\cos \Theta(-\kappa_k\eta(0))\neq 0.
\end{align*}
Since the different periods of functions $\cos \Theta(-\kappa_k\eta(0))$,  and $\tan \Theta(-\kappa_k\eta(0))$,
we derive
\begin{align*}
\mathrm{Sign}\Big(\f{c_{i}(\al_k)}{c_i(\al_{k+1})}\Big)=\mathrm{Sign}\Big(\f{\cos \Theta(-\kappa_k\eta(0))}{\cos \Theta(-\kappa_{k+1}\eta(0))}\Big)=-1.
\end{align*}
Thus, we can choose a subsequence $\{\kappa_{n_k}\}_{k=1}^{\infty}$ (it must be the odd terms or even terms of $\{\kappa_k\}_{k=1}^{\infty}$) from sequence $\{\kappa_k\}_{k=1}^{\infty}$ such that
\begin{align*}
\mathrm{sign}(\cos \Theta(-\kappa_{n_k}\eta(0)))=\mathrm{sign}(D_4)\Re((1+C_{P,1}))\Im(\frak{D}_0(\al_{n_k})).
\end{align*}
In the following, we still use notation $\{\kappa_k\}_{k=1}^{\infty} $ for convenience with distance 
\begin{align*}
\f32 \pi<w(Y_0)(\al_{k+1}-\al_k)<\f52 \pi.
\end{align*}
  By Lemma \ref{lem: Im-int}, we have
 \begin{align*}
 0<c_{i,k}\sim\al_k^{-\f76}|\Im(\frak{D}_0(\al_k)|\sim \al_k^{-3}e^{-\al_k w_0(\mathring Y(c_r))}.
 \end{align*} 
 Thus, we get $(\al_k,c_k)\in \mathbb{H}$ with $0<c_{i,k}\sim \al_k^{-3}e^{-\al_k w_0(\mathring Y(c_r))}$. Finally, we  construct the solution $P_k$ to \eqref{eq:ray-P1}. 
\end{proof}

\subsection{Proof of Proposition \ref{rem: Non-empty condition}}
\begin{proof}
We first compute the derivatives of $\eta$ at the turning point $Y_0$. By the Langer transform \eqref{def: Langer}, a direct calculation gives 
\begin{align*}
\pa_Y\eta(Y_0)=1, \quad \pa_Y^2\eta(Y_0)=\f15\f{\pa_Y^2\tilde F(Y_0)}{\pa_Y\tilde F(Y_0)},\quad \pa_Y^3\eta(Y_0)=\f17\f{\pa_Y^3\tilde F(Y_0)}{\pa_Y\tilde F(Y_0)}-\f{12}{175}\left(\f{\pa_Y^2\tilde F(Y_0)}{\pa_Y\tilde F(Y_0)}\right)^2.
\end{align*}
We recall 
\begin{align*}
\tilde Q_1(Y; c_r, M_a)=\f{\pa_Y^2 \bM_r}{\bM_r}-\f{2(\pa_Y\bM_r)^2}{\bM_r^2}-\f34\f{(\pa_Y^2\eta)^2}{(\pa_Y\eta)^2}-\f{\pa_Y^3 \eta}{2\pa_Y\eta},\quad Y\in[0,Y_0].
\end{align*}
Notice that $\tilde F(Y)= F_r(Y)$ for $Y\in[0, Y_0]$ and $F_r(Y)=1-\bM_r^2(Y)$. 
Taking $(Z, Y_0)=(0,0)$ with $c_r=T_0^\f12(0) M_a^{-1}$, we have $\bM_r(0)=-1, \pa_Y\eta(0)=1$ and
\begin{align*}
\tilde Q_1(0; T_0^\f12(0) M_a^{-1}, M_a)=&\left(\f{11}{14}-\f{3}{50}-\f{12}{175}\right)\f{\pa_Y^2 \bM_r(0)}{\bM_r(0)}-\left(\f{203}{100}+\f{6}{175}\right)\left(\f{\pa_Y\bM_r(0)}{\bM_r(0)}\right)^2\\
&-\left(\f{3}{100}+\f{6}{175}\right)\left(\f{\pa_Y^2\bM_r(0)}{\pa_Y\bM_r(0)}\right)^2-\f{1}{14}\f{\pa_Y^3\bM_r(0)}{\pa_Y\bM_r(0) }:=G(M_a).
\end{align*}
From the above equality, we know that $\tilde Q_1(0;T_0^\f12(0) M_a^{-1},  M_a)=G(M_a)$ is a nonzero  function of $M_a$ and is analytically related to $M_a$. Thus, the roots of $G(M_a)$ are isolated.\smallskip

\underline{$\bullet$ $U_B$ is  an analytic function.} Since the roots of $G(M_a)$ are isolated, there exists $M_0>0$  such that for  any $M_a\geq M_0$, it holds that $\tilde Q_1(0; T_0^\f12(0) M_a^{-1},  M_a)\neq 0$, where $M_0$ is the maximum of roots of $G(M_a)$ in the finite interval. Since $\tilde Q_1(Z; c_r, M_a)$ is continuos  in variable $(Z, c_r)$, there exists $\e_0>0$ such that $\tilde Q_1(Z, c_r, M_a)\neq 0$ for $0\leq Z+|c_r-T_0^\f12(0) M_a^{-1}|\leq 4\e_0$. Especially, we have $\tilde Q_1(Z; T_0^\f12(0) M_a^{-1}+\e_0, M_a)\neq 0$ for $Z\in[0, \e_0]$. Thus, we get $J(T_0^\f12(0) M_a^{-1}+\e_0)\neq 0$, which means that $J(c_r)$ is a nonzero function on $[T_0^\f12(0) M_a^{-1}, 1+M_a^{-1}]$. As  $U_B$ is analytic, we deduce  that $J(c_r)$ is also analytic, which implies that the roots of $J(c_r)$ are isolated and the number of roots is finite. As a result, there exist $C_1(M_a),\cdots, C_N(M_a)$ with $1-M_a^{-1}<C_1<\cdots<C_N<1$, such that 
\begin{align*}
(C_1, C_2)\cup (C_3, C_4)\cup \cdots (C_{N-1}, C_N)\subset \mathcal{H}(U_B, M_a)\cap(1-M_a^{-1},1).
\end{align*}
Thus, the statement in (1) is true.

\underline{$\bullet$ $U_B$ is the Blasius profile.} We claim that for any $M_a\geq 3$, it holds that $\tilde Q_1(0;T_0^\f12(0) M_a^{-1}, M_a)\neq 0$. For the Blasius profile, we have $U_B(0)=\pa_Y^2 U_B(0)=\pa_Y^3U_B(0)=0$. A direct calculation gives 
\begin{align*}
&\pa_Y T_0(0)=\pa_Y^3T_0(0)=0,\quad \pa_Y^2T_0=-(\gamma-1)M_a^2(\pa_Y U_B(0))^2,\\
&\bM(0)=-1, \quad \pa_Y \bM(0)=M_aT_0(0)^{-\f12}\pa_Y U_B(0),\quad \pa_Y^2 \bM(0)=\f12 T_0(0)^{-1}\pa_Y^2 T_0(0),\\
&\pa_Y^3 \bM(0)=-\f32M_aT_0(0)^{-\f32}\pa_Y^2 T_0(0)\pa_Y U_B(0).
\end{align*}
Thus, we obtain
\begin{align*}
\tilde Q_1(0,T_0^\f12(0) M_a^{-1}, M_a)=C_0(M_a)T_0(0)^{-1}M_a^2(\pa_Y U_B(0))^2,
\end{align*}
where for any $M_a\geq 3$
\begin{align*}
C_0(M_a)=\f27(\gamma-1)-(\gamma-1)(\f{3}{100}+\f{6}{175})-(\f{3}{100}+\f{6}{175})\f{1}{4 T_0(0)}-(\f{203}{100}+\f{6}{175})<0,
\end{align*}
which gives $\tilde Q_1(0;T_0^\f12(0) M_a^{-1}, M_a)<0$. Thus, the statement in (2) is true. 
\end{proof}

\subsection{Proof of the Theorem \ref{thm: main-1}, \ref{thm: main-2}-\ref{thm: main-3}}
With results in  Theorem \ref{thm: solve dispersion relation} and Proposition \ref{rem: Non-empty condition} at hand, we are in a position to prove main theorems.

\begin{proof}
In Theorem \ref{thm: solve dispersion relation}, we have constructed the multiple unstable modes $\{(\al_k, c_k)\}_{k=1}^{\infty}$ to the pressure system \eqref{eq:ray-P1}. By the relation \eqref{eq: rebuilt1}, we can rebuild $(\varrho_k,U_k,V_k,T_k)(Y)$ through $P_k(Y)$ and 
\begin{align*}
(\varrho_k,U_k,V_k,T_k, P_k)(Y)\in L^\infty(\mathbb R_+)\times W^{1,\infty}(\mathbb R_+)^2\times L^\infty(\mathbb R_+)\times W^{2,\infty}(\mathbb R_+)
\end{align*} 
which satisfies the system \eqref{eq: LCEuler-1}-\eqref{BC: (U,V, rho,T)}. Moreover, we have
\begin{align*}
	(\rho_k, u_k,v_k, \mathcal T_k, p_k)(t,X,Y)=(\varrho_k, U_k,V_k, T_k, P_k)(Y)e^{\mathrm i\alpha_k X}e^{-\mathrm i\alpha_k c_k t},
	\end{align*}
which are solutions to  the linearized Euler system  \eqref{eq: LCEuler} with $(\al_k, c_k)$ satisfying \eqref{bound: c_i}. The results in Theorem \ref{thm: main-2}-\ref{thm: main-3} are direct conclusions of Proposition \ref{rem: Non-empty condition}.
\end{proof}

\appendix

\section{Some estimates related to critical layer}\label{appendix: CL}

For this section, we are devoted to giving some important estimates related to $\frak{D}_0(\al,c)$, which are used to give the main part of the imaginary part through the dispersion relation.
\begin{lemma}\label{lem: G-sim}
Let $(\al, c)\in \mathbb{H}_0$ and  $f=\partial_Y\left(\frac{-H_1[A]}{1-\bar M^2}\right)$ in the formula of $G_f(Y)$ in  \eqref{def: G(Y)}. Then we have
	       \begin{align*}
0< \Re\Big(\int^{\bar Y_2^*}_{\bar Y_1^*} G_f(Y) dY\Big)\sim  \al^{\f{11}6} e^{-\al w_0(Y_c)}.
\end{align*}

\end{lemma}

\begin{proof}
As in Lemma \ref{lem: rayleigh-source term}, we write 
\begin{align*}
f=\partial_Y\left(\frac{-H_1[A]}{1-\bar M^2}\right)=f_1+R_1,\quad G_f(Y)=G_{f_1}(Y)+G_{R_1}(Y),
\end{align*}
where $f_1=-2\al^2\pa_Y\chi_1 A$ and $\|R_1\|_{L^\infty_{w_0}}\leq C\al^{\f56}$ with $\mathrm{supp}(R_1)\subset [Y_1-2\d_0,Y_1-\d_0]\cup[Y_2-\d_0, Y_2+2\d_0]$. By Lemma \ref{lem: est of G} (taking $\bar Y_1=Y_1+\d_0,~\bar Y_2=Y_2-\d_0, ~\d=3\d_0$), we get
\begin{align}\label{est: G_R1}
|G_{R_1}(Y)|\leq C\al^{\f56}e^{-\al w_0(Y_c)}. 
\end{align}

Taking $f=-2\al^2\partial_Y\chi_1 A(Y)$ in \eqref{def: G(Y)}, we have
\begin{align*}
G_{f_1} (Y)
=&-2\al^2\f{F(Y)}{\phi^2(Y)}\int_{Y_c}^{Y}\pa_{Y'} \chi_1\phi(Y')Ai(\kappa \eta(Y'))(\pa_{Y'}\eta(Y'))^{-\f12}dY'.
\end{align*}

For $Y\in \mathrm{supp} (\pa_Y \chi_1)$, it holds that $\eta(\pa_Y\eta)^2\sim 1$, which gives 
\begin{align}\label{sim: Ai}
(\pa_Y\eta(Y))^{-\f12}Ai(\kappa \eta(Y))\sim \f{1}{(\kappa \eta(Y))^\f14(\pa_Y\eta(Y))^{\f12}} e^{-\f23(\kappa \eta(Y))^\f32}\sim\al^{-\f16} e^{-\al w_0(Y)}.
\end{align}
 
Define a real function
\begin{align*}
\tilde G_{f_1}(Y)=\f{-2\al^2F_r(Y)\phi_r^2(Y)}{|\phi(Y)|^4}\int_{Y_c}^{Y}\pa_{Y'} \chi_1\phi_r(Y')Ai(\kappa \eta(Y'))(\pa_{Y'}\eta(Y'))^{-\f12}dY'.
\end{align*}
By Proposition \ref{pro: bound-(phi, phi_0)} and a similar argument in Lemma \ref{lem: est of G}, it is easy to get
\begin{align}\label{est: int-G-tilde G}
\Big|\int^{\bar Y_2^*}_{\bar Y_1^*} (G_{f_1}-\tilde G_{f_1}(Y)) dY\Big|\leq C\al^{\f{17}6}c_i e^{-\al w_0(Y_c)}.
\end{align}
Next, we focus on $\tilde G_{f_1}(Y)$. For $Y\in[\bar Y_1^*, \bar Y_c]$, we notice $\tilde G_{f_1}(Y)>0$, 
and  it follows from \eqref{est: G(Y)-1} and \eqref{est: G(Y)-2} that
\begin{align*}
0<\int^{Y_c}_{\bar Y_1^*}\tilde G_{f_1}(Y)dY\leq C\al^{\f{11}6}e^{-\al w_0(Y_c)}.
\end{align*}
On the other hand,  for $Y\in [Y_1+\f14\d_0,  Y_1+\f34 \d_0]$, it holds that $|\pa_Y\chi_1|\geq c_0 $. Then we deduce 
{\small
\begin{align*}
\tilde G_{f_1}(Y)\geq&C^{-1}\al^{\f{17}6} e^{-2\al |w_c(Y)|}\int_{Y}^{ Y_1+\f34 \d_0}e^{\al |w_c(Y')|}w^{-\al w_0(Y')}dY'
\geq C^{-1}\al^{\f{11}6}e^{-\al w_0(Y_c)}\Big(1-e^{-2\al \int_{Y}^{ Y_1+\f34 \d_0} F_r^\f12 dZ}\Big),
\end{align*}
}
which implies
\begin{align*}
\int^{Y_c}_{\bar Y_1^*}\tilde G_{f_1}(Y) dY\geq C^{-1}\al^{\f{11}6}e^{-\al w_0(Y_c)}.
\end{align*}
Thus, we have 
\begin{align}\label{est: tilde G-1}
0<\int^{Y_c}_{\bar Y_1^*} \tilde G_{f_1}(Y)dY\sim \al^{\f{11}6}e^{-\al w_0(Y_c)}.
\end{align}

By \eqref{est: G(Y)-3} and  \eqref{est: G(Y)-4},  we have  for $Y\in [ \bar Y_2, \bar Y_2^*] $: $\tilde G_{f_1}(Y)<0$ and 
\begin{align*}
 |\tilde G_{f_1}(Y)|\leq C\al^{\f{11}6} \al e^{-2\al w_c(Y)}e^{-\al w_0(Y_c)}\ll \al^{\f{11}6} e^{-\al w_0(Y_c)},
\end{align*}
which gives 
\begin{align}\label{est: tilde G-2}
\int_{Y_c}^{\bar Y_2^*} |\tilde G_{f_1}(Y)| dY\ll \al^{\f{11}6} e^{-\al w_0(Y_c)}.
\end{align}

It follows from  \eqref{est: tilde G-1} and \eqref{est: tilde G-2}  that
\begin{align*}
0<\int_{\bar Y_1^*}^{\bar Y_2^*} \tilde G_{f_1}(Y) dY\sim \al^{\f{11}6}e^{-\al w_0(Y_c)},
\end{align*}
which along with  \eqref{est: G_R1} and \eqref{est: int-G-tilde G} gives
\begin{align*}
0< \Re\Big(\int_{\bar Y_1^*}^{\bar Y_2^*}  G_f(Y) dY\Big)\sim \al^{\f{11}6} e^{-\al w_0(Y_c)}.
\end{align*}
\end{proof}

We introduce a quantity which is related to $\mathcal{K}_B[\varphi_0](Y)$
\begin{align}\label{def: cM}
    	\mathcal M=&-\int_{\bar Y_1^*}^{Y_c} \chi_2  \bM(Y) \phi(Y)\int_{\bar Y_1^*}^Y\f{F(Y')}{\phi^2(Y')}dY'dY\\
	\nonumber
&\qquad+\int_{ Y_c}^{\bar Y_2^*}  \chi_2 \bM(Y) \phi(Y)\int_Y^{\bar Y_2^*}\f{F(Y')}{\phi^2(Y')}dY'dY=I_1+I_2.
    \end{align}
To give the estimates for $\mathcal M$, both for the real part and the imaginary part, we set $F_r(Y)\geq \frac{9}{10}$ for any $Y\in[\bar Y_1^{**},Y_c+3\delta_0]$ and assume that $Y_c-Y_1=Y_2-Y_c=\delta_1>2\d_0$.

\begin{lemma}\label{lem:real-phi(0)}
Let $(\alpha,c)\in \mathbb H_0$ and $\phi$ be the solution to the system \eqref{eq: Rayleigh-homo} constructed in Proposition \ref{pro: bound-(phi, phi_0)}. Then there exist $Y_1^*\in [\bar Y_1^*+3\delta_0, Y_1-3\delta_0]$,  $Y_2^*\in [Y_2+3\delta_0,\bar Y_2^*-3\delta_0]$ and $\tilde Y\in(Y_c,Y_2^*-\delta_0)$ satisfying
 \begin{align}\label{est: bdd Im K_B}
e^{-\alpha w_0(Y_1)}e^{-2\alpha\int_{Y_1^*}^{Y_1-2\delta_0}F_r^\f12(Y)dY}+e^{-\alpha w_0(Y_2^*)}\ll e^{-\alpha w_0(\tilde Y)},
\end{align}
such that 
    \begin{align*}
    	&-C\al^{-1}\leq \mathrm{Re}(\mathcal M)\leq   -C^{-1}\al^{-1}e^{\alpha w_0(Y_c)}e^{-\alpha w_0(\tilde Y)}+\mathcal O(c_i^2|\log c_i|),
    \end{align*}
    or
    \begin{align*}
   &C\al^{-1}\geq \mathrm{Re}(\mathcal M)\geq   C^{-1}\al^{-1}e^{\alpha w_0(Y_c)}e^{-\alpha w_0(\tilde Y)}+\mathcal O(c_i^2|\log c_i|),
   \end{align*}
   where $\mathcal M$ is given in \eqref{eq:range-Y1*}.
       Moreover, we have 
    \begin{align*}
    	|\mathrm{Im}(\mathcal M)|\leq C c_i|\log c_i|.
    \end{align*}
   
\end{lemma}

\begin{proof}
	We recall that $\chi_2$ is a smooth cut-off function such that $\chi_2\equiv1$ for any $Y\in[Y_1^*,Y_2^*]$ and $\chi_2\equiv 0$ for any $Y\in [0,  Y_1^*-\delta_0]\cup[ Y_2^*+\delta_0,+\infty)$. Here $Y_1^*$ and $Y_2^*$ will be determined later. By Lemma \ref{lem: int-phi}, it is easy to obtain $|\mathcal M|\leq C\al^{-1}.$
	We introduce
\begin{align*}
&\tilde I_1=-\int_{\bar Y_1^*}^{Y_c}   \chi_2\bM_r(Z) \phi_r(Z)\int_{\bar Y_1^*}^Z\f{F_r(Y')\phi_r^2(Y')}{|\phi^4(Y')|}dY'dZ< 0,\\
&\tilde I_2=\int_{ Y_c}^{ \bar Y_2^*}     \chi_2\bM_r(Z) \phi_r(Z)\int^{\bar Y_2^*}_Z\f{F_r(Y')\phi_r^2(Y')}{|\phi^4(Y')|}dY'dZ> 0.
\end{align*}
We first choose $\tilde Y_1$ such that 
\begin{align}\label{eq:range-Y1*}
	2\int_{\tilde Y_1}^{Y_1-2\delta_0}F_r^\f12(Y)dY\geq \int_{Y_1}^{Y_c}F_r^{\f12}(Y)dY+\max\Big\{\int_{\tilde Y_1}^{Y_c}F_r^{\f12}(Y)dY,\int_{Y_c}^{2Y_c-\tilde Y_1}F_r^{\f12}(Y)dY\Big\}+4\delta_0,
\end{align}
which is guaranteed by the following: for any $Y_1-\tilde Y_1\geq 11\delta_0+3\d_1$,
\begin{align*}
2\int_{\tilde Y_1}^{Y_1-2\delta_0}F_r^\f12(Y)dY\geq&	\frac{9}{5}(Y_1-\tilde Y_1-2\delta_0)\geq (2Y_c-\tilde Y_1-Y_1)+4\delta_0\\
\geq& \int_{Y_1}^{Y_c}F_r^{\f12}(Y)dY+\max\Big\{\int_{\tilde Y_1}^{Y_c}F_r^{\f12}(Y)dY,\int_{Y_c}^{2Y_c-\tilde Y_1}F_r^{\f12}(Y)dY\Big\}+4\delta_0.
\end{align*}
From \eqref{eq:range-Y1*}, we can obtain that for any $Y\in[\bar Y_1^*,\tilde Y_1+2\delta_0]$,
\begin{align}
	e^{-\alpha w_0(Y_1)}e^{-2\alpha\int_Y^{Y_1-2\delta_0}F_r^\f12(Z)dZ}\leq e^{-\alpha w_0(Y_c)}\min\Big\{e^{-\alpha\int_Y^{Y_c}F_r^\f12(Z)dZ},e^{-\alpha\int_{Y_c}^{2Y_c-Y}F_r^\f12(Z)dZ}\Big\}e^{-4\delta_0\alpha.}
\end{align}

We introduce the following quantity $\tilde{\mathcal M}\in\mathbb R$,
{\small
\begin{align*}
	\tilde{\mathcal M}(Y;\tilde Y_1)=&\underbrace{-\int_{\tilde Y_1}^{Y_c}   \bM_r(Y) \phi_r(Y)\int_{\bar Y_1^*}^Y\f{F_r(Y')\phi_r^2(Y')}{|\phi^4(Y')|}dY'dY}_{<0}+\underbrace{\int_{ Y_c}^{Y} \bM_r(Z) \phi_r(Z)\int^{\bar Y_2^*}_Y\f{F_r(Y')\phi_r^2(Y')}{|\phi^4(Y')|}dY'dZ}_{>0} .
\end{align*}
}
We first notice that $\tilde {\mathcal M}(Y_c)<0$. Here we choose $Y_2^*=2Y_c-\tilde{Y}_1+3\delta_0$. Then we split our proof into the following two cases.\smallskip

\textbf{Case 1. $\tilde{\mathcal M}(Y_2^*)\leq 0$.} In this case, we take $Y_1^*=\tilde Y_1-\delta_0$. Then we have
\begin{align}\label{eq:I1+I2-1}
	\begin{split}
		\tilde I_1+\tilde I_2
	\leq&-\int_{ Y_1^*}^{\tilde Y_1}\bM_r(Y) \phi_r(Y)\int_{\bar Y_1^*}^Y\f{F_r(Y')\phi_r^2(Y')}{|\phi^4(Y')|}dY'dY\\
	&-\int_{ Y_1^*-\delta_0}^{Y_1^*} \chi_2  \bM_r(Y) \phi_r(Y)\int_{\bar Y_1^*}^Y\f{F_r(Y')\phi_r^2(Y')}{|\phi^4(Y')|}dY'dY   \\
		&+\int_{ Y_2^*}^{Y_2^*+\delta_0}\chi_2(Y) \bM_r(Y) \phi_r(Y)\int^{\bar Y_2^*}_Y\f{F_r(Y')\phi_r^2(Y')}{|\phi^4(Y')|}dY'dY.
	\end{split}
\end{align}
On the other hand,  by Proposition \ref{pro: bound-(phi, phi_0)}, we have
\begin{align}\label{eq:I_1-1}
	\begin{split}
		0<&\int_{ Y_1^*}^{\tilde Y_1}\bM_r(Y) \phi_r(Y)\int_{\bar Y_1^*}^Y\f{F_r(Y')\phi_r^2(Y')}{|\phi^4(Y')|}dY'dY\\
	\sim&\alpha\int_{Y_1^*}^{\tilde Y_1}e^{\alpha |w_c(Y)|}\int_{\bar Y_1^*}^Ye^{-2\alpha |w_c(Y')|}dY'
	\sim \al^{-1}e^{-\alpha \int_{\tilde Y_1}^{Y_c} F_r^\f12 dY}.
	\end{split}
\end{align}
By a similar argument as above, we can obtain 
{\small
\begin{align}\label{eq:I_1-2}
	\Big|\int_{ Y_1^*-\delta_0}^{Y_1^*} \chi_2(Y)  \bM_r(Y) \phi_r(Y)\int_{\bar Y_1^*}^Y\f{F_r(Y')\phi_r^2(Y')}{|\phi^4(Y')|}dY'dY\Big|\lesssim\alpha^{-1}e^{-\alpha\int_{Y_1^*}^{Y_c}F_r^\f12dY}\ll\al^{-1}e^{-\alpha \int_{\tilde Y_1}^{Y_c} F_r^\f12 dY},
\end{align}
}
and 
{\small
\begin{align}\label{eq:I_1-3}
	\Big|\int_{ Y_2^*}^{Y_2^*+\delta_0}\chi_2(Y) \bM_r(Y) \phi_r(Y)\int^{\bar Y_2^*}_Y\f{F_r(Y')\phi_r^2(Y')}{|\phi^4(Y')|}dY'dY\Big|\lesssim\alpha^{-1}e^{-\alpha\int_{Y_c}^{Y_2^*}F_r^\f12}\ll \al^{-1}e^{-\alpha \int_{\tilde Y_1}^{Y_c} F_r^\f12 dY}.
\end{align}
}
Then by \eqref{eq:I1+I2-1}-\eqref{eq:I_1-3}, we arrive at 
\begin{align}
	\tilde I_1+\tilde I_2\leq -C\al^{-1}e^{-\alpha \int_{\tilde Y_1}^{Y_c} F_r^\f12 dY}<0,
\end{align}
which implies 
\begin{align*}
	e^{-\alpha w_0(Y_c)}|\tilde I_1+\tilde I_2|\geq& \al^{-1}e^{-\alpha w_0(Y_c)}e^{-\alpha \int_{\tilde Y_1}^{Y_c} F_r^\f12 dY}\geq \alpha^{-1}e^{-\alpha w_0(Y_1)}e^{-2\alpha\int_{\tilde Y_1}^{Y_1-2\delta_0}F_r^\f12(Y) dY}e^{4\alpha\delta_0}\\
	\geq&\alpha^{-1}e^{-\alpha w_0(Y_1)}e^{-2\alpha\int_{Y_1^*}^{Y_1-2\delta_0}F_r^\f12(Y)dY}e^{4\alpha\delta_0}\gg e^{-\alpha w_0(Y_1)}e^{-2\alpha\int_{Y_1^*}^{Y_1-2\delta_0}F_r^\f12(Y)dY},
\end{align*}
and 
\begin{align*}
	e^{-\alpha w_0(Y_c)}|\tilde I_1+\tilde I_2|\geq& \al^{-1}e^{-\alpha w_0(Y_c)}e^{-\alpha \int_{\tilde Y_1}^{Y_c} F_r^\f12 dY}\geq \al^{-1}e^{-\alpha w_0(Y_2^*)}e^{4\delta_0\alpha}\gg e^{-\alpha w_0(Y_2^*)}.
\end{align*}
Therefore, we conclude that there exits $\tilde Y\in (Y_c, Y_2^*-\delta_0)$ such that 
\begin{align}
	0>e^{-\alpha w_0(Y_c)}(\tilde I_1+\tilde I_2)\geq -\alpha^{-1}e^{-\alpha w_0(\tilde Y)},
	\end{align}
	with 
	\begin{align*}
		e^{-\alpha w_0(Y_1)}e^{-2\alpha\int_{Y_1^*}^{Y_1-2\delta_0}F_r^\f12}+e^{-\alpha w_0(Y_2^*)}\ll e^{-\alpha w_0(\tilde Y)}.
	\end{align*}

\textbf{Case 2. $\tilde {\mathcal M}(Y_2^*)>0$.} 
In this case, we take $Y_1^*=\tilde Y_1+2\delta_0$ and we get 
\begin{align}\label{eq:I_1+I_2-2}
	\begin{split}
		\tilde I_1+\tilde I_2=& \tilde{\mathcal M}(Y_2^*)+\int_{\tilde{Y}_1}^{Y_1^*-\delta_0}\bM_r(Y) \phi_r(Y)\int_{\bar Y_1^*}^Y\f{F_r(Y')\phi_r^2(Y')}{|\phi^4(Y')|}dY'dY\\
		&+\int_{Y_1^*-\delta_0}^{Y_1^*}(1-\chi_2)\bM_r(Y) \phi_r(Y)\int_{\bar Y_1^*}^Y\f{F_r(Y')\phi_r^2(Y')}{|\phi^4(Y')|}dY'dY\\
		&+\int_{ Y_2^*}^{Y_2^*+\delta_0}\chi_2(Y) \bM_r(Y) \phi_r(Y)\int^{\bar Y_2^*}_Y\f{F_r(Y')\phi_r^2(Y')}{|\phi^4(Y')|}dY'dY\\
		\geq&\int_{\tilde{Y}_1}^{Y_1^*-\delta_0}\bM_r(Y) \phi_r(Y)\int_{\bar Y_1^*}^Y\f{F_r(Y')\phi_r^2(Y')}{|\phi^4(Y')|}dY'dY.
	\end{split}
\end{align}

On the other hand, by a similar argument as \eqref{eq:I_1-1}, we can deduce that 
\begin{align*}
	0<\int_{\tilde{Y}_1}^{Y_1^*-\delta_0} \bM_r(Y) \phi_r(Y)\int_{\bar Y_1^*}^Y\f{F_r(Y')\phi_r^2(Y')}{|\phi^4(Y')|}dY'dY\sim \al^{-1}e^{-\alpha \int_{\tilde{Y}_1+\delta_0}^{Y_c} F_r^\f12 dY},
\end{align*}
which along with \eqref{eq:I_1+I_2-2} implies
\begin{align*}
	\tilde I_1+\tilde I_2\geq  C\al^{-1}e^{-\alpha \int_{\tilde{Y}_1+\delta_0}^{Y_c} F_r^\f12(Y) dY}>0.
\end{align*}
Moreover, we have 
\begin{align*}
	e^{-\alpha w_0(Y_c)}|\tilde I_1+\tilde I_2|\geq& \al^{-1}e^{-\alpha w_0(Y_c)}e^{-\alpha \int_{\tilde Y_1+\delta_0}^{Y_c} F_r^\f12 dY}\geq \al^{-1}e^{-\alpha w_0(Y_c)}e^{-\alpha \int_{Y_1^*}^{Y_c} F_r^\f12 dY}e^{-\al\delta_0}\\
	\geq&\alpha^{-1}e^{-\alpha w_0(Y_1)}e^{-2\alpha\int_{Y_1^*}^{Y_1-2\delta_0}F_r^\f12(Y)dY}e^{\alpha\delta_0}\gg e^{-\alpha w_0(Y_1)}e^{-2\alpha\int_{Y_1^*}^{Y_1-2\delta_0}F_r^\f12(Y)dY},
\end{align*}
and 
\begin{align*}
	e^{-\alpha w_0(Y_c)}|\tilde I_1+\tilde I_2|\geq& \al^{-1}e^{-\alpha w_0(Y_c)}e^{-\alpha \int^{2Y_c-\tilde Y_1}_{Y_c} F_r^\f12 dY}\geq \al^{-1}e^{-\alpha w_0(Y_2^*)}e^{2\delta_0\alpha}\gg e^{-\alpha w_0(Y_2^*)}.
\end{align*}
Thus, we infer that  there exits $\tilde Y\in (Y_c, Y_2^*-\delta_0)$ such that 
\begin{align}
	0<e^{-\alpha w_0(Y_c)}(\tilde I_1+\tilde I_2)\leq \alpha^{-1}e^{-\alpha w_0(\tilde Y)},
	\end{align}
	with 
	\begin{align*}
		e^{-\alpha w_0(Y_1)}e^{-2\alpha\int_{Y_1^*}^{Y_1-2\delta_0}F_r^\f12}+e^{-\alpha w_0(Y_2^*)}\ll e^{-\alpha w_0(\tilde Y)}.
	\end{align*}

Now we turn to show the difference $I_1-\tilde I_1$ and $I_2-\tilde I_2$. We have
\begin{align*}
	|I_1-\tilde I_1|+c_i^{-1}|\mathrm{Re}I_1-\tilde I_1|\leq Cc_i\int_{Y_1^*-\delta_0}^{Y_c}|\phi|^{-1}(Y)dY\leq Cc_i|\log c_i|,
\end{align*}
Similarly, we  have 
\begin{align*}
	c_i|I_2-\tilde I_2|+|\Re I_2-\tilde I_2|\leq Cc_i^2|\log c_i|.
\end{align*}

The proof is completed.
\end{proof}

With results in Lemma \ref{lem: G-sim}-\ref{lem:real-phi(0)} at hand,
we are in a position to prove Lemma \ref{lem: Im-int}

\begin{proof}
	According to the definition of $\chi_2$, we first have 
	\begin{align*}
		\mathrm{supp}\left(\pa_Y\Big(\f{\pa_Y \chi_2 \bM^2}{1-\bM^2}\Big)\right)\subseteq\mathrm{supp}(\partial_Y\chi_2)=[Y_1^*-\delta_0,Y_1^*]\cup[Y_2^*,Y_2^*+\delta_0].
	\end{align*}
	Therefore, for any $Y\geq Y_2^*+\delta_0$, we have  $\mathcal K_B[\varphi]\equiv0$; for any $Y\in[Y_1^*-\delta_0,Y_2^*+\delta_0]$, we have 
	\begin{align*}
		|\mathcal K_B[\varphi_0](Y)|\leq C\int_Y^{\bar Y_2^*}|\varphi_0(Z)|dZ\leq C\alpha^{\f56}e^{-\alpha w_0(Y)},
	\end{align*}
and 	for any $Y\leq Y_1^*-\delta_0$, we have $\mathcal K_B[\varphi_0](Y)\equiv\mathcal K_B[\varphi_0](Y_1^*-\delta_0)$. Along with Proposition \ref{pro: rayleigh-varphi},  we obtain
\begin{align*}
|\mathcal K_B[\varphi](Y)|\leq&C\int_{Y_1^*-\d_0}^{ Y_1^*}|\varphi_0(Z)|dZ+C\int_{Y_2^*}^{Y_2^*+\d_0}|\varphi_0(Z)|dZ\\
\leq&Ce^{-\al w_0(Y_1^*-\d_0)}\|\varphi_0\|_{L^\infty_{w_0}([Y_1^*-\d_0, Y_1^*])}+Ce^{-\al w_0(Y_2^*)}\|\varphi_0\|_{L^\infty_{w_0}}\\
\leq&C\al^{-\f{1}{6}}\Big(e^{-\al w_0(Y_1^*-\d_0)}e^{-2\al \int_{Y_1^*}^{Y_1-2\d_0}F_r^\f12(Y') dY'}+\al e^{-\al w_0(Y_2^*)}\Big),
\end{align*}
by using Proposition \ref{pro: iteration-varphi}.
	
	Next we shall show the lower bound of $\mathrm{Im}(\mathcal K_B[\varphi_0](Y_1^*-\delta_0))$. For this purpose, we first notice that 
	\begin{align*}
\mathcal{L}_{\bar{M}}[\varphi_0]=\al^2 \bM \varphi_0+\partial_Y\Big(\frac{-H_1[A]}{1-\bar M^2}\Big)=\al^2 \bM \varphi_0-2\al^2\pa_Y \chi_1A+R_1,
\end{align*}
where $\|R_1\|_{L^\infty_{w_0}}+c_i^{-1}\|\Im R_1\|_{L^\infty_{w_0}}\leq C\al^{\f56}$. 
We get by integration by parts that
\begin{align*}
&\int_{Y_1^*-\delta_0}^{+\infty}\pa_Y\Big(\f{\pa_Y \chi_2 \bM^2}{1-\bM^2}\Big)\bM^{-1}\varphi_0(Y) dY=\int_{ Y_1^*-\delta_0}^{Y_2^*+\d_0} \chi_2 \Big(\f{\bM^2}{1-\bM^2}\Big(\f{\varphi_0(Y)}{\bM}\Big)'\Big)' dY\\
=&\int_{ Y_1^*-\delta_0}^{Y_2^*+\d_0}\chi_2 (\al^2 \bM \varphi_0(Y)-2\al^2\pa_Y \chi_1A(Y)+R_1(Y)) dY.
\end{align*}
On the other hand, it is easy to find that 
\begin{align}\label{est: Im-int-1}
	&\Big|\mathrm{Im}\Big(\int_{ Y_1^*-\delta_0}^{Y_2^*+\d_0}\al^2\partial_Y\chi_1A(Y)dY\Big)\Big|+\al \Big|\mathrm{Im}\Big(\int_{ Y_1^*-\delta_0}^{Y_2^*+\d_0}R_1(Y)dY\Big)\Big|\leq C\alpha^{\f{5}6}c_ie^{-\alpha w_0(Y_1^*-\delta_0)}.
\end{align}
Hence, we only need to pay attention to 
\begin{align*}
	\mathrm{Im}\Big(\int_{ Y_1^*-\delta_0}^{Y_2^*+\d_0}\chi_2 \al^2 \bM \varphi_0(Y)dY\Big).
\end{align*}

We denote $\varphi_0=\varphi_{01}+\varphi_{02}$, where 
\begin{align*}
&\varphi_{01}(Y)=
\left\{
\begin{aligned}
&\phi(Y)\int_{\bar Y_1^*}^Y G_f(Y') dY',\quad Y\in[\bar Y_1^*, Y_c],\\
&-\phi(Y)\int_Y^{\bar Y_2^*} G_f(Y') dY',\quad Y\in [Y_c, \bar Y_2^*],
\end{aligned}
\right.\\
&\varphi_{02}(Y)=
\left\{
\begin{aligned}
&\mu_f(c)\phi(Y)\int_{\bar Y_1^*}^Y\f{F(Y')}{\phi^2(Y')}dY',\quad Y\in[\bar Y_1^*, Y_c],\\
&-\mu_f(c) \phi(Y)\int_Y^{\bar Y_2^*}\f{F(Y')}{\phi^2(Y')}dY',\quad Y\in [Y_c, \bar Y_2^*],
\end{aligned}
\right.
\end{align*}
and $f=\partial_Y\left(\frac{-H_1[A]}{1-\bar M^2}\right)$ with $\|f\|_{L^\infty_{w_0}}\leq C\al^{\f{11}{6}}$. Here $\mu_f(c)$ is given in \eqref{def: mu(c)}.
By Lemma \ref{lem: est of G} and \eqref{sim: Ai}, we have for $Y\in [\bar Y_1^*,\bar Y_2^*]$,
\begin{align*}
|\varphi_{01}|\leq C\al^{\f56} e^{-\al w_0(Y)},\quad|\mathrm{Im}\varphi_{01}|\leq C\al^{\f{11}6}c_i e^{-\al w_0(Y)},
\end{align*}
which gives
\begin{align}\label{est: Im-int-2}
\Big|\mathrm{Im}\Big(\int_{\bar Y_1^*}^{Y_2^*+\d_0} \chi_2 \al^2 \bM \varphi_{01}(Y)dY\Big)\Big|\leq&C\al^{\f{17}6} c_ie^{-\al w_0(Y_1^*-\delta_0)}.
\end{align}
We focus on  $\varphi_{02}$. We write
\begin{align*}
\int_{ Y_1^*-\delta_0}^{Y_2^*+\d_0}\chi_2 \al^2 \bM \varphi_{02}(Y)dY=-\alpha^2\mu_f(c)\mathcal M,
\end{align*}
where $\mathcal M$ is defined in \eqref{def: cM}.  By Lemma \ref{lem: est of G}, Lemma \ref{lem: int-phi^(-2)} and \ref{lem: G-sim}, we obtain 
\begin{align*}
	|\Re (\mu_f(c))|\leq C\alpha^{\f56}e^{-\alpha w_0(Y_c)},\quad -\Im(\mu_f(c))\sim\alpha^{-\f{1}{6}}\partial_Y^2\bar M(Y_c)e^{-\alpha w_0(Y_c)},
\end{align*}
from which and Lemma \ref{lem:real-phi(0)}, we infer  
\begin{align*}
	\mathrm{Im}\Big(\int_{ Y_1^*-\delta_0}^{Y_2^*+\d_0}\chi_2 \al^2 \bM \varphi_{02}(Y)dY\Big)\sim\alpha^{\f{11}6}\partial_Y^2\bar M(Y_c)e^{-\alpha w_0(Y_c)}\mathrm{Re}(\mathcal M).
\end{align*}
This along with Lemma \ref{lem:real-phi(0)}  shows 
\begin{align*}
	C\al^{\f56}e^{-\alpha w_0(Y_c)}+C c_i^2|\log c_i|\geq &\Big|\mathrm{Im}\Big(\int_{ Y_1^*-\delta_0}^{Y_2^*+\d_0}\chi_2 \al^2 \bM \varphi_{02}(Y)dY\Big)\Big|\geq C^{-1}\alpha^{\f56}e^{-\alpha w_0(\tilde Y)}-C^{-1} c_i^2|\log c_i|,
\end{align*}
and 
\begin{align*}
	\mathrm{Sign}\Big(\mathrm{Im}\Big(\int_{ Y_1^*-\delta_0}^{Y_2^*+\d_0}\chi_2 \al^2 \bM \varphi_{02}(Y)dY\Big)\Big)=\mathrm{Sign}(\partial_Y^2\bar M(Y_c))\mathrm{Sign}\mathcal M.
\end{align*}

Therefore, we combine \eqref{est: Im-int-1}-\eqref{est: Im-int-2}	and use the assumption $c_i\ll\al^n e^{-\al w_0(\tilde Y)}$ for some $n>0$ to conclude  that
\begin{align*}
C\al^{\f56}e^{-\alpha w_0(Y_c)}\geq|\mathrm{Im}(\mathcal K_B[\varphi](Y_1^*-\d_0)|
\geq C^{-1}\alpha^{\f56}e^{-\alpha w_0(\tilde Y)}
\end{align*}
with 
\begin{align*}
\mathrm{Sign}(\mathrm{Im}(\mathcal K_B[\varphi]))=\mathrm{Sign}(\partial_Y^2\bar M(Y_c))\mathrm{Sign}\mathcal M.
\end{align*}
Then we choose $ \mathring{Y}\in[Y_c, \tilde Y]$ such that
\begin{align*}
|\mathrm{Im}(\mathcal K_B[\varphi](Y_1^*-\d_0)|\sim \al^{\f56}e^{-\al w_0(\mathring{Y})}.
\end{align*}

The proof is completed.	
\end{proof}

\section*{Acknowledgments}
N. Masmoudi is supported by NSF Grant DMS-1716466 and by Tamkeen under the NYU Abu Dhabi Research Institute grant of the center SITE. Y. Wang is partially supported by NSF of China under Grant 12101431. D. Wu is partially supported by NSF of China under Grant 12101245. Z. Zhang is partially supported by  NSF of China  under Grant 12288101.

\end{document}